\documentclass[10pt,oneside, draft]{amsart}
\usepackage[a4paper, textwidth=16cm]{geometry}
\usepackage[T1]{fontenc}
\usepackage[english]{babel}
\usepackage{xcolor} 
\usepackage{tikz} 
\usepackage[draft=false]{hyperref} 
\usepackage{mathtools}
\usepackage{amsmath,amsthm,amssymb}
\usepackage[all,cmtip]{xy}

\usepackage[capitalise]{cleveref}  
\newcommand{\clevertheorem}[3]{%
	\newtheorem{#1}[thm]{#2}
	\crefname{#1}{#2}{#3}
}
\numberwithin{equation}{section} 
\numberwithin{figure}{section} 

\theoremstyle{plain} 
\newtheorem{thm}{Theorem}[section]
\crefname{thm}{Theorem}{Theorems}
\newtheorem*{thm*}{Theorem}
\clevertheorem{proposition}{Proposition}{Propositions}
\clevertheorem{prop}{Proposition}{Propositions}
\newtheorem*{prop*}{Proposition}
\clevertheorem{theorem}{Theorem}{Theorems}
\clevertheorem{lemma}{Lemma}{Lemmas}

\clevertheorem{corollary}{Corollary}{Corollaries}
\clevertheorem{cor}{Corollary}{Corollaries}
\clevertheorem{conj}{Conjecture}{Conjectures}
\clevertheorem{questn}{Question}{Questions}

\theoremstyle{definition} 
\clevertheorem{definition}{Definition}{Definitions}
\clevertheorem{assumption}{Assumption}{Assumptions}
\clevertheorem{defn}{Definition}{Definitions}
\clevertheorem{notn}{Notation}{Notations}
\clevertheorem{conv}{Convention}{Conventions}

\clevertheorem{remark}{Remark}{Remarks}
\clevertheorem{rmk}{Remark}{Remarks}
\clevertheorem{warn}{Warning}{Warnings}
\clevertheorem{example}{Example}{Examples}
\clevertheorem{summ}{Summary}{Summaries}

\usepackage{amsmath,amsthm,amssymb,amsfonts,extarrows}
\usepackage{enumerate,amscd,amsxtra,MnSymbol}
\usepackage{mathrsfs}
\usepackage{color}
\usepackage[cmtip,all]{xy}
\usepackage{eucal}
\usepackage{bbold}
\usepackage{upgreek}
\usepackage{paralist}
\usepackage{tikz-cd}
\usepackage{dsfont}
\usepackage{bbm}
\usepackage[bbgreekl]{mathbbol}

\DeclareMathSymbol\bbDelta \mathord{bbold}{"01}
\DeclareMathSymbol\bDelta \mathord{bbold}{"01}

\newtheorem{remark*}{Remark}

\newtheorem{construction}[thm]{Construction}

\newtheorem{notation}[thm]{Notation}

\newcommand{\bE}{{\mathbb E}}

\renewcommand{\P}{{\mathbb P}}

\newcommand{\mA}{{\mathcal A}}
\newcommand{\mB}{{\mathcal B}}
\newcommand{\mC}{{\mathcal C}}
\newcommand{\mD}{{\mathcal D}}
\newcommand{\mE}{{\mathcal E}}
\newcommand{\mF}{{\mathcal F}}

\newcommand{\mJ}{{\mathcal J}}
\newcommand{\mK}{{\mathcal K}}

\newcommand{\mM}{{\mathcal M}}
\newcommand{\mN}{{\mathcal N}}
\newcommand{\mO}{{\mathcal O}}
\newcommand{\mP}{{\mathcal P}}
\newcommand{\mQ}{{\mathcal Q}}

\newcommand{\mS}{{\mathcal S}}
\newcommand{\mT}{{\mathcal T}}
\newcommand{\mU}{{\mathcal U}}
\newcommand{\mV}{{\mathcal V}}
\newcommand{\mW}{{\mathcal W}}
\newcommand{\mX}{{\mathcal X}}
\newcommand{\mY}{{\mathcal Y}}
\newcommand{\mZ}{{\mathcal Z}}

\newcommand{\A}{{\mathrm A}}
\newcommand{\B}{{\mathrm B}}
\newcommand{\C}{{\mathrm C}}
\newcommand{\D}{{\mathrm D}}

\newcommand{\F}{{\mathrm F}}
\newcommand{\G}{{\mathrm G}}
\newcommand{\rH}{{\mathrm H}}

\newcommand{\K}{{\mathrm K}}
\renewcommand{\L}{{\mathrm L}}

\newcommand{\N}{{\mathrm N}}
\renewcommand{\P}{{\mathrm P}}
\newcommand{\Q}{{\mathrm Q}}
\newcommand{\R}{{\mathrm R}}
\newcommand{\rS}{{\mathrm S}}
\newcommand{\T}{{\mathrm T}}
\newcommand{\U}{{\mathrm U}}
\newcommand{\V}{{\mathrm V}}
\newcommand{\W}{{\mathrm W}}
\newcommand{\X}{{\mathrm X}}
\newcommand{\Y}{{\mathrm Y}}
\newcommand{\Z}{{\mathrm Z}}

\newcommand{\rc}{{c}}

\newcommand{\y}{\mathrm{y}}

\newcommand{\bj}{\mathrm{j}}
\newcommand{\bi}{\mathrm{i}}
\newcommand{\m}{\mathrm{m}}
\newcommand{\bk}{\mathrm{k}}

\newcommand{\q}{\mathrm{q}}
\newcommand{\g}{\mathrm{g}}
\newcommand{\n}{\mathrm{n}}

\newcommand{\op}{\mathrm{op}}
\newcommand{\s}{s}
\newcommand{\colim}{\mathrm{colim}}
\newcommand{\Mod}{{\mathrm{Mod}}}
\newcommand{\LMod}{{\mathrm{LMod}}}
\newcommand{\RMod}{{\mathrm{RMod}}}
\newcommand{\BMod}{{\mathrm{BMod}}}

\newcommand{\Ass}{  {\mathrm {   Ass  } }   }

\newcommand{\rev}{{\mathrm{rev}}}

\newcommand{\Env}{{\mathrm{Env}}}  
\newcommand{\ot}{\otimes}

\newcommand{\id}{\mathrm{id}}
\newcommand{\Cat}{\mathrm{Cat}}

\newcommand{\Set}{\mathrm{Set}}

\newcommand{\Alg}{\mathrm{Alg}}

\newcommand{\Comm}{\mathrm{Comm}}

\newcommand{\rep}{\mathrm{rep}}
\newcommand{\lan}{\mathrm{lan}}
\newcommand{\Pseu}{\mathrm{Pseu}}

\newcommand{\Mon}{\mathrm{Mon}}
\newcommand{\Fun}{\mathrm{Fun}}
\newcommand{\Op}{{\mathrm{Op}}}

\newcommand{\tu}{{\mathbb 1}}

\newcommand{\Mul}{{\mathrm{Mul}}}
\newcommand{\ev}{{\mathrm{ev}}}
\newcommand{\Ind}{{\mathrm{Ind}}}

\newcommand{\Enr}{{\mathrm{Enr}}} 

\newcommand{\LinFun}{{\mathrm{LinFun}}}

\newcommand{\Mor}{{\mathrm{Mor}}} 

\newcommand{\map}{{\mathrm{map}}}

\newcommand{\f}{{\mathrm{f}}}

\newcommand{\Ho}{{\mathrm{Ho}}}

\newcommand{\triv}{{\mathrm{triv}}}  
  
\newcommand{\LM}{{\mathrm{LM}}}  
  
\newcommand{\BM}{{\mathrm{BM}}}  

\newcommand{\mi}{{\mathrm{min}}}  
\newcommand{\ma}{{\mathrm{max}}}  
\newcommand{\Min}{{\mathrm{Min}}}  
\newcommand{\Max}{{\mathrm{Max}}}  
\newcommand{\Tw}{{\mathrm{Tw}}} 
\newcommand{\Act}{{\mathrm{Act}}} 
\newcommand{\cc}{{\mathrm{cc}}} 

\usepackage{pdfpages}

\title{On bienriched $\infty$-categories}

\author{Hadrian Heine, \\ University of Oslo, Norway, \\ hadriah@math.uio.no}

\begin{document}
	
\maketitle

\begin{abstract}
We extend Lurie's definition of enriched $\infty$-categories to the notions of left enriched, right enriched and bienriched $\infty$-categories, which generalize the notions of closed left tensored, right tensored and bitensored $\infty$-categories and share many desirable features with them.
We use bienriched $\infty$-categories to endow the $\infty$-category of enriched functors
with enrichment that generalizes both the internal hom of the tensor product of enriched $\infty$-categories when the latter exists, and the free cocompletion under colimits and tensors.
As an application we construct enriched Kan-extensions from operadic Kan-extensions, compute the monad for enriched functors, prove an end formula for morphism objects of enriched $\infty$-categories of enriched functors and a coend formula for the relative tensor product of enriched profunctors and construct transfer of enrichment from scalar extension of presentably bitensored $\infty$-categories.
In particular, we develop an independent theory of enriched $\infty$-categories for Lurie's model of enriched $\infty$-categories.

\end{abstract}

\tableofcontents


\section{Introduction}

\vspace{2mm}
Gepner-Haugseng \cite{GEPNER2015575} develop a theory of enriched $\infty$-categories
that both at the same time extends the theory of $\infty$-categories and the theory of non-commutative higher algebra: the authors prove that $\infty$-categories enriched in spaces are equivalent to Segal spaces \cite[Theorem 4.4.6.]{GEPNER2015575}, a model for $\infty$-categories, and define enriched $\infty$-categories as a many object version of $A_\infty$-algebras \cite[Definition 2.4.5.]{GEPNER2015575}, a homotopy coherent version of associative algebras, with which enriched $\infty$-categories share several properties:
any $A_\infty$-algebra in a combinatorial monoidal model category $\mM$ is equivalent to
an associative algebra in $\mM$ \cite[Theorem 4.1.8.4.]{lurie.higheralgebra}.
Similarly, every $\infty$-category enriched in the underlying monoidal $\infty$-category
of $\mM$ is equivalent to a $\mM$-enriched category \cite[Theorem 1.1.]{MR3402334}.
This covers most examples of interest since every presentably monoidal $\infty$-category $\mV$ underlies a monoidal combinatorial model category \cite[Remark 4.1.8.9.]{lurie.higheralgebra}.
Consequently, as long as one works with $\infty$-categories enriched in a presentably monoidal $\infty$-category, nothing new arises when passing from enriched category theory to enriched $\infty$-category theory so far as one is just interested in enriched $\infty$-categories but not in enriched functors. On the other hand $\infty$-categories of enriched functors are generally tremendously hard to model by their strict counterparts:
even in the most basic example of interest, enrichment in spaces, 
one finds that the $\infty$-category of (spaces-enriched) functors is not modeled by the category of simplicially enriched functors since Bergner's model structure on simplicial categories fails to be compatible with the cartesian product \cite[Definition A 3.4.9.]{lurie.HTT}. From this perspective enriched $\infty$-categories of enriched functors are the real topic and challenge of enriched $\infty$-category theory \cite{BENMOSHE2024107625}, \cite{berman2020enriched}, \cite{heine2024equivalence}, \cite{HINICH2020107129}, \cite{hinich2021colimits} and good control about enriched $\infty$-categories of enriched functors is highly desirable.

From the conceptual point of view there are several obstructions to the existence of enrichment on the $\infty$-category of enriched functors:
if $\mV$ is a presentably braided monoidal $\infty$-category, there is a natural candidate for a $\mV$-enriched $\infty$-category of $\mV$-enriched functors as the internal hom for a tensor product of $\mV$-enriched $\infty$-categories:
Gepner-Haugseng \cite[Corollary 5.6.12.]{GEPNER2015575}, Haugseng \cite{haugseng2023tensor}
construct a tensor product for enriched $\infty$-categories, which exhibits the $\infty$-category of small $\infty$-categories enriched in any presentably $\bE_{\n+1}$-monoidal $\infty$-category for $\n \geq 1$ as closed $\bE_\n$-monoidal.
On the other hand if $\mV$ is a presentably monoidal $\infty$-category, an important situation in non-commutative algebra and higher category theory \cite{campion2023gray}, there is no need for any $\mV$-enrichment on the $\infty$-category of $\mV$-enriched functors. However, the $\infty$-category of $\mV$-enriched presheaves on any small $\mV$-enriched $\infty$-category is always $\mV$-enriched as long as $\mV$ is presentably monoidal. This relies on the fact that $\mV$ has more structure than being a $\mV$-enriched $\infty$-category via its closed left $\mV$-action: the monoidal structure on $\mV$ endows $\mV$ with a closed $\mV,\mV$-biaction. 

To construct and systematically study enrichment on $\infty$-categories of enriched functors
we extend the concepts of left, right and bitensored $\infty$-categories to concepts of left, right and bienriched $\infty$-categories that generalize Lurie's notion of enriched $\infty$-category \cite[Definition 4.2.1.28.]{lurie.higheralgebra}:
Lurie defines enriched $\infty$-categories as a weakening of the notion of $\infty$-category left tensored over a monoidal $\infty$-category.
A left action of a monoidal $\infty$-category $\mV$ on an $\infty$-category $\mC$ gives rise to a non-symmetric $\infty$-operad $\mO$ whose colors are the objects of $\mV$ and $\mC$, and whose space of multimorphisms $$\Mul_\mO(\V_1,...,\V_\n,\X, \Y) $$ for colors $\V_1,...,\V_\n \in \mV$, $\n \geq 0$ and $\X,\Y \in \mC$ is the space of morphisms in $\mC$ from the left tensor $\V_1 \ot ... \ot \V_\n \ot \X $ to $\Y$. 
For $\mV$ and $\mC$ contractible we write $\LM$ for $\mO$. So $\LM$ is a non-symmetric $\infty$-operad with two colors $\mathfrak{a}, \mathfrak{m}$ whose algebras are left modules. 
By functoriality the $\infty$-operad $\mO$ comes equipped with a map of $\infty$-operads $\mO \to \LM$, 
which completely determines the left action of $\mV$ on $\mC$: for any objects $\V$ of $\mV$ and $\X $ of $\mC$ one can reconstruct the left tensor $\V \ot \X$ as the object corepresenting the functor $ \Mul_\mO(\V,\X;-): \mC \to \mS$. The left action is closed if and only if for every $\X,\Y \in \mC$ there is an object 
$\Mor_\mC(\X,\Y) $ of morphisms $\X \to \Y$ in $\mC$ representing the 
presheaf $\Mul_\mO(-,\X;\Y): \mV^\op \to \mS$. 
An arbitrary map of non-symmetric $\infty$-operads $\mO \to \LM$
whose fibers over $\mathfrak{a}$ and $\mathfrak{m}$ we denote by $\mV,\mC$, respectively, represents a closed left action of a closed monoidal $\infty$-category 
if the $\infty$-operad $\mV$ is a closed monoidal $\infty$-category, for every objects $\V \in \mV, \X,\Y \in \mC$ there is a left tensor $\V \ot \X \in \mC$
and a morphism object $\Mor_\mC(\X,\Y) \in \mV $ and the following natural maps are equivalences:
$$\Mor_\mC(\X,\Y)^\V \simeq \Mor_\mC(\V \ot \X,\Y),$$
\begin{equation}\label{psul}
\Mul_\mO(\V_1 \ot ... \ot \V_\n,\X;\Y) \simeq \Mul_\mO(\V_1,...,\V_\n,\X;\Y).\end{equation}
Lurie's notion of enriched $\infty$-category arises by removing the existence of left tensors from a closed left action: an $\infty$-category enriched in a monoidal $\infty$-category $\mV$ is a $\LM$-operad $\mO \to \LM$ whose fibers over $\mathfrak{a}$ and $\mathfrak{m}$ we denote by $\mV,\mC$, respectively, such that $\mV$ is a monoidal $\infty$-category, every pair of objects of $ \mC$ admits a morphism object and equivalence (\ref{psul}) holds.
If equivalence (\ref{psul}) holds but not all morphism objects exist, Lurie calls 
the $\LM$-operad $\mO \to \LM$ a pseudo-enriched $\infty$-category \cite[Definition 4.2.1.25.]{lurie.higheralgebra}.
He calls an arbitrary $\LM$-operad a weakly enriched $\infty$-category \cite[Notation 4.2.1.22.]{lurie.higheralgebra}. Pseudo-enriched and weakly enriched $\infty$-categories are crucial when studying enrichment on $\infty$-categories of enriched functors: for an arbitrary monoidal $\infty$-category $\mV$ there is no need that the $\infty$-category of $\mV$-enriched presheaves is $\mV$-enriched since $\mV$ itself is an example of an $\infty$-category of $\mV$-enriched presheaves.
But the $\infty$-category of presheaves is always pseudo-enriched in $\mV$. More generally, the $\infty$-category of presheaves enriched in any $\infty$-operad $\mV$ is weakly enriched in $\mV$.
In analogy the concept of bienriched $\infty$-category arises by removing the existence of bitensors from a closed biaction:
a biaction of two monoidal $\infty$-categories $\mV,\mW $ on an $\infty$-category $\mC$ gives rise to a non-symmetric $\infty$-operad $\mO$ whose colors are the objects of $\mV, \mW, \mC$ and whose space of multimorphisms $$\Mul_\mO(\V_1,...,\V_\n,\X, \W_1,...,\W_\m, \Y) $$ for $\V_1,...,\V_\n \in \mV, \W_1,...,\W_\m \in \mW$, $\n,\m \geq 0$ and $\X,\Y \in \mC$ is the space of morphisms $\V_1 \ot ... \ot \V_\n \ot \X \ot \W_1 \ot ...\ot \W_\m \to \Y$ in $\mC.$ The $\infty$-operad $\mO$ comes equipped with a map of non-symmetric $\infty$-operads $ \mO \to \BM$ to the $\infty$-operad $\BM$ governing bimodules,
which completely determines the biaction:
for every objects $\V$ of $\mV$, $\W \in \mW$ and $\X $ of $\mC$ the bitensor $\V \ot \X \ot \W$ corepresents the functor $$\Mul_\mO(\V,\X, \W;-): \mC \to \mS.$$ The biaction is closed if and only if for every $\X,\Y \in \mC$ the presheaf \begin{equation}\label{bbbp}
\Mul_\mO(-,\X,-;\Y): \mV^\op \times \mW^\op \to \mS\end{equation} is representable component-wise.
We define left enriched, right enriched and bienriched $\infty$-categories by removing the existence of bitensors from a closed biaction: left enriched, right enriched $\infty$-categories are maps of non-symmetric $\infty$-operads $\mO \to \BM$ 
such that for every $\X,\Y \in \mC$ the presheaf (\ref{bbbp}) is representable in the first component, in the second component, respectively, and a respective left, right pseudo-enrichedness condition holds
that is analogous to equivalence (\ref{psul}).
We define bienriched $\infty$-categories as maps of non-symmetric $\infty$-operads $\mO \to \BM$ that are simoultaneously left and right enriched 
$\infty$-categories.
By \cite[Proposition 4.8.1.17.]{lurie.higheralgebra} for every presentably monoidal $\infty$-categories $\mV,\mW$
the tensor product $\mV \ot \mW$ of presentable $\infty$-categories
identifies with the full subcategory of presheaves on $\mV \times \mW$ that are representable in both variables. Consequently, in this situation a bienriched $\infty$-category has morphism objects in $\mV\ot \mW.$ In general, a presheaf on $\mV \times \mW$ representable in both variables uniquely determines an adjunction $\mW \rightleftarrows\mV^\op,$
which we denote by \begin{equation}\label{asol}
\L\Mor_\mO(\X,\Y): \mW \rightleftarrows\mV^\op : \R\Mor_\mO(\X,\Y).\end{equation}
If $\mW$ is the $\infty$-category of spaces endowed with the cartesian structure, $\mV,\mW$-bienriched $\infty$-categories agree with $\mV$-enriched $\infty$-categories (Corollary \ref{coronn}) and adjunction (\ref{asol}) is uniquely determined by the value of the left adjoint $\L\Mor_\mO(\X,\Y)$ at the final space, which is the morphism object of $\X,\Y$.
Another example of $\mV,\mV$-bienriched $\infty$-category is the full subcategory $B\tu_\mV \subset \mV$ spanned by the tensor unit of a closed monoidal $\infty$-category $\mV$. Here  adjunction (\ref{asol}) for $\X=\Y=\tu_\mV$ is the "lax duality" adjunction $ \mV \rightleftarrows \mV^\op$ whose left (right) adjoint takes the right (left) internal hom to the tensor unit (Remark \ref{duall}).

Since pseudo-enriched and weakly enriched $\infty$-categories are crucial when studying enrichment on $\infty$-categories of enriched functors, 
we generalize Lurie's notions of pseudo-enriched and weakly enriched $\infty$-categories to notions of left pseudo-enriched, right pseudo-enriched and bipseudo-enriched $\infty$-categories and weakly left enriched, weakly right enriched and weakly bienriched $\infty$-categories that fit into the following diagram, which we draw for the bienriched case:

\begin{equation*} 
\begin{xy}
\xymatrix{
& \{\underset{\text{$\infty$-categories}}{\text{closed bitensored}}\} \ar[rd]^{ }\ar[ld]^{ }
\\
\{\underset{\text{$\infty$-categories}}{\text{bitensored}}\}\ar[d]\ar[rd] && \{\underset{\text{in a monoidal $\infty$-category}}{\text{$\infty$-categories bienriched}}\} \ar[d] \ar[ld]
\\
\{\underset{\text{$\infty$-categories}}{\text{oplax bitensored}}\}  \ar[rd] & \{\underset{\text{$\infty$-categories}}{\text{bipseudo-enriched}}\} \ar[d] & \{\underset{\text{in an $\infty$-operad}}{\text{$\infty$-categories bienriched}}\} \ar[ld] 
\\ & \{\underset{\text{$\infty$-categories}}{\text{weakly bienriched}}\}
}
\end{xy} 
\end{equation*}

In this work we develop a theory of left, right and bienriched $\infty$-categories that studies the structures in the latter diagram and contains the theory of enriched $\infty$-categories for Lurie's model as a special case. In particular, we build a theory of enriched $\infty$-categories for Lurie's model of enrichment, which we use in \cite{heine2024higher} to develop a theory of weighted colimits, and which is independent from the equivalent theories \cite{HEINE2023108941} of Gepner-Haugseng \cite{GEPNER2015575} and Hinich \cite{HINICH2020107129}.

To show the usefulness of left, right and bienriched $\infty$-categories we demonstrate that several principles of left, right and bitensored $\infty$-categories extend to left, right and bienriched $\infty$-categories: 
given an $\infty$-category $\mD$ with biaction of two monoidal $\infty$-categories $\mV,\mW$
and an $\infty$-category $\mC$ with a left $\mV$-action the $\infty$-category $\LinFun_\mV(\mC,\mD)$ of left $\mV$-linear functors $\mC \to \mD$, i.e. functors preserving the left $\mV$-action, carries a right $\mW$-action that applies object-wise the right $\mW$-action on the target $\mD$.
If $\mV$ is a braided monoidal $\infty$-category, the tensor product functor $\ot: \mV \times \mV \to \mV$ is a monoidal functor so that every left $\mV$-action on a monoidal $\infty$-category $\mD$ gives rise via restriction along the tensor product of $\mV$ to a left $\mV \ot \mV$-action on $\mD.$ By a general principle about the relationship between left actions and biactions a left $\mV \ot \mV$-action on $\mD$ is equivalently given by a $\mV, \mV^\rev$-biaction on $\mD,$ where $\mV^\rev$ is the reversed monoidal structure.
We extend these principles about left, right and biactions to principles about left, right and bienriched $\infty$-categories to formally treat the $\infty$-category of $\mV$-enriched functors like the familiar $\infty$-category of $\mV$-linear functors:
\begin{theorem}\label{1}(Theorem \ref{psinho})
Let $\mV$ be a locally small monoidal $\infty$-category and $\mW$ a closed monoidal $\infty$-category that admits small limits. Let $\mC$ be a small weakly left $\mV$-enriched $\infty$-category and $\mD$ a $\mV,\mW$-bienriched $\infty$-category. 
The $\infty$-category $$\L\Enr\Fun_\mV(\mC,\mD)$$ of left $\mV$-enriched functors $\mC \to \mD $ is a right $\mW$-enriched $\infty$-category.
\end{theorem}

\begin{theorem}\label{2} (Theorem \ref{Enrbi})
Let $\mV,\mW $ be presentably monoidal $\infty$-categories.
There is a canonical equivalence
$${_{\mV}\B\Enr_\mW} \simeq {_{\mV \otimes \mW^\rev}\L\Enr}$$
between the $\infty$-categories of $\mV,\mW$-bienriched $\infty$-categories and left $\mV\ot\mW^\rev$-enriched $\infty$-categories, where $\mV\ot\mW^\rev$ is the tensor product of presentable $\infty$-categories.
	
\end{theorem}
By Theorem \ref{2} for any $\infty$-category $\mD$ left enriched in a presentably braided monoidal $\infty$-category $\mV$ the restricted enrichment of $\mD$ along the left adjoint monoidal tensor product functor $\ot: \mV \times \mV \to \mV$ 
is an $\infty$-category bienriched in $\mV,\mV^\rev$.
So by Theorem \ref{1} for any small left $\mV$-enriched $\infty$-category $\mC$ the $\infty$-category $ \L\Enr_\mV(\mC,\mD)$ of left $\mV$-enriched functors $\mC \to \mD $ is a right $\mV^\rev$-enriched $\infty$-category, which is canonically identified with a left $\mV$-enriched $\infty$-category. We prove that this left $\mV$-enrichment on the $\infty$-category of left $\mV$-enriched functors $\mC \to \mD $ is in fact the internal hom for the tensor product of $\mV$-enriched $\infty$-categories:
\begin{theorem}\label{3}(Theorem \ref{cloff}) Let $\mV$ be a presentably symmetric monoidal $\infty$-category and $\mC,\mD$ small $\mV$-enriched $\infty$-categories. The left $\mV$-enriched $\infty$-category $ \L\Enr_\mV(\mC,\mD) $ provided by Theorems \ref{1} and \ref{2} is the internal hom for the tensor product of $\mV$-enriched $\infty$-categories.\end{theorem}

Moreover we prove a refinement of Theorem \ref{1} that considers actions compatible with the enrichments and more generally pseudo-enrichments:

\begin{theorem}\label{15}(Theorem \ref{vvvl}, Proposition \ref{klmou})
Let $\mU, \mV, \mW$ be monoidal $\infty$-categories, $\mC$ an $\infty$-category weakly bienriched in $ \mV,\mU$ that is right tensored over $\mU$ 
and $\mD$ a $\mV,\mW$-bipseudo-enriched $\infty$-category. 
Then $$\L\Enr\Fun_\mV(\mC,\mD)$$ is bipseudo-enriched in $\mU,\mW$ and left tensored over $\mU$. The left 
$\mU$-action sends a left $\mV$-enriched functor $\phi: \mC \to \mD $ and an object $X$ of $\mU$ to the left $\mV$-enriched functor $\mC \xrightarrow{(-)\ot X}\mC \xrightarrow{\phi} \mD.$
\end{theorem}

Similarly, in our framework of bienriched $\infty$-categories we develop many constructions of enriched $\infty$-category theory via familiar constructions of the theory of left tensored, right tensored and bitensored $\infty$-categories. 
The $\infty$-category of presheaves on a small $\infty$-category $\mC$ is the free cocompletion under small colimits \cite[Theorem 5.1.5.6.]{lurie.HTT}, 
the initial $\infty$-category with small colimits into which $\mC$ embeds.
We construct enriched analoga of the free cocompletion.
Lurie \cite[Construction 2.2.4.1.]{lurie.higheralgebra} constructs for every $\infty$-operad $\mO$ and $\mO$-operad $\mB$ an $\mO$-monoidal envelope, the initial $\mO$-monoidal $\infty$-category into which $\mB$ embeds. 
We use the free cocompletion of the $\BM$-monoidal envelope of a small $\mV,\mW$-bienriched $\infty$-category $\mC$ 
to construct the free cocompletion under small colimits and left $\mV$-tensors, right $\mW$-tensors and $\mV, \mW$-bitensors, respectively: 
\begin{theorem}\label{4}(Theorem \ref{unipor3}). 
\begin{enumerate}
\item Let $\mV,\mW$ be presentably monoidal $\infty$-categories and $\mC$ a small $\mV, \mW$-bienriched $\infty$-category. There is a $\mV,\mW$-enriched embedding $\mC \to \mP_{\mV, \mW}(\mC)$ such that for any $\infty$-category $\mD$ presentably bitensored over $\mV,\mW$ the induced functor $$ \LinFun^\L_{\mV,\mW}(\mP_{\mV,\mW}(\mC),\mD) \to \Enr\Fun_{\mV,\mW}(\mC,\mD)$$ is an equivalence, where the left hand side is the $\infty$-category of left adjoint $\mV, \mW$-linear functors and the right hand side is the $\infty$-category of $\mV, \mW$-enriched functors.	
	
\item Let $\mV$ be a presentably monoidal $\infty$-category, $\mW$ a small monoidal $\infty$-category and $\mC$ a small $\infty$-category weakly bienriched in $\mV,\mW$ that exhibits $\mC$ as left $\mV$-enriched.
There is a $\mV,\mW$-enriched embedding $\mC \to \mP_{\mV}(\mC)$ such that 
for any $\infty$-category $\mD$ weakly bienriched in $\mV,\mW$ that exhibits $\mD$ as left tensored over $\mV$ the induced functor $$ \L\LinFun^\L_{\mV,\mW}(\mP_{\mV}(\mC),\mD) \to \Enr\Fun_{\mV,\mW}(\mC,\mD)$$ is an equivalence, where the left hand side is the $\infty$-category of $\mV, \mW$-enriched functors that preserve left tensors 
and small colimits.

\item Let $\mW$ be a presentably monoidal $\infty$-category, $\mV$ a small monoidal $\infty$-category and $\mC$ a small weakly bienriched $\infty$-category that exhibits $\mC$
as right $\mW$-enriched. There is a $\mV, \mW$-enriched embedding $\mC \to \mP_{\mW}(\mC)$ such that for any $\infty$-category $\mD$ weakly bienriched in $\mV,\mW$ that exhibits $\mD$ as right tensored over $\mW$ the induced functor $$ \R\LinFun^\L_{\mV,\mW}(\mP_{\mW}(\mC),\mD) \to \Enr\Fun_{\mV,\mW}(\mC,\mD)$$ is an equivalence, where the left hand side is the $\infty$-category of $\mV, \mW$-enriched functors that preserve right tensors 
and small colimits.

\end{enumerate}
\end{theorem}


We use the $\BM$-monoidal envelope to construct transfer of enrichment for left, right and bienriched $\infty$-categories from scalar extension of presentably left, right and bitensored $\infty$-categories (Theorem \ref{bica}), and use the construction of free cocompletion under small colimits and left tensors, right tensors and bitensors of Theorem \ref{4} to reveal the following fundamental relationship between transfer of enrichment and scalar extension of presentably left, right and bitensored $\infty$-categories:

\begin{theorem}(Theorem \ref{bica}, Proposition \ref{bicay})
Let $ \mC$ be a small $\infty$-category bienriched in presentably monoidal $\infty$-categories $\mV,\mW$ and $\alpha: \mV \to \mV', \beta: \mW \to \mW'$ monoidal functors between presentably monoidal $\infty$-categories that admit right adjoints $\gamma, \delta,$ respectively.

\begin{enumerate}
\item There is an adjunction $$(\alpha, \beta)_!: {_{\mV} \B\Enr}_{\mW} \rightleftarrows {_{\mV'} \B\Enr}_{\mW'}: (\alpha, \beta)^*, $$
where the right adjoint restricts and the left adjoint transfers bi-enrichment.

\item There is a canonical equivalence $ (\alpha, \beta)^* \simeq (\phi, \psi)_!,$
where $(\phi, \psi)_!$ transfers bi-enrichment along the right adjoints.

\item Let $ \mD$ be a small $\infty$-category bienriched in $\mV', \mW'$. A $\mV, \mW$-enriched functor $\mC \to (\alpha, \beta)^*(\mD)$
exhibits $\mD$ as transfer of enrichment, i.e. induces an equivalence
$(\alpha,\beta)_!(\mC) \simeq \mD$, if and only if it is essentially surjective 
and the induced $\mV,\mW$-linear left adjoint functor $$\mP_{\mV,\mW}(\mC) \to (\alpha, \beta)^*(\mP_{\mV',\mW'}(\mD))$$ induces an equivalence $\mV' \ot_{\mV} \mP_{\mV,\mW}(\mC) \ot_{\mW} \mW' \simeq \mP_{\mV',\mW'}(\mD)$.
\end{enumerate}
\end{theorem}

We 
link the universal property of free cocompletion of Theorem \ref{4} to another universal property of enriched presheaves: we use Theorem \ref{1}
to construct 
for every left $\mV$-enriched $\infty$-category an opposite right $\mV$-enriched $\infty$-category (Definition \ref{notori}, Proposition \ref{Dungo}). 
Since any closed monoidal $\infty$-category $\mV$ is canonically $\mV,\mV$-bienriched, by Theorem \ref{1} for any small left $\mV$-enriched $\infty$-category $\mC$ 
the $\infty$-category $\R\Enr_\mV(\mC^\op,\mV) $ of right $\mV$-enriched functors $\mC^\op \to \mV$ is left $\mV$-enriched. This left $\mV$-enrichment satisfies the following universal property: 
left $\mV$-enriched functors $ \mD \to \R\Enr_\mV(\mC^\op,\mV) $ from any left $\mV$-enriched $\infty$-category $\mD$ 
correspond to $\mV,\mV$-enriched functors $\mD \times \mC^\op \to \mV$, where the $\mV,\mV$-bi-enrichment on the product $\mD \times \mC^\op $ is induced from the left $\mV$-enrichment on $\mD$ and the right $\mV$-enrichment on $\mC^\op.$
We link this universal property 
that characterizes enriched functors to the $\infty$-category of enriched presheaves with the universal property of free cocompletion that characterizes enriched functors starting at the $\infty$-category of enriched presheaves:
\begin{theorem}\label{5}(Theorem \ref{unitol}) Let $\mV$ be a presentably monoidal $\infty$-category and $\mC$ a small left $\mV$-enriched $\infty$-category.
There is a left $\mV$-enriched equivalence $$\mP_\mV(\mC) \simeq \R\Enr_\mV(\mC^\op,\mV).$$
\end{theorem} 


By Theorem \ref{5} the enriched $\infty$-category of enriched presheaves carries two universal properties. By the universal property of enriched cocompletion every left $\mV$-enriched functor $\phi: \mC \to \mD$ between small left $\mV$-enriched $\infty$-categories uniquely extends to a small colimits preserving 
left $\mV$-linear functor $\bar{\phi}: \mP_{\mV}(\mC)\to \mP_{\mV}(\mD).$
By the second universal property there is a left $\mV$-enriched functor $\phi^*: \mP_{\mV}(\mD)\to \mP_{\mV}(\mC)$ precomposing with $\phi.$
Theorem \ref{5} implies the following corollary, which was asked by \cite[Question 1.5.]{BENMOSHE2024107625}:
\begin{corollary}(Corollary \ref{saewo}) \label{6} Let $\mV$ be a presentably monoidal $\infty$-category
and $\phi: \mC \to \mD$ a left $\mV$-enriched functor between small left $\mV$-enriched $\infty$-categories.
There is a left $\mV$-enriched adjunction $$\bar{\phi}: \mP_{\mV}(\mC)\rightleftarrows \mP_{\mV}(\mD): \phi^*.$$
\end{corollary} 


We use the description of the internal hom for the tensor product of enriched
$\infty$-categories (Theorem \ref{3}) and the identification of enriched presheaf 
$\infty$-categories (Theorem \ref{5}) to compute morphism objects in enriched
$\infty$-categories of enriched functors. 
Let $\mC$ be a small left $\mV$-enriched $\infty$-category and $\mD$ a $\mV,\mW$-bitensored $\infty$-category compatible with small colimits. Restriction to the space of objects $\mC^\simeq$ defines a right $\mW$-enriched functor \begin{equation}\label{aaal}
\L\Enr\Fun_\mV(\mC,\mD) \to \Fun(\mC^\simeq,\mD)\end{equation} from left $\mV$-enriched functors $\mC \to \mD $ to functors $\mC^\simeq \to \mD,$ where $\mC^\simeq$ is the maximal subspace in $\mC.$
We prove that this functor is monadic and compute the monad:

\begin{theorem}\label{Tff} (Theorem \ref{expli}, Corollary \ref{explicas})
Let $\mC $ be a small left $\mV$-enriched $\infty$-category and $\mD $ an $\infty$-category bitensored over $\mV,\mW$ compatible with small colimits.		
The right $\mW$-enriched functor (\ref{aaal}) admits a right $\mW$-enriched left adjoint $\phi$ and for every functor $\F : \mC^\simeq \to \mD$ there is a canonical equivalence of left $\mV$-enriched functors $\mC \to \mD:$
$$ \colim_{\Z \in \mM^\simeq}(\L\Mor_{\mM}(\Z,-) \ot \F(\Z)) \simeq \phi(\F).$$

\end{theorem}
 
We use Theorem \ref{Tff} to derive monadic resolutions of left $\mV$-enriched functors
and deduce the following description of morphism objects in the right $\mW$-enriched $\infty$-category $\L\Enr\Fun_\mV(\mC,\mD)  $ of left $\mV$-enriched functors
$\mC \to \mD:$

\begin{corollary}\label{55}(Corollary \ref{enrhom}) Let $\mV$ be a locally small monoidal $\infty$-category and $\mW$ a closed monoidal $\infty$-category that admits small limits. Let $\mC$ be a small left $\mV$-enriched $\infty$-category and $\mD$ a $\mV,\mW$-bienriched $\infty$-category that admits left tensors.
For every left $\mV$-enriched functors $\F,\G : \mC \to \mD$ there is a canonical equivalence in $\mW:$ $$\R\Mor_{\L\Enr\Fun_{\mV}(\mC,\mD)}(\F,\G) \simeq $$$$ \underset{[\n]\in \Delta^\op}{\lim} \lim_{\Z_1, ...., \Z_\n \in \mC^\simeq}  \R\Mor_{\mD}(\L\Mor_{\mC}(\Z_{\n-1},\Z_\n) \ot  ...\ot \L\Mor_{\mC}(\Z_1,\Z_2) \ot \F(\Z_1),\G(\Z_\n)).$$
	
\end{corollary}

We use the latter limit formula for morphism objects to derive a limit formula for the relative tensor product of $\mV$-enriched profunctors.
For right $\mV$-enriched $\infty$-categories $\mC,\mD$ we define $\mV$-enriched profunctors $\mC \to \mD$ as $\mV,\mV$-enriched functors $\mD^\op \times \mC \to \mV$, which extends the usual definition from enrichment in symmetric monoidal $\infty$-categories to enrichment in not necessarily symmetric monoidal $\infty$-categories. Enriched profunctors have several different descriptions: by Theorem \ref{5} a $\mV,\mV$-enriched profunctor $\mC \to \mD$ is likewise a right $\mV$-enriched functor $\mC \to \mP_\mV(\mD)$, which by Theorem \ref{4} is the same as a 
left adjoint right $\mV$-linear functor $\mP_\mV(\mC) \to \mP_\mV(\mD)$.
The latter description gives rise to a relative tensor product of enriched profunctors that extends the relative tensor product of bimodules: the relative tensor product of two $\mV$-enriched profunctors $\F: \mC \to \mD,\G: \mD \to \mE$ corresponding to left adjoint right $\mV$-linear functors $\mP_\mV(\mC) \to \mP_\mV(\mD), \mP_\mV(\mD) \to \mP_\mV(\mE)$ is the $\mV$-enriched profunctor $\F\ot_\mD\G: \mC \to \mE$ corresponding to the composition $\mP_\mV(\mC) \to \mP_\mV(\mD) \to \mP_\mV(\mE)$.
We use Corollary \ref{55} to deduce the following colimit formula for the relative tensor product of profunctors that extends the colimit formula for the relative tensor product of bimodules:
\begin{corollary}(Corollary \ref{reltu}) Let $\mC,\mD,\mE$ be right $\mV$-enriched $\infty$-categories and $\F: \mC \to \mD, \G: \mD \to \mE$ be $\mV$-enriched profunctors.
For every $\X \in \mC, \Y \in \mE$ there is a canonical equivalence $$ (\F \ot_\mD \G)(\Y,\X) \simeq $$$$\underset{[\n]\in \Delta^\op}{\colim}(\colim_{\Z_1, ...., \Z_\n \in \mD^\simeq} \F(\Z_1,\X) \ot\L\Mor_{\mD}(\Z_2,\Z_1) \ot .... \ot \L\Mor_{\mD}(\Z_{\n},\Z_{\n-1})\ot \G(\Y,\Z_{\n}))).$$
\end{corollary}

In case that $\mV$ is a presentably symmetric monoidal $\infty$-category, we prove a refinement of Corollary \ref{55}, which describes morphism objects in enriched functor $\infty$-categories via enriched ends: Gepner-Haugseng-Nikolaus \cite[Definition 2.9.]{articles} define ends for functors of $\infty$-categories via the twisted arrow $\infty$-category and prove that for any functors $\F,\G: \mC \to \mD$ of $\infty$-categories the space of natural transformations $\F \to \G$ 
is computed as the end $\int_{\X \in \mC} \map_\mD(\F(\X),\G(\X)).$ 
The authors define the end of a functor $\rH: \mC^\op \times \mC \to \mD$ as
the limit of the composition $\Tw(\mC) \to \mC^\op \times \mC \xrightarrow{\rH} \mD, $ where $\Tw(\mC) \to \mC^\op \times \mC$ is the twisted arrow left fibration that classifies the mapping space functor $\mC^\op \times \mC \to \mS, (\X,\Y) \mapsto \map_\mC(\X,\Y).$ 
This definition characterizes the end $\int_{\X\in \mC} \rH(\X,\X)$ as the object of $\mD$ representing the presheaf $\X \mapsto \map_{\Fun(\mC^\op \times \mC,\mS)}(\map_\mC, \map_\mD(\X,-) \circ \rH)$ on $\mC$.
The latter universal property generalizes to $\mV$-enriched $\infty$-categories
for any presentably symmetric monoidal $\infty$-category $\mV$ by replacing mapping spaces by morphism objects:
for any $\mV$-enriched $\infty$-categories $\mC,\mD$ and any $\mV$-enriched functor $\rH: \mC^\op \ot \mC \to \mD$ the $\mV$-enriched end of $\rH$
is the object of $\mD$ representing the $\mV$-enriched presheaf $\X \mapsto \Mor_{\Enr\Fun_\mV(\mC^\op \otimes \mC,\mV)}(\Mor_\mC, \Mor_\mD(\X,-) \circ \rH)$ on $\mC.$
We prove the following enriched end formula:
\begin{theorem}\label{44} (Theorem \ref{end}) Let $\mV$ be a presentably symmetric monoidal $\infty$-category 
and $\F,\G: \mC,\mD$ be $\mV$-enriched functors.	
There is a canonical equivalence
$$ \Mor_{\Enr\Fun_{\mV}(\mC,\mD)}(\F,\G) \simeq \int_{\X\in\mC} \Mor_\mD(\F(\X),\G(\X)).$$
	
\end{theorem}

We use Theorem \ref{44} to obtain the following coend formula for the relative tensor product:

\begin{corollary}(Corollary \ref{endrel})
Let $\mV$ be a presentably symmetric monoidal $\infty$-category, $\mC,\mD,\mE $ small right $\mV$-enriched $\infty$-categories and $\F:\mC \to \mD, \G:\mD \to \mE$ be $\mV$-enriched profunctors. For every $\X \in \mC, \Y \in \mE$ there is a canonical equivalence in $\mV: $$$ (\F \ot_{\mD} \G)(\X,\Y)\simeq \int^{\Z\in\mD} \F(\Z,\X) \ot \G(\Y,\Z).$$ 
	
\end{corollary}

\vspace{2mm}

\subsection{Notation and terminology}

We fix a hierarchy of Grothendieck universes whose objects we call small, large, very large, etc.
We call a space small, large, etc. if its set of path components and its homotopy groups are for any choice of base point. We call an $\infty$-category small, large, etc. if its maximal subspace and all its mapping spaces are.

\vspace{2mm}

We write 
\begin{itemize}
\item $\Set$ for the category of small sets.
\item $\Delta$ for (a skeleton of) the category of finite, non-empty, partially ordered sets and order preserving maps, whose objects we denote by $[\n] = \{0 < ... < \n\}$ for $\n \geq 0$.
\item $\mS$ for the $\infty$-category of small spaces.
\item $ \Cat_\infty$ for the $\infty$-category of small $\infty$-categories.
\item $\Cat_\infty^{\rc \rc} $ for the $\infty$-category of large $\infty$-categories with small colimits and small colimits preserving functors.
\end{itemize}

\vspace{2mm}

We often indicate $\infty$-categories of large objects by $\widehat{(-)}$, for example we write $\widehat{\mS}, \widehat{\Cat}_\infty$ for the $\infty$-categories of large spaces, $\infty$-categories.

For any $\infty$-category $\mC$ containing objects $\A, \B$ we write
\begin{itemize}
\item $\mC(\A,\B)$ for the space of maps $\A \to \B$ in $\mC$,
\item $\mC_{/\A}$ for the $\infty$-category of objects over $\A$,
\item $\Ho(\mC)$ for its homotopy category,
\item $\mC^{\triangleleft}, \mC^{\triangleright}$ for the $\infty$-category arising from $\mC$ by adding an initial, final object, respectively,
\item $\mC^\simeq $ for the maximal subspace in $\mC$.
\end{itemize}

Note that $\Ho(\Cat_\infty)$ is cartesian closed and for small $\infty$-categories $\mC,\mD$ we write $\Fun(\mC,\mD)$ for the internal hom, the $\infty$-category of functors $\mC \to \mD.$ 

We often call a fully faithful functor $\mC \to \mD$ an embedding.
We call a functor $\phi: \mC \to \mD$ an inclusion (of a subcategory)
if one of the following equivalent conditions holds:
\begin{itemize}
\item For any $\infty$-category $\mB$ the induced map
$\Cat_\infty(\mB,\mC) \to \Cat_\infty(\mB,\mD)$ is an embedding.
\item The functor $\phi: \mC \to \mD$ induces an embedding on maximal subspaces and on all mapping spaces.

\end{itemize}





Let $\rS$ be an $\infty$-category and $\mE \subset \Fun([1],\rS)$ a full subcategory. We call a functor $\X \to \rS$ a cocartesian fibration relative to $\mE$ if for every morphism $[1] \to \rS$ that belongs to $\mE$ the pullback
$[1] \times_\rS \X \to [1]$ is a cocartesian fibration whose cocartesian morphisms are preserved by the projection $[1] \times_\rS \X \to \X.$
We call a functor $\X \to \Y$ over $\rS$ a map of cocartesian fibrations relative to $\mE$ if it preserves cocartesian lifts of morphisms of $\mE.$
We write $\Cat_{\infty/\rS}^\mE \subset \Cat_{\infty/\rS}$ for the subcategory of cocartesian fibrations relative to $\mE$ and maps of such.
(see \cite[Definition 2.3.]{heine2023monadicity} for details).

\subsubsection*{Acknowledgements}We thank David Gepner and Markus Spitzweck for helpful discussions. We thank Julie Rasmusen for carefully reading a first draft of this paper.

\label{rel}

\section{Enrichment}
In this section we recall the theory of weakly enriched $\infty$-categories of \cite{HEINE2023108941} and further develope this theory for our needs.

\subsection{Weak enrichment}

We start with defining weakly enriched $\infty$-categories.
To define such we first need to introduce non-symmetric $\infty$-operads
\cite[Definition 2.2.6.]{GEPNER2015575}, \cite[Definition 2.16.]{HEINE2023108941}.

\begin{notation}\label{ooop}
Let $\Ass:= \Delta^\op$ be the category of finite non-empty totally ordered sets and order preserving maps.
We call a map $[\n] \to [\m]$ in $\Ass$
\begin{itemize}
\item inert if it corresponds to a map of $\Delta$ of the form $[\m] \simeq \{\bi, \bi+1,..., \bi+\m \} \subset [\n]$ for some $\bi \geq 0.$
\item active if it corresponds to a map of $\Delta$, which preserves the minimum and maximum.
\end{itemize}
\end{notation}
\begin{remark}
For every $\n \geq 0$ there are $\n$ inert morphisms $[\n] \to [1]$, where the $\bi$-th inert morphism $[\n] \to [1]$ for $1 \leq \bi \leq \n$ corresponds to the map $[1] \simeq \{\bi-1, \bi\} \subset [\n]$.
\end{remark}

\begin{definition}
A morphism in $ \Ass$ preserves the minimum (maximum) if it corresponds to a map $[\m] \to [\n]$ in $\Delta$ sending $0$ to $0$
(sending $\m$ to $\n$).
	
\end{definition}

\begin{definition}\label{ek}A (non-symmetric) $\infty$-operad is a cocartesian fibration $\phi: \mV^\ot \to \Ass$ relative to the collection of inert morphisms such that the following conditions hold, where we set $\mV:=\mV^\ot_{[1]}:$

\begin{itemize}
\item For every $\n \geq 0$ the family $[\n]\to [1]$ of all inert morphisms in $\Ass$ induces an equivalence $\mV^\ot_{[\n]} \to \mV^{\times \n}.$

\item For every $\Y,\X \in \mV^\ot $ lying over $[\m], [\n] \in \Ass $ the family
$\X \to \X_\bj$ for $1 \leq \bj \leq \n $ of $\phi$-cocartesian lifts of
the inert morphisms $[\n]\to [1]$ induces 
a pullback square
\begin{equation*} 
\begin{xy}
\xymatrix{
\mV^\ot(\Y,\X) \ar[d]^{} \ar[r]^{ }
& \prod_{1 \leq \bj \leq \n} \mV^\ot(\Y,\X_\bj) \ar[d]^{} 
\\  \Ass([\m], [\n]) 
\ar[r]^{} & \prod_{1 \leq \bj \leq \n} \Ass([\m], [1]). 
}
\end{xy} 
\end{equation*}

\end{itemize}	
	
\end{definition}
\begin{notation}
For every $\infty$-operad $\phi: \mV^\ot \to \Ass$ we call $\mV:= \mV^\ot_{[1]}$ the underlying $\infty$-category.
The $\infty$-category $\mV^\ot_{[0]}$ is contractible. We write $\emptyset$ for the unique object.
\end{notation}

\begin{example}\label{empp}
Let $\emptyset^\ot \subset \Ass$ be the full subcategory spanned by $[0] \in \Ass.$
Then $\emptyset^\ot$ is contractible and $\emptyset^\ot \subset \Ass$ is an $\infty$-operad that is the initial $\infty$-operad.
	
\end{example}

\begin{notation}
	
Let $\mV^\ot \to \Ass$ be an $\infty$-operad and $\V_1,..., \V_{\n},\W \in \mV$ for $\n \geq 0 $. Let
$$\Mul_{\mV}(\V_1,..., \V_\n;\W)$$ be the full subspace of $\mV^\ot(\V,\W)$ spanned by the morphisms $\V \to \W$ in $\mV^\ot$ lying over the active morphism $[1] \to [\n]$ in $\Delta,$ where $\V \in \mV^\ot_{[\n]} \simeq \mV^{\times \n}$ corresponds to $(\V_1,..., \V_\n) $.
\end{notation}

\begin{definition}
An $\infty$-operad $\phi: \mV^\ot \to \Ass$ is a monoidal $\infty$-category if $\phi$ is a cocartesian fibration.
\end{definition}

\begin{remark}
A cocartesian fibration $\phi: \mV^\ot \to \Ass$ is a monoidal $\infty$-category if and only if the first condition
of Definition \ref{ek} holds. The second condition is automatic if $\phi$ is a cocartesian fibration \cite[Remark 2.17.]{HEINE2023108941}.
\end{remark}

\begin{definition}
\begin{enumerate}
\item An $\infty$-operad $\phi: \mV^\ot \to \Ass$ is locally small if the mapping spaces of $\mV^\ot$ are small.

\item An $\infty$-operad $\phi: \mV^\ot \to \Ass$ is small if it is locally small and $\mV$ is small.

\end{enumerate}	
\end{definition}
\begin{remark}\label{oprer}

By the first axiom of Definition \ref{ek} an $\infty$-operad $\phi: \mV^\ot \to \Ass$ is small if and only if $\mV^\ot$ is small.
By the second axiom an $\infty$-operad $\mV^\ot \to \Ass$ is locally small if and only if the multi-morphism spaces of $ \mV^\ot \to \Ass $ are small.	
So an $\infty$-operad $\mV^\ot \to \Ass$ is small if and only if $\mV$ is small and the multi-morphism spaces of $ \mV^\ot \to \Ass $ are small.	
	
\end{remark}

\begin{notation}

For every monoidal $\infty$-category $\mV^\ot \to \Ass $ we write $\ot: \mV \times \mV \simeq \mV^\ot_{[2]} \to \mV^\ot_{[1]}= \mV$ for the functor induced by the unique active map $[2] \to [1]$ in $\Ass$ and we write $\tu_\mV: \mV^\ot_{[0]} \to \mV^\ot_{[1]}=\mV$ for the functor induced by the unique active map $[0] \to [1]$ in $\Ass$.
	
\end{notation}

\begin{definition}\label{compost}
Let $\mK \subset \Cat_\infty$ be a full subcategory.

\begin{enumerate}
\item
An $\infty$-operad $\mV^\ot \to \Ass$ is compatible with $\mK$-indexed colimits if $\mV$ admits $\mK$-indexed colimits and for every $\V_1,..., \V_\n, \V \in \mV$ for $\n \geq0$ and $0 \leq \bi \leq \n$ the presheaf $\Mul_\mV(\V_1,..,\V_\bi,-, \V_{\bi+1},..., \V_\n;\V)$ on $\mV$ preserves $\mK$-indexed limits.

\item A monoidal $\infty$-category is compatible with $\mK$-indexed colimits
if it is compatible with $\mK$-indexed colimits as an $\infty$-operad.
\end{enumerate}
\end{definition}

\begin{remark}
A monoidal $\infty$-category $\mV^\ot \to \Ass$ is compatible with $\mK$-indexed colimits
if and only if $\mV$ admits $\mK$-indexed colimits and such are preserved by the tensor product component-wise.

\end{remark}

\begin{definition}
	
A monoidal $\infty$-category $\mV^\ot \to \Ass$ is
\begin{enumerate}
\item closed if the tensor product $\ot: \mV \times \mV \to \mV$ is left adjoint component-wise.
\item presentably if it is closed and $\mV$ is presentable.
\item $\kappa$-compactly generated for a regular cardinal $\kappa$ if 
$\mV^\ot \to \Ass $ is compatible with small colimits, $\mV$ is $\kappa$-compactly generated and the monoidal structure of $\mV$ restricts to the full subcategory $\mV^\kappa \subset \mV$ of $\kappa$-compact objects.
\end{enumerate}
\end{definition}

\begin{remark}

By \cite[Proposition 7.15.]{Rune} every presentably monoidal $\infty$-category is $\kappa$-compactly generated for some regular cardinal $\kappa.$
\end{remark}

We also consider maps of $\infty$-operads:
\begin{definition}
Let $\mV^\ot \to \Ass, \mW^\ot \to \Ass$ be $\infty$-operads.
A map of $\infty$-operads $\mV^\ot \to \mW^\ot$ is a map of cocartesian fibrations
$\mV^\ot \to \mW^\ot$ relative to the collection of inert morphisms of $\Ass$.
\end{definition}

\begin{definition}
	
A map of $\infty$-operads $\mV^\ot \to \mW^\ot$ is an embedding
if it is fully faithful.
\end{definition}
\begin{definition}
Let $\mV^\ot \to \Ass, \mW^\ot \to \Ass$ be monoidal $\infty$-categories.
\begin{itemize}
\item A lax monoidal functor $\mV^\ot \to \mW^\ot$ is a map of $\infty$-operads $\mV^\ot \to \mW^\ot$.
\item A monoidal functor $\mV^\ot \to \mW^\ot$ is a map of cocartesian fibrations $\mV^\ot \to \mW^\ot$ over $\Ass.$	
\end{itemize}	
\end{definition}

\begin{notation}We fix the following notation:
\begin{itemize}
	
\item Let $\Op_\infty \subset \Cat_{\infty/\Ass}$ be the subcategory of $\infty$-operads and maps of $\infty$-operads.

\item Let $\Mon \subset\Op_\infty$ be the subcategory of monoidal $\infty$-categories and monoidal functors.

\item Let $\L\Mon \subset \widehat{\Mon}$ be the subcategory of monoidal $\infty$-categories compatible with small colimits and monoidal functors preserving small colimits.

\end{itemize}
	
\end{notation}


\begin{notation}
For every $\infty$-operads $\mV^\ot \to \Ass, \mW^\ot \to \Ass$ let $$\Alg_\mV(\mW) \subset \Fun_\Ass(\mV^\ot,\mW^\ot)$$ be the full subcategory of maps of $\infty$-operads.
If $\mV^\ot \to \Ass, \mW^\ot \to \Ass$ are monoidal $\infty$-categories,
let $$\Fun^{\ot,\L}(\mV,\mW) \subset\Fun^\ot(\mV,\mW) \subset \Alg_\mV(\mW)$$ be the full subcategories of monoidal functors (that induce on underlying $\infty$-categories a left adjoint).
\end{notation}
 
\begin{definition}\label{monob} Let $\mD$ be an $\infty$-category that admits finite products. 
\begin{enumerate}
\item A monoid object in $\mD$ is a functor $\beta: \Ass \to \mD$ such that for every $\n \geq 0$ the family $[\n]\to [1]$ of all inert morphisms in $\Ass$ induces an equivalence $\beta([\n]) \to \beta([1])^{\times\n}.$

\item Let $\beta$ be a monoid object in $\mD.$ A left $\beta$-action object in $\mD$ is a map $\alpha \to \beta $ of functors $ \Ass \to \mD$ such that 
for any $\n \geq 0$ the morphism $\{\n\}\subset [\n]$ in $\Delta$ induces an equivalence $\alpha([\n]) \to \beta([\n]) \times \alpha([0]).$

\end{enumerate} \end{definition}


\begin{example}\label{eopl}Monoid objects in $\Cat_\infty$ are precisely classified by monoidal $\infty$-categories.
\end{example} 

\begin{remark}\label{carst}

By \cite[Proposition 2.39.]{HEINE2023108941} for every $\infty$-category $\mD$ that admits finite products there is a monoidal $\infty$-category $\mD^\times \to \Ass$, the cartesian monoidal structure, whose tensor unit is the final object and whose tensor product is the product, and such that associative algebras in $\mD^\times$ precisely correspond to monoid objects in $\mD$. Moreover if $\beta$ is a monoid object in $\mD$ corresponding to an associative algebra $\beta'$ in $\mD^\times,$ by \cite[Remark 3.28.]{HEINE2023108941} a left $\beta$-action object in $\mD$ corresponds to a map $ \gamma \to \beta'$ of functors $ \Ass \to \mD^\times$ that lies over the map $(-)^{\triangleright}\to \id $ of functors $ \Ass \to \Ass$ such that 
$\gamma$ is a map of cocartesian fibrations relative to the collection of inert morphisms preserving the minimum.
\end{remark}
 
\vspace{2mm}
Next we define weakly bienriched $\infty$-categories \cite[Definition 3.3.]{HEINE2023108941}:

\begin{definition}\label{bla}
Let $\mV^\ot \to \Ass, \mW^\ot \to \Ass$ be $\infty$-operads.
An $\infty$-category weakly bienriched in $\mV, \mW$
is a map $\phi=(\phi_1,\phi_2): \mM^\circledast \to \mV^\ot \times \mW^\ot $ of cocartesian fibrations relative to the collection of inert morphisms of $\Ass \times \Ass$ whose first component preserves the maximum and whose second component preserves the minimum such that the following conditions hold:
\begin{enumerate}
\item for every $\n,\m \geq 0$ the map $[0]\simeq \{\n\} \subset [\n]$ in the first component and the map $[0]\simeq \{0\} \subset [\m]$ in the second component induce an equivalence
$$ \theta: \mM^\circledast_{[\n][\m]} \to \mV^\ot_{[\n]} \times  \mM^\circledast_{[0][0]} \times\mW^\ot_{[\m]},$$
\vspace{1mm}
\item for every $\X,\Y \in \mM^\circledast$ lying over $([\m'], [\n']), ([\m], [\n]) \in \Ass \times \Ass$ the cocartesian lift $\Y \to \Y'$ of the map $[0]\simeq \{\m\} \subset [\m]$ and $[0]\simeq \{0\} \subset [\n]$ induces a
pullback square
\begin{equation*} 
\begin{xy}
\xymatrix{
\mM^\circledast(\X,\Y)  \ar[d]^{} \ar[r]^{ }
& \mM^\circledast(\X,\Y') \ar[d]^{} 
\\ \mV^\ot(\phi_1(\X),\phi_1(\Y)) \times \mW^\ot(\phi_2(\X),\phi_2(\Y))  
\ar[r]^{}  & \mV^\ot(\phi_1(\X),\phi_1(\Y')) \times \mW^\ot(\phi_2(\X),\phi_2(\Y')). 
}
\end{xy} 
\end{equation*}
\end{enumerate}

\end{definition}

\begin{remark}
In \cite[Definition 3.3.]{HEINE2023108941} we define weakly bienriched $\infty$-categories under the name weakly bitensored $\infty$-categories
and reserve the notion of weakly bienriched $\infty$-categories for a different structure of enrichment based on Gepner-Haugseng's model of enriched $\infty$-categories \cite{GEPNER2015575}.
In \cite{HEINE2023108941} we prove an equivalence between weakly bienriched $\infty$-categories and weakly bitensored $\infty$-categories
justifying our choice of terminology. Moreover we prefer to use this terminology since
we specialize the notion of weakly bienriched $\infty$-categories of Definition \ref{bla} to several notions of enrichment.	
\end{remark}

\begin{notation}

For every weakly bienriched $\infty$-category $\phi: \mM^\circledast \to \mV^\ot \times \mW^\ot$ we call $\mM:= \mM^\circledast_{[0][0]} $ the underlying $\infty$-category of $\phi$ and say that $\phi$ exhibits $\mM$ as weakly bienriched in $\mV,\mW$.

\end{notation}

\begin{example}
Let $\mV^\ot \to \Ass$ be an $\infty$-operad.
We write $\mV^\circledast \to \Ass \times \Ass$ for the pullback of $\mV^\ot \to \Ass$ along the functor $\Ass \times \Ass \to \Ass, \ ([\n],[\m]) \mapsto [\n]\ast [\m]$.
The two functors $\Ass \times \Ass \times [1] \to \Ass$ corresponding to the natural transformations $(-)\ast \emptyset \to (-)\ast (-), \ \emptyset \ast (-) \to (-)\ast (-)$ send the morphism $\id_{[\n],[\m]}, 0 \to 1$ to
an inert one and so give rise to functors
$\mV^\circledast \to \mV^\ot \times \Ass, \ \mV^\circledast \to \Ass \times \mV^\ot $ over $\Ass\times \Ass$. The resulting functor $\mV^\circledast \to \mV^\ot \times \mV^\ot$ is an $\infty$-category weakly bienriched in $\mV,\mV.$

\end{example}


\begin{example}
Let $\mM^\circledast \to \mV^\ot \times \mW^\ot$ be a weakly bienriched $\infty$-category and $\mN \subset \mM$ a full subcategory.
Let $\mN^\circledast \subset \mM^\circledast$ be the full subcategory spanned by all objects of $\mM^\circledast$ lying over some
$(\V,\W)\in \mV^\ot \times \mW^\ot$ corresponding some object of $\mN \subset \mM.$
The restriction $\mN^\circledast \subset \mM^\circledast \to \mV^\ot \times \mW^\ot$ is a weakly bienriched $\infty$-category, whose underlying $\infty$-category is $\mN$.
We call $\mN^\circledast \to \mV^\ot \times \mW^\ot$ the full weakly bienriched subcategory of $ \mM^\circledast $ spanned by $\mN.$

\end{example}

\begin{example}\label{euuz}
Let $\mM^\circledast \to \mV^\ot \times \mW^\ot$ be a weakly bienriched $\infty$-category and $\K$ an $\infty$-category.
\begin{itemize}
\item 
The functor $\K \times \mM^\circledast \to \mV^\ot \times \mW^\ot$ 	
is a weakly bienriched $\infty$-category that exhibits $\K \times \mM$
as weakly bienriched in $\mV,\mW.$

\item The pullback $(\mM^\K)^\circledast:= \mV^\ot \times \mW^\ot\times_{\Fun(\K,\mV^\ot \times \mW^\ot)} \Fun(\K,\mM^\circledast) \to \mV^\ot \times \mW^\ot$	along the diagonal functor $\mV^\ot \times \mW^\ot\to \Fun(\K,\mV^\ot \times \mW^\ot)$ is a weakly bienriched $\infty$-category that exhibits $\Fun(\K,\mM)$
as weakly bienriched in $\mV,\mW.$
\end{itemize}
\end{example}

\begin{definition}
An $\infty$-category weakly left enriched in $\mV$ is an $\infty$-category weakly bienriched in $\mV, \emptyset$.
An $\infty$-category weakly right enriched in $\mV$ is an $\infty$-category weakly bienriched in $ \emptyset, \mV$.

\end{definition}


\begin{notation}\label{mult}
Let $\mM^\circledast \to \mV^\ot \times \mW^\ot$ be a weakly bienriched $\infty$-category and $\V_1,..., \V_{\n} \in \mV, \ \X, \Y \in \mM, \W_1,...,\W_\m \in \mW$ for some $\n,\m \geq 0 $. Let
$$\Mul_{\mM}(\V_1,..., \V_\n,\X,\W_1,...,\W_\m; \Y)$$ be the full subspace of $\mM^\circledast(\Z,\Y)$ spanned by the morphisms $\Z \to \Y$ in $\mM^\circledast$ lying over the map $[0] \simeq \{0 \} \subset [\n],[0] \simeq \{\m \} \subset [\m]$ in $\Delta \times \Delta,$ where $\Z \in \mM_{[\n],[\m]}^\circledast \simeq \mV^{\times \n} \times \mM \times \mW^{\times\m}$ corresponds to $(\V_1,..., \V_\n,\X,\W_1,...,\W_\m) $.
\end{notation}

\begin{remark}For every $\infty$-operad $\mV^\ot \to \Ass$ the projection $\mM^\circledast:=\mV^\circledast \to \mV^\ot$ gives rise to an embedding
$\mV^\circledast(\Z,\Y) \to \mV^\ot(\Z,\Y) $
that restricts to an equivalence
$$\Mul_{\mM}(\V_1,..., \V_\n,\X,\W_1,...,\W_\m; \Y) \simeq \Mul_{\mV}(\V_1,..., \V_\n,\X,\W_1,...,\W_\m; \Y).$$
\end{remark}

\begin{definition}Let $\phi: \mM^\circledast \to \mV^\ot \times \mW^\ot $ be a weakly bienriched $\infty$-category.
\begin{enumerate}
\item We call $\phi$ locally small if the $\infty$-categories $\mV^\ot, \mW^\ot, \mM^\circledast$ are locally small.

\item We call $\phi$ small if $\phi$ is locally small and $\mM$ is small.

\item We call $\phi$ absolute small if $\phi$ is small and $\mV,\mW$ are small.

\end{enumerate}

\end{definition}

\begin{remark}Let $\phi: \mM^\circledast \to \mV^\ot \times \mW^\ot $ be a weakly bienriched $\infty$-category.
The first axiom of Definition \ref{bla} implies that $\phi$ is absolute small if and only if $ \mM^\circledast, \mV^\ot, \mW^\ot$ are small.
Remark \ref{oprer} and the axioms of Definition \ref{bla} imply that $\phi$ is locally small if and only if the multi-morphism spaces (Notation \ref{mult}) of $\mV^\ot \to \Ass, \mW^\ot \to \Ass, \phi $ are small.	So $\phi$ is small if and only if $\mM$ is small and the multi-morphism spaces of $\mV^\ot \to \Ass, \mW^\ot \to \Ass, \phi  $ are small.		
\end{remark}

\subsection{Tensors and cotensors}

Next we specialize to an important class of weakly bienriched $\infty$-categories, which is easier in nature.

The next definition is \cite[Definition 3.12.]{HEINE2023108941}:

\begin{definition}\label{leftten}
Let $\phi: \mM^\circledast \to \mV^\ot \times \mW^\ot$ be a weakly bienriched $\infty$-category.

\begin{enumerate}
\item We say that $\phi: \mM^\circledast \to \mV^\ot \times \mW^\ot$ exhibits $\mM$ as left tensored over $\mV$ if $\mV^\ot \to \Ass$ is a monoidal $\infty$-category and $\phi$ is a map of cocartesian fibrations over $\Ass$
via projection to the first factor.


\item We say that $\phi: \mM^\circledast \to \mV^\ot \times \mW^\ot$ exhibits $\mM$ as right tensored over $\mW$ if $\mW^\ot \to \Ass$ is a monoidal $\infty$-category and $\phi$ is a map of cocartesian fibrations over $\Ass$
via projection to the second factor.


\item We say that $\phi: \mM^\circledast \to \mV^\ot \times \mW^\ot$ exhibits $\mM$ as bitensored over $\mV,\mW$ if $\phi$ exhibits $\mM$ as left tensored over $\mV$ and right tensored over $\mW$.



\end{enumerate}

\end{definition} 

\begin{remark}

Let $\mV^\ot \to \Ass, \mW^\ot \to \Ass$ be monoidal $\infty$-categories and $\phi: \mM^\circledast \to \mV^\ot \times \mW^\ot$ a map of cocartesian fibrations over $\Ass \times \Ass.$
Then $\phi: \mM^\circledast \to \mV^\ot \times \mW^\ot$ exhibits $\mM$ as bitensored over $\mV,\mW$ if and only if condition (1) of Definition \ref{bla} holds. Condition (2) is then automatic. 
\end{remark}

We apply Definition \ref{leftten} in particular to the case that $\mV^\ot=\emptyset^\ot$ or $\mW^\ot=\emptyset^\ot$.

\begin{example}
	
Let $\mM^\circledast \to \mV^\ot $ be a (weakly) left tensored $\infty$-category
and $\mN^\circledast \to \mW^\ot $ a (weakly) right tensored $\infty$-category.
The functor $\mM^\circledast \times \mN^\circledast \to \mV^\ot \times \mW^\ot $ is a (weakly) bitensored $\infty$-category.



	
\end{example}

\begin{definition}
A left tensored, right tensored, bitensored $\infty$-category $\mM^\circledast \to \mV^\ot \times \mW^\ot$, respectively, is small if $\mM^\circledast, \mV^\ot, \mW^\ot$ are small.
	
\end{definition}

\begin{definition}Let $\kappa$ be a small regular cardinal. 
\begin{enumerate}
\item A left tensored $\infty$-category $ \mM^\circledast \to \mV^\ot \times \mW^\ot$ is compatible with $\kappa$-small colimits if $\mV^\ot \to \Ass$ is compatible with $\kappa$-small colimits, $\mM$ admits $\kappa$-small colimits, for every $\V \in \mV, \X \in \mM$ the functors $(-) \ot \X: \mV \to \mM, \V \ot (-): \mM \to \mM$ preserve $\kappa$-small colimits and for every $\W_1,...,\W_\m \in \mW$
for $\m \geq0$ and $\Y \in \mM$ the functor $\Mul_\mM(-,\W_1,...,\W_\m;\Y): \mM^\op \to \mS$ preserves $\kappa$-small limits.

\item A right tensored $\infty$-category $\mM^\circledast \to \mV^\ot \times \mW^\ot$ is compatible with $\kappa$-small colimits if $(\mM^\rev)^\circledast \to (\mW^\rev)^\ot \times (\mV^\rev)^\ot $ is compatible with $\kappa$-small colimits.

\item A bitensored $\infty$-category is compatible with $\kappa$-small colimits if it is a left and right tensored $\infty$-category compatible with $\kappa$-small colimits, i.e. $\mV^\ot \to \Ass, \mW^\ot \to \Ass$ are compatible with $\kappa$-small colimits, $\mM$ admits $\kappa$-small colimits and for every $\V \in \mV, \W \in \mW, \X \in \mM$ the functors $(-)\ot \X: \mV \to \mM, \X \ot (-): \mW \to \mM, \V \ot (-), (-) \ot \W: \mM \to \mM$ preserve $\kappa$-small colimits.



\end{enumerate}

\end{definition}

\begin{definition}Let $\kappa$ be a small regular cardinal.

\begin{enumerate}
\item A presentably left tensored $\infty$-category is a left tensored $\infty$-category  $\phi: \mM^\circledast \to \mV^\ot \times \mW^\ot$ compatible with small colimits such that $\mV,\mM$ are presentable.

\item A presentably right tensored $\infty$-category is a right tensored $\infty$-category  $\phi: \mM^\circledast \to \mV^\ot \times \mW^\ot$ compatible with small colimits such that $\mM, \mW$ are presentable.

\item A presentably bitensored $\infty$-category is a presentably left tensored and  presentably right tensored $\infty$-category.

\item A left tensored $\infty$-category $\mM^\circledast \to \mV^\ot \times \mW^\ot$ is $\kappa$-compactly generated if $\mV^\ot \to \Ass$ and $\mM$ are $\kappa$-compactly generated and the left $\mV$-action on $\mM$ restricts to a left $\mV^\kappa$-action on $\mM^\kappa$. 

\item A right tensored $\infty$-category $\mM^\circledast \to \mV^\ot \times \mW^\ot$ is $\kappa$-compactly generated if $\mW^\ot \to \Ass$ and $\mM$ are $\kappa$-compactly generated and the right $\mW$-action on $\mM$ restricts to a right $\mW^\kappa$-action on $\mM^\kappa$. 

\item A bitensored $\infty$-category $\mM^\circledast \to \mV^\ot \times \mW^\ot$ is $\kappa$-compactly generated if the left $\mV$-action and right $\mW$-action are
$\kappa$-compactly generated.

\end{enumerate}	
\end{definition}

By \cite[Proposition 7.15.]{Rune} every presentably left tensored, right tensored, bitensored $\infty$-category, respectively, is $\kappa$-compactly generated for some regular cardinal $\kappa$.


Next we generalize the notions of left, right and bitensored $\infty$-categories (compatible with small colimits) by introducing the notions of left and right tensors and conical colimits.

\begin{definition}\label{Defo}
	
Let $\mM^\circledast \to \mV^\ot\times \mW^\ot$ be a weakly bienriched $\infty$-category and $\V \in \mV, \W \in \mW, \X,\Y \in \mM$.
	
\begin{enumerate}

\item A multi-morphism $\psi \in \Mul_\mM(\V,\X;\Y)$ exhibits $\Y$ as the left tensor of $\V, \X$ if for every $\Z \in \mM, \V_1,...,\V_\n \in \mV, \W_1,...,\W_\m \in \mW$ for $\n,\m \geq 0 $ the following map is an equivalence:
\begin{equation*}\label{tenss}
\Mul_\mM(\V_1,...,\V_\n, \Y, \W_1,...,\W_\m; \Z) \to \Mul_{\mM}(\V_1,...,\V_\n,\V,\X, \W_1,...,\W_\m; \Z).
\end{equation*}
We write $\V \ot \X $ for $\Y$. 

\item A multi-morphism $\psi \in \Mul_\mM(\X,\W;\Y)$ exhibits $\Y$ as the right tensor of $\X, \W$ if for every $\Z \in \mM, \V_1,...,\V_\n \in \mV, \W_1,...,\W_\m \in \mW$ for $\n,\m \geq 0 $ the following map is an equivalence:
\begin{equation*}\label{tenss}
\Mul_\mM(\V_1,...,\V_\n, \Y, \W_1,...,\W_\m; \Z) \to \Mul_{\mM}(\V_1,...,\V_\n,\X,\W, \W_1,...,\W_\m; \Z).
\end{equation*}
We write $\X\ot\W $ for $\Y$. 


\vspace{1mm}


\item A multi-morphism $\tau \in \Mul_\mM(\V,\Y;\X)$ exhibits $\Y$ as the left cotensor of $\V, \X$ if for every $\Z \in \mM, \V_1,...,\V_\n \in \mV, \W_1,...,\W_\m\in \mW$ for $\n,\m \geq 0 $ the following map is an equivalence: $$\Mul_\mM(\V_1,...,\V_\n, \Z, \W_1,...,\W_\m; \Y) \to \Mul_{\mM}(\V,\V_1,...,\V_\n,\Z, \W_1,...,\W_\m; \X).
$$ 

We write $^{\V}\X$ for $\Y$.

\item A multi-morphism $\tau \in \Mul_\mM(\Y,\W;\X)$ exhibits $\Y$ as the right cotensor of $\X, \W$ if for every $\Z \in \mM, \V_1,...,\V_\n \in \mV, \W_1,...,\W_\m\in \mW$ for $\n,\m \geq 0 $ the following map is an equivalence: $$\Mul_\mM(\V_1,...,\V_\n, \Z, \W_1,...,\W_\m; \Y) \to \Mul_{\mM}(\V_1,...,\V_\n,\Z, \W_1,...,\W_\m, \W; \X).
$$ 

We write $\X^\W$ for $\Y$.

\end{enumerate}
\end{definition}

\begin{remark}\label{reuj}

Let $\mM^\circledast \to \mV^\ot\times \mW^\ot$ be a weakly bienriched $\infty$-category, $\V \in \mV, \W \in \mW$ and $ \X \in \mM$. 
If the respective left and right tensors exist,
there is a canonical equivalence
$ (\V \ot \X) \ot \W \simeq \V \ot (\X\ot \W).$
In this case we refer to $ (\V \ot \X) \ot \W \simeq \V \ot (\X\ot \W)$ as the bitensor of $\V,\X,\W.$
Similarly, if the respective left and right cotensors exist,
there is a canonical equivalence
$ (\X^\V)^\W \simeq (\X^\W)^\V$ and we refer to the latter object as the bicotensor of $\V,\X,\W.$
\end{remark}

\begin{definition}

A weakly bienriched $\infty$-category $\mM^\circledast \to \mV^\ot \times \mW^\ot$ admits
\begin{itemize}


\item left (co)tensors if for every object $\V \in \mV, \X \in \mM$ there is a left (co)tensor of $\V$ and $\X$.

\item right (co)tensors if for every object $\W \in \mW, \X \in \mM$ there is a right (co)tensor of $\X$ and $\W.$

\end{itemize}


\begin{example}
Every weakly bienriched $\infty$-category that is left, right, bitensored admits left tensors, right tensors, left and right tensors, respectively.	
	
\end{example}

\begin{remark}\label{locre}
Let $\phi: \mM^\circledast \to \mV^\ot \times \mW^\ot$ be a weakly bienriched $\infty$-category that admits left and right tensors.
Then $\phi$ is a locally cocartesian fibration. 
	
\end{remark}

\end{definition}

Next we define conical colimits and conical limits:


\begin{definition}Let $\mM^\circledast \to \mV^\ot \times \mW^\ot$ be a weakly bienriched $\infty$-category.
	
\begin{enumerate}
\item Let $\F: \K^\triangleright \to \mM$ be a functor. 
We say that $\F$ is a conical colimit diagram 
if for every $\V_1,...,\V_\n \in \mV, \W_1,...,\W_\m \in \mW$ for $\n,\m \geq 0$
and $\Y \in \mM$ the presheaf $\Mul_\mM(\V_1,...,\V_\n,-, \W_1,...,\W_\m;\Y)$ on $\mM$ sends $\F^\op: (\K^\op)^\triangleleft \to \mM^\op$ to a limit diagram.
We say that $\mM$ admits $\K$-indexed conical colimits if $\mM$ admits the conical colimit of every functor starting at $\K.$

\item Let $\F: \K^\triangleleft \to \mM$ be a functor. 
We say that $\F$ is a conical limit diagram 
if for every $\V_1,...,\V_\n \in \mV, \W_1,...,\W_\m \in \mW$ for $\n,\m \geq 0$
and $\X \in \mM$ the presheaf $\Mul_\mM(\V_1,...,\V_\n,\X, \W_1,...,\W_\m;-)$
on $\mM$ sends $\F$ to a limit diagram.
We say that $\mM$ admits $\K$-indexed conical limits if $\mM$ admits the conical limit of every functor starting at $\K.$

\end{enumerate}
\end{definition}

\begin{remark}\label{laulo}
Let $\mM^\circledast \to \mV^\ot \times \mW^\ot$ be a weakly bienriched $\infty$-category. 

\begin{enumerate}
\item Every conical colimit diagram is a colimit diagram.

\item If $\mM^\circledast \to \mV^\ot \times \mW^\ot$ admits left and right cotensors,
every colimit diagram is conical. 

\item Assume that $\mM^\circledast \to \mV^\ot \times \mW^\ot$ admits left tensors
and let $\G: \K \to \mM$ be a functor.
The colimit of $\G$ is conical if and only if forming left tensors preserves the colimit of $\G$ and for every $\W_1,...,\W_\m \in \mW$ for $\m \geq 0$
and $\Y \in \mM$ the presheaf $\Mul_\mM(-, \W_1,...,\W_\m;\Y)$ on $\mM$ preserves the limit of $\G^\op.$

\item Assume that $\mM^\circledast \to \mV^\ot \times \mW^\ot$ has left and right tensors and let $\G: \K \to \mM$ be a functor.
The colimit of $\G$ is conical if and only if forming left and right tensors preserves the colimit of $\G$.
\end{enumerate}
	
\end{remark}




\begin{definition}Let $\mK \subset \Cat_\infty$ be a full subcategory.
A weakly bienriched $\infty$-category $\phi: \mM^\circledast \to \mV^\ot \times \mW^\ot$ is compatible with $\mK$-indexed colimits if $\phi$ admits $\mK$-indexed conical colimits and for every $\X, \Y \in \mM$, $\V_1,...,\V_\n \in \mV, \W_1,...,\W_\m \in \mW$ for $\n,\m \geq0$ and $0 \leq \bi \leq \n, 0 \leq \bj \leq \m$ the presheaves 
$$  \Mul_\mM(\V_1,..,\V_\bi,-, \V_{\bi+1},...,\V_\n,\X,\W_1,...,\W_\m;\Y)$$$$ \Mul_\mM(\V_1,...,\V_\n,\X,\W_1,..,\W_\bj,-, \W_{\bj+1},...,\W_\m;\Y)$$ on $\mV,\mW$, respectively, preserve $\mK$-indexed limits.
	
\end{definition}

\subsection{Enriched functors}

Next we define maps of weakly bienriched $\infty$-categories.
The next definition is \cite[Definition 3.18., 3.19.]{HEINE2023108941}:
\begin{definition}\label{linmapp}
Let $\alpha: \mV^\ot \to \mV'^\ot,\beta: \mW^\ot \to \mW'^\ot$ be maps of $\infty$-operads and $\phi: \mM^\circledast \to \mV^\ot \times \mW^\ot, \phi':\mM'^\circledast \to \mV'^\ot \times \mW'^\ot $ weakly bienriched $\infty$-categories.
An enriched functor $\phi \to \phi'$ is a commutative square of $\infty$-categories over $\Ass \times \Ass$ 
\begin{equation*} 
\begin{xy}
\xymatrix{
\mM^\circledast  \ar[d]^{\phi} \ar[rr]^{\gamma}
&&\mM'^\circledast \ar[d]^{\phi'} 
\\ \mV^\ot \times \mW^\ot
\ar[rr]^{\alpha \times \beta}  && \mV'^\ot \times \mW'^\ot
}
\end{xy} 
\end{equation*}
such that $\gamma$ preserves cocartesian lifts of inert morphisms of $\Ass \times \Ass$ whose first component preserves the maximum and whose second component preserves the minimum.

\begin{itemize}
\item An enriched functor $\phi \to \phi'$ is an embedding if $\alpha, \beta, \gamma$ are fully faithful.	
 
\item An enriched functor $\phi \to \phi'$ is left (right) linear if it preserves left (right) tensors.

\item An enriched functor $\phi \to \phi'$ is linear if it is left and right linear.

\item An enriched functor $\phi \to \phi'$ is left $\mV$-enriched if $\alpha$ is the identity.

\item An enriched functor $\phi \to \phi'$ is right $\mW$-enriched if $\beta$ is the identity.

\item An enriched functor $\phi \to \phi'$ is $\mV,\mW$-enriched if it is left $\mV$-enriched and right $\mW$-enriched.

\item An enriched functor $\phi \to \phi'$ is left $\mV$-linear if it is left linear
and left $\mV$-enriched.

\item An enriched functor $\phi \to \phi'$ is right $\mW$-linear if it is right linear
and right $\mW$-enriched.

\item An enriched functor $\phi \to \phi'$ is $\mV,\mW$-linear if it is
left $\mV$-linear and right $\mW$-linear (or equivalently $\mV,\mW$-enriched and linear).


\end{itemize}
\end{definition}

\begin{notation}Let $\mM^\circledast \to \mV^\ot \times \mW^\ot, \mN^\circledast \to \mV'^\ot \times \mW'^\ot$ be weakly bienriched $\infty$-categories.


Let $$\Enr\Fun(\mM,\mN) \subset (\Fun(\mV^\ot, \mV'^\ot) \times \Fun(\mW^\ot, \mW'^\ot)) \times_{\Fun(\mM^\circledast,\mV'^\ot \times \mW'^\ot) } \Fun(\mM^\circledast,\mN^\circledast) $$ be the full subcategory of enriched functors.





\end{notation} 

\begin{notation}\label{enrfunc}
Let $\mM^\circledast \to \mV^\ot \times \mW^\ot, \mN^\circledast \to \mV^\ot \times \mW^\ot$ be weakly bienriched $\infty$-categories.

Let $$ \Enr\Fun_{\mV,\mW}(\mM,\mN) \subset \Fun_{\mV^\ot \times \mW^\ot}(\mM^\circledast,\mN^\circledast)$$ be the full subcategory of $\mV,\mW$-enriched functors.

\end{notation}

\begin{construction}

Let $\mM^\circledast \to \mV^\ot \times \mW^\ot, \mN^\circledast \to \mV^\ot \times \mW^\ot$ be weakly bienriched $\infty$-categories.

Restriction along the embedding $\mM \subset \mM^\circledast$
induces a functor
\begin{equation}\label{forgetol} \Enr\Fun_{\mV,\mW}(\mM,\mN) \to \Fun(\mM,\mN). \end{equation}
    
\end{construction}

\begin{lemma}\label{conserva}

Let $\mM^\circledast \to \mV^\ot \times \mW^\ot, \mN^\circledast \to \mV^\ot \times \mW^\ot$ be weakly bienriched $\infty$-categories.
The functor $$\Enr\Fun_{\mV,\mW}(\mM,\mN) \to \Fun(\mM,\mN)$$ of \ref{forgetol} is conservative.
    
\end{lemma}

\begin{proof}

Let $\F, \G: \mM^\circledast \to \mN^\circledast $ be $\mV,\mW$-enriched functors and $ \alpha$ a morphism $\F \to \G$ in $\Enr\Fun_{\mV,\mW}(\mM,\mN)$, i.e. a natural transformation 
$\F \to \G$ of functors $\mM^\circledast \to \mN^\circledast$ over
$\mV^\ot \times \mW^\ot$.
Let $\X \in \mM^\circledast$ lying over $([\n],[\m]) \in \Ass \times \Ass$. By \cref{bla} there is a cocartesian lift $\f: \X \to \X'$ in $\mM^\circledast$ of the morphism in $\Ass \times \Ass $ whose first component is the map $[0]\simeq \{\n\} \subset [\n]$ in $\Delta$ and whose second component is the map $[0]\simeq \{0\} \subset [\m]$ in $\Delta$. In particular, $\X' \in \mM.$
By definition the enriched functors $\F,\G$ send the cocartesian lift $\f: \X \to \X'$ to a morphism cocartesian over $\Ass \times \Ass$.
Hence in the commutative square
$$\begin{xy}
\xymatrix{
\F(\X) \ar[d]^{\F(\f)} \ar[rr]^{\alpha_\X}
&&\G(\X) \ar[d]^{\G(\f)}
\\ \F(\X')
\ar[rr]^{\alpha_{\X'}}  && \G(\X')
}
\end{xy} $$
both vertical morphisms are cocartesian over $\Ass \times \Ass.$
Therefore the morphism $\alpha_{\X'}$ is the image of $\alpha_\X$ under the induced functor $\mN^\circledast_{[\n],[\m]} \to \mN^\circledast_{[0],[0]} =\mN.$
Thus $\alpha_\X$ is sent by the 
induced functor \begin{equation}\label{erfmu}
\mN^\circledast_{[\n],[\m]} \to \mV_{[n]}^{\ot} \times \mN \times \mW_{[m]}^{\ot}\end{equation} to the triple $(\id_\Y,\alpha_{\X'}, \id_\Z),$ where $\Y \in \mV^\ot, \Z \in \mW^\ot $ are the images of $\X$.

By \cref{bla} (1) the induced functor \ref{erfmu} is an equivalence.
Hence $\alpha_\X: \F(\X) \to \G(\X) $ is an equivalence
if $\alpha_{\X'}: \F(\X') \to \G(\X') $ is an equivalence.

\end{proof}

\begin{notation}Let $\mM^\circledast \to \mV^\ot \times \mW^\ot, \mN^\circledast \to \mV^\ot \times \mW^\ot$ be weakly bienriched $\infty$-categories.

\begin{enumerate}

\item Let $$\L\LinFun_{\mV, \mW}(\mM,\mN) \subset \Enr\Fun_{\mV, \mW}(\mM,\mN),$$$$ \R\LinFun_{\mV, \mW}(\mM,\mN)\subset \Enr\Fun_{\mV, \mW}(\mM,\mN),$$$$ \LinFun_{\mV, \mW}(\mM,\mN) \subset \Enr\Fun_{\mV, \mW}(\mM,\mN)$$ be the full subcategories of left linear, right linear, linear $\mV,\mW$-enriched functors, respectively.

\item 
Let $$ \LinFun^\L_{\mV, \mW}(\mM,\mN) \subset \LinFun_{\mV, \mW}(\mM,\mN)$$ be the full subcategory of $\mV,\mW$-linear functors whose underlying functor admits a right adjoint.

\end{enumerate}
\end{notation} 

For the next remark we use the notation of Example \ref{euuz}:

\begin{remark}\label{2-cat}
Let $\mM^\circledast \to \mV^\ot \times \mW^\ot, \mN^\circledast \to \mV^\ot \times \mW^\ot$ be weakly bienriched $\infty$-categories and $\K$ an $\infty$-category.
The canonical equivalences $$ \Fun(\K,\Fun_{\mV^\ot\times\mW^\ot}(\mM^\circledast,\mN^\circledast)) \simeq \Fun_{\mV^\ot\times\mW^\ot}(\K\times \mM^\circledast,\mN^\circledast)
\simeq\Fun_{\mV^\ot\times\mW^\ot}(\mM^\circledast,(\mN^\K)^\circledast)$$
restrict to equivalences
$$ \Fun(\K,\Enr\Fun_{\mV,\mW}(\mM,\mN)) \simeq \Enr\Fun_{\mV,\mW}(\K\times\mM,\mN)
\simeq\Enr\Fun_{\mV,\mW}(\mM,\mN^\K).$$

Let $ \mN^\circledast \to \mV'^\ot \times \mW'^\ot$ be a weakly bienriched $\infty$-category. The canonical equivalence $$ \Fun(\K,(\Fun(\mV^\ot, \mV'^\ot) \times \Fun(\mW^\ot, \mW'^\ot)) \times_{\Fun(\mM^\circledast,\mV'^\ot \times \mW'^\ot) } \Fun(\mM^\circledast,\mN^\circledast)) \simeq$$$$ (\Fun(\K \times \mV^\ot, \mV'^\ot) \times \Fun(\K \times \mW^\ot, \mW'^\ot)) \times_{\Fun(\K \times \mM^\circledast,\mV'^\ot \times \mW'^\ot) } \Fun(\K \times \mM^\circledast,\mN^\circledast)$$
restricts to an equivalence
$$ \Fun(\K,\Enr\Fun(\mM,\mN)) \simeq \Enr\Fun(\K\times\mM,\mN).$$

\end{remark}

We will often use the following proposition, which is \cite[Proposition 4.2.4.2.]{lurie.higheralgebra}:

\begin{proposition}\label{Line}
For every bitensored $\infty$-category $\mM^\circledast \to \mV^\ot \times \mW^\ot$
and $\infty$-category $\K$ the forgetful functor
$$\LinFun_{\mV, \mW}(\mV \times \K \times \mW,\mM) \to \Fun(\K,\mM)$$ is an equivalence. 
	
\end{proposition}


\begin{lemma}\label{colas}

Let $\mV^\ot \to \Ass, \mW^\ot \to \Ass$ be $\infty$-operads, $\mM^\circledast \to \mV^\ot \times \mW^\ot, \mN^\circledast \to \mV^\ot \times \mW^\ot$ weakly bienriched $\infty$-categories such that $\mN$ admits left and right tensors and $\mK \subset \Cat_\infty$ a full subcategory. 
\begin{enumerate}
\item If $\mN$ admits $\mK$-indexed conical colimits, 
the $\infty$-categories $$ \Enr\Fun_{\mV,\mW}(\mM, \mN), \ \L\LinFun_{\mV,\mW}(\mM, \mN), \ \R\LinFun_{\mV,\mW}(\mM, \mN), \ \LinFun_{\mV,\mW}(\mM, \mN)$$ admit $\mK$-indexed colimits and the forgetful functors $$ \Enr\Fun_{\mV,\mW}(\mM, \mN) \to \Fun(\mM, \mN), \ \L\LinFun_{\mV,\mW}(\mM, \mN) \to \Fun(\mM, \mN),$$$$ \R\LinFun_{\mV,\mW}(\mM, \mN) \to \Fun(\mM, \mN), \ \LinFun_{\mV,\mW}(\mM, \mN) \to \Fun(\mM, \mN)$$ 
and embeddings $$ \LinFun_{\mV,\mW}(\mM, \mN) \subset  \Enr\Fun_{\mV,\mW}(\mM, \mN),$$ $$\L\LinFun_{\mV,\mW}(\mM, \mN) \subset  \Enr\Fun_{\mV,\mW}(\mM, \mN), $$$$ \R\LinFun_{\mV,\mW}(\mM, \mN) \subset \Enr\Fun_{\mV,\mW}(\mM, \mN)$$ preserve $\mK$-indexed colimits.

\item If for every $\V \in \mV, \W \in \mW$ the functors $\V \ot (-), (-) \ot \W: \mN \to \mN$ are accessible, $\mM^\circledast \to \mV^\ot \times \mW^\ot$ is absolute small and $\mN$ is accessible, then $$ \Enr\Fun_{\mV,\mW}(\mM, \mN), \ \L\LinFun_{\mV,\mW}(\mM, \mN), \ \R\LinFun_{\mV,\mW}(\mM, \mN)$$ are accessible.

\item If for every $\V \in \mV, \W \in \mW$ the functors $\V \ot (-), (-) \ot \W: \mN \to \mN$ preserve small colimits, $\mM^\circledast \to \mV^\ot \times \mW^\ot$ is absolute small and $\mN$ is presentable, then $$ \Enr\Fun_{\mV,\mW}(\mM, \mN), \ \L\LinFun_{\mV,\mW}(\mM, \mN), \ \R\LinFun_{\mV,\mW}(\mM, \mN)$$ are presentable and the embeddings $$ \LinFun_{\mV,\mW}(\mM, \mN) \subset  \Enr\Fun_{\mV,\mW}(\mM, \mN), $$$$ \L\LinFun_{\mV,\mW}(\mM, \mN) \subset  \Enr\Fun_{\mV,\mW}(\mM, \mN), $$$$ \R\LinFun_{\mV,\mW}(\mM, \mN) \subset  \Enr\Fun_{\mV,\mW}(\mM, \mN)$$ admit right adjoints.


\end{enumerate}

\end{lemma}

\begin{proof}
This follows from \cite[Proposition 5.4.7.11., Remark 5.4.7.13.]{lurie.HTT} in view of Remark \ref{locre}.
Moreover the last part of (1) follows from the first part of (1).
(3) follows from (1) and (2).
	
\end{proof}

\begin{notation}

Let $$\omega\B\Enr \subset (\Op_{\infty} \times \Op_{\infty}) \times_{ \Cat_{\infty / \Ass \times \Ass} } \Fun([1], \Cat_{\infty / \Ass \times \Ass}) $$ be the subcategory of weakly bienriched $\infty$-categories.
\end{notation}

\begin{notation}
Evaluation at the target restricts to a forgetful functor $\omega\B\Enr \to \Op_\infty \times \Op_\infty$ 
whose fibers over $\infty$-operads 
$\mV^\ot \to \Ass, \mW^\ot \to \Ass$ we denote by $ _\mV\omega\B\Enr_{\mW}.$
\end{notation}

\begin{remark}\label{fori}

There is a canonical equivalence
$$ \omega\B\Enr_{\emptyset, \emptyset} \simeq \Cat_\infty, \ \mM^\circledast \to \emptyset^\ot \times \emptyset^\ot \mapsto \mM^\circledast$$
inverse to the functor $\K \mapsto \emptyset^\ot \times \K \times \emptyset^\ot \simeq \K.$

\end{remark}

\begin{remark}\label{2-catt}

There is a canonical left action of $\Cat_\infty$ on $\omega\B\Enr$
that sends $\K, \mM^\circledast \to \mV^\ot \times \mW^\ot$ to
$\K \times \mM^\circledast \to \mV^\ot \times \mW^\ot.$
The forgetful functor $\omega\B\Enr \to \Op_\infty \times \Op_\infty$ is $\Cat_\infty$-linear, where the target carries the trivial action.
The forgetful functor $\omega\B\Enr \to \Cat_\infty$ is $\Cat_\infty$-linear, where the target carries the left action
induced by the cartesian product.
Thus for any small $\infty$-operads $\mV^\ot \to \Ass, \mW^\ot \to \Ass$ the fiber $_\mV \omega\B\Enr_\mW$ carries an induced left $\Cat_\infty$-action that admits the same description and the forgetful functor $_\mV \omega\B\Enr_\mW \to \Cat_\infty$ is left $\Cat_\infty$-linear, where the target carries the left action induced by the cartesian product. These actions are closed by Remark \ref{2-cat}.

	
\end{remark}

\begin{proposition}\label{pr} For every small $\infty$-operads $\mV^\ot \to \Ass, \mW^\ot \to \Ass$ the $\infty$-category ${_\mV\omega\B\Enr}_{\mW}$ is compactly generated.
\end{proposition}

\begin{proof}
Let $\Cat^{\max, \min}_{\infty/\mV^\ot \times \mW^\ot} \subset \Cat_{\infty/\mV^\ot \times \mW^\ot}$ be the subcategory of cocartesian fibrations relative to the collection of pairs of inert morphisms whose first component preserves the maximum and whose second component preserves the minimum, and functors over $\mV^\ot \times \mW^\ot$ preserving cocartesian lifts of inert morphisms whose first component preserves the maximum and whose second component preserves the minimum.
By definition ${_\mV\omega\B\Enr}_{\mW} \subset \Cat_{\infty/\mV^\ot \times \mW^\ot}$ is a full subcategory of $\Cat_{\infty/\mV^\ot \times \mW^\ot}^{\max, \min}$, which is closed under filtered colimits since filtered colimits commute with finite limits.
Let $\mathfrak{P}$ be the collection of functors $[1]= [0]^{\triangleleft} \to \mV^\ot$ selecting an inert morphism of $\mV^\ot \times \mW^\ot$ lying over a morphism of the form $[0] \simeq \{\n\} \subset [\n]$ in $\Delta$
in the first component and a morphism of the form $[0] \simeq \{0\} \subset [\n]$ in $\Delta$ in the second component. 
By definition ${_\mV\omega\B\Enr}_{\mW}$ is the full subcategory of $\Cat_{\infty/\mV^\ot \times \mW^\ot}^{\max, \min}$ spanned by the $\mathfrak{P}$-fibered objects in the sense of \cite[Definition B.0.19.]{lurie.higheralgebra}, which are the local objects of a localization \cite[Proposition B.2.9.]{lurie.higheralgebra}. 
By \cite[Theorem B.0.20.]{lurie.higheralgebra} the $\infty$-category $\Cat_{\infty/\mV^\ot \times \mW^\ot}^{\max, \min}$ is compactly generated and so also the accessible localization ${_\mV\omega\B\Enr}_{\mW}$ is compactly generated.

\end{proof}


\begin{notation}
Let $$ \LMod, \RMod, \BMod \subset \omega\B\Enr $$
be the subcategories of left tensored, right tensored, bitensored $\infty$-categories and enriched functors preserving left tensors, right tensors, left and right tensors, respectively.



\end{notation}

The next proposition is \cite[ Proposition 3.32.]{HEINE2023108941} whose short proof we include for completeness:

\begin{proposition}\label{bicaa}
The forgetful functor $$\gamma: \omega\B\Enr \to \Op_\infty \times \Op_\infty$$
is a cartesian fibration that restricts to a cartesian fibration $\BMod \to \Mon \times \Mon $ with the same cartesian morphisms.
An enriched functor 
corresponding to a commutative square
\begin{equation*} 
\begin{xy}
\xymatrix{
\mM^\circledast \ar[d] \ar[rr]^{\psi}
&&\mN^\circledast \ar[d] 
\\ \mV^\ot \times \mV'^\ot
\ar[rr]^{\alpha \times \alpha'}  && \mW^\ot \times \mW'^\ot
}
\end{xy} 
\end{equation*}
is $\gamma$-cartesian if and only if the square is a pullback square, i.e. if $\psi$ induces an equivalence
$\mM \simeq \mN$ and for any $\V_1,...,\V_\n \in \mV,\X,\Y \in \mM, \W_1,...,\W_\m\in \mW$ for $\n,\m \geq 0 $ the following map is an equivalence: $$\Mul_\mM(\V_1,..., \V_\n,\X,\W_1,...,\W_\m;\Y) \to \Mul_\mN(\alpha(\V_1),...,\alpha(\V_\n),\psi(\X), \alpha'(\W_1),...,\alpha'(\W_\m);\psi(\Y)).$$

\end{proposition}

\begin{proof}
Since $\Cat_{\infty / \Ass \times \Ass}$ admits pullbacks, the functor $\Fun([1], \Cat_{\infty / \Ass \times \Ass}) \to \Cat_{\infty / \Ass \times \Ass}$ evaluating at the target is a cartesian fibration whose cartesian morphisms are precisely the pullback squares in $\Cat_{\infty / \Ass \times \Ass}.$
Thus the pullback $(\Op_{\infty} \times \Op_{\infty}) \times_{ \Cat_{\infty / \Ass \times \Ass} } \Fun([1], \Cat_{\infty / \Ass \times \Ass})\to \Op_{\infty} \times \Op_{\infty}$ is a cartesian fibration,
which restricts to cartesian fibrations $\omega\B\Enr \to \Op_\infty \times \Op_\infty, \BMod \to \Mon \times \Mon$ with the same cartesian morphisms.	
	
\end{proof}

\begin{notation}\label{invo}
The category $\Ass=\Delta^\op$ carries a canonical involution sending $[\n] $ to $[\n]$ and a map $\f:[\n] \to [\m]$ to the map $[\n] \to [\m], \bi \mapsto \m-\f(\n-\bi).$ The involution on $\Ass$ induces an involution on
$\Cat_{\infty/\Ass}$ that restricts to an involution $(-)^\rev$ on $\Op_{\infty} \subset \Cat_{\infty/\Ass}.$ 
Moreover the involution on $\Ass$ induces an involution on
$\Ass \times \Ass$ by applying the involution on $\Ass$ to each factor and switching the factors, which induces an involution on
$(\Cat_{\infty / \Ass} \times \Cat_{\infty / \Ass}) \times_{ \Cat_{\infty / \Ass \times \Ass} } \Fun([1], \Cat_{\infty / \Ass \times \Ass}) $ that restricts to
involutions $(-)^\rev$ on $\omega\B\Enr$ and $\BMod$, under which $ \LMod$ corresponds to $ \RMod$.

\end{notation}


\begin{notation}
Let $ \cc\cc\BMod \subset \widehat{\BMod} $ be the subcategory of bitensored $\infty$-categories compatible with small colimits and linear functors preserving small colimits.\end{notation}


The next proposition follows from \cite[Proposition 8.29., Proposition 3.42]{HEINE2023108941}:

\begin{proposition}\label{laan}
Let $\psi: \mM^\circledast \to \mN^\circledast $ be an enriched functor from an absolute small to a locally small weakly bienriched $\infty$-category lying over maps of $\infty$-operads $\alpha: \mV^\ot \to \mV'^\ot, \beta: \mW^\ot \to \mW'^\ot$.
Let $\mO^\circledast \to \mV''^\ot \times \mW''^\ot$ be a bitensored $\infty$-category compatible with small colimits and $\alpha': \mV'^\ot \to \mV''^\ot, \beta': \mW'^\ot \to \mW''^\ot$ maps of $\infty$-operads.
If $ \alpha_!: \Alg_{\mV}(\mV'') \to \Alg_{\mV'}(\mV'')$ 
sends $\alpha' \circ \alpha$ to $\alpha'$ and $ \beta_!: \Alg_{\mW}(\mW'') \to \Alg_{\mW'}(\mW'')$ sends $\beta' \circ \beta$ to $\beta',$
the induced functor $$\Enr\Fun_{\mV',\mW'}(\mN,(\alpha', \beta')^*(\mO)) \to \Enr\Fun_{\mV, \mW}(\mM,(\alpha, \beta)^*((\alpha', \beta')^*(\mO))) $$ admits a left adjoint $\psi_!$, which is fully faithful if $\psi$ is an embedding.
For every $\mV,\mW$-enriched functor $\F: \mM^\circledast \to (\alpha, \beta)^*((\alpha', \beta')^*(\mO))$ and $\X \in \mN$ there is a caonical equivalence
$$ \hspace{8mm}\psi_!(\F)(\X) \simeq \colim_{\V_1,..., \V_\n, \psi(\Y), \W_1,..., \W_{\m} \to \X}\ \bigotimes_{\bi=1}^\n \alpha(\V_\bi) \ot \F(\Y) \ot \bigotimes_{\bj=1}^\m \beta(\W_\bj).$$

\end{proposition} 

\subsection{Enriched adjunctions}


	
	

Next we define enriched adjunctions.

\begin{definition}
An enriched adjunction is an adjunction in the $(\infty,2)$-category $\omega\B\Enr$ of Remark \ref{2-catt}.	
	
\end{definition}

\begin{definition}\label{ajlt} Let $\mV^\ot \to \Ass, \mW^\ot \to \Ass$ be small $\infty$-operads. An $\mV,\mW$-enriched adjunction is an adjunction in the $(\infty,2)$-category $_\mV\omega\B\Enr_\mW$	of Remark \ref{2-catt}.	
	
\end{definition}

\begin{remark}\label{enrra}
	
Let $\F: \mM^\circledast \to \mN^\circledast$ be an enriched functor lying over maps of $\infty$-operads $\alpha: \mV^\ot \to \mV'^\ot, \beta: \mW^\ot \to \mW'^\ot$
and $ \G: \mN^\circledast \to \mM^\circledast$ an enriched functor lying over maps of $\infty$-operads $\gamma: \mV'^\ot \to \mV^\ot, \delta: \mW'^\ot \to \mW^\ot$
and $\eta: \id \to \G \circ \F$ a morphism in $\Enr\Fun(\mM,\mM)$
lying over maps $\eta': \id \to \gamma \circ \alpha, \eta'': \id \to \delta \circ \beta$.
The morphism $\eta: \id \to \G \circ \F$ exhibits $\F$ as an enriched left adjoint of $\G$ (or $\G$ as an enriched right adjoint of $\F$) if $\eta$ exhibits
$\F$ as a left adjoint of $\G$ and $\eta'$ exhibits $\alpha$ as a left adjoint of $\gamma$ relative to $\Ass$ and $\eta''$ exhibits $\beta$ as a left adjoint of $\delta$ relative to $\Ass$.	
If $\alpha, \beta, \eta', \eta''$ are the identities, a morphism $\eta: \id \to \G \circ \F$ 
exhibits $\F$ as a $\mV,\mW$-enriched left adjoint of $\G$ if and only if it exhibits $\F$ as a left adjoint of $\G$ relative to $\mV^\ot \times \mW^\ot.$
The analogous holds for the counit.
\end{remark}

\begin{notation}Let $\mM^\circledast \to \mV^\ot \times \mW^\ot, \mN^\circledast \to \mV^\ot \times \mW^\ot$ be weakly bienriched $\infty$-categories.
Let
$$\Enr\Fun^\L_{\mV, \mW}(\mM,\mN), \Enr\Fun^\R_{\mV, \mW}(\mM,\mN) \subset \Enr\Fun_{\mV, \mW}(\mM,\mN)$$ be the full subcategories of $\mV, \mW$-enriched functors that admit a $\mV,\mW$-enriched left (right) adjoint.
	
\end{notation}

\begin{remark}
Stefanich \cite[Definition 4.3.2.]{stefanich2020presentable} introduces conical colimits in the model of Gepner-Haugseng. 
\end{remark}

\begin{remark}
The notion of enriched adjunction of Definition \ref{ajlt} is also discussed in \cite[§6.2.6.]{heine2019restricted} and \cite[A.6.]{heine2021real}.
Stefanich \cite[§4.1., §4.2.]{stefanich2020presentable} studies enriched adjunctions in the model of Gepner-Haugseng.
	
\end{remark}

Remark \ref{enrra} implies the following:
\begin{remark}\label{enradj}
Let $\mM^\circledast \to \mV^\ot \times \mW^\ot, \mN^\circledast \to \mV'^\ot \times \mW'^\ot$ be weakly bienriched $\infty$-categories. 
	
\begin{enumerate}
\item 
An enriched functor $ \G: \mN^\circledast \to \mM^\circledast$ lying over maps of $\infty$-operads $\gamma: \mV'^\ot \to \mV^\ot, \delta: \mW'^\ot \to \mW^\ot$
admits an enriched left adjoint if $\gamma, \delta$ admit left adjoints $\alpha, \beta$ relative to $\Ass$, respectively, and for every $\X \in \mM$ there is an $\Y \in \mN$
and a morphism $\X \to \G(\Y) $ in $\mM$ such that for any $\V_1,...,\V_\n \in \mV, \W_1,..., \W_\m \in \mW$ for $\n, \m \geq 0$ and $\Z \in \mN$ the following map is an equivalence:
$$ \Mul_\mN(\alpha(\V_1),...,\alpha(\V_\n),\Y, \beta(\W_1),..., \beta(\W_\m);\Z) \to$$$$ \Mul_\mM(\gamma(\alpha(\V_1)),...,\gamma(\alpha(\V_\n)),\G(\Y), \delta(\beta(\W_1)),..., \delta(\beta(\W_\m));\G(\Z))\to $$$$ \Mul_\mM(\V_1,...,\V_\n,\X, \W_1,..., \W_\m;\G(\Z)).$$

\item An enriched functor $ \F: \mM^\circledast \to \mN^\circledast$ lying over maps of $\infty$-operads $\alpha: \mV^\ot \to \mV'^\ot, \beta: \mW^\ot \to \mW'^\ot$
admits an enriched right adjoint if $\alpha, \beta$ admit right adjoints $\gamma, \delta$ relative to $\Ass$, respectively, and for every $\Y \in \mN$ there is an $\X \in \mM$
and a morphism $\F(\X) \to \Y $ in $\mN$ such that for any $\V_1,...,\V_\n \in \mV, \W_1,..., \W_\m \in \mW$ for $\n, \m \geq 0$ and $\Z \in \mM$ the following map is an equivalence:
$$\hspace{12mm}\Mul_\mM(\V_1,...,\V_\n,\Z, \W_1,..., \W_\m;\X) \to \Mul_\mN(\alpha(\V_1),...,\alpha(\V_\n),\F(\Z), \beta(\W_1),..., \beta(\W_\m);\F(\X))$$$$ \to \Mul_\mN(\alpha(\V_1),...,\alpha(\V_\n),\F(\Z), \beta(\W_1),..., \beta(\W_\m);\Y).$$
\end{enumerate}
\end{remark}

\begin{lemma}\label{Adj}
Let $\G: \mN^\circledast \to \mM^\circledast$ be an enriched functor between weakly bienriched $\infty$-categories that admit left and right tensors that lies over maps of $\infty$-operads $\gamma: \mV'^\ot \to \mV^\ot, \delta: \mW'^\ot \to \mW^\ot$ that admit left adjoints $\alpha, \beta$ relative to $\Ass$, respectively.
Then $\G: \mN^\circledast \to \mM^\circledast$ admits an enriched left adjoint if and only if the underlying functor $\mN \to \mM$ admits a left adjoint
$\F:\mM \to \mN$ and for every $\V_1,...,\V_\n \in \mV, \W_1,...,\W_\m \in \mW$
for $\n,\m \geq 0$ and $\X \in \mM$ the following morphism is an equivalence: $$\F(\V_1 \ot ...\ot \V_\n \ot \X \ot \W_1 \ot ...\ot \W_\m) \to \alpha(\V_1) \ot ...\ot \alpha(\V_\n) \ot \F(\X)\ot \beta(\W_1) \ot ...\ot \beta(\W_\m).$$ 
	
\end{lemma}

\begin{proof}
For every $\V_1,...,\V_\n \in \mV, \W_1,..., \W_\m \in \mW$ for $\n, \m \geq 0$ and $\Z \in \mN$ the map
$$ \Mul_\mN(\alpha(\V_1),...,\alpha(\V_\n),\F(\X), \beta(\W_1),..., \beta(\W_\m);\Z) \to$$$$ \Mul_\mM(\gamma(\alpha(\V_1)),...,\gamma(\alpha(\V_\n)),\G(\F(\X)), \delta(\beta(\W_1)),..., \delta(\beta(\W_\m));\G(\Z))\to $$$$ \Mul_\mM(\V_1,...,\V_\n,\X, \W_1,..., \W_\m;\G(\Z)).$$
identifies with the composition
$$ \mN(\alpha(\V_1) \ot ...\ot \alpha(\V_\n)\ot\F(\X)\ot \beta(\W_1)\ot...\ot \beta(\W_\m),\Z) \to \mN(\F(\V_1\ot...\ot\V_\n\ot\X\ot \W_1\ot...\ot \W_\m),\Z) $$$$\simeq \mM(\V_1 \ot...\ot\V_\n\ot \X\ot \W_1\ot...\ot \W_\m,\G(\Z)). $$	
	
\end{proof}

\begin{lemma}\label{Adj2}
Let $\F: \mM^\circledast \to \mN^\circledast$ be an enriched 
functor that starts at a weakly bienriched $\infty$-category that admits left and right tensors and lies over maps of $\infty$-operads $\alpha: \mV^\ot \to \mV'^\ot, \beta: \mW^\ot \to \mW'^\ot$ that admit right adjoints $\gamma, \delta$ relative to $\Ass$, respectively.
Then $\F: \mM^\circledast \to \mN^\circledast$ admits an 
enriched right adjoint if and only if the underlying functor $\mM \to \mN$ admits a right adjoint $\G:\mN \to \mM$ and $\F$ is linear.
\end{lemma}

\begin{proof}
For every $\V_1,...,\V_\n \in \mV, \W_1,..., \W_\m \in \mW$ for $\n, \m \geq 0$ and $\Z \in \mN$ the map
$$ \Mul_\mM(\V_1,...,\V_\n,\X, \W_1,..., \W_\m;\G(\Z)) \to 
\Mul_\mN(\alpha(\V_1),...,\alpha(\V_\n),\F(\X), \beta(\W_1),..., \beta(\W_\m);\F(\G(\Z))) \to$$$$ \Mul_\mN(\alpha(\V_1),...,\alpha(\V_\n),\F(\X), \beta(\W_1),..., \beta(\W_\m);\Z)$$
identifies with the composition
$$ \mM(\V_1 \ot...\ot\V_\n\ot \X\ot \W_1\ot...\ot \W_\m,\G(\Z))
\simeq \mN(\F(\V_1\ot...\ot\V_\n\ot\X\ot \W_1\ot...\ot \W_\m),\Z) \to $$$$ \mN(\alpha(\V_1) \ot ...\ot \alpha(\V_\n)\ot\F(\X) \ot \beta(\W_1)\ot...\ot \beta(\W_\m),\Z).$$	
	
\end{proof}

\begin{notation}
	
Let $\mV^\ot \to \Ass,\mW^\ot\to\Ass$ be small $\infty$-operads.
Let ${_\mV\omega\B\Enr^\L_\mW}, {_\mV \omega\B\Enr^\R_\mW} \subset {_\mV\omega\B\Enr_\mW}  $ be the subcategories with the same objects and with morphisms the $\mV,\mW$-enriched functors that admit a $\mV,\mW$-enriched left adjoint, right adjoint, respectively.
\end{notation}

\begin{lemma}\label{kurt}Let $\mV^\ot \to \Ass,\mW^\ot\to\Ass$ be small $\infty$-operads.
There is a canonical equivalence $${_\mV\omega\B\Enr^\L_\mW} \simeq ({_\mV\omega\B\Enr^\R_\mW})^\op .$$
\end{lemma}
\begin{proof}
	
For every small $\infty$-category $\rS$ let $\Cat_{\infty / \rS}^\L, \Cat_{\infty / \rS}^\R \subset \Cat_{\infty / \rS}$ be the wide subcategories whose morphisms are the functors over $\rS$ that admit a left, right adjoint relative to $\rS$, respectively.
There is an equivalence $\Cat_{\infty / \rS}^\L \simeq (\Cat_{\infty / \rS}^\R)^\op:$
a functor $\mB \to \Cat_{\infty / \rS}$ is classified by a map $\mX \to \rS\times \mB$ of cocartesian fibrations over $\mB$, which is a map of bicartesian fibrations over $\mB$
if and only if the functor $\mB \to \Cat_{\infty / \rS}$ lands in $\Cat^\L_{\infty / \rS} $
\cite[Corollary 5.2.2.5.]{lurie.HTT}.
A functor $\mB^\op \to \Cat_{\infty / \rS}$ is classified by a map $\mY \to \rS\times \mB$ of cartesian fibrations over $\mB$, which is a map of bicartesian fibrations over $\mB$
if and only if the functor $\mB \to \Cat_{\infty / \rS}$ lands in $\Cat^\R_{\infty / \rS} $
\cite[Corollary 5.2.2.5.]{lurie.HTT}.
So there is an equivalence
$ \Fun(\mB, \Cat_{\infty / \rS})^\simeq \simeq \Fun(\mB^\op, \Cat_{\infty / \rS}^\R)^\simeq \simeq \Fun(\mB, (\Cat_{\infty / \rS}^\R)^\op)^\simeq$
representing an equivalence $\Cat_{\infty / \rS}^\L \simeq (\Cat_{\infty / \rS}^\R)^\op.$
The equivalence $\Cat_{\infty / \mV^\ot \times \mW^\ot}^\L \simeq (\Cat_{\infty / \mV^\ot \times \mW^\ot}^\R)^\op$ restricts to an equivalence
${_\mV\omega\B\Enr^\L_\mW} \simeq ({_\mV\omega\B\Enr^\R_\mW})^\op.$
	
\end{proof} 

\begin{remark}\label{indadj}
	
Let $\mO^\circledast \to \mV^\ot \times \mW^\ot$ be a weakly bienriched $\infty$-category, $\F: \mM^\circledast \to \mN^\circledast, \G: \mN^\circledast \to \mM^\circledast$ be $\mV,\mW$-enriched functors 
and $\eta: \id \to \G \circ \F$ a morphism in $\Enr\Fun_{\mV,\mW}(\mM,\mM)$
that exhibits $\F$ as a $\mV,\mW$-enriched left adjoint of $\G$. 
Then $$\Enr\Fun_{\mV,\mW}(\eta,\mO) : \Enr\Fun_{\mV,\mW}(\F,\mO) \circ \Enr\Fun_{\mV,\mW}(\G,\mO) \to \id$$
exhibits $ \Enr\Fun_{\mV,\mW}(\G,\mO)$ as a left adjoint of $\Enr\Fun_{\mV,\mW}(\F,\mO)$ and $$\Enr\Fun_{\mV,\mW}(\mO,\eta) : \id \to \Enr\Fun_{\mV,\mW}(\mO,\G) \circ \Enr\Fun_{\mV,\mW}(\mO,\F) $$
exhibits $ \Enr\Fun_{\mV,\mW}(\mO,\F)$ as a left adjoint of $\Enr\Fun_{\mV,\mW}(\mO,\G)$. 
	
\end{remark}

\subsection{Trivial weak enrichment}

\begin{definition}Let $\mV^\ot \to \Ass, \mW^\ot \to \Ass$ be $\infty$-operads.

\begin{enumerate}
\item A weakly bienriched $\infty$-category $\mM^\circledast \to \mV^\ot \times \mW^\ot$ 
is left trivial if for every $\n>0,\m \geq 0$ and $\V_1,...,\V_\n \in \mV, \W_1,...,\W_\m \in \mW,\X,\Y \in \mM$ the space $\Mul_{\mM}(\V_1,...,\V_\n,\X, \W_1,...,\W_\m;\Y)$ is empty.	
\item A weakly bienriched $\infty$-category $\mM^\circledast \to \mV^\ot \times \mW^\ot$ 
right trivial if for every $\n \geq 0,\m > 0$ and $\V_1,...,\V_\n \in \mV, \W_1,...,\W_\m \in \mW,\X,\Y \in \mM$ the space $\Mul_{\mM}(\V_1,...,\V_\n,\X, \W_1,...,\W_\m;\Y)$ is empty.	
\item A weakly bienriched $\infty$-category is trivial if it is left trivial and right trivial.
\end{enumerate}	

\end{definition}

\begin{remark}\label{tent}
A weakly bienriched $\infty$-category $\mM^\circledast \to \mV^\ot \times \mW^\ot$ is left trivial if and only if the functor $\mM^\circledast \to \mV^\ot$ factors through the subcategory $\triv_\mV^\circledast \subset \mV^\ot.$
And dually for right triviality.

\end{remark}	

\begin{remark}\label{Poll}
Let $\mV'^\ot \to \mV^\ot, \mW'^\ot \to \mW^\ot$ be maps of $\infty$-operads. For every left trivial, right trivial, trivial weakly bienriched $\infty$-category $\mM^\circledast \to \mV^\ot \times \mW^\ot$ the pullback
$\mV'^\ot \times_{\mV^\ot} \mM^\circledast \times_{\mW^\ot}\mW'^\ot \to \mV'^\ot \times \mW'^\ot$ is a left trivial, right trivial, trivial weakly bienriched $\infty$-category, respectively.

\end{remark}

Next we consider examples:

\begin{notation}

Let $\Ass_\mi,\Ass_\ma \subset\Ass$ be the subcategories with the same objects and with morphisms the inert morphisms preserving the minimum, maximum, respectively.
\end{notation}

\begin{example}
\begin{enumerate}
\item The inclusion $\Ass_\ma \subset \Ass$ exhibits $[0]$ as weakly left enriched in $[0]$, where the left enrichment is left trivial.

\item The inclusion $\Ass_\mi \subset \Ass$ exhibits $[0]$ as weakly right enriched in $[0]$, where the right enrichment is right trivial.


\end{enumerate}		
\end{example}

\begin{notation}Let $\mV^\ot \to \Ass, \mW^\ot \to \Ass$ be $\infty$-operads.
Let $$_\mV\triv^\circledast :=\mV^\ot \times_\Ass \Ass_\ma \to \mV^\ot, \ \triv_\mW^\circledast :=\mW^\ot \times_\Ass \Ass_\mi \to \mW^\ot, $$$$ \triv_{\mV,\mW}^\circledast :=\triv_\mV^\circledast \times \triv_\mW^\circledast 
\to \mV^\ot \times \mW^\ot.$$
\end{notation}




Remark \ref{Poll} gives the following example:

\begin{example}

Let $\mV^\ot \to \Ass, \mW^\ot \to \Ass$ be $\infty$-operads.
Then $_\mV\triv^\circledast \to \mV^\ot, \triv_\mW^\circledast \to \mW^\ot, \triv_{\mV,\mW}^\circledast\to \mV^\ot \times \mW^\ot$
are trivial weakly left, weakly right, weakly bienriched $\infty$-categories, respectively.

\end{example}

\begin{remark}\label{aii}

The object $[0]$ is a final object of the subcategories $\Ass_\ma, \Ass_\mi \subset \Ass$ 
since for every $\n \geq0$ there is a unique inert order preserving map $[0]\to[\n]$ 
preserving the maximum, minimum, respectively.	
The functors $_\mV\triv^\circledast \to \Ass_\ma, \triv_\mW^\circledast \to \Ass_\mi$ are cocartesian fibrations whose fiber over the final object is contractible. Thus the unique object of ${_\mV \triv^\circledast}, \triv_\mW^\circledast $ lying over the final object is a final object of $_\mV\triv^\circledast, \triv_\mW^\circledast,$ respectively.
\end{remark} 

\begin{definition}\label{Triv}
Let $\mV^\ot \to \Ass, \mW^\ot \to \Ass$ be $\infty$-operads
and $\K$ a small $\infty$-category.
The trivial weakly bienriched $\infty$-category on $\K$ is 
$$\K_{\mV, \mW}^\circledast:= \triv_\mV^\circledast \times \K \times \triv_\mW^\circledast \to \mV^\ot \times \mW^\ot.$$

\end{definition}

\begin{remark}
The weakly bienriched $\infty$-category $\K_{\mV, \mW}^\circledast \to \mV^\ot \times \mW^\ot$ is trivial and $\K_{\mV, \mW} \simeq \K$.
\end{remark}

\begin{proposition}\label{ljnbfg}
Let	$\mN^\circledast \to \mV^\ot \times \mW^\ot$ be a weakly bienriched $\infty$-category.
\begin{enumerate}
\item Let $\mM^\circledast \to \mV^\ot $ be a weakly left enriched $\infty$-category.
The following functor is an equivalence: $$\Enr\Fun_{\mV, \mW}(\mM \times \triv_\mW, \mN) \to \Enr\Fun_{\mV,\emptyset}(\mM, \mN).$$

\item Let $\mM^\circledast \to \mW^\ot $ be a weakly right enriched $\infty$-category.
The following functor is an equivalence: $$\Enr\Fun_{\mV, \mW}(_\mV\triv \times \mM, \mN) \to \Enr\Fun_{\emptyset, \mW}(\mM, \mN).$$

\item The following functor is an equivalence: $$\Enr\Fun_{\mV, \mW}(_\mV\triv \times \triv_\mW, \mN) \to \mN.$$

\end{enumerate}
\end{proposition}

\begin{proof}
We prove (1). Statemement (2) is dual. (3) follows from (1) and (2).
Let $\mN_\triv^\circledast:= \mN^\circledast \times_{(\mV^\ot \times \mW^\ot)} (\mV^\ot \times \triv_\mW^\circledast) \to \mV^\ot \times \mW^\ot$.
The embedding $\mN_\triv^\circledast \subset \mN^\circledast$ induces an equivalence $$\Enr\Fun_{\mV, \mW}(\mM \times \triv_\mW, \mN_\triv) \simeq \Enr\Fun_{\mV, \mW}(\mM \times \triv_\mW, \mN).$$
Consequently, it is enough to see that the canonical functor $$\Enr\Fun_{\mV, \mW}(\mM \times \triv_\mW, \mN_\triv) \to \Enr\Fun_{\mV,\emptyset}(\mM, \mN) $$ is an equivalence.
The functor $\mN^\circledast \to \mV^\ot \times \mW^\ot$ over $\mW^\ot$
is a map of cocartesian fibrations relative to the subcategory $\triv_\mW^\circledast \subset \mW^\ot.$
Thus the pullback $\mN^\circledast_\triv \to \mV^\ot \times \triv_\mW^\circledast$
is map of cocartesian fibrations over $\triv_\mW^\circledast$ and so classifies a functor $\triv_\mW^\circledast \to \Cat_{\infty / \mV^\ot}.$
By Remark \ref{aii} the $\infty$-category $\triv_\mW^\circledast$ admits a final object
so that the latter functor lifts to a functor 
$\kappa: \triv_\mW^\circledast \to (\Cat_{\infty / \mV^\ot})_{/\mN^\circledast_{[0]}}.$
This functor sends an object $\W \in \triv_\mW^\circledast$ to
the functor $\lambda: \mN^\circledast_\W \to \mN^\circledast_{\W'}\simeq \mN^\circledast_{[0]}$ over $\mV^\ot$ induced by the unique inert morphism $\W \to \W'$ lying over the unique inert map
$[\n]\to [0]$ in $\Ass$ preserving the maximum.
By the axioms of a weakly bienriched $\infty$-category the functor $\lambda$ is an equivalence so that $\kappa$ is the final object in the $\infty$-category of functors
$\triv_\mW^\circledast \to (\Cat_{\infty / \mV^\ot})_{/\mN^\circledast_{[0]}} $. 
Hence there is an equivalence $\mN^\circledast_\triv \simeq \mN^\circledast_{[0]} \times \triv_\mW^\circledast$ over $\mV^\ot \times \triv_\mW^\circledast$
whose pullback to $\emptyset^\circledast \subset \mW^\ot$
is the identity by construction of $\lambda.$
So it suffices to show that the canonical functor $$\Enr\Fun_{\mV, \mW}(\mM \times \triv_\mW, \mN \times \triv_\mW) \to \Enr\Fun_{\mV,\emptyset}(\mM, \mN) $$ is an equivalence. 
This functor is the restriction of the composition of the canonical equivalence
$$\theta: \Fun_{\mV^\ot \times \mW^\ot}(\mM^\circledast \times \triv^\circledast_\mW, \mN^\circledast \times \triv^\circledast_\mW) \simeq \Fun_{\mV^\ot \times \triv_\mW^\circledast}(\mM^\circledast \times \triv^\circledast_\mW, \mN^\circledast \times \triv^\circledast_\mW) $$$$\simeq \Fun_{\mV^\ot}(\mM^\circledast \times \triv^\circledast_\mW, \mN^\circledast) \simeq \Fun(\triv^\circledast_\mW,\Fun_{\mV^\ot}(\mM^\circledast, \mN^\circledast))$$
and the functor $$\rho: \Fun(\triv^\circledast_\mW,\Fun_{\mV^\ot}(\mM^\circledast, \mN^\circledast)) \to \Fun_{\mV^\ot}(\mM^\circledast, \mN^\circledast)$$
evaluating at the final object of $\triv^\circledast_\mW.$
The equivalence $\theta$ restricts to an equivalence
$$\Enr\Fun_{\mV, \mW}(\mM \times \triv_\mW, \mN \times \triv_\mW) \simeq \Fun'(\triv^\circledast_\mW,\Enr\Fun_{\mV,\emptyset}(\mM, \mN)),$$
where the right hand side is the full subcategory of $\Fun(\triv^\circledast_\mW,\Enr\Fun_{\mV,\emptyset}(\mM, \mN))$
of functors inverting every morphism in $\triv^\circledast_\mW$.
The functor $\rho$ restricts to an equivalence
$\Fun'(\triv^\circledast_\mW,\Enr\Fun_{\mV,\emptyset}(\mM, \mN)) \to \Enr\Fun_{\mV,\emptyset}(\mM, \mN)$
since $\triv^\circledast_\mW$ admits a final object and so is weakly contractible.

\end{proof}

\begin{corollary}\label{asio}Let $\mV^\ot \to \Ass, \mW^\ot \to \Ass$ be small $\infty$-operads.

\begin{enumerate}
\item The functor ${_\mV\omega\B\Enr_\mW} \to {_\mV\omega\B\Enr_\emptyset}$
taking pullback along $\emptyset^\circledast \subset \mW^\circledast$
admits a fully faithful left adjoint whose essential image precisely consists of the
right trivial weakly bienriched $\infty$-categories.

\item The functor ${_\mV\omega\B\Enr_\mW} \to {_\emptyset\omega\B\Enr_\mW}$
taking pullback along $\emptyset^\circledast \subset \mV^\circledast$
admits a fully faithful left adjoint whose essential image precisely consists of the
left trivial weakly bienriched $\infty$-categories.

\item The functor ${_\mV\omega\B\Enr_\mW} \to \Cat_\infty, \mM^\circledast \to \mV^\ot \times \mW^\ot \mapsto \mM$
admits a fully faithful left adjoint whose essential image precisely consists of the trivial weakly bienriched $\infty$-categories.

\end{enumerate}	


\end{corollary}


\subsection{Tensored envelopes}

In the following we introduce a tool that reduces questions about weakly enriched $\infty$-categories to tensored $\infty$-categories. We follow \cite[§ 3.3.]{HEINE2023108941}. 

\begin{notation}
Let $\text{Min},\text{Max} \subset \Fun([1],\Ass)$ be the full subcategories of morphisms
preserving the minimum, maximum, respectively, and $\Ass= \Min \cap \Max$ the intersection.

\end{notation}

\begin{remark}\label{heh}
Every morphism in $\Ass$ uniquely factors as an inert morphism followed by a 
morphism preserving the minimum and maximum.
Similarly, every morphism in $\Ass$ uniquely factors as an inert morphism preserving the maximum (minimum) followed by a morphism preserving the minimum (maximum).
By \cite[Lemma 5.2.8.19.]{lurie.HTT} this guarantees that the embeddings $\Act, \text{Min},\text{Max} \subset \Fun([1],\Ass)$ admit left adjoints, where a morphism in $\Fun([1],\Ass) $ with local target is a local equivalence
if and only if its image in $\Ass$ under evaluation at the source is inert,
is inert and preserves the maximum, is inert and preserves the minimum, respectively, and its image in $\Ass$ under evaluation at the target is an equivalence.
\end{remark}

The next notation is \cite[Notation 3.98.]{HEINE2023108941}:

\begin{notation}\label{ene}
Let $\mM^\circledast \to \mV^\ot \times \mW^\ot$ be a weakly bienriched $\infty$-category.
Let
\begin{itemize}	
\item $\Env(\mV)^\ot:= \Act \times_{\Fun(\{0\}, \Ass)} \mV^\ot\to \Fun(\{1\},\Ass).$

\item $\L\Env(\mM)^\circledast:= \mathrm{Min} \times_{\Fun(\{0\},\Ass)} \mM^\circledast \to \Fun(\{1\},\Ass)$
\item $\R\Env(\mM)^\circledast:= \mM^\circledast \times_{\Fun(\{0\},\Ass)} \Max\to \Fun(\{1\},\Ass)$
	
\item $\B\Env(\mM)^\circledast:= \text{Min} \times_{\Fun(\{0\}, \Ass)} \mM^\circledast \times_{\Fun(\{0\}, \Ass) } \Max \to \Fun(\{1\},\Ass) \times \Fun([1], \Ass).$ 
\end{itemize}

\end{notation}

	
\begin{remark}
The diagonal embedding $\Ass \subset \Act$ induces embeddings $$\mV^\ot \subset \Env(\mV)^\ot, \mM^\circledast \subset \L\Env(\mM)^\circledast,\L\Env(\mM)^\circledast \subset \B\Env(\mM)^\circledast, \mM^\circledast \subset \R\Env(\mM)^\circledast,\R\Env(\mM)^\circledast \subset \B\Env(\mM)^\circledast.$$
	
\end{remark}

\begin{remark}\label{hehh}
Due to Remark \ref{heh} for any weakly bienriched $\infty$-category 
$\mM^\circledast \to \mV^\ot$ the embeddings
$$ \Env(\mV)^\ot \subset \Fun([1],\Ass) \times_{\Fun(\{0\}, \Ass)} \mV^\ot,$$$$ \L\Env(\mM)^\circledast \subset  \Fun([1],\Ass) \times_{\Fun(\{0\},\Ass)} \mM^\circledast, \R\Env(\mM)^\circledast \subset \mM^\circledast \times_{\Fun(\{0\},\Ass)} \Fun([1],\Ass), $$$$\B\Env(\mM)^\circledast \subset \Fun([1],\Ass) \times_{\Fun(\{0\}, \Ass)} \mM^\circledast \times_{\Fun(\{0\}, \Ass) } \Fun([1],\Ass)$$
admit left adjoints relative to $\Fun(\{1\},\Ass)$, relative to $\Fun(\{1\},\Ass) \times \Fun(\{1\},\Ass),$ respectively. This implies that the following functors are cocartesian fibrations:
$$ \Env(\mV)^\ot \to \Fun(\{1\},\Ass), \L\Env(\mM)^\circledast \to \Fun(\{1\},\Ass), \R\Env(\mM)^\circledast \to \Fun(\{1\},\Ass),$$$$ \B\Env(\mM)^\circledast \to \Fun(\{1\},\Ass) \times \Fun(\{1\},\Ass).$$ 

\end{remark}

\begin{construction}
	
There is a map
$$\B\Env(\mM)^\circledast:= \text{Min} \times_{\Fun(\{0\}, \Ass)} \mM^\circledast \times_{\Fun(\{0\}, \Ass) } \Max
\subset $$$$\Fun([1], \Ass) \times_{\Fun(\{0\}, \Ass)} \mM^\circledast \times_{\Fun(\{0\}, \Ass) } \Fun([1], \Ass)\to $$$$ (\Fun([1],\Ass) \times_{\Fun(\{0\},\Ass)} \mV^\ot) \times (\Fun([1],\Ass) \times_{\Fun(\{0\},\Ass)} \mW^\ot) \to $$$$\Env(\mV)^\ot \times \Env(\mW)^\ot= (\Act \times_{\Fun(\{0\},\Ass)} \mV^\ot) \times (\Act \times_{\Fun(\{0\},\Ass)} \mW^\ot)$$
of cocartesian fibrations over $\Ass\times \Ass, $ where the last functor is induced by the localization functors of Remark \ref{hehh}.
There are maps
$$\L\Env(\mM)^\circledast:= \text{Min} \times_{\Fun(\{0\}, \Ass)} \mM^\circledast \subset \Fun([1], \Ass) \times_{\Fun(\{0\}, \Ass)} \mM^\circledast \to $$$$ (\Fun([1],\Ass) \times_{\Fun(\{0\},\Ass)} \mV^\ot) \times \mW^\ot \to \Env(\mV)^\ot \times \mW^\ot= (\Act \times_{\Fun(\{0\},\Ass)} \mV^\ot) \times \mW^\ot,$$
$$\R\Env(\mM)^\circledast:= \mM^\circledast \times_{\Fun(\{0\}, \Ass)} \Max \subset\mM^\circledast \times_{\Fun(\{0\}, \Ass)} \Fun([1], \Ass) \to $$$$\mV^\ot \times (\Fun([1],\Ass) \times_{\Fun(\{0\},\Ass)} \mW^\ot) \to \mV^\ot \times \Env(\mW)^\ot =\mV^\ot \times (\Act \times_{\Fun(\{0\},\Ass)} \mW^\ot)$$
of cocartesian fibrations over $\Ass$, where the last functors are induced by the localization functors of Remark \ref{hehh}.
\end{construction}

\begin{definition}Let $\mM^\circledast \to \mV^\ot \times \mW^\ot$ be a weakly bienriched $\infty$-category.

\begin{enumerate}
\item The monoidal envelope of $\mV^\ot \to \Ass$
is the functor $$ \Env(\mV)^\ot \to \Act \to \Fun(\{1\},\Ass) $$

\item The left tensored envelope is the functor $$\L\Env(\mM)^\circledast \to \Env(\mV)^\ot \times \mW^\ot.$$

\item The right tensored envelope is the functor $$\R\Env(\mM)^\circledast \to  \R\Env(\mM)^\circledast \to \mV^\ot \times \Env(\mW)^\ot.$$

\item The bitensored envelope is the functor
$$ \B\Env(\mM)^\circledast \to \Env(\mV)^\ot \times \Env(\mW)^\ot.$$
\end{enumerate}	
\end{definition}

The next Proposition \ref{unb} justifies this terminology and is \cite[Proposition 3.92, Proposition 3.101.]{HEINE2023108941}:

\begin{proposition}\label{unb}\label{envv}\label{bitte} Let $\mV^\ot$ be an $\infty$-operad and $\mM^\circledast \to \mV^\ot \times \mW^\ot$ a weakly bienriched $\infty$-category.
\begin{enumerate}
\item The functor $\Env(\mV)^\ot \to \Act \to \Fun(\{1\}, \Ass) $
is a monoidal $\infty$-category and the embedding $\mV^\ot \subset \Env(\mV)^\ot$ is a map of  $\infty$-operads. For any monoidal $\infty$-category $\mW^\ot \to \Ass $
the functor $\Alg_{\Env(\mV)}(\mW) \to \Alg_{\mV}(\mW)$
admits a fully faithful left adjoint that lands in the full subcategory of monoidal functors. So the induced functor $\Fun^{\ot}(\Env(\mV), \mW) \to \Alg_{\mV}(\mW)$ is an equivalence.

\item The functor $\L\Env(\mM)^\circledast \to \Env(\mV)^\ot \times \mW^\ot$ is a left  tensored $\infty$-category and the embedding $\mM^\circledast \subset \L\Env(\mM)^\circledast$ is an enriched functor. For any left tensored $\infty$-category $\mN^\ot \to \Env(\mV)^\ot \times \mW^\ot$ the functor $\Enr\Fun_{\Env(\mV),\mW}(\L\Env(\mM),\mN) \to \Enr\Fun_{\mV,\mW}(\mM,\mN)$
admits a fully faithful left adjoint that lands in the full subcategory of enriched functors preserving left tensors. So the induced functor $\L\LinFun_{\Env(\mV),\mW}(\L\Env(\mM),\mN) \to \Enr\Fun_{\mV,\mW}(\mM,\mN)$ is an equivalence.

\item The functor $\R\Env(\mM)^\circledast \to \mV^\ot \times \Env(\mW)^\ot $ is a right tensored $\infty$-category and the embedding $\mM^\circledast \subset \R\Env(\mV)^\circledast$ is an enriched functor. For any right tensored $\infty$-category $\mN^\ot \to  \mV^\ot \times \Env(\mW)^\ot$ the functor $\Enr\Fun_{\mV, \Env(\mW)}(\R\Env(\mM),\mN) \to \Enr\Fun_{\mV,\mW}(\mM,\mN)$admits a fully faithful left adjoint that lands in the full subcategory of enriched functors preserving right tensors. So the functor $\R\LinFun_{\mV, \Env(\mW)}(\R\Env(\mM),\mN) \to \Enr\Fun_{\mV,\mW}(\mM,\mN)$ is an equivalence.

\item The functor $\B\Env(\mM)^\circledast \to \Env(\mV)^\ot \times \Env(\mW)^\ot $ is a bitensored $\infty$-category and the embedding $\mM^\circledast \subset \B\Env(\mM)^\circledast$ is an enriched functor. For any bitensored $\infty$-category $\mN^\ot \to \Env(\mV)^\ot \times \Env(\mW)^\ot$ the functor $\Enr\Fun_{\Env(\mV), \Env(\mW)}(\B\Env(\mM),\mN) \to \Enr\Fun_{\mV,\mW}(\mM,\mN)$
admits a fully faithful left adjoint that lands in the full subcategory of linear functors. So the induced functor $\LinFun_{\mV, \mW}(\B\Env(\mM),\mN) \to \Enr\Fun_{\mV,\mW}(\mM,\mN)$ is an equivalence.

\end{enumerate}
	
\end{proposition} 

\begin{remark} Let $\mM^\circledast \to \mV^\ot \times \mW^\ot$ be a weakly bienriched $\infty$-category.
By construction there is a canonical $\mV, \Env(\mW)$-linear equivalence $$\R\Env(\mM)^\circledast \simeq (\L\Env(\mM^\rev)^\rev)^\circledast.$$
\end{remark}

\begin{remark}\label{huiii}
By construction every object of $\B\Env(\mM)$ is equivalent to $\V_1 \ot ...\ot \V_\n \ot \X \ot \W_1 \ot ...\ot\W_\m$ for
$\n,\m \geq 0$ and $\V_1,...,\V_\n \in \mV \subset \Env(\mV), \X \in \mM \subset \B\Env(\mM), \W_1,...,\W_\m \in \mW \subset \Env(\mW).$ 	
The left tensored envelope of $\mM^\circledast \to \mV^\ot \times \mW^\ot$ is the full weakly bienriched subcategory $$\L\Env(\mM)^\circledast \subset \B\Env(\mM)^\circledast \times_{\Env(\mW)^\ot} \mW^\ot \to \Env(\mV)^\ot \times \mW^\ot$$ spanned by the objects of the form
$ \V_1 \ot ... \ot \V_\n \ot \X$ for $\n \geq 0$ and $\V_1,...,\V_\n \in \mV$
and $\X \in \mM.$
	
\end{remark}
 



\begin{remark}\label{envdecom} Let $\mM^\circledast \to \mV^\ot \times \mW^\ot$ be a weakly bienriched $\infty$-category.
By construction there is a canonical $\Env(\mV), \Env(\mW)$-linear equivalence:
$$\L\Env(\R\Env(\mM))^\circledast \simeq \B\Env(\mM)^\circledast \simeq \R\Env(\L\Env(\mM))^\circledast.$$
\end{remark}

\begin{remark}\label{bienv}
	
Let $\mM^\circledast \to \mV^\ot$ be a weakly left enriched $\infty$-category
and $\mN^\circledast \to \mW^\ot$ a weakly right enriched $\infty$-category.
By construction there is a $\Env(\mV), \Env(\mW)$-linear equivalence:
$$ \B\Env(\mM \times \mN)^\circledast= \mathrm{Min} \times_{\Fun(\{0\},\Ass)} \mM^\circledast \times \mN^\circledast \times_{\Fun(\{0\},\Ass)}\Max \simeq \L\Env(\mM)^\circledast \times \R\Env(\mN)^\circledast.$$	
\end{remark}

The next lemma is \cite[Lemma 3.93.]{HEINE2023108941}:

\begin{lemma}\label{looocx}
Let $\mV^\ot \to \Ass$ be a monoidal $\infty$-category.
The embedding $\mV^\ot \subset \Env(\mV)^\ot$ admits a left adjoint relative to $\Ass$.
\end{lemma}

The next lemma is \cite[Lemma 3.105.]{HEINE2023108941}:

\begin{lemma}\label{loccyy}
Let $\mM^\circledast \to \mV^\ot \times \mW^\ot$ be a right tensored $\infty$-category.
The embedding $\L\Env(\mM)^\circledast \subset \B\Env(\mM)^\circledast$ admits a left 
$\Env(\mV)$-enriched left adjoint covering 
the left adjoint 
of the embedding $\mW^\ot \subset \Env(\mW)^\ot$.
The enriched adjunction $$\B\Env(\mM)^\circledast \rightleftarrows \L\Env(\mM)^\circledast$$ restricts to an enriched adjunction $\R\Env(\mM)^\circledast \rightleftarrows \mM^\circledast$. 
So for any bitensored $\infty$-category $\mM^\circledast \to \mV^\ot \times \mW^\ot$ the embedding $\mM^\circledast \subset \B\Env(\mM)^\circledast$ has an enriched left adjoint covering the left adjoints 
of the embeddings $\mV^\ot \subset \Env(\mV)^\ot, \mW^\ot \subset \Env(\mW)^\ot.$


\end{lemma}

\begin{remark}\label{rati} Let $\mM^\circledast \to \mV^\ot \times \mW^\ot$ be a  right tensored $\infty$-category.
Lemma \ref{loccyy} implies that $\L\Env(\mM)^\circledast \to \Env(\mV)^\ot \times \mW^\ot$ is a bitensored $\infty$-category and $\mM^\circledast \subset \L\Env(\mM)^\circledast$ is right $\mW$-linear.
\end{remark}

Next we consider the free cocompletion of the tensored envelope. 

\begin{notation}\label{Ind} Let $\kappa$ be a small regular cardinal and $\mC$ a small $\infty$-category.
Let $\Ind_\kappa(\mC) \subset \mP(\mC)$ be the full subcategory generated by $\mC$ under small $\kappa$-filtered colimits.

\end{notation}

\begin{example}
	
For $\kappa=\emptyset$ we find that $\Ind_\emptyset(\mC)=\mP(\mC).$

\end{example}

\begin{remark}\label{Indd} 
By \cite[Corollary 5.3.5.4.]{lurie.higheralgebra} for every small $\infty$-category $\mC$
that admits $\kappa$-small colimits the full subcategory $\Ind_\kappa(\mC) \subset \mP(\mC)$ precisely consists of the functors $\mC^\op \to \mS$ preserving $\kappa$-small limits. 
Thus $\Ind_\kappa(\mC)$ is a $\kappa$-accessible localization with respect to the set of maps $\{ \colim(\y \circ \rH) \to \y(\colim(\rH)) \mid \rH:\K \to \mC, \ \K \ \kappa\text{-small} \}$, where $\y: \mC \subset \mP(\mC)$ is the Yoneda-embedding. Hence $\Ind_\kappa(\mC)$ is a presentable $\infty$-category.

\end{remark}

\begin{proposition}\label{presta} \label{Day}Let $\kappa$ be a small regular cardinal
and $\mV^\ot \to \Ass$ a small $\infty$-operad.

\begin{enumerate}
\item There is an $\infty$-operad $\Ind_\kappa(\mV)^\ot \to \Ass $ compatible with small $\kappa$-filtered colimits and an embedding $\mV^\ot \to \Ind_\kappa(\mV)^\ot$ inducing the embedding $\mV \to \Ind_\kappa(\mV)$ on underlying $\infty$-categories.

\item For every $\infty$-operad $\mW^\ot \to \Ass$ compatible with small colimits the induced functor 
\begin{equation}\label{jjjp}
\Alg_{\Ind_\kappa(\mV)}(\mW) \to \Alg_{\mV}(\mW)\end{equation} admits a fully faithful left adjoint that lands in the full subcategory $\Alg^{\kappa-\mathrm{fil}}_{\Ind_\kappa(\mV)}(\mW)$ of maps of $\infty$-operads preserving small $\kappa$-filtered colimits.
Thus the following functor is an equivalence:
\begin{equation}\label{ejjt}
\Alg^{\kappa-\mathrm{fil}}_{\Ind_\kappa(\mV)}(\mW) \to \Alg_{\mV}(\mW).
\end{equation}

\item If $\mV$ admits $\kappa$-small colimits, the functor (\ref{ejjt}) restricts to an equivalence $$\Alg^\L_{\Ind_\kappa(\mV)}(\mW) \to \Alg^\kappa_{\mV}(\mW).$$

\item If $\mV^\ot \to \Ass$ is compatible with $\kappa$-small colimits, $\Ind_\kappa(\mV)^\ot \to \Ass$ is compatible with small colimits.
	
\item If $\mV^\ot \to \Ass$ is a monoidal $\infty$-category, $\Ind_\kappa(\mV)^\ot \to \Ass$ is a monoidal $\infty$-category, the embedding $\mV^\ot \to \Ind_\kappa(\mV)^\ot$
is monoidal and for every monoidal $\infty$-category $\mW^\ot \to \Ass$ compatible with small colimits the functor (\ref{ejjt}) reflects monoidal functors. 
\end{enumerate}

\end{proposition}

\begin{proof} By \cite[Corollary 8.31.]{HEINE2023108941} for every small monoidal $\infty$-category $\mV^\ot \to \Ass$ there is a monoidal $\infty$-category $\Ind_\kappa(\mV)^\ot \to \Ass $ compatible with small $\kappa$-filtered colimits and a monoidal embedding $\mV^\ot \to \Ind_\kappa(\mV)^\ot$ inducing the embedding $\mV \to \Ind_\kappa(\mV)$ on underlying $\infty$-categories such that for every monoidal $\infty$-category $\mW^\ot \to \Ass$ compatible with small colimits condition (2) holds.
We set $\mP(\mV)^\ot:=\Ind_\emptyset(\mV)^\ot$ and have a monoidal embedding
$\mV^\ot \to \mP(\mV)^\ot$ lying over the Yoneda-embedding.

(1): Let $\mV^\ot \to \Ass$ be a small $\infty$-operad and $\Ind_\kappa(\mV)^\ot \subset \mP(\Env(\mV))^\ot $ the full suboperad spanned by  the essential image of the embedding $\Ind_\kappa(\mV)\subset \mP(\mV)\subset \mP(\Env(\mV))$.
Then $\Ind_\kappa(\mV)^\ot \to \Ass$ is compatible with small $\kappa$-filtered colimits since $\mP(\Env(\mV))^\ot \to \Ass$ is a monoidal $\infty$-category compatible with small colimits and the embedding $\Ind_\kappa(\mV) \subset \mP(\Env(\mV))$ preserves small $\kappa$-filtered colimits. The embedding of $\infty$-operads $\mV^\ot \subset \Env(\mV)^\ot \to \mP(\Env(\mV))^\ot$ induces an embedding $\mV^\ot \to \Ind_\kappa(\mV)^\ot$.

(5): We prove that if $\mV$ is a small monoidal $\infty$-category, both definitions of
$\Ind_\kappa(\mV)$ agree:
the localization $\Env(\mV)^\ot \to \mV^\ot$ relative to $\Ass$ of Lemma \ref{looocx} induces a localization $\mP(\Env(\mV))^\ot \to \mP(\mV)^\ot$ relative to $\Ass$ that exhibits $\mP(\mV)^\ot$ as a full suboperad of $\mP(\Env(\mV))^\ot$.
The monoidal embedding $\mV^\ot \to \mP(\mV)^\ot$ uniquely extends to a monoidal functor $ \Ind_\kappa(\mV)^\ot \to \mP(\mV)^\ot $ that preserves small $\kappa$-filtered colimits and so induces on underlying $\infty$-categories the canonical embedding.
So both definitions agree. 

(4): Let $\mV^\ot \to \Ass$ be a small $\infty$-operad compatible with 
$\kappa$-small colimits. We show that the embedding \begin{equation}\label{oof}
\Ind_\kappa(\mV)^\ot \subset \mP(\mV)^\ot:=\Ind_\emptyset(\mV)^\ot \end{equation} admits a left adjoint relative to $\Ass$. This implies that $\Ind_\kappa(\mV)^\ot \to \Ass$ is compatible with small colimits.
In view of Remark \ref{Indd} it is enough to see that for every $\kappa$-small
$\infty$-category $\K$, functor $\rH: \K \to \mV$ and objects $\V_1,...,\V_\n \in \mP\Env(\mV)$ for $\n\geq 0$ and $\Y \in \Ind_\kappa(\mV)$ and $1 \leq \bi \leq \n$
the induced map
$$\theta: \Mul_{\mP\Env(\mV)}(\V_1,...,\V_\bi,\colim(\rH),\V_{\bi+1},...,\V_\n;\Y) \to \lim \Mul_{\mP\Env(\mV)}(\V_1,...,\V_\bi, \rH(-),\V_{\bi+1},...,\V_\n;\Y) $$
is an equivalence.
Since $\mP\Env(\mV)^\ot \to \Ass$ is a monoidal $\infty$-category compatible with small colimits and $\mP\Env(\mV)$ is generated under small colimits by tensor products of objects of $\mV$, we can moreover assume that $\V_1,...,\V_\n \in \mV.$
By \cite[Corollary 5.3.5.4.]{lurie.HTT} the object $\Y$ of $\Ind_\kappa(\mV)$ is the $\kappa$-filtered colimit of a 
functor $\lambda$ taking values in $\mV$, which is preserved by the $\kappa$-filtered colimits preserving embedding $\Ind_\kappa(\mV) \subset \mP(\mV)$ and the small colimits preserving embedding $\mP(\mV) \subset \mP\Env(\mV).$
The map $\theta$ factors as
$$\Mul_{\mP\Env(\mV)}(\V_1,...,\V_\bi,\colim(\rH),\V_{\bi+1},...,\V_\n;\Y) \simeq$$$$ \colim\Mul_{\mP\Env(\mV)}(\V_1,...,\V_\bi,\colim(\rH),\V_{\bi+1},...,\V_\n;\lambda(-))\to $$$$\colim(\lim \Mul_{\mP\Env(\mV)}(\V_1,...,\V_\bi, \rH(-),\V_{\bi+1},...,\V_\n;\lambda(-))) \to $$$$\lim(\colim \Mul_{\mP\Env(\mV)}(\V_1,...,\V_\bi, \rH(-),\V_{\bi+1},...,\V_\n;\lambda(-))) \simeq$$$$ \lim \Mul_{\mP\Env(\mV)}(\V_1,...,\V_\bi, \rH(-),\V_{\bi+1},...,\V_\n;\Y).$$
The third map in the composition commuting colimit with limit is an equivalence
since $\kappa$-filtered colimits commute with $\kappa$-small limits in the $\infty$-category of spaces.  
So for proving that $\theta$ is an equivalence we can assume that $\Y \in \mV.$
In this case $\theta$ is an equivalence since $\mV^\ot \to \Ass$ is a small $\infty$-operad compatible with $\kappa$-small colimits.
	
(2): Let $\mW^\ot \to \Ass$ be an $\infty$-operad compatible with small colimits.
The following functors $$\Alg_{\Ind_\kappa(\Env(\mV))}(\widehat{\mP}(\Env(\mW)))\to \Alg_{\Ind_\kappa(\mV)}(\widehat{\mP}(\Env(\mW))), $$$$ \Alg_{\Ind_\kappa(\mV)}(\widehat{\mP}(\Env(\mW)))  \to \Alg_{\mV}(\widehat{\mP}(\Env(\mW)))$$ admit fully faithful left adjoints $\lambda, \tau$, respectively, by Proposition \ref{laan}.
So $\tau$ factors as $\lambda \circ \tau$ followed by the first functor. 
By Proposition \ref{envv} and what we have proven the functor $\lambda \circ \tau$ lands in the full subcategory of monoidal functors preserving small $\kappa$-filtered colimits. Thus $\tau$ takes values in the full subcategory of maps of $\infty$-operads preserving small $\kappa$-filtered colimits and therefore sends 
$\Alg_{\mV}(\widehat{\mP}(\mW))$ to $\Alg_{\Ind_\kappa(\mV)}(\widehat{\mP}(\mW))$
since the embedding $\widehat{\mP}(\mW) \subset \widehat{\mP}(\Env(\mW))$ preserves small colimits and $\Ind_\kappa(\mV)$ is generated by $\mV$ under small $\kappa$-filtered colimits.
Consequently, $\tau$ restricts to a left adjoint $\tau'$ of the induced functor 
\begin{equation}\label{ooo}\Alg_{\Ind_\kappa(\mV)}(\widehat{\mP}(\mW)) \to \Alg_{\mV}(\widehat{\mP}(\mW)),\end{equation}
which is fully faithful and lands in the full subcategory of maps of $\infty$-operads preserving small $\kappa$-filtered colimits.
Let $\sigma$ be the strongly inaccessible cardinal corresponding to the small universe.
The localization $\widehat{\Ind}_\sigma(\mW)^\ot \subset \widehat{\mP}(\mW)^\ot$ relative to $\Ass$ of (\ref{oof}) (applied to $\mW^\ot$ and $\sigma$) induces two localizations $\Alg_{\Ind_\kappa(\mV)}(\widehat{\Ind}_\sigma(\mW)) \subset \Alg_{\Ind_\kappa(\mV)}(\widehat{\mP}(\mW)), \Alg_{\mV}(\widehat{\Ind}_\sigma(\mW)) \subset \Alg_{\mV}(\widehat{\mP}(\mW)).$
The composed adjunction $$ \Alg_{\mV}(\widehat{\mP}(\mW)) \rightleftarrows \Alg_{\Ind_\kappa(\mV)}(\widehat{\mP}(\mW))
\rightleftarrows \Alg_{\Ind_\kappa(\mV)}(\widehat{\Ind}_\sigma(\mW))$$
restricts to an adjunction
$ \phi: \Alg_{\mV}(\widehat{\Ind}_\sigma(\mW)) \rightleftarrows \Alg_{\Ind_\kappa(\mV)}(\widehat{\Ind}_\sigma(\mW)).$
The left adjoint $\phi$ lands in the full subcategory of maps of $\infty$-operads preserving small $\kappa$-filtered colimits. 
The functor (\ref{ooo}) preserves local equivalences for the respective object-wise localizations so that the unit of the latter adjunction is a local equivalence
between local objects and so an equivalence. Hence $\phi$ is fully faithful.

Since $\mW$ is closed in $\widehat{\Ind}_\sigma(\mW)$ under small colimits and $\Ind_\kappa(\mV)$ is generated by $\mV$ under small $\kappa$-filtered colimits,
the adjunction $ \phi: \Alg_{\mV}(\widehat{\Ind}_\sigma(\mW)) \rightleftarrows \Alg_{\Ind_\kappa(\mV)}(\widehat{\Ind}_\sigma(\mW))$ restricts to an adjunction
$ \Alg_{\mV}(\mW) \rightleftarrows \Alg_{\Ind_\kappa(\mV)}(\mW),$
where the left adjoint lands in the full subcategory of maps of $\infty$-operads preserving small $\kappa$-filtered colimits. (3) follows from \cite[Proposition 5.5.1.9.]{lurie.HTT}.
\end{proof}

\begin{proposition}\label{cool} Let $\kappa$ be a small regular cardinal and $\mM^\circledast \to \mV^\ot \times \mW^\ot$ an absolute small weakly bienriched $\infty$-category. 
	
\begin{enumerate}

\item There is a weakly bienriched $\infty$-category $\Ind_\kappa(\mM)^\circledast \to \Ind_\kappa(\mV)^\ot \times \Ind_\kappa(\mW)^\ot$ compatible with small $\kappa$-filtered colimits and a $\mV,\mW$-enriched embedding \begin{equation}\label{jett}
\mM^\circledast \to \mV^\ot \times_{\Ind_\kappa(\mV)^\ot} \Ind_\kappa(\mM)^\circledast \times_{\Ind_\kappa(\mW)^\ot} \mW^\ot\end{equation}  inducing the embedding $\mM \to \Ind_\kappa(\mM)$ on underlying $\infty$-categories.
\item For every weakly bienriched $\infty$-category $\mN^\circledast \to \Ind_\kappa(\mV)^\ot \times \Ind_\kappa(\mW)^\ot$ compatible with small colimits the functor
$$
\Enr\Fun_{\Ind_\kappa(\mV),\Ind_\kappa(\mW)}(\Ind_\kappa(\mM), \mN) \to \Enr\Fun_{\mV,\mW}(\mM,\mN)$$
admits a fully faithful left adjoint that lands in the full subcategory $$\Enr\Fun^{\kappa-\mathrm{fil}}_{\Ind_\kappa(\mV),\Ind_\kappa(\mW)}(\Ind_\kappa(\mM), \mN) $$
of enriched functors preserving small $\kappa$-filtered colimits.
So the next functor is an equivalence:
\begin{equation}\label{exol2}\Enr\Fun^{\kappa-\mathrm{fil}}_{\Ind_\kappa(\mV),\Ind_\kappa(\mW)}(\Ind_\kappa(\mM), \mN) \to \Enr\Fun_{\mV,\mW}(\mM,\mN).\end{equation}
\item If $\mM$ admits $\kappa$-small colimits, equivalence (\ref{exol2}) restricts to an equivalence
$$\Enr\Fun^{\L}_{\Ind_\kappa(\mV),\Ind_\kappa(\mW)}(\Ind_\kappa(\mM), \mN) \to \Enr\Fun^\kappa_{\mV,\mW}(\mM,\mN).$$
\item If $\mM^\circledast \to \mV^\ot \times \mW^\ot$ is compatible with $\kappa$-small colimits, $\Ind_\kappa(\mM)^\circledast \to \Ind_\kappa(\mV)^\ot \times \Ind_\kappa(\mW)^\ot$ is compatible with small colimits.

\item If $\mM^\circledast \to \mV^\ot \times \mW^\ot$ is a left tensored $\infty$-category, $\Ind_\kappa(\mM)^\circledast \to \Ind_\kappa(\mV)^\ot \times \Ind_\kappa(\mW)^\ot$ is a left tensored $\infty$-category, (\ref{jett}) is left linear
and for every left tensored $\infty$-category $\mN^\circledast \to \Ind_\kappa(\mV)^\ot \times \Ind_\kappa(\mW)^\ot$ compatible with small colimits 
the functor (\ref{exol2}) reflects left linear functors.

\end{enumerate}

\end{proposition}

\begin{proof}
One proves (1)-(4) completely analogous to Proposition \ref{presta}.
By \cite[Corollary 8.31.]{HEINE2023108941} for every absolute small weakly bitensored $\infty$-category $\mM^\circledast \to \mV^\ot \times \mW^\ot $ there is a bitensored $\infty$-category $\Ind_\kappa(\mM)^\circledast \to \Ind_\kappa(\mV)^\ot \times \Ind_\kappa(\mW)^\ot $ compatible with small $\kappa$-filtered colimits and a linear embedding $\mM^\circledast \to \Ind_\kappa(\mM^\circledast$ inducing the embedding $\mM \to \Ind_\kappa(\mM)$ on underlying $\infty$-categories such that for every bitensored $\infty$-category $\mN^\circledast \to \mV^\ot \times \mW^\ot$ compatible with small colimits condition (2) holds.
We set $\mP(\mM)^\circledast:=\Ind_\emptyset(\mM)^\circledast \to \mP(\mV)^\ot \times \mP(\mW)^\ot$ and have a monoidal embedding $\mM^\circledast \to \mP(\mM)^\circledast$ lying over the Yoneda-embedding.

Let $\mM^\circledast \to \mV^\ot \times \mW^\ot$ be an absolute small weakly bienriched $\infty$-category. Let $\Ind_\kappa(\mM)^\circledast \subset \mP(\Env(\mV))^\ot $ be the full suboperad spanned by the essential image of the embedding $\Ind_\kappa(\mM)\subset \mP(\mM)\subset \mP(\B\Env(\mM))$.
The linear embedding $\mM^\circledast \subset \B\Env(\mM)^\circledast \to \mP(\B\Env(\mM))^\circledast$ induces an embedding $\mM^\circledast \to \Ind_\kappa(\mM)^\circledast$. This proves (1). The proofs of (2)-(4) are analogous to Proposition \ref{presta}.
We prove (5): we prove that if $\mM^\circledast \to \mV^\ot \times \mW^\ot$ is a left tensored $\infty$-category, $\Ind_\kappa(\mM)^\circledast \to \Ind_\kappa(\mV)^\ot \times \Ind_\kappa(\mW)^\ot$ is a left tensored $\infty$-category and (\ref{jett}) is left linear.
It is enough to assume that $\kappa=\emptyset$ because $\Ind_\kappa(\mM)$ is generated by $\mM$ under small $\kappa$-filtered colimits, which are preserved by the embedding
$\Ind_\kappa(\mM) \subset \mP(\mM)$, so that the left action on $\mP(\mM)$ compatible with small colimits restricts to $\Ind_\kappa(\mM).$ 
So let $\kappa=\emptyset.$

By Proposition \ref{loccyy} the embedding $\R\Env(\mM)^\circledast \subset \B\Env(\mM)^\circledast$ admits an $\Env(\mW)$-linear enriched left adjoint
lying over the left adjoint 
of the embedding $\mV^\ot \subset \Env(\mV)^\ot$ of Lemma \ref{looocx}. 
By (2) there is a right $\mP(\Env(\mW))$-linear enriched left adjoint $\phi: \mP(\B\Env(\mM))^\circledast \to \mP(\R\Env(\mM))^\circledast$
lying over the monoidal functor $\mP(\Env(\mV))^\ot \to \mP(\mV)^\ot$,
whose underlying functor 
is left adjoint to the embedding $\mP(\R\Env(\mM)) \subset \mP(\B\Env(\mM))$. 
Thus the enriched right adjoint of $\phi$ is an embedding. 
The embedding $\mP(\mM)^\circledast \subset \mP\B\Env(\mM)^\circledast$
lands in $\mP\R\Env(\mM)^\circledast$ 
and we get a left $\mP(\mV)$-enriched embedding $\mM^\circledast \subset \mP\R\Env(\mM)^\circledast$ lying over 
$\mP(\mW)^\ot \subset \mP\Env(\mW)^\ot$. Since $\mP\R\Env(\mM)^\circledast \to \mP(\mV)^\ot \times \mP\Env(\mW)^\ot$ is a bitensored $\infty$-category compatible with small colimits and the embedding $\mP(\mM) \subset \mP\R\Env(\mM)$ preserves small colimits, it is enough to see that $\mP(\mM)\subset \mP\R\Env(\mM)$ is closed under the left $\mP(\mV)$-action. 
As this left action is compatible with small colimits, it is enough to see that $\mM \subset \mP\R\Env(\mM)$ is closed under the left $\mV$-action. 
This holds because the embeddings $\R\Env(\mM)^\circledast \subset \mV^\ot \times_{\mP(\mV)^\ot}\mP\R\Env(\mM)^\circledast$ and $\mM^\circledast \subset \R\Env(\mM)^\circledast$ are left $\mV$-linear by Propositions \ref{Day} and \ref{loccyy}.
It follows that $\mP(\mM)^\circledast \to \mP(\mV)^\ot \times \mP(\mW)^\ot$ is a left tensored $\infty$-category compatible with small colimits and the embedding
$\mP(\mM)^\circledast \subset \mP\R\Env(\mM)^\circledast$ is left $\mP(\mV)$-linear.
Since the embeddings $\mM^\circledast \subset \R\Env(\mM)^\circledast$ and $\R\Env(\mM)^\circledast \subset \mV^\ot \times_{\mP(\mV)^\ot} \mP\R\Env(\mM)^\circledast$ are left $\mV$-linear, the embedding $\mM^\circledast \subset \mV^\ot \times_{\mP(\mV)^\ot} \mP(\mM)^\circledast$ is left $\mV$-linear, too.
	
\end{proof}


\begin{notation}
Let  $\mM^\circledast \to \mV^\ot \times \mW^\ot$ be an absolute small weakly bienriched $\infty$-category. Let $$\mP(\mV)^\ot:=\Ind_\emptyset(\mV)^\ot, \hspace{4mm}
\mP(\mM)^\circledast:= \Ind_\emptyset(\mM)^\circledast \to \mP(\mV)^\ot \times \mP(\mW)^\ot.$$

\end{notation}

\begin{notation}\label{zgbbl}Let $\mV^\ot \to \Ass, \mW^\ot \to \Ass$ be $\infty$-operads and $\mM^\circledast \to \mV^\ot \times \mW^\ot$ a weakly bienriched $\infty$-category.

\begin{enumerate}
\item We call $$\mP\Env(\mV)^\ot:=\mP(\Env(\mV))^\ot \to \Ass$$ the closed monoidal envelope of $\mV^\ot \to \Ass$.

\item We call $$\mP\L\Env(\mM)^\circledast := \mP(\L\Env(\mM))^\circledast \to \mP\Env(\mV)^\ot \times \mP(\mW)^\ot$$ the closed left tensored envelope.

\item We call $$\mP\R\Env(\mM)^\circledast := \mP(\R\Env(\mM))^\circledast \to \mP(\mV)^\ot \times \mP\Env(\mW)^\ot$$ the closed right tensored envelope.

\item We call $$\mP\B\Env(\mM)^\circledast := \mP(\B\Env(\mM))^\circledast \to \mP\Env(\mV)^\ot \times \mP\Env(\mW)^\ot$$ the closed bitensored envelope.

\end{enumerate}
\end{notation}

\begin{remark}By Proposition \ref{cool} the closed left tensored, right tensored, bitensored
envelopes 
are presentably left tensored, right tensored, bitensored $\infty$-categories.
	
\end{remark}

Propositions \ref{presta}, \ref{cool} and \ref{bitte} give the following corollary:

\begin{corollary}\label{envvcor} 
	
\begin{enumerate}
\item Let $\mV^\ot \to \Ass$ be a small $\infty$-operad. For every monoidal $\infty$-category $\mW^\ot \to \Ass$ compatible with small colimits
the induced functor $$\Alg_{\mP\Env(\mV)}(\mW) \to \Alg_{\mV}(\mW)$$
admits a fully faithful left adjoint that lands in the full subcategory 
of monoidal functors that preserve small colimits. In particular, the following induced functor is an equivalence:
$$\Fun^{\ot,\L}(\mP\Env(\mV), \mW) \to \Alg_{\mV}(\mW).$$


\item Let $\mM^\circledast \to \mV^\ot \times \mW^\ot$ be a weakly bienriched $\infty$-category. For every bitensored $\infty$-category $\mN^\circledast \to\mP\Env(\mV)^\ot \times \mP\Env(\mW)^\ot $ compatible with small colimits the induced functor $$\Enr\Fun_{\mP\Env(\mV),\mP\Env(\mW)}(\mP\B\Env(\mM), \mN) \to \Enr\Fun_{\mV,\mW}(\mM, \mN)$$
admits a fully faithful left adjoint that lands in the full subcategory 
of $\mP\Env(\mV), \mP\Env(\mW)$-linear functors preserving small colimits. In particular, the following functor is an equivalence: $$\LinFun^\L_{\mP\Env(\mV),\mP\Env(\mW)}(\mP\B\Env(\mM), \mN) \to \Enr\Fun_{\mV,\mW}(\mM, \mN).$$

\end{enumerate}	

\end{corollary}

\begin{remark}\label{switch}
By the universal property of Corollary \ref{envvcor} (2) there is a canonical equivalence
$ \mP\B\Env(\mM^\rev)^\circledast \simeq (\mP\B\Env(\mM)^\rev)^\circledast$
that restricts to an equivalence
$ \mP\L\Env(\mM^\rev)^\circledast \simeq (\mP\R\Env(\mM)^\rev)^\circledast.$
	
\end{remark}

Remark \ref{envdecom} gives the following one:
\begin{corollary}\label{envdecom2}
Let $\mM^\circledast \to \mV^\ot \times \mW^\ot$ be a weakly bienriched $\infty$-category. There is an equivalence
$$\mP\L\Env(\R\Env(\mM))^\circledast \simeq \mP\B\Env(\mM)^\circledast$$
of $\infty$-categories presentably bitensored over $\mP\Env(\mV),\mP\Env(\mW).$	
\end{corollary}

Using tensored envelopes we can prove the following refinement of Corollary \ref{kurt}:

\begin{proposition}\label{adjeq}
Let $\mM^\circledast \to \mV^\ot \times\mW^\ot, \mN^\circledast \to \mV^\ot \times\mW^\ot$ be weakly bienriched $\infty$-categories.
There is a canonical equivalence
$$\Enr\Fun^\L_{\mV, \mW}(\mM,\mN)^\op \simeq \Enr\Fun_{\mV,\mW}^\R(\mN,\mM).$$	
\end{proposition}

\begin{proof}
We can assume that $\mM^\circledast \to \mV^\ot \times\mW^\ot, \mN^\circledast \to \mV^\ot \times\mW^\ot$ are absolute small. 
Let $\K$ be a small $\infty$-category and $\iota: \mM^\circledast \subset \mP\B\Env(\mM)^\circledast$ the canonical embedding.
The functor $\K^\op \times \K \xrightarrow{\K(-,-)} \mS \xrightarrow{\iota} \Enr\Fun_{\mV,\mW}(\mM,\mP\B\Env(\mM))$
corresponds to a $\mV,\mW$-enriched functor $\K^\op \times \mM^\circledast \to (\mP\B\Env(\mM)^\circledast)^\K$, which by Corollary \ref{envvcor} corresponds to a left adjoint $\mP\Env(\mV),\mP\Env(\mW)$-linear functor 
$$\mP\B\Env(\K^\op \times \mM)^\circledast \to (\mP\B\Env(\mM)^\circledast)^\K.$$
The latter is an equivalence as it induces on underlying $\infty$-categories the canonical equivalence 
$$\mP\B\Env(\K^\op \times \mM) \simeq \mP(\K^\op \times \B\Env(\mM)) \simeq \Fun(\K, \mP\B\Env(\mM)).$$
By Corollary \ref{kurt} 
there is a canonical equivalence
$\Enr\Fun^\L_{\mV, \mW}(\mM,\mN)^\simeq \simeq \Enr\Fun_{\mV,\mW}^\R(\mN,\mM)^\simeq.$
Hence there is a canonical equivalence 
$$\lambda: \Cat_\infty(\K^\op, \Enr\Fun^\L_{\mP\Env(\mV), \mP\Env(\mW)}(\mP\B\Env(\mM),\mP\B\Env(\mN))) \simeq$$$$
\Cat_\infty(\K^\op, \Enr\Fun_{\mV, \mW}(\mM,\mP\B\Env(\mN))) \simeq
\Enr\Fun_{\mV, \mW}(\K^\op \times \mM,\mP\B\Env(\mN))^\simeq\simeq$$$$ \Enr\Fun^\L_{\mP\Env(\mV), \mP\Env(\mW)}(\mP\B\Env(\mM)^\K,\mP\B\Env(\mN))^\simeq\simeq$$$$\Enr\Fun^{\R}_{\mP\Env(\mV), \mP\Env(\mW)}(\mP\B\Env(\mN), \mP\B\Env(\mM)^\K)^\simeq \simeq $$$$
\Cat_\infty(\K, \Enr\Fun^{\R}_{\mP\Env(\mV), \mP\Env(\mW)}(\mP\B\Env(\mN), \mP\B\Env(\mM))).$$
Let $$\Enr\Fun^\L_{\mP\Env(\mV), \mP\Env(\mW)}(\mP\B\Env(\mM),\mP\B\Env(\mN))'\subset \Enr\Fun^\L_{\mP\Env(\mV), \mP\Env(\mW)}(\mP\B\Env(\mM),\mP\B\Env(\mN)),$$$$ \Enr\Fun^{\R}_{\mP\Env(\mV), \mP\Env(\mW)}(\mP\B\Env(\mN), \mP\B\Env(\mM))'\subset \Enr\Fun^{\R}_{\mP\Env(\mV), \mP\Env(\mW)}(\mP\B\Env(\mN), \mP\B\Env(\mM))$$ be the full subcategories of left adjoint enriched functors sending $\mM$ to $\mN$ whose right adjoint sends $\mN$ to $\mM$, of right adjoint enriched functors sending $\mN$ to $\mM$ whose left adjoint sends $\mM$ to $\mN$, respectively.
Corollary \ref{envvcor} implies that the embedding $\mM^\circledast \subset \mP\B\Env(\mM)^\circledast$ induces an equivalence
$$\Enr\Fun^\L_{\mV, \mW}(\mM, \mN) \simeq \Enr\Fun^\L_{\mP\Env(\mV), \mP\Env(\mW)}(\mP\B\Env(\mM), \mP\B\Env(\mN))'.$$
The embedding $ \mN^\circledast \subset \mP\B\Env(\mN)^\circledast$ induces an equivalence
$$\Enr\Fun^{\R}_{\mV, \mW}(\mN, \mM) \simeq \Enr\Fun^{\R}_{\mP\Env(\mV), \mP\Env(\mW)}(\mP\B\Env(\mN), \mP\B\Env(\mM))'$$
by Proposition \ref{cool} since a right adjoint functor between presheaf $\infty$-categories preserves small colimits if the left adjoint preserves representables.
The equivalence $\lambda$ restricts to an equivalence
$$\Cat_\infty(\K^\op, \Enr\Fun^\L_{\mP\Env(\mV), \mP\Env(\mW)}(\mP\B\Env(\mM),\mP\B\Env(\mN))') \simeq $$$$
\Cat_\infty(\K, \Enr\Fun^{\R}_{\mP\Env(\mV), \mP\Env(\mW)}(\mP\B\Env(\mN), \mP\B\Env(\mM))')$$
and we obtain an equivalence 
$$\Cat_\infty(\K^\op, \Enr\Fun^\L_{\mV, \mW}(\mM,\mN)) \simeq \Cat_\infty(\K, \Enr\Fun^\R_{\mV, \mW}(\mN, \mM))$$
representing the desired equivalence $\Enr\Fun^\L_{\mV, \mW}(\mM,\mN)^\op \simeq \Enr\Fun^\R_{\mV, \mW}(\mN, \mM)$. 

\end{proof}


\subsection{Pseudo-enrichment}

Before defining left, right and bienriched $\infty$-categories we define the more general class of left, right and bipseudoenriched $\infty$-categories that contains all tensored $\infty$-categories.

\vspace{1mm}
\begin{definition}\label{Lu}Let $\phi: \mM^\circledast\to \mV^\ot \times \mW^\ot$ be a weakly bienriched $\infty$-category.
	
\begin{enumerate}
\item We say that $\phi$ exhibits $\mM$ as left pseudo-enriched in $\mV$ if $\mV^\ot \to\Ass$ is a monoidal $\infty$-category and for every $\X,\Y \in \mM$ and $\V_1, ..., \V_\n \in \mV, \W_1, ..., \W_\m \in \mW$ for $\n,\m \geq 0$ the map
$$ \Mul_{\mM}(\V_1 \ot ... \ot \V_\n, \X, \W_1, ..., \W_\m; \Y) \to \Mul_{\mM}(\V_1, ..., \V_\n, \X, \W_1, ..., \W_\m ; \Y)$$ 
induced by the active morphism $\V_1, ..., \V_\n \to \V_1 \ot  ... \ot \V_\n $
in $\mV^\ot$ is an equivalence.
\vspace{1mm}

\item We say that $\phi$ exhibits $\mM$ as right pseudo-enriched in $\mW$ if $\mW^\ot \to\Ass$ is a monoidal $\infty$-category and for every $\X,\Y \in \mM$ and $\V_1, ..., \V_\n \in \mV, \W_1, ..., \W_\m \in \mW$ for $\n,\m \geq 0$ the map
$$ \Mul_{\mM}(\V_1, ..., \V_\n, \X, \W_1 \ot ... \ot \W_\m; \Y) \to \Mul_{\mM}(\V_1, ...,\V_\n, \X, \W_1, ..., \W_\m ; \Y)$$ 
induced by the active morphism $\W_1, ..., \W_\m \to \W_1 \ot  ... \ot \W_\m $
in $\mW^\ot$ is an equivalence.

\item We say that $\phi$ exhibits $\mM$ as bipseudo-enriched in $\mV, \mW$ if $\mV^\ot \to\Ass, \mW^\ot \to\Ass $ are monoidal $\infty$-categories and for any $\X,\Y \in \mM$ and $\V_1, ..., \V_\n \in \mV, \W_1, ..., \W_\m \in \mW$ for $\n,\m \geq 0$ the following map is an equivalence:
$$ \Mul_{\mM}(\V_1 \ot ... \ot \V_\n, \X, \W_1 \ot ...\ot \W_\m; \Y) \to \Mul_{\mM}(\V_1, ..., \V_\n, \X, \W_1, ..., \W_\m ; \Y).$$

\end{enumerate}

\end{definition}

\begin{remark}\label{both}
A weakly bienriched $\infty$-category $\mM^\circledast\to \mV^\ot \times \mW^\ot$ exhibits $\mM$ as bipseudo-enriched in $\mV, \mW$ if and only if it exhibits $\mM$ as left pseudo-enriched in $\mV$ and right pseudo-enriched in $\mW$.
	
\end{remark}

	

\begin{lemma}\label{zzzz} 
	
\begin{enumerate}
\item A weakly bienriched $\infty$-category $\mM^\circledast \to \mV^\ot \times \mW^\ot$ is a left tensored $\infty$-category if and only if it is a left pseudo-enriched $\infty$-category and admits left tensors.	

\item A weakly bienriched $\infty$-category $\mM^\circledast \to \mV^\ot \times \mW^\ot$ is a right tensored $\infty$-category if and only if it is a right pseudo-enriched $\infty$-category and admits right tensors.	
	
\item A weakly bienriched $\infty$-category $\mM^\circledast \to \mV^\ot \times \mW^\ot$ is a bitensored $\infty$-category if and only if it is a bipseudo-enriched $\infty$-category and admits bitensors.
\end{enumerate}	
\end{lemma}

\begin{proof}
(2) is dual to (1) and (3) follows from (1) and (2).
We prove (1): Every left tensored $\infty$-category $\mM^\circledast \to \mV^\ot \times \mW^\ot$ is a left pseudo-enriched $\infty$-category. 
By left pseudo-enrichedness for every $\V \in \mV, \X \in \mM$ the natural morphism $\V,\X \to \V \ot \X$ in $\mM^\circledast$ exhibits $ \V \ot \X$ as the left tensor of $\V,\X$ if for every
$\V' \in \mV, \Z \in \mM, \W_1,...,\W_\m \in \mW$ for $\m \geq 0$ the following map is an equivalence:
\begin{equation*}
\Mul_\mM(\V', \V \ot \X, \W_1,...,\W_\m; \Z) \to \Mul_{\mM}(\V' \ot \V,\X,\W_1,...,\W_\m; \Z).
\end{equation*}
This follows since $\V' \ot (\V\ot\X)\simeq (\V' \ot \V) \ot\X.$

Conversely, if $\mM^\circledast \to \mV^\ot \times \mW^\ot$ admits left tensors, for every $\X \in \mM, \V_1,...,\V_\n \in \mV$ for $\n \geq 0 $
the functor $ \Mul_\mM(\V_1,..., \V_\n, \X;-): \mM \to \mS$ is corepresentable by some object $\Y$ by iterately taking the left tensor. 
For every $\V'_1,..., \V'_\bk \in \mV, \W_1,...,\W_\m \in \mW, \Z \in \mM$ for $\bk, \m \geq 0 $ the canonical map
$$\Mul_\mM(\V'_1,..., \V'_\bk, \Y,\W_1,...,\W_\m;\Z) \to \Mul_\mM(\V'_1,..., \V'_\bk, \V_1,...,\V_\n, \X,\W_1,...,\W_\m;\Z) $$ is an equivalence.
By left pseudo-enrichedness
for every $\V^\bj_1,..., \V^\bj_{\n_\bj}$ for $1 \leq \bj \leq \ell$ 
and $\W_1,...,\W_\m \in \mW$ for $\m \geq 0 $ the object corepresenting the functor
$$ \Mul_\mM(\{\V^\bj_1,..., \V^\bj_{\n_\bj}\}^\ell_{\bj=1} ,\V_1,..., \V_\n, \X;-): \mM \to \mS$$
corepresents the functor 
$ \Mul_\mM(\{\V^\bj_1 \ot ...\ot \V^\bj_{\n_\bj}\}^\ell_{\bj=1}, \Y;-): \mM \to \mS$.
This implies that $\mM^\circledast \to \mV^\ot \times \mW^\ot$ is a left tensored  $\infty$-category.

\end{proof}




\begin{notation}
Let $$ \L\P\Enr, \R\P\Enr, \B\P\Enr \subset \omega\B\Enr $$ be the full subcategories of left pseudo-enriched, right pseudo-enriched, bipseudo-enriched $\infty$-categories, respectively.
	
\end{notation}

\begin{example}
Let $\mM^\circledast \to \mV^\ot$ be a weakly left enriched $\infty$-category and $\mN^\circledast \to \mW^\ot$ a weakly right enriched $\infty$-category.
\begin{enumerate}
\item If $\mM^\circledast \to \mV^\ot$ exhibits $\mM$ as left pseudo-enriched in $\mV$, then
$\mM^\circledast \times \mN^\circledast \to \mV^\ot \times \mW^\ot$ exhibits
$\mM \times \mN$ as left pseudo-enriched in $\mV$.
\item If $\mN^\circledast \to \mW^\ot$ exhibits $\mN$ as right pseudo-enriched in $\mW$, then
$\mM^\circledast \times \mN^\circledast \to \mV^\ot \times \mW^\ot$ exhibits
$\mM \times \mN$ as right pseudo-enriched in $\mW$.
\item If $\mM^\circledast \to \mV^\ot$ exhibits $\mM$ as left pseudo-enriched in $\mV$ and $\mN^\circledast \to \mW^\ot$ exhibits $\mN$ as right pseudo-enriched in $\mW$, then
$\mM^\circledast \times \mN^\circledast \to \mV^\ot \times \mW^\ot$ exhibits
$\mM \times \mN$ as bipseudo-enriched in $\mV, \mW$.
\end{enumerate}
\end{example}

\begin{construction}\label{Enros}
Let 
$\mM^\circledast \to \mV^\ot \times \mW^\ot$ be a weakly bienriched $\infty$-category.
Since $\mP\B\Env(\mM)^\circledast \to \mP\Env(\mV)^\ot \times \mP\Env(\mW)^\ot$
is a closed bitensored $\infty$-category, by Proposition \ref{Line} evaluation at the tensor units gives an equivalence
$$\Enr\Fun^\L_{\mP\Env(\mV),\mP\Env(\mW)}(\mP\Env(\mV)\ot \mP\Env(\mW),\mP\B\Env(\mM)) \simeq \mP\B\Env(\mM).$$
The graph of $\mM$ is the $\mV,\mW$-enriched functor $\Gamma_\mM: \mM^\circledast \times \mM^\op \to \mP\Env(\mV)\ot \mP\Env(\mW)$ corresponding to the composition 
$$\mM^\op \subset \mP\B\Env(\mM)^\op \simeq \Enr\Fun^\L_{\mP\Env(\mV),\mP\Env(\mW)}(\mP\Env(\mV)\ot \mP\Env(\mW),\mP\B\Env(\mM))^\op$$$$
\simeq \Enr\Fun^\R_{\mP\Env(\mV),\mP\Env(\mW)}(\mP\B\Env(\mM), \mP\Env(\mV)\ot \mP\Env(\mW))$$$$\to \Enr\Fun_{\mV,\mW}(\mM, \mP\Env(\mV)\ot \mP\Env(\mW)),$$
where the middle functor is by Proposition \ref{adjeq} and the last functor is restriction.
\end{construction}
\begin{remark}
For every $\X\in \mM$ the $\mV,\mW$-enriched functor $\Gamma_\mM(-,\X): \mM^\circledast \to (\mP\Env(\mV)\ot \mP\Env(\mW))^\circledast$ is the restriction of the enriched right adjoint of the left adjoint $\mP\Env(\mV),\mP\Env(\mW)$-linear functor $$(\mP\Env(\mV)\ot \mP\Env(\mW))^\circledast \to \mP\B\Env(\mM)^\circledast, (\V,\W) \mapsto \V\ot\X\ot\W.$$


So for every $\X, \Y \in \mM$ and $\V_1 ,...,\V_\n \in \mV, \W_1,...,\W_\m \in \mW$ for $\n, \m \geq 0$ there is an equivalence
$$ \Gamma_\mM(\X,\Y)(\V_1 \ot ... \ot \V_\n, \W_1 \ot ...\ot \W_\m) \simeq \mP\B\Env(\mM)(\V_1 \ot ... \ot \V_\n \ot \X \ot \W_1,...,\W_\m, \Y) $$$$ \simeq \Mul_\mM(\V_1, ..., \V_\n, \X, \W_1,...,\W_\m; \Y).$$

\end{remark}

\begin{remark}\label{rhhhj} an absolute small weakly bienriched $\infty$-category $\mM^\circledast \to \mV^\ot \times \mW^\ot$ exhibits $\mM$ as 
\begin{enumerate}
		
\item left pseudo-enriched in $\mV$ if and only if $\mV^\ot \to \Ass$ is a monoidal $\infty$-category and for every $\X,\Y \in \mM$ the object $$\Gamma_\mM(\X,\Y)\in \mP(\Env(\mV) \times \Env(\mW)) \simeq \Fun(\Env(\mW)^\op, \mP\Env(\mV))$$ lies in $\Fun(\Env(\mW)^\op, \mP(\mV)).$
		
\item right pseudo-enriched in $\mW$ if and only if $\mW^\ot \to \Ass$ is a monoidal $\infty$-category and for every $\X,\Y \in \mM$ the object $$\Gamma_\mM(\X,\Y)\in \mP(\Env(\mV) \times \Env(\mW)) \simeq \Fun(\Env(\mV)^\op, \mP\Env(\mW))$$ 
lies in $\Fun(\Env(\mV)^\op, \mP(\mW)).$
			
\item bipseudo-enriched in $\mV, \mW$ if and only if $\mV^\ot \to \Ass, \mW^\ot \to \Ass$ are monoidal $\infty$-categories and for every $\X,\Y \in \mM$ the object $$\Gamma_\mM(\X,\Y)\in \mP(\Env(\mV) \times \Env(\mW))$$ lies in $\mP(\mV \times \mW).$
\end{enumerate}
	
\end{remark}

\subsection{Enrichment}

Now we are ready to define enriched $\infty$-categories:

\begin{definition}
Let $\mM^\circledast\to \mV^\ot \times \mW^\ot $ be a weakly bienriched $\infty$-category.

\begin{enumerate}
\item A left multi-morphism object of $\W_1,..., \W_\m, \X, \Y \in \mM$ for $\m \geq 0$ is an object $$\L\Mul\Mor_{\mM}(\X, \W_1,..., \W_\m; \Y) \in \mV $$ such that there is a multi-morphism $\beta \in \Mul_{\mM}(\L\Mul\Mor_{\mM}(\X, \W_1,..., \W_\m; \Y), \X, \W_1,..., \W_\m; \Y) $ that induces for every objects $\V_1, ..., \V_\n \in \mV$ for $\n \geq 0$ an equivalence
$$\hspace{3mm} \Mul_{\mV}(\V_1, ..., \V_\n;  \L\Mul\Mor_{\mM}(\X, \W_1,..., \W_\m; \Y)) \simeq \Mul_{\mM}(\V_1, ..., \V_\n, \X, \W_1, ..., \W_\m; \Y).$$	
	
\item A right multi-morphism object of $\V_1,..., \V_\n, \X, \Y \in \mM$ for $\n \geq 0$ is an object $$\R\Mul\Mor_{\mM}(\V_1,..., \V_\n, \X; \Y) \in \mW $$ such that there is a multi-morphism $\alpha \in \Mul_{\mM}(\V_1,..., \V_\n, \X, \R\Mul\Mor_{\mM}(\V_1,..., \V_\n, \X; \Y); \Y) $ that induces for every objects $\W_1, ..., \W_\m \in \mW$ for $\m \geq 0$ an equivalence
$$\hspace{3mm}\Mul_{\mW}(\W_1, ..., \W_\m;  \R\Mul\Mor_{\mM}(\V_1,..., \V_\n, \X; \Y)) \simeq \Mul_{\mM}(\V_1, ..., \V_\n, \X, \W_1, ..., \W_\m; \Y).$$


\end{enumerate}

\end{definition}

\begin{definition}
Let $\mM^\circledast\to \mV^\ot \times \mW^\ot $ be a weakly bienriched $\infty$-category and $\X, \Y \in \mM$.	

\begin{enumerate}
\item A left morphism object of $\X,\Y $ in $\mM$ is a factorization
$\L\Mor_\mM(\X,\Y): \Env(\mW)^\op \to \mV$ of the functor 
$\Env(\mW)^\op \to \mP\Env(\mV)$ corresponding to $\Gamma_\mM(\X,\Y) \in \mP(\Env(\mV) \times \Env(\mW))$.

\item A right morphism object of $\X,\Y $ in $\mM$ is a factorization
$\R\Mor_\mM(\X,\Y): \Env(\mV)^\op \to \mW$ of the functor 
$\Env(\mV)^\op \to \mP\Env(\mW)$ corresponding to $\Gamma_\mM(\X,\Y) \in \mP(\Env(\mV) \times \Env(\mW))$.

\end{enumerate}

\end{definition}

\begin{remark}Let $\mM^\circledast\to \mV^\ot \times \mW^\ot $ be a weakly bienriched $\infty$-category and $\X, \Y \in \mM, \W_1, ..., \W_\m \in \mW$ for $\m \geq 0$.
A left multi-morphism object $\L\Mul\Mor_{\mM}(\W_1,..., \W_\m, \X; \Y) \in \mV$
represents the presheaf $\Gamma_\mM(\X,\Y)(\W_1 \ot ... \ot \W_\m,-) \in \mP\Env(\mV).$ 
Consequently, there is a left morphism object $\L\Mor_\mM(\X,\Y): \Env(\mW)^\op \to \mV$
if and only if for every $ \W_1, ..., \W_\m \in \mW$ for $\m \geq 0$ there is a left multi-morphism object $\L\Mul\Mor_{\mM}(\W_1,..., \W_\m, \X; \Y) \in \mV$.
In this case there is a canonical equivalence
$$ \L\Mor_\mM(\X,\Y)(\W_1 \ot ... \ot \W_\m) \simeq \L\Mul\Mor_{\mM}(\X,\W_1,..., \W_\m; \Y).$$

The similar holds for right (multi-) morphism objects. 
\end{remark}

\begin{lemma}\label{ato} Let $\mM^\circledast \to \mV^\ot \times \mW^\ot$ be a small bitensored $\infty$-category.
The linear embedding $\mM^\circledast \subset \mP(\mM)^\circledast$ preserves left and right multi-morphism objects.
\end{lemma}

\begin{proof}
First note that by presentability of the biaction, $\mP(\mM)$ admits all left and right multi-morphism objects.
We prove the case of right multi-morphism objects. The case of left multi-morphism objects is similar. Let $\X,\Y \in \mM, \V_1,...,\V_\n \in \mV$ for $\n \geq 0$ that admit a right multi-morphism object $ \R\Mul\Mor_{\mM}(\V_1,..., \V_\n, \X; \Y) \in \mW $ in $\mM$. We like to see that the induced morphism
$$ \R\Mul\Mor_{\mM}(\V_1,..., \V_\n, \X; \Y) \to \R\Mul\Mor_{\mP(\mM)}(\V_1,..., \V_\n, \X; \Y)$$ in $\mP(\mW)$
is an equivalence.
Evaluating the latter morphism at $\W \in \mW$ gives the canonical equivalence
$$\Mul_{\mM}(\V_1,..., \V_\n, \X, \W; \Y) \simeq \mW(\W,\R\Mul\Mor_{\mM}(\V_1,..., \V_\n, \X; \Y)) \to $$$$\mP(\mW)(\W,\R\Mul\Mor_{\mP(\mM)}(\V_1,..., \V_\n, \X; \Y)) \simeq\Mul_{\mP(\mM)}(\V_1,..., \V_\n, \X, \W; \Y).$$
	
\end{proof}

\begin{proposition}\label{morpre}
Let $\mM^\circledast \to \mV^\ot \times \mW^\ot$ be a weakly bienriched $\infty$-category.
The embeddings $$\mM^\circledast \subset \B\Env(\mM)^\circledast, \B\Env(\mM)^\circledast \subset \mP\B\Env(\mM)^\circledast$$ of weakly bienriched $\infty$-categories preserve left and right multi-morphism objects.

\end{proposition}

\begin{proof}
We first prove that the embedding $\iota: \mM^\circledast \subset \B\Env(\mM)^\circledast$ 
preserves left and right multi-morphism objects.
We prove the case of right multi-morphism objects. The case of left multi-morphism objects is similar.
Let $\V_1,..., \V_\n \in \mV, \X,\Y \in \mM$ for $\n \geq 0$
and assume there is a right multi-morphism object $\R\Mul\Mor_{\mM}(\V_1,..., \V_\n, \X; \Y).$ Let $\bj_\mV: \mV \subset \Env(\mV), \bj_\mW: \mW \subset \Env(\mW)$ be the canonical embeddings.
Using that $\B\Env(\mM)^\circledast \to \Env(\mV)^\ot \times \Env(\mW)^\ot$
is a bitensored $\infty$-category and so a right pseudo-enriched $\infty$-category, it is enough to see that for every $\W \in \Env(\mW)$ the following induced map is an equivalence: $$\Env(\mW)(\W,  \bi_\mW(\R\Mul\Mor_{\mM}(\V_1,..., \V_\n, \X; \Y)))
\to \Mul_{\mP\Env(\mM)}(\bj_\mV(\V_1),..., \bj_\mV(\V_\n), \iota(\X), \W).$$
As $\W \simeq \bj_\mW(\W_1) \ot ... \ot \bj_\mW(\W_\m)$ for $ \W_1,..., \W_\m \in \mW$ and $ \m \geq 0$, the latter map identifies with the map $$\Mul_{\Env(\mW)}(\bj_\mW(\W_1),..., \bj_\mW(\W_\m),  \bi_\mW(\R\Mul\Mor_{\mM}(\V_1,..., \V_\n, \X; \Y)))$$$$
\to \Mul_{\mP\Env(\mM)}(\bj_\mV(\V_1),..., \bj_\mV(\V_\n), \iota(\X), \bj_\mW(\W_1),..., \bj_\mW(\W_\m); \iota(\Y)).$$
The latter is an equivalence by definition of right multi-morphism object.
By Lemma \ref{ato} the embedding $\B\Env(\mM)^\circledast \subset \mP\B\Env(\mM)^\circledast$ preserves left and right multi-morphism objects.

		
\end{proof}

\begin{definition}
	
\begin{enumerate}
\item A weakly bienriched $\infty$-category $\mM^\circledast \to \mV^\ot \times \mW^\ot$ exhibits $\mM$ as left enriched in $\mV$ if for every $\X, \Y \in \mM$ and $\W_1,..., \W_\m \in \mW $ for $\m \geq 0$ there is a left multi-morphism object $$\L\Mul\Mor_{\mM}(\X, \W_1,..., \W_\m; \Y) \in \mV.$$

\item A weakly bienriched $\infty$-category $\mM^\circledast \to \mV^\ot \times \mW^\ot$ exhibits $\mM$ as right enriched in $\mW$ if for every $\X, \Y \in \mM$ and $\V_1,..., \V_\n \in \mV $ for $\n \geq 0$ there is a right multi-morphism object $$\R\Mul\Mor_{\mM}(\V_1,..., \V_\n, \X; \Y) \in \mW.$$

\item A weakly bienriched $\infty$-category $\mM^\circledast \to \mV^\ot \times \mW^\ot$ exhibits $\mM$ as bienriched in $\mV, \mW$ if it exhibits $\mM$ as left enriched in $\mV$ and right enriched in $\mW$.

\end{enumerate}

\end{definition}

\begin{example}\label{Excoq}

\begin{enumerate}
\item Every $\infty$-category $\mM^\circledast \to \mV^\ot $ with a closed left action of a monoidal $\infty$-category is a left $\mV$-enriched $\infty$-category. The left multimorphism object of any $\X,\Y \in \mM$ coincides with the left morphism object of $\X,\Y $ and $\L\Mor_\mM(\X,-): \mM \to \mV$ is right adjoint to the left tensor $(-) \ot \X: \mV \to \mM.$

\item Every $\infty$-category $\mM^\circledast \to \mW^\ot $ with a closed right action of a monoidal $\infty$-category is a right $\mW$-enriched $\infty$-category.
The right multimorphism object of any $\X,\Y \in \mM$ coincides with the right morphism object of $\X,\Y $ and $\R\Mor_\mM(\X,-): \mM \to \mW$ is right adjoint to the right tensor $X \ot (-): \mW \to \mM.$

\end{enumerate}

\end{example}

\begin{notation}\label{notaenr}
Let $$\L\Enr, \R\Enr, \B\Enr \subset \omega\B\Enr$$ be the full subcategories of left enriched, right enriched, bienriched $\infty$-categories.

\end{notation}

\begin{remark}\label{2-cattt}

By \cref{2-catt} the $\infty$-category 
$\omega\B\Enr$ and for every $\infty$-operads $\mV^\ot \to \Ass, \mW^\ot \to \Ass$ the $\infty$-category 
$_\mV \omega\B\Enr_\mW$ carry closed left $\Cat_\infty$-actions and so by \cref{Excoq} are left $\Cat_\infty$-enriched $\infty$-categories, known as $(\infty,2)$-categories.
For every weakly bienriched $\infty$-categories 
$\mM^\circledast \to \mV^\ot \times \mW^\ot, \mN^\circledast \to \mV^\ot \times \mW^\ot$ the $\infty$-category $\Enr\Fun_{\mV,\mW}(\mM,\mN)$ is the left morphism $\infty$-category in $_\mV \omega\B\Enr_\mW$.
Hence also the full subcategories $${\L\Enr}, {\R\Enr}, {\B\Enr} \subset {\omega\B\Enr}, {_\mV \L\Enr_\mW}, {_\mV \R\Enr_\mW}, {_\mV \B\Enr_\mW} \subset {_\mV \omega\B\Enr_\mW} $$ of left enriched, right enriched, bienriched $\infty$-categories are $(\infty,2)$-categories.

By \cref{2-catt} the forgetful functors $\omega\B\Enr \to \Cat_\infty, 
{_\mV \omega\B\Enr_\mW} \to \Cat_\infty$ are left $\Cat_\infty$-linear and so left $\Cat_\infty$-enriched functors.
The forgetful functor $
{_\mV \omega\B\Enr_\mW} \to \Cat_\infty$ induces on left morphism $\infty$-categories the canonical functor $\Enr\Fun_{\mV,\mW}(\mM,\mN) \to \Fun(\mM,\mN)$.
Hence also the forgetful functors $$ {_\mV \L\Enr_\mW} \to \Cat_\infty, {_\mV \R\Enr_\mW}\to \Cat_\infty, {_\mV \B\Enr_\mW} \to \Cat_\infty$$ are left $\Cat_\infty$-enriched functors.
    
\end{remark}

For the next corollaries we say that a functor of
$(\infty,2)$-categories, i.e. a left $\infty\Cat$-enriched functor,
is locally conservative if it induces on morphism $\infty$-categories conservative functors.

\begin{corollary}\label{loccons}
Let $\mV^\ot \to \Ass, \mW^\ot \to \Ass$ be small $\infty$-operads.
The left $\infty\Cat$-enriched forgetful functor $
{_\mV \omega\B\Enr_\mW} \to \Cat_\infty$ 
is locally conservative.
\end{corollary}
\begin{proof}

The forgetful functor $
{_\mV \omega\B\Enr_\mW} \to \Cat_\infty$ induces on left morphism $\infty$-categories between any weakly bienriched $\infty$-categories 
$\mM^\circledast \to \mV^\ot \times \mW^\ot, \mN^\circledast \to \mV^\ot \times \mW^\ot$ the canonical functor $\Enr\Fun_{\mV,\mW}(\mM,\mN) \to \Fun(\mM,\mN)$, which is conservative by \cref{conserva}.
    
\end{proof}

\begin{corollary}
Let $\mV^\ot \to \Ass, \mW^\ot \to \Ass$ be small $\infty$-operads.
The left $\infty\Cat$-enriched forgetful functors $$ {_\mV \L\Enr_\mW} \to \Cat_\infty, {_\mV \R\Enr_\mW}\to \Cat_\infty, {_\mV \B\Enr_\mW} \to \Cat_\infty$$
are locally conservative.
\end{corollary}

\begin{proof}
The left $\infty\Cat$-enriched forgetful functors $$ {_\mV \L\Enr_\mW} \to \Cat_\infty, {_\mV \R\Enr_\mW}\to \Cat_\infty, {_\mV \B\Enr_\mW} \to \Cat_\infty$$ are the restriction along embeddings of the forgetful functor $
{_\mV \omega\B\Enr_\mW} \to \Cat_\infty$, which is locally conservative by \cref{loccons}.
    
\end{proof}

\begin{example}\label{Excopo}

\begin{enumerate}
\item Let $\mM^\circledast \to \mV^\ot $ be a weakly left enriched $\infty$-category.
Then $\mM^\circledast \to \mV^\ot $ is a left $\mV$-enriched $\infty$-category if and only if it is a $\mV$-enriched $\infty$-category in the sense of \cite[Definition 3.112.]{HEINE2023108941}.
The left multimorphism object of any $\X,\Y \in \mM$ agrees with the left morphism object of $\X,\Y $,
which coincides with the morphism object of $\X,\Y $ in the sense of \cite[Definition 3.111.]{HEINE2023108941}.

\item Let $\mM^\circledast \to \mW^\ot $ be a weakly right enriched $\infty$-category.
Then $\mM^\circledast \to \mW^\ot$ is a right $\mW$-enriched $\infty$-category if and only if $(\mM^\rev)^\circledast \to (\mW^\rev)^\circledast$ is a left $\mW^\rev$-enriched $\infty$-category and so by (1) a $\mW^\rev$-enriched $\infty$-category.

\end{enumerate}

\end{example}

\begin{remark}\label{twocat} Let $\mV^\ot \to \Ass$ be an $\infty$-operad.
In view of \cref{Excopo} the $\infty$-category of left $\mV$-enriched $\infty$-categories is precisely the $\infty$-category of 
$\mV$-enriched $\infty$-categories in the sense of 
\cite[Definition 3.112.]{HEINE2023108941}.
In \cite[Proposition 6.10.]{HEINE2023108941} we construct an equivalence between the $\infty$-category ${_\mV\L\Enr_\emptyset}$
and the $\infty$-category of $\mV$-enriched $\infty$-categories in the sense of Gepner-Haugseng \cite{GEPNER2015575}. Macpherson \cite[Theorem 1.1.]{MR4185309} proved that the latter $\infty$-category is equivalent to the $\infty$-category of $\mV$-enriched $\infty$-categories in the sense of Hinich \cite{HINICH2020107129}.
Moreover we extend the equivalence of \cite[Proposition 6.10.]{HEINE2023108941} to an equivalence of $(\infty,2)$-categories
\cite[Theorem 8.3.]{https://doi.org/10.48550/arxiv.2009.02428}.
By \cref{2-cattt} the $\infty$-category ${_\mV\L\Enr_\emptyset}$
is left enriched in $\infty\Cat$, where the left morphism
$\infty$-category between two left $\mV$-enriched $\infty$-categories
$\mM^\circledast \to \mV^\ot, \mN^\circledast \to \mV^\ot$
is the $\infty$-category
$ \Enr\Fun_{\mV, \emptyset}(\mM,\mN) $ of left $\mV$-enriched functors of \cref{enrfunc}.
In \cite[Theorem 6.16.]{HEINE2023108941}, \cite[Corollary 4.19.]{heine2024equivalence} we prove that the $\infty$-category
$ \Enr\Fun_{\mV, \emptyset}(\mM,\mN) $
is equivalent to the $\infty$-category of $\mV$-enriched functors constructed by Hinich \cite[Definition 6.1.3.]{HINICH2020107129}.

\end{remark}




\begin{example}\label{Exco}
	
Let $\mM^\circledast \to \mV^\ot \times \mW^\ot$ be a weakly bienriched $\infty$-category
whose underlying weakly left enriched $\infty$-category $\mN^\circledast \to \mV^\ot$ exhibits $\mM$ as left $\mV$-enriched.

\begin{enumerate}
\item If $\mM^\circledast \to \mV^\ot \times \mW^\ot$ admits right tensors, $\mM^\circledast \to \mV^\ot \times \mW^\ot$ exhibits $\mM$ as left $\mV$-enriched.
For every $\X,\Y \in \mM$ and $\W_1, ..., \W_\m$ for $\m \geq 0$ the left multi-morphism object
is $$\L\Mul\Mor_\mM(\X,\W_1,...,\W_\m;\Y) = \L\Mor_\mN(((-)\ot\W_\m) \circ ... \circ ((-)\ot\W_1)(\X) ,\Y).$$

\item If $\mM^\circledast \to \mV^\ot \times \mW^\ot$ admits right cotensors, $\mM^\circledast \to \mV^\ot \times \mW^\ot$ exhibits $\mM$ as left $\mV$-enriched.
For every $\X,\Y \in \mM$ and $\W_1, ..., \W_\m$ for $\m \geq 0$ the left multi-morphism object
is $$\L\Mul\Mor_\mM(\X,\W_1,...,\W_\m;\Y) = \L\Mor_\mN(\X, (-)^{\W_\m} \circ ... \circ (-)^{\W_1}(\Y)).$$

\item If $\mM^\circledast \to \mV^\ot \times \mW^\ot$ admits right tensors and right cotensors, $\mM^\circledast \to \mV^\ot \times \mW^\ot$ exhibits $\mM$ as left $\mV$-enriched
and for every $\X,\Y \in \mM$ and $\W_1, ..., \W_\m$ for $\m \geq 0$ there is a canonical equivalence
$$ \L\Mor_\mN(((-)\ot\W_\m) \circ ... \circ ((-)\ot\W_1)(\X) ,\Y) \simeq \L\Mor_\mN(\X, (-)^{\W_\m} \circ ... \circ (-)^{\W_1}(\Y)). $$

\end{enumerate}	
	
\end{example}


\begin{remark}
Let $\mM^\circledast \to \mV^\ot \times \mW^\ot$ be a bienriched $\infty$-category.
The underlying weakly left enriched $\infty$-category $\mN^\circledast \to \mV^\ot$ of
$\mM^\circledast \to \mV^\ot \times \mW^\ot$ exhibits $\mM$ as left $\mV$-enriched,
where for any $\X,\Y \in \mM$ there is a canonical equivalence
$\L\Mor_\mN(\X,\Y)\simeq \L\Mul\Mor_\mM(\X;\Y).$
If $\mM^\circledast \to \mV^\ot \times \mW^\ot$ admits right tensors or right cotensors, the left multi-morphism spaces of $\mM^\circledast \to \mV^\ot \times \mW^\ot$ are 
given by the left multi-morphism spaces of $\mN^\circledast \to \mV^\ot$ by Example \ref{Exco}.
	
\end{remark}

\begin{lemma}
A weakly bienriched $\infty$-category $\mM^\circledast \to \mV^\ot \times \mW^\ot$ that exhibits $\mM$ as left enriched in a monoidal $\infty$-category, right enriched in a monoidal $\infty$-category, bienriched in monoidal $\infty$-categories, respectively, exhibits $\mM$ as left pseudo-enriched in $\mV$,
right pseudo-enriched in $\mW$, bipseudo-enriched in $\mV, \mW$, respectively.
	
\end{lemma}

\begin{proof}We prove the case of left enrichment. The case of right enrichment is dual and the case of bi-enrichment follows from the cases of left and right enrichment
using Remark \ref{both}.
For every $\X,\Y \in \mM$ and $\V_1, ..., \V_\n \in \mV, \W_1, ..., \W_\m \in \mW$ for $\n,\m \geq 0$ the map
$$ \Mul_{\mM}(\V_1 \ot ... \ot \V_\n, \X, \W_1, ..., \W_\m; \Y) \to \Mul_{\mM}(\V_1, ..., \V_\n, \X, \W_1, ..., \W_\m ; \Y)$$ 
induced by the active morphism $\V_1, ..., \V_\n \to \V_1 \ot  ... \ot \V_\n $
in $\mV^\ot$ identifies with the induced map
$$ \Mul_{\mV}(\V_1 \ot ... \ot \V_\n, \L\Mul\Mor_\mM(\X, \W_1, ..., \W_\m; \Y)) \to \Mul_{\mV}(\V_1, ..., \V_\n, \L\Mul\Mor_\mM(\X, \W_1, ..., \W_\m; \Y)),$$ 
which is an equivalence since the active morphism $\V_1, ..., \V_\n \to \V_1 \ot  ... \ot \V_\n $ is cocartesian over $\Ass.$

	
\end{proof}

\begin{remark}
Let $\mC, \mD$ be small $\infty$-categories.
A presheaf $\alpha \in \mP(\mC \times \mD)$
is representable in both variables if for every $\X \in \mC, \Y \in \mD$
the presheaves $\alpha(\X,-): \mD^\op \to \mS, \alpha(-,\Y): \mC^\op \to \mS$ are representable. Let $\mP_\rep(\mC\times \mD)\subset \mP(\mC\times\mD)$ be the full subcategory of presheaves representable in both variables.

The following conditions are equivalent for a presheaf $\alpha \in \mP(\mC \times \mD)$:
\begin{enumerate}
\item The presheaf $\alpha \in \mP(\mC \times \mD)$ is representable in both variables.
\item The corresponding functors $\mC^\op \to \mP(\mD), 
\mD^\op \to \mP(\mC)$ induce functors $\F^\op: \mC^\op \to \mD, \G: \mD^\op \to \mC.$

\item The corresponding functor $\mD^\op \to \mP(\mC)$ induces a functor $ \G: \mD^\op \to \mC$ that admits a left adjoint.
In this case for every $\X \in \mC$ the presheaf $\mC(\X,-) \circ \G \simeq \alpha(\X,-) \in \mP(\mD)$ is represented by $\F(\X)$.
\item The corresponding functor $\mC^\op \to \mP(\mD)$ induces a functor $\F^\op: \mC^\op \to \mD$ that admits a left adjoint.
In this case for every $\Y \in \mD$ the presheaf $\mD(\Y,-) \circ \F^\op \simeq \alpha(-,\Y) \in \mP(\mC)$ is represented by $\G(\Y)$.

\end{enumerate}
Consequently, the canonical equivalence $\mP(\mC\times \mD) \simeq \Fun(\mC^\op, \mP(\mD))$ restricts to an equivalence $$\mP_\rep(\mC\times \mD) \simeq \Fun^\R(\mC^\op,\mD).$$
The canonical equivalence $\mP(\mC\times \mD) \simeq \mP(\mD\times\mC)$
restricts to an equivalence $$\Fun^\R(\mC^\op,\mD) \simeq \mP_\rep(\mC\times\mD) \simeq \mP_\rep(\mD\times\mC) \simeq \Fun^\R(\mD^\op,\mC).$$

Let $\mC, \mD$ be presentable $\infty$-categories. 
The universal functor $\mC \times \mD \to \mC \ot \mD$ preserving small colimits component-wise induces an embedding $\Fun^\R((\mC \otimes \mD)^\op, \mS) \to \mP(\mC \times \mD)$, which restricts to equivalences
$$\mC \ot \mD \simeq \Fun^\R((\mC \otimes \mD)^\op, \mS) \simeq \mP_\rep(\mC \times \mD) \simeq \Fun^\R(\mC^\op,\mD)$$ by the adjoint functor theorem \cite[Proposition 5.5.2.2.]{lurie.HTT}, where the first equivalence is induced by the Yoneda-embedding.
The universal functor $\mC \times \mD \to \mC \ot \mD \simeq \Fun^\R(\mC^\op,\mD)$
preserving small colimits component-wise sends $\C,\D$ to the image of $\D$ under the left adjoint $\lan_\C$ of the functor $\Fun^\R(\mC^\op,\mD) \to \mD$ evaluating at $\C.$
	
\end{remark}

\begin{remark}\label{EEnr}
Let $\mM^\circledast \to \mV^\ot \times \mW^\ot$ be a weakly bienriched $\infty$-category and $\X,\Y \in \mM$ such that there is a left and a right morphism object $\L\Mor_\mM(\X,\Y), \R\Mor_\mM(\X,\Y).$ Then $$ \Gamma_\mM(\X,\Y)\in \mP(\Env(\mV) \times \Env(\mW)) \simeq \Fun(\Env(\mV)^\op, \mP\Env(\mW))$$
is the image of $\R\Mor_\mM(\X,\Y) \in \Fun(\Env(\mV)^\op, \mW)$ and $$\Gamma_\mM(\X,\Y)\in \mP(\Env(\mV) \times \Env(\mW)) \simeq \Fun(\Env(\mW)^\op, \mP\Env(\mV)) $$
is the image of $ \L\Mor_\mM(\X,\Y) \in \Fun(\Env(\mW)^\op, \mV).$ 
Thus $ \R\Mor_\mM(\X,\Y), \L\Mor_\mM(\X,\Y)$ lie in $$\Fun^\R(\Env(\mV)^\op, \Env(\mW)) \simeq \Fun^\R(\Env(\mW)^\op, \Env(\mV))$$
so that there is an adjunction 
$$\L\Mor_{\mM}(\X,\Y)^\op: \Env(\mW) \rightleftarrows \Env(\mV)^\op: \R\Mor_{\mM}(\X,\Y)$$
whose left adjoint lands in $\mV^\op$ and right adjoint lands in $\mW.$

If $\mM^\circledast \to \mV^\ot \times \mW^\ot$ is bipseudo-enriched,
$ \Gamma_\mM(\X,\Y)\in \mP(\Env(\mV) \times \Env(\mW))$ lies in $\mP(\mV \times \mW)$, the left morphism object $\L\Mor_{\mM}(\X,\Y): \Env(\mW)^\op \to \Env(\mV)$
factors as $ \Env(\mW)^\op \to \mW^\op \xrightarrow{\L\Mor_{\mM}(\X,\Y)_{\mid \mW^\op}} \Env(\mV)$, the right morphism object $\R\Mor_{\mM}(\X,\Y):\Env(\mV)^\op \to \Env(\mW) $ 
factors as $$ \Env(\mV)^\op \to \mV^\op \xrightarrow{\R\Mor_{\mM}(\X,\Y)_{\mid \mV^\op}} \Env(\mW)$$ and there is an adjunction 
$$\L\Mor_{\mM}(\X,\Y)_{\mid \mW}^\op: \mW \rightleftarrows \mV^\op: \R\Mor_{\mM}(\X,\Y)_{\mid \mV}.$$

\end{remark}

\begin{remark}\label{duall}
Let $\mV^\ot \to \Ass$ be a closed monoidal $\infty$-category and $B\tu_\mV^\circledast \subset \mV^\circledast$ the full weakly bienriched subcategory spanned by the tensor unit.	
Then $B\tu_\mV^\circledast \to \mV^\ot \times \mV^\ot$ is a bienriched $\infty$-category and by the latter remark there is an adjunction
$\mV \rightleftarrows \mV^\op,$ 
which identifies with the canonical adjunction
$ \L\Mor_\mV(-,\tu_\mV): \mV \rightleftarrows \mV^\op: \R\Mor_\mV(-,\tu_\mV)$
(see \cite[Lemma 3.11.]{heine2023topological} for this adjunction).
	
\end{remark}

\begin{remark}Let $\mV^\ot \to \Ass, \mW^\ot \to \Ass$ be presentably monoidal $\infty$-categories, $\mM^\circledast \to \mV^\ot \times \mW^\ot$ a weakly bienriched $\infty$-category and $\X, \Y \in \mM.$
Under the canonical equivalence
$\mV \ot \mW \simeq \Fun^\R(\mV^\op,\mW)$ the left morphism object $\L\Mor_\mM(\X,\Y)$
is identified with an object of $\mV \ot \mW$, which we denote by the same name.
Let $\alpha$ be the universal monoidal functor $\mV \times \mW \to \mV \ot \mW$
preserving small colimits component-wise.
Since for every $\V \in \mV, \W \in \mW$ the object $\alpha(\V,\W)\in \mV \ot \mW$ corresponds to $\lan_\V(\W) \in \Fun^\R(\mV^\op,\mW)$, there is a canonical equivalence
$$\mV \ot \mW(\alpha(\V,\W), \L\Mor_{\mM}(\X, \Y)) \simeq \Fun^\R(\mV^\op,\mW)(\lan_\V(\W),\L\Mor_{\mM}(\X, \Y)) \simeq$$$$ \mW(\W,\L\Mor_{\mM}(\X, \Y)(\V))\simeq \mW(\W,\L\Mul\Mor_{\mM}(\V, \X; \Y)) \simeq \Mul_{\mM}(\V, \X, \W; \Y).$$
Consequently, $\Mul_{\mM}(-, \X,-; \Y) \in \mP(\mV \times \mW) $ is the image of 
$\L\Mor_\mM(\X,\Y) \in \mV \ot \mW$ under the embedding $$\alpha^*: \mV \ot \mW \simeq \mP(\mV \ot \mW) \to \mP(\mV \times \mW).$$
Thus for every $\V \in \mV, \W \in \mW$ there is a canonical equivalence
\begin{equation}\label{eqppp}
\mV \ot \mW(\alpha(\V,\W), \L\Mor_{\mM}(\X, \Y)) \simeq \mM(\V \ot \X \ot \W,\Y) \simeq \Mul_{\mM}(\V,\X,\W;\Y).\end{equation}
	
\end{remark}

\begin{example}\label{Biolk}

Let $\mV^\ot \to \Ass, \mW^\ot \to \Ass$ be presentably monoidal $\infty$-categories and 
$\mM^\circledast \to \mV^\ot \times \mW^\ot$ a bitensored $\infty$-category
such that for every $\X \in \mM$ the unique $\mV, \mW$-linear functor $\mV \ot \mW \to \mM$ sending the tensor unit to $\X$ of Proposition \ref{Line} admits a right adjoint $\gamma_\X.$
Then $\mM^\circledast \to \mV^\ot \times \mW^\ot$ exhibits $\mM$ as bienriched in $\mV,\mW.$
The right adjoint sends every $\Y \in \mM$ to a left morphism object $\L\Mor_{\mM}(\X, \Y) \in \mV \ot \mW $ by equivalence (\ref{eqppp}).
In particular, every presentably bitensored $\infty$-category $\mM^\circledast \to \mV^\ot \times \mW^\ot$ is a bienriched $\infty$-category.



\end{example}

\begin{remark}
By Example \ref{Biolk} for every weakly bienriched $\infty$-category $\mM^\circledast \to \mV^\ot \times \mW^\ot$ and every $\X, \Y \in \mM$ the object
$\Gamma_\mM(\X, \Y) \in \mP(\Env(\mV)\times \Env(\mW)) \simeq \mP\Env(\mV) \otimes \mP\Env(\mW) $ is the left morphism object $\L\Mor_{\bar{\mM}}(\X,\Y)$ of $\X$ and $\Y$,
where $\bar{\mM}^\circledast \subset \mP\B\Env(\mM)^\circledast \to \mP\Env(\mV)^\ot \times \mP\Env(\mW)^\ot$ is the full weakly bienriched subcategory spanned by $\mM$.

\end{remark}

Moreover we use the following terminology:

\vspace{1mm}
\begin{definition}\label{Luel}Let $\kappa, \tau$ be regular cardinals
and $\phi: \mM^\circledast\to \mV^\ot \times \mW^\ot$ a weakly bienriched $\infty$-category.
	
\begin{enumerate}
\item We say that $\phi$ exhibits $\mM$ as left $\kappa$-enriched in $\mV$ if $\mV^\ot \to\Ass$ is a monoidal $\infty$-category compatible with $\kappa$-small colimits, $\phi$ is a left pseudo-enriched $\infty$-category and for every $\X,\Y \in \mM$ and $\W_1,...,\W_\m \in \mW$ for $\m \geq0$ the presheaf
$\Mul_\mM(-,\X, \W_1,...,\W_\m;\Y)$ on $\mM$ preserves $\kappa$-small limits.
	
\vspace{1mm}
		
\item We say that $\phi$ exhibits $\mM$ as right $\tau$-enriched in $\mW$ if $\mW^\ot \to\Ass$ is a monoidal $\infty$-category compatible with $\tau$-small colimits, $\phi$ is a right pseudo-enriched $\infty$-category and for every $\X,\Y \in \mM$ and $\V_1,...,\V_\n \in \mV$ for $\n \geq0$ the presheaf
$\Mul_\mM(\V_1,...,\V_\n,\X,-;\Y)$ on $\mM$ preserves $\tau$-small limits.
		
\vspace{1mm}
		
\item We say that $\phi$ exhibits $\mM$ as $\kappa, \tau$-bienriched in $\mV, \mW$ if it exhibits $\mM$ as left $\kappa$-enriched in $\mV$ and right $\tau$-enriched in $\mW.$
		
\end{enumerate}
	
\end{definition}

\begin{notation}Let $\kappa, \tau$ be small regular cardinals.
Let $$ ^\kappa\L\Enr, \R\Enr^\tau, {^\kappa\B\Enr}^{\tau} \subset \omega\B\Enr $$ be the full subcategories of left $\kappa$-enriched, right $\tau$-enriched, $\kappa, \tau$-bienriched $\infty$-categories, respectively.
	
\end{notation}

\begin{example}
	
A weakly bienriched $\infty$-category $\mM^\circledast\to \mV^\ot \times \mW^\ot$ exhibits $\mM$ as left $\emptyset$-enriched in $\mV$ if it exhibits $\mM$ as left pseudo-enriched in $\mV.$
A weakly bienriched $\infty$-category $\mM^\circledast\to \mV^\ot \times \mW^\ot$ exhibits $\mM$ as right $\emptyset$-enriched in $\mW$ if it exhibits $\mM$ as right pseudo-enriched in $\mW.$
	
\end{example}

\begin{example}\label{laulo2}
Let $\kappa$ be a small regular cardinal. 
A weakly bienriched $\infty$-category $\mM^\circledast \to \mV^\ot \times \mW^\ot$ is a left (right) tensored $\infty$-category compatible with $\kappa$-small colimits if and only if $\mM^\circledast \to \mV^\ot \times \mW^\ot$ is a left (right) $\kappa$-enriched $\infty$-category that admits left (right) tensors and $\kappa$-small conical colimits.
	
\end{example}

\begin{definition}Let $\sigma$ be the strongly inaccessible cardinal corresponding to the small universe and $\phi: \mM^\circledast\to \mV^\ot \times \mW^\ot$ a weakly bienriched $\infty$-category.
	
\begin{enumerate}
\item We say that $\phi$ exhibits $\mM$ as left quasi-enriched in $\mV$
if $\phi$ exhibits $\mM$ as left $\sigma$--enriched in $\mV$.
		
\item We say that $\phi$ exhibits $\mM$ as right quasi-enriched in $\mW$
if $\phi$ exhibits $\mM$ as right $\sigma$-enriched in $\mW$.
		
\item We say that $\phi$ exhibits $\mM$ as biquasi-enriched in $\mV, \mW$
if $\phi$ exhibits $\mM$ as $\sigma, \sigma$-bienriched in $\mV, \mW$.
\end{enumerate}
	
\end{definition}

\begin{notation}
Let $$ \L\Q\Enr, \R\Q\Enr, \B\Q\Enr \subset \omega\widehat{\B\Enr} $$ be the full subcategories of left quasi-enriched, right quasi-enriched, biquasi-enriched $\infty$-categories, respectively, whose underlying $\infty$-category is small.
	
\end{notation}

\begin{remark}
Let $\mV^\ot \to \Ass, \mW^\ot \to \Ass$ be presentably monoidal $\infty$-categories.
An $\infty$-category left quasi-enriched in $\mV$ is an $\infty$-category left enriched in $\mV$ if and only if it is locally small.
An $\infty$-category right quasi-enriched in $\mW$ is an $\infty$-category right enriched in $\mW$ if and only if it is locally small.
An $\infty$-category biquasi-enriched in $\mV, \mW$ is an $\infty$-category bienriched in $\mV, \mW$ if and only if it is locally small.
	
\end{remark}

\section{Enriched presheaves via tensored envelopes}\label{loal}

\subsection{Generalized enriched presheaves}

In this section we construct the $\infty$-category of
enriched presheaves as a localization of the enveloping $\infty$-category with closed biaction (Proposition \ref{pseuso}, Proposition \ref{unipor}, Theorem \ref{unipor3}).
For that we introduce the following concept that treats weak enrichment, pseudo-enrichment and enrichment on one footage.

\begin{definition}\label{locpa}

A small localization pair is a pair $(\mV^\ot \to \Ass, \rS)$,
where $\mV^\ot \to \Ass$ is a small $\infty$-operad and $\rS$ is a set of morphisms
of $\mP\Env(\mV)$ 
such that the saturated closure $\bar{\rS}$ of $\rS$ is closed under the tensor product 
and for every $\V \in \mV$ and $\f \in \rS$ 
the map $\mP\Env(\mV)(\f,\V)$ is an equivalence.
\end{definition}
\begin{remark}

Let $(\mV^\ot \to \Ass, \rS)$ be a small localization pair.
Since $\mP\Env(\mV)^\ot \to \Ass$ is a presentably monoidal $\infty$-category, $\rS$ is a set and $\bar{\rS}$ is closed under the tensor product, the
embedding $\rS^{-1}\mP\Env(\mV)^\ot \subset \mP\Env(\mV)^\ot$ of the full suboperad spanned by the $\rS$-local objects admits a left adjoint relative to $\Ass.$
Moreover by definition the embedding of $\infty$-operads $\mV^\ot \subset \mP\Env(\mV)^\ot $ induces an embedding of $\infty$-operads $\mV^\ot \subset \rS^{-1}\mP\Env(\mV)^\ot.$
\end{remark}


\begin{definition}
Let $(\mV^\ot \to \Ass, \rS)$, $(\mW^\ot \to \Ass, \T)$	be small localization pairs.

\begin{enumerate}
\item A weakly bienriched $\infty$-category $\mM^\circledast \to \mV^\ot \times \mW^\ot$ exhibits $\mM$ as left $\rS$-enriched in $\mV$
if for every $\f \in \rS, \X, \Y \in \mM, \W \in \mP\Env(\mW)$
the induced map $\mP\B\Env(\mM)(\f \ot \X \ot \W, \Y)$ is an equivalence. 

\item A weakly bienriched $\infty$-category $\mM^\circledast \to \mV^\ot \times \mW^\ot$ exhibits $\mM$ as right $\T$-enriched in $\mW$
if for every $\g \in \T, \X, \Y \in \mM, \V \in \mP\Env(\mV)$
the map $\mP\B\Env(\mM)(\V \ot \X \ot \g, \Y)$ is an equivalence. 

\item A weakly bienriched $\infty$-category $\mM^\circledast \to \mV^\ot \times \mW^\ot$ exhibits $\mM$ as $\rS, \T$-bienriched in $\mV, \mW$ if it exhibits $\mM$ as left $\rS$-enriched in $\mV$ and right $\T$-enriched in $\mW$.
\end{enumerate}	
\end{definition}

\begin{remark}
A weakly bienriched $\infty$-category $\mM^\circledast \to \mV^\ot \times \mW^\ot$ exhibits $\mM$ as $\rS, \T$-bienriched in $\mV, \mW$ if and only if 
for every $\f \in \bar{\rS}, \g \in \bar{\T}, \X, \Y \in \mM$ the map $\mP\B\Env(\mM)(\f \ot \X \ot \g, \Y)$ is an equivalence.

\end{remark}

\begin{notation}Let $(\mV^\ot \to \Ass, \rS)$, $(\mW^\ot \to \Ass, \T)$	be small localization pairs.
Let $$_\mV ^\rS\L\Enr_{\mW}, \ {_\mV\R\Enr}^{\T}_{\mW}, {^\rS_\mV\B\Enr^\T_{\mW}} \subset {_\mV\omega\B\Enr}_{\mW} $$ be the full subcategories of left $\rS$-enriched, right $\T$-enriched, $\rS, \T$-bienriched $\infty$-categories, respectively.

\end{notation}

\begin{example}Let $\mU^\ot \to \Ass$ be a small $\infty$-operad.
The pair $(\mU^\ot \to \Ass, \emptyset)$ is a small localization pair,
where $\emptyset^{-1}\mP\Env(\mU)^\ot= \mP\Env(\mU)^\ot.$
Let $(\mV^\ot \to \Ass, \rS)$, $(\mW^\ot \to \Ass, \T)$	be small localization pairs.
Then $$_\mV ^\rS\L\Enr_{\mW}=  {^\rS_\mV\B\Enr^\emptyset_{\mW}}, \ {_\mV\R\Enr}^{\T}_{\mW}= {^\emptyset_\mV\B\Enr^\T_{\mW}}.$$


\end{example}

\begin{notation}\label{enr}
Let $\kappa$ be a small regular cardinal and $\mV^\ot \to \Ass$ a small monoidal $\infty$-category compatible with $\kappa$-small colimits. 
Let $\Enr_\kappa$ be the set of morphisms
$$\V'_1, ... , \V'_\ell, \V_1, ..., \V_\n, \V''_1, ... , \V''_\bk \to \V'_1, ... , \V'_\ell,\V_1 \ot ...\ot \V_\n, \V''_1, ... , \V''_\bk,$$ 
$$ \V'_1, ... , \V'_\ell,\colim(\iota \circ \F), \V''_1, ... , \V''_\bk \to \V'_1, ... , \V'_\ell, \iota(\colim(\F)),\V''_1, ... , \V''_\bk$$
for $\V'_1, ... , \V'_\ell, \V_1,...,\V_\n, \V''_1, ... , \V''_\bk \in \mV$ and $\ell,\n,\bk \geq0, $ where $\iota: \mV \subset \mP\Env(\mV)$ is the embedding, $\F: \K \to \mV$ is a functor and $\K$ is $\kappa$-small.
\end{notation}


\begin{notation}Let $\kappa, \tau$ be small regular cardinals.

\begin{enumerate}
\item If $\mV^\ot \to \Ass$ is a small monoidal $\infty$-category compatible with $\kappa$-small colimits and $\mW^\ot \to \Ass$ is a small $\infty$-operad, let $\L\Enr_\kappa:=(\Enr_\kappa,\emptyset)$.

\item If $\mV^\ot \to \Ass$ is a small $\infty$-operad and $\mW^\ot \to \Ass$ is a small monoidal $\infty$-category compatible with $\kappa$-small colimits, let $\R\Enr_\kappa:=(\emptyset,\Enr_\kappa)$.

\item If $\mV^\ot \to \Ass, \mW^\ot \to \Ass$ are small monoidal $\infty$-categories  compatible with $\kappa$-small colimits, $\tau$-small colimits, respectively, let $\B\Enr_{\kappa,\tau}:=(\Enr_\kappa,\Enr_\tau)$.
If $\kappa,\tau$ are the the strongly inaccessible cardinal corresponding to the small universe, we drop $\kappa, \tau$, respectively.



\end{enumerate}	

\end{notation}

\begin{example}\label{Exaso}
Let $\kappa$ be a small regular cardinal.

\begin{enumerate}
\item For every small monoidal $\infty$-category $\mV^\ot \to \Ass$ compatible with $\kappa$-small colimits the pair $(\mV^\ot \to \Ass, \Enr_\kappa)$ is a small localization pair, where $\Enr_\kappa^{-1}\mP\Env(\mV)^\ot= \Ind_\kappa(\mV)^\ot$,
which is a monoidal localization of $\mP\Env(\mV)^\ot \to \Ass.$
A weakly bienriched $\infty$-category $\mM^\circledast \to \mV^\ot \times \mW^\ot$ exhibits $\mM$ as left $\Enr_\kappa$-enriched in $\mV$ if and only if
it exhibits $\mM$ as left $\kappa$-enriched in $\mV$.

\item If $\mV^\ot \to\Ass$ is a presentably monoidal $\infty$-category,
a locally small weakly bienriched $\infty$-category $\mM^\circledast \to \mV^\ot \times \mW^\ot$ exhibits $\mM$ as left $\Enr$-enriched in $\mV$ if and only if it exhibits $\mM$ as left enriched in $\mV$.
\end{enumerate}
\end{example}

\begin{definition}
Let $(\mV^\ot \to \Ass, \rS)$, $(\mV'^\ot \to \Ass, \rS')$ be small localization pairs.	
A map of localization pairs $(\mV^\ot \to \Ass, \rS) \to (\mV'^\ot \to \Ass, \rS')$
is a map of $\infty$-operads $\mV^\ot \to \mV'^\ot$ such that the induced left adjoint monoidal functor $\mP\Env(\mV) \to \mP\Env(\mV')$
sends morphisms of $\rS$ to morphisms of $\rS'.$

\end{definition}

\begin{example}



Let $\mV^\ot \to \Ass, \mV'^\ot \to \Ass$ be small monoidal $\infty$-categories compatible with $\kappa$-small colimits, compatible with $\kappa'$-small colimits,
respectively, for small regular cardinals $\kappa < \kappa'$.
A map of $\infty$-operads $\mV^\ot \to \mV'^\ot$ is a map of localization pairs
$(\mV^\ot \to \Ass, \Enr_\kappa) \to (\mV'^\ot \to \Ass, \Enr_{\kappa'})$
if and only if it is a monoidal functor preserving $\kappa$-small colimits.

\end{example}

\begin{lemma}\label{Lemut}

Let $\mC, \mD$ be presentable $\infty$-categories, $\rS$ a set of morphisms of $\mC$ and $\T$ a set of morphisms of $\mD.$
Let $\bar{\rS}, \bar{\T}$ be the saturated classes generated by $\rS, \T$
and $\bar{\rS} \ot \bar{\T}$ the image of the set
$\bar{\rS} \times \bar{\T}$ under the universal functor $\mC \times \mD \to \mC\ot \mD.$
The left adjoint functor $\mC \ot \mD \to \rS^{-1}(\mC)\ot \T^{-1}\mD$
admits a fully faithful right adjoint whose essential image are the $\bar{\rS} \ot \bar{\T}$-local objects.

\end{lemma}

\begin{proof}

By \cite[Proposition 4.8.1.17.]{lurie.higheralgebra} there is an equivalence
$\mC \ot \mD \simeq \Fun^\R(\mC^\op, \mD)$,
where the right hand side is the full subcategory of right adjoint functors $\mC^\op \to \mD.$
The functor $\mC \ot \mD \to \rS^{-1}(\mC)\ot \T^{-1}\mD$ is left adjoint to the embedding $\Fun^\R(\rS^{-1}(\mC)^\op, \T^{-1}\mD) \subset \Fun^\R(\mC^\op, \T^{-1}\mD) \to \Fun^\R(\mC^\op, \mD)$
whose essential image are the small limits preserving functors $\mC^\op \to \mD$
that land in $\T^{-1}\mD$ and invert $\rS.$

The functor $\mC \times \mD \to \mC\ot \mD \simeq \Fun^\R(\mC^\op, \mD)$
sends $\X, \Y$ to $\mC^\op \xrightarrow{\mC(-,\X)}\mS \xrightarrow{(-)\ot \Y} \mD.$
An object $\F \in \mC \ot \mD \simeq \Fun^\R(\mC^\op, \mD)$ is local with respect to
$\bar{\rS} \ot \bar{\T}$ if and only if it is local with respect to $\f \ot \Y, \X \ot \g$ for any $\f \in \bar{\rS}, \g \in \bar{\T}, \X \in \mC, \Y \in \mD$.
There are canonical equivalences $$\mC \ot \mD(\f \ot \Y,\F) \simeq \Fun^\R(\mC^\op, \mD)(\mC(-,\f) \ot \Y,\F) \simeq \mD(\Y, \F(\f)),$$$$ \mC \ot \mD(\X \ot \g,\F) \simeq \Fun^\R(\mC^\op, \mD)(\mC(-,\X) \ot \g,\F) \simeq \mD(\g, \F(\X)).$$
So $\F$ is local with respect to
$\f \ot \Y$ for any $\f \in \bar{\rS}, \Y \in \mD$ if and only if $\F$ inverts maps of $\bar{\rS}$ (equivalently of $\rS$).
So $\F$ is local with respect to
$\X \ot \g$ for any $\X \in \mC, \g \in \bar{\T}$ if and only if $\F$ lands in $\T^{-1}\mD$.

\end{proof}

\begin{corollary}\label{corg} 

Let $(\mV^\ot \to \Ass, \rS)$, $(\mW^\ot \to \Ass, \T)$	be small localization pairs.

\begin{enumerate}
\item A weakly bienriched $\infty$-category $\mM^\circledast \to \mV^\ot \times \mW^\ot$ exhibits $\mM$ as left $\rS$-enriched if for every $ \X, \Y \in \mM, \W_1,...,\W_\m \in \mW$ for $\m \geq 0$ the left multi-morphism object 
$$ \L\Mul\Mor_{\bar{\mM}}(\W_1,...,\W_\m,\X;\Y)$$ belongs to $\rS^{-1}\mP\Env(\mV).$

\item A weakly bienriched $\infty$-category $\mM^\circledast \to \mV^\ot \times \mW^\ot$ exhibits $\mM$ as right $\T$-enriched if for every $ \X, \Y \in \mM, \V_1,...,\V_\n \in \mV$ for $\n \geq 0$ the right multi-morphism object 
$$ \R\Mul\Mor_{\bar{\mM}}(\V_1,...,\V_\n,\X;\Y)$$ belongs to $\T^{-1}\mP\Env(\mW).$

\item A weakly bienriched $\infty$-category $\mM^\circledast \to \mV^\ot \times \mW^\ot$ exhibits $\mM$ as $\rS, \T$-bienriched if for every $ \X, \Y \in \mM, \V_1,...,\V_\n \in \mV$ for $\n \geq 0$ the morphism object 
$ \Mor_{\bar{\mM}}(\X,\Y)$ belongs to $\rS^{-1}\mP\Env(\mV) \ot \T^{-1}\mP\Env(\mW).$

\end{enumerate}

\end{corollary}

\begin{proof}
Statements (1) and (2) are clear.
We prove (3): for any $\X, \Y \in \mP\B\Env(\mM), \f \in \bar{\rS}, \g \in \bar{\T} $
the map $\mP\B\Env(\mM)(\f \ot \X \ot \g, \Y) $ identifies with the map $\mP\Env(\mV) \otimes \mP\Env(\mW)(\f \ot \g, \Mor_{\mP\B\Env(\mM)}(\X,\Y))$.
So the result follows from Lemma \ref{Lemut}.	

\end{proof}


\begin{notation}\label{locpat} Let $(\mV^\ot \to \Ass, \rS), (\mW^\ot \to \Ass, \T)$ be small localization pairs (Definition \ref{locpa}) and $\mM^\circledast \to \mV^\ot \times \mW^\ot$ a small weakly bienriched $\infty$-category.
	
\begin{enumerate}
\item Let $$\mP\B\Env(\mM)^\circledast_{\rS,\T} \subset \rS^{-1}\mP\Env(\mV)^\ot\times_{\mP\Env(\mV)^\ot}\mP\B\Env(\mM)^\circledast\times_{\mP\Env(\mW)^\ot}\T^{-1}\mP\Env(\mW)^\ot$$
be the full weakly bienriched subcategory spanned by the presheaves on
$\B\Env(\mM)$ that are local with respect to the set of morphisms of the form $\f \ot \X \ot \W_1 ... \ot \W_\m $ and $\V_1 \ot ... \ot \V_\n \ot \X \ot \g$ 
for $\f \in \rS, \g \in \T, \X \in \mM$, $\V_1,...,\V_\n \in \mV, \W_1,...,\W_\m \in \mW$ for $\n, \m \geq0$.

\item Let $$\mP\widetilde{\B\Env}(\mM)^\circledast_{\rS,\T}:= \mV^\ot \times_{\rS^{-1}\mP\Env(\mV)^\ot}\mP\B\Env(\mM)^\circledast_{\rS,\T}\times_{\T^{-1}\mP\Env(\mW)^\ot} \mW^\ot.$$
		
If $\rS=\emptyset$ or $\T=\emptyset$, we drop $\rS,\T$ from the notation, respectively.
		
\end{enumerate}	
	
\end{notation}

\begin{remark}
As a consequence of Remark \ref{switch} there is a canonical equivalence
$$(\mP\B\Env(\mM)_{\rS,\T}^\rev)^\circledast \simeq \mP\B\Env(\mM^\rev)_{\T,\rS}^\circledast.$$
	
\end{remark}

\begin{remark}\label{reipp}
Since $\mP\B\Env(\mM)$ is presentable, the embedding $\mP\B\Env(\mM)_{\rS,\T} \subset \mP\B\Env(\mM)$ admits a left adjoint.
The full subcategory $\mP\B\Env(\mM)_{\rS,\T} \subset \mP\B\Env(\mM)$ precisely consists of the presheaves on $\B\Env(\mM)$ that are local with respect to the collection $\mQ$ of morphisms of the form $\f \ot \X \ot \g $ for $\f \in \bar{\rS}, \g \in \bar{\T}$ and $\X \in \mM.$ The biaction of $\mP\Env(\mV), \mP\Env(\mW)$ on $\mP\B\Env(\mM)$ sends morphisms of $\bar{\rS}, \mQ, \bar{\T}$ to $\mQ$.
This implies that the embedding $\mP\B\Env(\mM)^\circledast_{\rS, \T} \subset \mP\B\Env(\mM)^\circledast$ admits an enriched left adjoint such that the unit lies over the units of the localizations $\rS^{-1}\mP\Env(\mV)^\ot \subset \mP\Env(\mV)^\ot, \T^{-1}\mP\Env(\mW)^\ot \subset \mP\Env(\mW)^\ot$. In particular,
$\mP\B\Env(\mM)^\circledast_{\rS, \T} \to \rS^{-1}\mP\Env(\mV)^\ot \times \T^{-1}\mP\Env(\mW)^\ot $ is a presentably bitensored $\infty$-category.
	
\end{remark}

\begin{remark}\label{embil}
	
A weakly bienriched $\infty$-category $\mM^\circledast \to \mV^\ot \times \mW^\ot$ is a $\rS, \T$-bienriched $\infty$-category if and only if
$\mM^\circledast \subset \mP\B\Env(\mM)^\circledast$ lands in
$\mP\B\Env(\mM)^\circledast_{\rS,\T}$.
\end{remark}
\begin{remark}\label{Functol}
Let $\mM^\circledast \to \mN^\circledast$ be an enriched functor between absolute small weakly bienriched $\infty$-categories lying over maps of small localization pairs
$$(\mV^\ot \to \Ass, \rS)\to (\mV'^\ot \to \Ass, \rS'), (\mW^\ot \to \Ass, \T)\to (\mW'^\ot \to \Ass, \T').$$
The induced left adjoint linear functor $\mP\B\Env(\mM)^\circledast \to \mP\B\Env(\mN)^\circledast$ preserves local equivalences for the corresponding localizations and so descends to a left adjoint linear functor $$\mP\B\Env(\mM)_{\rS,\T}^\circledast \to \mP\B\Env(\mN)_{\rS',\T'}^\circledast.$$
	
\end{remark}

\begin{lemma}\label{Leom}
Let $(\mV^\ot \to \Ass, \rS), (\mW^\ot \to \Ass, \T)$ be localization pairs, $\mM^\circledast \to \mV^\ot \times \mW^\ot$ an absolute small weakly bienriched $\infty$-category and $\Y \in \mP\B\Env(\mM).$ The following conditions are equivalent:
	
\begin{enumerate}
\item The object $\Y$ belongs to $\mP\B\Env(\mM)_{\rS,\T}.$
		
\item For every $\X \in \mM$ the object $\Mor_{\mP\B\Env(\mM)}(\X,\Y)$ belongs to $\rS^{-1}\mP\Env(\mV) \ot \T^{-1}\mP\Env(\mW).$
		
\item The object $\Y$ is local with respect to all morphisms $\f \ot \X \ot \g $ for $\f \in \bar{\rS}, \g \in \bar{\T}$ and $\X \in \mP\B\Env(\mM).$
		
\vspace{1mm}
\item For any $\X \in \mP\B\Env(\mM)$ the object $\Mor_{\mP\B\Env(\mM)}(\X,\Y)$ belongs to $\rS^{-1}\mP\Env(\mV) \ot \T^{-1}\mP\Env(\mW).$
\end{enumerate}
	
\end{lemma}

\begin{proof}Lemma \ref{Lemut} implies that (1) is equivalent to (2), and
that (3) is equivalent to (4).
Condition (1) is equivalent to (3) since $\bar{\rS}, \bar{\T}$ are closed under the tensor product and $\mP\B\Env(\mM)$ is generated by $\mM$ under small colimits and the $\mP\Env(\mV)$, $\mP\Env(\mW)$-biaction.
	
\end{proof}

\begin{proposition}\label{pseuso}
Let $(\mV^\ot \to \Ass, \rS), (\mW^\ot \to \Ass, \T)$ be small localization pairs, $\mM^\circledast \to \mV^\ot \times \mW^\ot$ a small $\rS, \T$-bienriched $\infty$-category and $\rho: \mN^\circledast \to \rS^{-1}\mP\Env(\mV)^\ot \times \T^{-1}\mP\Env(\mW)^\ot$ a bitensored $\infty$-category compatible with small colimits. 
	
\begin{enumerate}
\item The enriched embedding $\mM^\circledast \subset \mP\B\Env(\mM)^\circledast_{\rS,\T}$ induces a functor
$$\Enr\Fun_{\rS^{-1}\mP\Env(\mV),\T^{-1}\mP\Env(\mW)}(\mP\B\Env(\mM)_{\rS,\T},\mN) \to \Enr\Fun_{\mV,\mW}(\mM,\mN)$$
that admits a fully faithful left adjoint.
The left adjoint lands in the full subcategory $$\LinFun^\L_{\rS^{-1}\mP\Env(\mV),\T^{-1}\mP\Env(\mW)}(\mP\B\Env(\mM)_{\rS,\T},\mN).$$
		
\item The following induced functor is an equivalence: $$\LinFun^\L_{\rS^{-1}\mP\Env(\mV),\T^{-1}\mP\Env(\mW)}(\mP\B\Env(\mM)_{\rS,\T},\mN) \to \Enr\Fun_{\mV,\mW}(\mM,\mN).$$
		
\item There is a canonical $\rS^{-1}\mP\Env(\mV),\T^{-1}\mP\Env(\mW)$-linear equivalence
$$ \mP\B\Env(\mM)_{\rS,\T}^\circledast \simeq(\rS^{-1}\mP\Env(\mV) \ot_{\mP\Env(\mV)}\mP\B\Env(\mM) \ot_{\mP\Env(\mW)}\T^{-1}\mP\Env(\mW))^\circledast.$$
\end{enumerate}

\end{proposition}

\begin{proof}We first prove (2).
Let $\mN'^\circledast \to \mP\Env(\mV)^\ot \times \mP\Env(\mW)^\ot$ be the pullback of $\rho$ along the left adjoints of the embeddings $\rS^{-1}\mP\Env(\mV)^\ot \subset \mP\Env(\mV)^\ot, \T^{-1}\mP\Env(\mW)^\ot \subset \mP\Env(\mW)^\ot.$
Consider $$\LinFun^\L_{\rS^{-1}\mP\Env(\mV),\T^{-1}\mP\Env(\mW)}(\mP\B\Env(\mM)_{\rS,\T},\mN) \xrightarrow{\rho} \LinFun^\L_{\mP\Env(\mV),\mP\Env(\mW)}(\mP\B\Env(\mM),\mN')$$$$ \to \Enr\Fun_{\mV,\mW}(\mM,\mN).$$
	
The first functor in the composition is an equivalence by the description of local equivalences.
This proves (3). The second functor is an equivalence by Propositions \ref{bitte} and \ref{Day}.
	
(1): The left adjoint of the functor $$\Enr\Fun_{\rS^{-1}\mP\Env(\mV),\T^{-1}\mP\Env(\mW)}(\mP\B\Env(\mM)_{\rS,\T},\mN) \to \Enr\Fun_{\mV,\mW}(\mM,\mN)$$ factors as
$$ \Enr\Fun_{\mV,\mW}(\mM,\mN) \xrightarrow{\alpha} \Enr\Fun_{\mP\Env(\mV),\mP\Env(\mW)}(\mP\B\Env(\mM),\mN') \xrightarrow{\beta }$$$$ \Enr\Fun_{\rS^{-1}\mP\Env(\mV),\T^{-1}\mP\Env(\mW)}(\mP\B\Env(\mM)_{\rS,\T},\mN),$$
where $\beta$ takes restriction and $\alpha$ is the fully faithful left adjoint of $$\Enr\Fun_{\mP\Env(\mV),\mP\Env(\mW)}(\mP\B\Env(\mM),\mN') \to \Enr\Fun_{\mV,\mW}(\mM,\mN)$$ that lands in
$\LinFun^\L_{\mP\Env(\mV),\mP\Env(\mW)}(\mP\B\Env(\mM),\mN')$
(Proposition \ref{cool}).
Since $\rho$ is essentially surjective, $\beta$ takes $\LinFun^\L_{\mP\Env(\mV),\mP\Env(\mW)}(\mP\B\Env(\mM),\mN') $ to $\LinFun^\L_{\rS^{-1}\mP\Env(\mV),\T^{-1}\mP\Env(\mW)}(\mP\B\Env(\mM)_{\rS,\T},\mN)$.
	
	
\end{proof}

\subsection{Weak and pseudo-enrichment as enrichment}\label{looal}

In the following we show that weakly enriched and pseudo-
enriched $\infty$-categories can be viewed as
$\infty$-categories enriched in an appropriate way.

\begin{proposition}\label{eqq} Let $(\mV^\ot \to \Ass, \rS), (\mW^\ot \to \Ass, \T)$ be small localization pairs.

\begin{enumerate}

\item The functor 
$${_{\rS^{-1}\mP\Env(\mV)}\omega\B\Enr}_{\mW} \to {_\mV\omega\B\Enr}_{\mW},$$ which takes pullback along the embedding $\mV^\ot \subset \rS^{-1}\mP\Env(\mV)^\ot$, restricts to an equivalence
$$ {_{\rS^{-1}\mP\Env(\mV)}\L\Enr}_{\mW} \to {^\rS_\mV\L\Enr}_{\mW}.$$

\item The functor 
$${_\mV\omega\B\Enr}_{\T^{-1}\mP\Env(\mW)} \to {_\mV\omega\B\Enr}_{\mW},$$ which takes pullback along the embedding $\mW^\ot \subset \T^{-1}\mP\Env(\mW)^\ot$, restricts to an equivalence
$$ {_\mV\R\Enr}_{\T^{-1}\mP\Env(\mW)} \to {_\mV\R\Enr}^\T_{\mW}. $$
\item The functor 
$$\rho: {_{\rS^{-1}\mP\Env(\mV)}\omega\B\Enr}_{\T^{-1}\mP\Env(\mW)} \to {_\mV\omega\B\Enr}_{\mW},$$ which takes pullback along the embeddings $\mV^\ot \subset \rS^{-1}\mP\Env(\mV)^\ot, \mW^\ot \subset \T^{-1}\mP\Env(\mW)^\ot$, restricts to an equivalence
$$ {_{\rS^{-1}\mP\Env(\mV)}\B\Enr}_{\T^{-1}\mP\Env(\mW)}  \to {^\rS_\mV\B\Enr^\T_{\mW}}.$$

\end{enumerate}

\end{proposition}

\begin{proof}(1): We first show that the functor 
${_{\rS^{-1}\mP\Env(\mV)}\L\Enr}_{\mW} \to {^\rS_\mV\L\Enr}_{\mW}$ is conservative.
Let $\F: \mM^\circledast \to \mN^\circledast$ be a $\rS^{-1}\mP\Env(\mV), \mW$-enriched functor between left enriched $\infty$-categories whose pullback to $\mV^\ot$ is an equivalence. Then $\F$ induces an equivalence on underlying $\infty$-categories and it is enough to see that $\F$ induces an equivalence on multi-morphism objects.
Since $\rS^{-1}\mP\Env(\mV)$ is generated under small colimits by the tensor products of objects in the essential image of $\bj: \mV \subset \mP\Env(\mV)$, it is enough to see that for every $\X,\Y \in \mM$ and $\V_1,...,\V_\n \in \mV, \W_1,...,\W_\m \in \mW$ for $\n, \m \geq 0$ the following map is an equivalence:
$$\P\Env(\mV)(\bj(\V_1) \ot ... \ot \bj(\V_\n), \Mul\Mor_\mM(\X,\W_1,...,\W_\m,\Y)) \to $$$$ \P\Env(\mV)(\bj(\V_1) \ot ... \ot \bj(\V_\n), \Mul\Mor_\mN(\F(\X),\W_1,...,\W_\m,\F(\Y))).$$
The latter map identifies with following map which is an equivalence by assumption:
$$\Mul_\mM(\bj(\V_1), ..., \bj(\V_\n),\X,\W_1, ..., \W_\m,\Y) \to \Mul_\mN(\bj(\V_1), ..., \bj(\V_\n),\F(\X),\W_1, ..., \W_\m,\F(\Y)).$$

Let $\mM^\circledast \to \mV^\ot \times \mW^\ot$ be an absolute small left $\rS$-enriched $\infty$-category.
There is an embedding $\mM^\circledast \subset \mP\B\Env(\mM)_{\rS}^\circledast.$
Let $\bar{\mM}^\circledast \to \rS^{-1}\mP\Env(\mV)^\ot \times \mW^\ot$
be the full weakly bienriched subcategory of $\mP\B\Env(\mM)_{\rS}^\circledast \to \rS^{-1}\mP\Env(\mV)^\ot \times \mW^\ot$ spanned by $\mM.$
Since the latter is a presentably bitensored $\infty$-category,
$\bar{\mM}^\circledast \to \rS^{-1}\mP\Env(\mV)^\ot \times \mW^\ot$ is a left enriched $\infty$-category whose pullback to $\mV^\ot \times \mW^\ot$ is $\mM^\circledast \to \mV^\ot \times \mW^\ot$.
We prove that for any small left enriched $\infty$-category $\mN^\circledast \to \rS^{-1}\mP\Env(\mV)^\ot \times \mW^\ot$ the induced functor
$$\alpha_\mN: \Enr\Fun_{\rS^{-1}\mP\Env(\mV), \mW}(\bar{\mM},\mN) \to \Enr\Fun_{\mV, \mW}(\mM,\mN)$$ is an equivalence.
The functor $\alpha_\mN$ is the pullback of the functor
$$\Enr\Fun_{\rS^{-1}\mP\Env(\mV), \mP\Env(\mW)}(\bar{\mM},\mP\B\Env(\mN)_{\L\Enr}) \xrightarrow{\alpha_{\mP\B\Env(\mN)_{\L\Enr}}} \Enr\Fun_{\mV, \mP\Env(\mW)}(\mM,\mP\B\Env(\mN)_{\L\Enr}).$$ 
So it is enough to see that the latter functor is an equivalence. Since $\alpha_{\mP\B\Env(\mN)_{\L\Enr}}$ is conservative, it is an equivalence if it admits a fully faithful left adjoint. 
This follows from Proposition \ref{laan} and Proposition \ref{envv} (2)
because the left adjoint of the functor
$$\Alg_{\rS^{-1}\mP\Env(\mV)}(\rS^{-1}\mP\Env(\mV)) \xrightarrow{\L^*}\Alg_{\mP\Env(\mV)}(\rS^{-1}\mP\Env(\mV)) \xrightarrow{\kappa} \Alg_\mV(\rS^{-1}\mP\Env(\mV))$$ sends the map of $\infty$-operads
$\tau: \mV^\ot \to \mP\Env(\mV)^\ot \to \rS^{-1}\mP\Env(\mV)^\ot$ to the identity,
where $\L: \mP\Env(\mV)^\ot \leftrightarrows \rS^{-1}\mP\Env(\mV)^\ot: \iota$ is the monoidal localization.
Indeed, the left adjoint of $\kappa$ sends $\tau$ to the monoidal localization functor $ \L: \mP\Env(\mV)^\ot \to \rS^{-1}\mP\Env(\mV)^\ot$,
which is sent by the left adjoint of $\L^*$ given by $\iota^*$ 
to $\L \circ \iota \simeq \id.$

We complete the proof by showing that the functor 
${_{\rS^{-1}\mP\Env(\mV)}\omega\B\Enr}_{\mW} \to {_\mV\omega\B\Enr}_{\mW}$ restricts to a functor
$ {_{\rS^{-1}\mP\Env(\mV)}\L\Enr}_{\mW} \to {_\mV\L\Enr}^\rS_{\mW}.$
If this is shown, the latter functor is an equivalence since it is conservative and admits a fully faithful left adjoint as $\alpha_\mN$ is an equivalence.
First consider the case that $\rS$ is empty.
In this case there is nothing to show and therefore the induced functor $\lambda: {_{\mP\Env(\mV)}\L\Enr}_{\mW} \to {_\mV\omega\B\Enr}_{\mW}$ is an equivalence.
For any left enriched $\infty$-category $\beta: \mM^\circledast \to \rS^{-1}\mP\Env(\mV)^\ot \times \mW^\ot$ let $\mN^\circledast \to \mV^\ot \times \mW^\ot$ be the pullback of $\beta$ to $\mV^\ot$ and $\mM'^\circledast \to \mP\Env(\mV)^\ot \times \mW^\ot$ the pullback of $\beta$ along the left adjoint monoidal functor
$\L: \mP\Env(\mV)^\ot \to \rS^{-1}\mP\Env(\mV)^\ot. $ 
The weakly bienriched $\infty$-category $\mM'^\circledast \to \mP\Env(\mV)^\ot \times \mW^\ot$ exhibits $\mM$ as left enriched pulling back along a left adjoint.
Because the pullback of $\mM'^\circledast \to \mP\Env(\mV)^\ot \times \mW^\ot$ to $\mV^\ot \times \mW^\ot $ is $\mN^\circledast \to \mV^\ot \times \mW^\ot$ and $\lambda$ is an equivalence, there is a canonical $ \mP\Env(\mV), \mW$-enriched equivalence $\bar{\mN}^\circledast \simeq \mM'^\circledast$.
The weakly bienriched $\infty$-category $\mN^\circledast \to \mV^\ot \times \mW^\ot$ exhibits $\mN$ as left $\rS$-enriched in $\mV$ because for every $\f \in \rS, \X, \Y \in \mN, \W_1,..., \W_\m \in \mW$ the induced map $$\Mul_{\bar{\mN}}(\f, \X, \W_1,..., \W_\m, \Y) \simeq \Mul_{\mM'}(\f, \X, \W_1,..., \W_\m, \Y) \simeq \Mul_{\mM}(\L(\f), \X, \W_1,..., \W_\m, \Y) $$ is an equivalence. (2) is dual to (1).

(3): By (1) the induced functor $${_{\rS^{-1}\mP\Env(\mV)}\L\Enr}_{\T^{-1}\mP\Env(\mW)} \to {^\rS_\mV\L\Enr}_{ \T^{-1}\mP\Env(\mW)}$$
is an equivalence, which restricts to an equivalence
$${_{\rS^{-1}\mP\Env(\mV)}\B\Enr}_{\T^{-1}\mP\Env(\mW)} \to {^\rS_\mV\L\Enr}_{ \T^{-1}\mP\Env(\mW)} \cap {_\mV\R\Enr}_{\T^{-1}\mP\Env(\mW)}.$$

By (2) the induced functor $${_\mV\R\Enr}_{\T^{-1}\mP\Env(\mW)} \to {_\mV\R\P\Enr}^{\T}_{\mW}$$ is an equivalence, which restricts to an equivalence 
$${^\rS_\mV\L\Enr}_{\T^{-1}\mP\Env(\mW)} \cap {_\mV\R\Enr}_{\T^{-1}\mP\Env(\mW)} \to {^\rS_\mV\L\Enr}_{\mW} \cap {_\mV\R\P\Enr}^{\T}_{\mW} = {^\rS_\mV\B\P\Enr^\T_\mW}.$$

\end{proof}

Proposition \ref{eqq} gives the following corollaries:	
	
\begin{corollary}\label{coronn}

Let $\mV^\ot \to \Ass, \mW^\ot \to \Ass$ be small $\infty$-operads. 
\begin{enumerate}
	
\item The functor 
$$ {_{\mP\Env(\mV)}\L\Enr}_{\mW} \to {_\mV\omega\B\Enr}_{\mW}$$ taking pullback along the embedding $\mV^\ot \subset \mP\Env(\mV)^\ot$ is an equivalence, which restricts to an equivalence 
$$ {_{\mP\Env(\mV)}\B\Enr}_{\mW} \to {_\mV\R\Enr}_{\mW}.$$

\vspace{1mm}

\item The functor 
$$ {_\mV\R\Enr}_{\mP\Env(\mW)} \to {_\mV\omega\B\Enr}_{\mW}$$ taking pullback along the embedding $ \mW^\ot \subset \mP\Env(\mW)^\ot$ is an equivalence,
which restricts to an equivalence 
$$ {_\mV\B\Enr}_{\mP\Env(\mW)} \to {_\mV\L\Enr}_{\mW}.$$
\vspace{1mm}

\item The functor $$ {_{\mP\Env(\mV)}\B\Enr}_{\mP\Env(\mW)} \to {_\mV\omega\B\Enr}_{\mW}$$ taking pullback along the embeddings $\mV^\ot \subset \mP\Env(\mV)^\ot, \mW^\ot \subset \mP\Env(\mW)^\ot$ is an equivalence.

\vspace{1mm}

\item Let $\mV^\ot \to \Ass$ be a small monoidal $\infty$-category. The functor 
$${_{\mP(\mV)}\L\Enr}_{\mW} \to {_\mV\L\P\Enr}_{\mW}$$ taking pullback along the monoidal embedding $\mV^\ot \subset \mP(\mV)^\ot$ is an equivalence.

\vspace{1mm}

\item Let $\mW^\ot \to \Ass$ be a small monoidal $\infty$-category. The functor 
$${_\mV\R\Enr}_{\mP(\mW)} \to {_\mV\R\P\Enr}_{\mW}$$ taking pullback along the monoidal embedding $\mW^\ot \subset \mP(\mW)^\ot$ is an equivalence.

\vspace{1mm}

\item Let $\mV^\ot \to \Ass, \mW^\ot \to \Ass$ be small monoidal $\infty$-categories. The functor 
$${_{\mP(\mV)}\B\Enr}_{\mP(\mW)} \to {_\mV\B\P\Enr}_{\mW}$$ taking pullback along the monoidal embeddings $\mV^\ot \subset \mP(\mV)^\ot, \mW^\ot \subset \mP(\mW)^\ot$
is an equivalence.

\vspace{1mm}

\item The functor $$ {_{\Env(\mV)}\B\P\Enr}_{\Env(\mW)} \to {_\mV\omega\B\Enr}_{\mW}$$ taking pullback along the embeddings $\mV^\ot \subset \Env(\mV)^\ot, \mW^\ot \subset \Env(\mW)^\ot$ is an equivalence.

\vspace{1mm}

\item Let $\mV^\ot \to \Ass$ be a small monoidal $\infty$-category. The functor 
$${_{\mV}\B\P\Enr}_{*} \to {_\mV\L\P\Enr}_{\emptyset}$$ taking pullback along the embedding $\emptyset^\ot \subset \mW^\ot$ is an equivalence.

\vspace{1mm}

\item Let $\mW^\ot \to \Ass$ be a small monoidal $\infty$-category. The functor 
$${_{*}\B\P\Enr}_{\mW} \to {_\emptyset\R\P\Enr}_{\mW}$$ taking pullback along the embedding $\emptyset^\ot \subset \mW^\ot$ is an equivalence.
\end{enumerate}
\end{corollary}

Remark \ref{fori} and Corollary \ref{coronn} give the following corollary:

\begin{corollary}\label{uhol}

\begin{enumerate}
\item The functors $_\mS\L\Enr_\emptyset \to \Cat_\infty, {_\emptyset}\R\Enr_{\mS} \to \Cat_\infty, {_\mS\B\Enr}_{\mS} \to \Cat_\infty $ 
taking pullback along the embedding of $\infty$-operads $\emptyset^\ot \subset \mS^\times$ are equivalences.

\item The functors $_*\L\P\Enr_\emptyset \to \Cat_\infty, {_\emptyset}\R\P\Enr_{*} \to \Cat_\infty, {_*\B\P\Enr}_{*} \to \Cat_\infty $ 
taking pullback along the embedding of $\infty$-operads $\emptyset^\ot \subset\Ass$ are equivalences.	
\end{enumerate}
	
		
	
\end{corollary}

\begin{corollary}\label{cosqa}
Let $\kappa, \tau$ be small regular cardinals.

\begin{enumerate}
\item Let $\mV^\ot \to \Ass$ be a small monoidal $\infty$-category compatible with $\kappa$-small colimits and $ \mW^\ot \to \Ass $ a small $\infty$-operad.
The functor $${_{\Ind_\kappa(\mV)}\L\Enr}_{\mW} \to {_{\mV}^\kappa\L\Enr_\mW} $$ taking pullback along the monoidal embedding $\mV^\ot \subset \Ind_\kappa(\mV)^\ot$ 
is an equivalence.

\item Let $\mW^\ot \to \Ass$ be a small monoidal $\infty$-category compatible with $\tau$-small colimits and $ \mV^\ot \to \Ass $ a small $\infty$-operad.
The functor $$ {_\mV\R\Enr}_{\Ind_\tau(\mW)} \to {_\mV\R\Enr^\tau_{\mW}}$$ taking pullback along the monoidal embedding $\mW^\ot \subset \Ind_\tau(\mW)^\ot$ is an equivalence.

\item Let $\mV^\ot \to \Ass, \mW^\ot \to \Ass$ be small monoidal $\infty$-categories compatible with $\kappa$-small colimits, $\tau$-small colimits, respectively.
The functor $$ {_{\Ind_\kappa(\mV)}\B\Enr}_{\Ind_\tau(\mW)} \to {^\kappa_{\mV}\B\Enr^{\tau}_{\mW}}$$ taking pullback along the monoidal embeddings $\mV^\ot \subset \Ind_\kappa(\mV)^\ot, \mW^\ot \subset \Ind_\tau(\mW)^\ot$ is an equivalence.

\end{enumerate}

\end{corollary}

\begin{corollary}\label{sma}\label{presy}
	
\begin{enumerate}
\item Let $\mV^\ot \to \Ass$ be a presentably monoidal $\infty$-category,
$\mW^\ot \to \Ass$ a small $\infty$-operad and $\mM^\circledast \to \mV^\ot \times \mW^\ot, \mN^\circledast \to \mV^\ot \times \mW^\ot$ left enriched $\infty$-categories such that $\mM,\mN$ are small.
The $\infty$-category $\Enr\Fun_{\mV, \mW}(\mM,\mN)$ is small.

\item Let $\mV^\ot \to \Ass$ be a small $\infty$-operad, $\mW^\ot \to \Ass$
a presentably monoidal $\infty$-category and $\mM^\circledast \to \mV^\ot \times \mW^\ot, \mN^\circledast \to \mV^\ot \times \mW^\ot$ right enriched $\infty$-categories such that $\mM,\mN$ are small.
The $\infty$-category $\Enr\Fun_{\mV, \mW}(\mM,\mN)$ is small.

\item Let $\mM^\circledast \to \mV^\ot \times \mW^\ot, \mN^\circledast \to \mV^\ot \times \mW^\ot$ be $\infty$-categories bienriched in presentably monoidal $\infty$-categories such that $\mM,\mN$ are small. The $\infty$-category $\Enr\Fun_{\mV, \mW}(\mM,\mN)$ is small.

\end{enumerate}

\end{corollary}
For the next corollary we fix the following notation:

\begin{notation}\label{nity}
Let $\kappa$ be a small regular cardinal and $\mV^\ot \to \Ass$ a small monoidal $\infty$-category.

\begin{enumerate}
\item Let $$\mP_\kappa(\mV)^\ot \subset \mP(\mV)^\ot$$ be the full suboperad spanned by the full subcategory of $\mP(\mV)$ generated by $\mV$ under $\kappa$-small colimits.
For $\kappa$ the large strongly inaccessible cardinal corresponding to the small universe
(appyling the notation in a larger universe) we drop $\kappa$ from the notation.

\item Let $$\mP^\kappa(\mV)^\ot \subset \mP(\mV)^\ot$$ be the full suboperad spanned by the $\kappa$-compact objects of $\mP(\mV)$.
\end{enumerate}
\end{notation}

\begin{remark}
Since the Yoneda-embedding is monoidal and $ \mP(\mV)^\ot \to \Ass$ is compatible with small colimits, $\mP_\kappa(\mV)^\ot \to \Ass$ is a monoidal full subcategory of $\mP(\mV)^\ot\to\Ass$ compatible with $\kappa$-small colimits. The full subcategory $\mP(\mV)^\kappa$ of $\kappa$-compact objects of $\mP(\mV)$ is the smallest full subcategory of $\mP(\mV)$ containing $\mP_\kappa(\mV)$ and closed under retracts \cite[Proposition 5.3.4.17.]{lurie.HTT}. 
Thus $\mP(\mV)^\kappa$ is closed under the tensor product of $\mP(\mV)$
so that $(\mP(\mV)^\kappa)^\ot \to \Ass$ is monoidal $\infty$-category compatible with $\kappa$-small colimits. 
\end{remark}


\begin{remark}\label{comge}\label{comge2}
Since every object of $\mP(\mV)$ is a small $\kappa$-filtered colimit of objects of $\mP_\kappa(\mV)$, we find that $\mP(\mV)^\ot \to \Ass$ is a $\kappa$-compactly generated monoidal $\infty$-category.
The monoidal $\infty$-category $\Ind_\kappa(\mV)^\ot \to \Ass$ is a localization of $\mP(\mV)^\ot$ relative to $\Ass$. Consequently, a monoidal $\infty$-category is $\kappa$-compactly generated
if and only if it is a $\kappa$-accessible localization relative to $\Ass$ of presheaves on a small monoidal $\infty$-category.
Similarly, a bitensored $\infty$-category is $\kappa$-compactly generated
if and only if it is an enriched $\kappa$-accessible localization of presheaves on a small bitensored $\infty$-category.


\end{remark}


\begin{corollary}\label{cosqai}
Let $\kappa, \tau$ be small regular cardinals and $\mV^\ot \to \Ass, \mW^\ot \to \Ass$ be small $\infty$-operads.
\begin{enumerate}
\item The functor $${_{\mP\Env_\kappa(\mV)}^\kappa\L\Enr_\mW} \to {_{\mV}\omega\B\Enr}_{\mW} $$ taking pullback along the embedding $\mV^\ot \subset \mP\Env_\kappa(\mV)^\ot$ 
is an equivalence.
\item The functor $${_{\mV}\R\Enr^\tau_{\mP\Env_\tau(\mW)}} \to {_{\mV}\omega\B\Enr}_{\mW}  $$ taking pullback along the embedding $\mW^\ot \subset \mP\Env_\tau(\mW)^\ot$ 
is an equivalence.
\item The functor $${_{\mP\Env_\kappa(\mV)}^\kappa\B\Enr^{\tau}_{\mP\Env_\tau(\mW)}} \to {_{\mV}\omega\B\Enr}_{\mW}  $$ taking pullback along the embeddings $\mV^\ot \subset \mP\Env_\kappa(\mV)^\ot,\mW^\ot \subset \mP\Env_\tau(\mW)^\ot $ 
is an equivalence.
\end{enumerate}

\end{corollary}

\begin{proof}

The functor $ {_{\mP\Env(\mV)}\L\Enr}_{\mW} \to {_\mV\omega\B\Enr}_{\mW}$ taking pullback along the embedding $\mV^\ot \subset \mP\Env(\mV)^\ot$,
which is an equivalence by Corollary \ref{coronn} (1),
factors as $ {_{\mP\Env(\mV)}\L\Enr}_{\mW} \to {_{\mP\Env_\kappa(\mV)}^\kappa\L\Enr_\mW} \to {_\mV\omega\B\Enr}_{\mW}.$
By Corollary \ref{cosqa} (1) the first functor in the composition is an equivalence
since the embedding $$\mP\Env_\kappa(\mV) \subset \mP\Env(\mV)$$
induces an equivalence $\Ind_\kappa(\mP\Env_\kappa(\mV)) \simeq \mP\Env(\mV)$
as both satisfy the same universal property.
The proof of (2) and (3) are similar and follow from Corollary \ref{coronn}
and \ref{cosqa} (2) and (3), respectively.

\end{proof}

\begin{corollary}\label{ihot} Let $(\mV^\ot \to \Ass, \rS), (\mW^\ot \to \Ass, \T)$ be small localization pairs.
	
\begin{enumerate}
\item The functor $${_{\rS^{-1}\mP\Env(\mV)}\L\Q\widehat{\Enr}_\mW} \to {^\rS_\mV\widehat{\L\Enr}_\mW} $$ taking pullback along the embedding $\mV^\ot \subset \rS^{-1}\mP\Env(\mV)^\ot$ is an equivalence.
\item The functor $${_\mV\R\Q\widehat{\Enr}_{_{\T^{-1}\mP\Env(\mW)}}} \to {_\mV\widehat{\R\Enr}^\T_\mW} $$ taking pullback along the embedding $\mW^\ot \subset \T^{-1}\mP\Env(\mW)^\ot$ is an equivalence.
\item The functor $${_{\rS^{-1}\mP\Env(\mV)}\B\Q\widehat{\Enr}_{_{\T^{-1}\mP\Env(\mW)}}} \to {^\rS_\mV\widehat{\B\Enr}^\T_\mW}$$ taking pullback along the embeddings $\mV^\ot \subset \rS^{-1}\mP\Env(\mV)^\ot,\mW^\ot \subset \T^{-1}\mP\Env(\mW)^\ot $ is an equivalence.

\end{enumerate}
	
\end{corollary}

\begin{proof}The monoidal embedding $\mP\Env(\mV)^\ot \subset \widehat{\mP}\Env(\mV)^\ot$ induces a monoidal embedding $\rS^{-1}\mP\Env(\mV)^\ot \subset \rS^{-1}\widehat{\mP}\Env(\mV)^\ot$. We prove that the latter induces a monoidal equivalence $$\theta: \widehat{\Ind}_\sigma(\rS^{-1}\mP\Env(\mV))^\ot \simeq 
\rS^{-1}\widehat{\mP}\Env(\mV)^\ot,$$ where $\sigma$ is the strongly inaccessible cardinal corresponding to the small universe.
For every monoidal $\infty$-category $\mU^\ot \to \Ass$ compatible with large colimits the functor $$ \Fun^\ot(\rS^{-1}\widehat{\mP}\Env(\mV),\mU) \to  \Fun^\ot(\widehat{\Ind}_\sigma(\rS^{-1}\mP\Env(\mV)),\mU)$$ induced by $\theta$
factors as the following equivalences:
$$ \Fun^{\ot,\L}(\rS^{-1}\widehat{\mP}\Env(\mV),\mU) \simeq \Fun^{\ot,\L,\rS}(\widehat{\mP}\Env(\mV),\mU) \simeq \Fun^{\ot,\sigma, \rS}(\mP\Env(\mV),\mU) $$$$ \simeq \Fun^{\ot,\sigma}(\rS^{-1}\mP\Env(\mV),\mU) \simeq \Fun^{\ot,\L}(\widehat{\Ind}_\sigma(\rS^{-1}\mP\Env(\mV)),\mU),$$
where the superscript $\rS$ refers to monoidal functors inverting $\rS$.

Similarly, there is a canonical equivalence $\widehat{\Ind}_\sigma(\T^{-1}\mP\Env(\mW))^\ot \simeq 
\T^{-1}\widehat{\mP}\Env(\mW)^\ot.$ Thus the statement follows from Proposition \ref{eqq} and Corollary \ref{cosqa}.
	
\end{proof}	
	 

Next we compare different sorts of tensored $\infty$-categories
(Corollary \ref{emcolf}).

\begin{lemma}\label{lekko}
Let $\kappa$ be a small regular cardinal and $\mV^\ot \to \mV'^\ot, \mW^\ot \to \mW'^\ot$ embeddings of $\infty$-operads into a monoidal $\infty$-category compatible with $\kappa$-small colimits such that $\mV', \mW'$ are generated under $\kappa$-small colimits by tensor products of objects in the essential image.


\begin{enumerate}
\item A left $\kappa$-enriched $\infty$-category $ \mN^\circledast \to \mV'^\ot \times \mW^\ot$ 
is left tensored compatible with $\kappa$-small colimits if and only if the pullback $\mV^\ot \times_{\mV'^\ot} \mN^\circledast \to \mV^\ot \times \mW^\ot $ admits $\kappa$-small conical colimits and left tensors. 

\item A left $\kappa$-enriched $\infty$-category $ \mN^\circledast \to \mV'^\ot \times \mW^\ot$ admits right tensors if and only if the pullback $\mV^\ot \times_{\mV'^\ot} \mN^\circledast \to \mV^\ot \times \mW^\ot $ admits right tensors. 

\item A right $\kappa$-enriched $\infty$-category $ \mN^\circledast \to \mV^\ot \times \mW'^\ot$ is right tensored compatible with $\kappa$-small colimits if and only if the pullback $\mN^\circledast \times_{\mW'^\ot} \mW^\ot \to \mV^\ot \times \mW^\ot $ admits $\kappa$-small conical colimits and right tensors.

\item A right $\kappa$-enriched $\infty$-category $ \mN^\circledast \to \mV^\ot \times \mW'^\ot$ admits left tensors if and only if the pullback $\mN^\circledast \times_{\mW'^\ot} \mW^\ot \to \mV^\ot \times \mW^\ot $ admits left tensors.

\item A $\kappa, \kappa$-bienriched $\infty$-category $ \mN^\circledast \to \mV'^\ot \times \mW'^\ot$ is bitensored compatible with $\kappa$-small colimits if and only if the pullback $\mV^\ot \times_{\mV'^\ot} \mN^\circledast \times_{\mW'^\ot} \mW^\ot \to \mV^\ot \times \mW^\ot $ admits left and right tensors and $\kappa$-small conical colimits.
\end{enumerate}

\end{lemma}

\begin{proof}
(5) follows immediately from (1) - (4). The proof of (3) is similar to (1).
The proof of (4) is similar to (2).
We prove (1) and (2). We start with proving (2).
Let $\W \in \mW,\X \in \mN$ and $\X \ot \W$ the right tensor for $\phi: \mN^\circledast \to \mV'^\ot \times \mW^\ot$.
Then there is an object in $\Mul_{\mN}(\X,\W;\X \ot \W)$ and by left pseudo-enrichedness of $\phi$ it is enough to prove that for every $\V' \in \mV', \W_1,...,\W_\m\in \mW$ for $\m \geq0$ the following induced map is an equivalence:
$$ \rho: \Mul_\mN(\V',\X\ot \W, \W_1,...,\W_\m;\Y) \to \Mul_\mN(\V',\X, \W, \W_1,...,\W_\m;\Y).$$
By left $\kappa$-enrichedness of $\phi$ we can moreover assume that
$\V' \simeq \V_1 \ot ... \ot \V_\n$ for $\V_1,...,\V_\n \in \mV$ and $\n \geq0$ since $\mV'$ is generated by tensor products of $\mV$ under $\kappa$-small colimits.
Again by left pseudo-enrichedness of $\phi$ for this choice of $\V'$ the map $\rho$ identifies with the map
$$ \Mul_\mN(\V_1,...\V_\n,\X\ot \W, \W_1,...,\W_\m;\Y) \to \Mul_\mN(\V_1,...,\V_\n,\X, \W, \W_1,...,\W_\m;\Y),$$
which is an equivalence by universal property of the right tensor.

We continue with proving (1). The only-if direction is clear. We prove the if-direction. Assume that $ \mN^\circledast \to \mV'^\ot \times \mW^\ot$ exhibits $\mN$ as left $\kappa$-enriched and the pullback $\mV^\ot \times_{\mV'^\ot} \mN^\circledast \to \mV^\ot \times \mW^\ot $ admits left tensors and $\kappa$-small conical colimits. 
We like to see that $ \mN^\circledast \to \mV'^\ot \times \mW^\ot$ admits left tensors and $\kappa$-small conical colimits.
Since $\mV^\ot \times_{\mV'^\ot} \mN^\circledast \to \mV^\ot \times \mW^\ot $ admits left tensors, by iterately taking left tensors for every $\V_1,...,\V_\n \in \mV, \X \in \mN$ for $\n \geq 0$ the functor $ \Mul_\mN(\V_1,...,\V_\n,\X;-): \mN \to \mS$ is corepresentable.
This implies that for every $\V \in \mV', \X \in \mN$ the functor $ \Mul_\mN(\V,\X;-): \mN \to \mS$ is corepresentable, too.
This holds since by assumption $\mV'$ is generated under $\kappa$-small colimits by tensor products of objects of $\mV$, the $\infty$-category $\mN$ admits $\kappa$-small colimits, and for every $\X \in \mV'$
the functor $\mV'^\op \to \Fun(\mN,\mS), \V' \mapsto \Mul_\mN(\V',\X;-)$ preserves $\kappa$-small limits. The latter holds as for every $\Y \in \mN, \V \in \mV' $ the functor $\Mul_\mN(-,\X;\Y): \mV'^\op \to \mS$ preserves $\kappa$-small limits
by left $\kappa$-enrichedness.
Let $\V \ot \X$ be the object corepresenting the functor  $ \Mul_\mN(\V,\X;-): \mN \to \mS$. We obtain a functor $\ot: \mV' \times \mN \to \mN$ that sends $\V_1 \ot ... \ot \V_\n \in \mV'$ for $ \V_1,..., \V_\n \in \mV$ for $\n \geq 0$ to $(\V_1 \ot (-)) \circ ... \circ (\V_\n \ot(-)): \mN \to \mN.$
Moreover we obtain a natural transformation $\alpha: (-,\X) \to (-)\ot\X$
of functors $\mV' \to \mN^\circledast$
that sends any $\U' \in \mV'$ to the universal multi-morphism $\U', \X \to \U' \ot \X $ in $\mN^\circledast$.
We prove that for every $\U' \in \mV', \X \in \mN$
the universal morphism $\U', \X \to \U' \ot \X $ in $\mN^\circledast$ exhibits $\U' \ot \X$ as the left tensor of $\U'$ and $\X$.
For every $\X \in \mN$ the natural transformation $\alpha: (-,\X) \to (-)\ot \X$ induces for every $\V' \in \mV', \W_1,..., \W_\m \in \mW$ for $\m \geq 0$ a natural transformation 
$$\beta: \Mul_\mN(\V',(-) \ot \X, \W_1,..., \W_\m; \Y) \to \Mul_\mN(\V',-,\X, \W_1,..., \W_\m; \Y) $$ 
of functors $\mV'^\op \times \mV'^\op \to \mS.$
By left pseudo-enrichedness of $ \mN^\circledast \to \mV'^\ot \times \mW^\ot $ the universal morphism $\U', \X \to \U' \ot \X $ in $\mN^\circledast$ exhibits $\U' \ot \X$ as the left tensor of $\U'$ and $\X$ if
$\beta$ is an equivalence. 
By left $\kappa$-enrichedness of $ \mN^\circledast \to \mV'^\ot \times \mW^\ot $ we can assume that $\V' \simeq\V_1 \ot ...\ot \V_\bk$ for some $\V_1,...,\V_\bk\in\mV$ and $ \bk \geq 0 $.
Moreover by left $\kappa$-enrichedness of $ \mN^\circledast \to \mV'^\ot \times \mW^\ot $ the target of $\beta$ preserves $\kappa$-small limits.
We prove next that also the source of $\beta$
preserves $\kappa$-small limits.
The functor $(-)\ot \X:\mV' \to \mN$ preserves $\kappa$-small colimits because the functor $\Mul_\mN(-,\X;\Y): \mV'^\op \to \mS$ preserves $\kappa$-small limits by 
left $\kappa$-enrichedness of $ \mN^\circledast \to \mV'^\ot \times \mW^\ot $.
Therefore it is enough to see that the functor $\Mul_\mN(\V',-, \W_1,..., \W_\m; \Y): \mN^\op \to \mS$ preserves $\kappa$-small limits.
By left pseudo-enrichedness of $ \mN^\circledast \to \mV'^\ot \times \mW^\ot $
the canonical map $\V_1,...,\V_\bk \to \V' $ in $ \mV'^\ot$
identifies the functor $\Mul_\mN(\V',-, \W_1,..., \W_\m; \Y): \mN^\op \to \mS$ with the functor
$\Mul_\mN(\V_1,...,\V_\bk,-, \W_1,..., \W_\m; \Y) : \mN^\op \to \mS$, which preserves $\kappa$-small limits because the pullback $\mV^\ot \times_{\mV'^\ot} \mN^\circledast \to \mV^\ot \times \mW^\ot $ admits $\kappa$-small conical colimits.
Hence source and target of $\beta$ preserve $\kappa$-small limits.
This implies that it is enough to prove that for every $\U_1,...,\U_\n \in \mV$ for $\n \geq0$ the map
$\beta_{\U_1 \ot ... \ot \U_\n}$ is an equivalence.
The map $\beta_{\U_1 \ot ... \ot \U_\n}$ factors as equivalences
$$\Mul_\mN(\V', (\U_1 \ot ... \ot \U_\n) \ot \X, \W_1,..., \W_\m; \Y) \simeq\Mul_\mN(\V_1,...,\V_\bk, (\U_1 \ot ... \ot \U_\n) \ot \X, \W_1,..., \W_\m; \Y) $$$$\simeq \Mul_\mN(\V_1, ..., \V_\bk,\U_1, (\U_2 \ot ... \ot \U_{\n}) \ot \X, \W_1,..., \W_\m; \Y) $$$$\simeq
\Mul_\mN(\V_1, ..., \V_\bk,\U_1, \U_2, (\U_3 \ot ... \ot \U_{\n}) \ot \X, \W_1,..., \W_\m; \Y) \simeq ... $$
$$ \simeq \Mul_\mN(\V_1, ..., \V_\bk, \U_1,..., \U_\n ,\X, \W_1,..., \W_\m; \Y)\simeq \Mul_\mN(\V', \U_1 \ot ... \ot \U_\n ,\X, \W_1,..., \W_\m; \Y).$$

It remains to prove that for every $\V' \in \mV'$ the functor $\V' \ot (-): \mN \to \mN$ preserves $\kappa$-small colimits. By assumption this holds for every $\V' \in \mV$ and so also for every $\V' \simeq \V_1 \ot ...\ot \V_\n$ for $\V_1,..., \V_\n \in \mV$ for $\n \geq 0.$
The full subcategory of $\mV'$ spanned by all $\V'$ such that 
the functor $\V' \ot (-): \mN \to \mN$ preserves $\kappa$-small colimits is closed under $\kappa$-small colimits because the functor
$\mV' \to \Fun(\mN, \mN): \V' \mapsto \V' \ot (-): \mN \to \mN$ preserves $\kappa$-small colimits as the functor $\Mul_\mN(-,\X;\Y): \mV'^\op \to \mS$ preserves $\kappa$-small limits by left $\kappa$-enrichedness of $ \mN^\circledast \to \mV'^\ot \times \mW^\ot $. So the result follows from the fact that $\mV'$ is generated under $\kappa$-small colimits by tensor products of objects of $\mV.$

\end{proof}

For the next corollary we use Notation \ref{nity}:

\begin{corollary}\label{emcolf}
Let $\kappa$ be a small regular cardinal and $\mV^\ot \to \Ass, \mW^\ot \to \Ass$ be small $\infty$-operads.
\begin{enumerate}
\item The functor $${_{\mP\Env_\kappa(\mV)}^\kappa\L\Enr_\mW} \to {_{\mV}\omega\B\Enr}_{\mW} $$ taking pullback along the embedding $\mV^\ot \subset \mP\Env_\kappa(\mV)^\ot$ 
restricts to an equivalence between left tensored $\infty$-categories compatible with $\kappa$-small colimits and weakly bienriched $\infty$-categories that admit left tensors and 
$\kappa$-small conical colimits. 

\item The functor $${_{\mV}\R\Enr^\kappa_{\mP\Env_\kappa(\mW)}} \to {_{\mV}\omega\B\Enr}_{\mW} $$ taking pullback along the embedding $\mW^\ot \subset \mP\Env_\kappa(\mW)^\ot$
restricts to an equivalence between right tensored $\infty$-categories compatible with $\kappa$-small colimits and weakly bienriched $\infty$-categories that admit right tensors and 
$\kappa$-small conical colimits. 
 
\item The functor $${_{\mP\Env_\kappa(\mV)}^\kappa\B\Enr^{\kappa}_{\mP\Env_\tau(\mW)}} \to {_{\mV}\omega\B\Enr}_{\mW}  $$ taking pullback along the embeddings $\mV^\ot \subset \mP\Env_\kappa(\mV)^\ot,\mW^\ot \subset \mP\Env_\kappa(\mW)^\ot $ 
restricts to an equivalence between bitensored $\infty$-categories compatible with $\kappa$-small colimits and weakly bienriched $\infty$-categories that admit left and right tensors and 
$\kappa$-small conical colimits. 

\end{enumerate}	
	
\end{corollary}




We also need the following variant of Notation \ref{locpat}:

\begin{notation}\label{locpat2} Let $(\mV^\ot \to \Ass, \rS), (\mW^\ot \to \Ass, \T)$ be small localization pairs and $\mM^\circledast \to \mV^\ot \times \mW^\ot$ a small weakly bienriched $\infty$-category.

\begin{enumerate}
\item Let $$\mP\L\Env(\mM)^\circledast_{\rS} \subset \rS^{-1}\mP\Env(\mV)^\ot\times_{\mP\Env(\mV)^\ot}\mP\B\Env(\mM)^\circledast_{\rS}\times_{\mP\Env(\mW)^\ot}\mW^\ot$$
be the full weakly bienriched subcategory generated by $\mM$ under small colimits and the left action.

\item Let $$\mP\R\Env(\mM)^\circledast_{\T} \subset \mV^\ot\times_{\mP\Env(\mV)^\ot}\mP\B\Env(\mM)^\circledast_{\T}\times_{\mP\Env(\mW)^\ot}{\T^{-1}\mP\Env(\mW)^\ot}$$
be the full weakly bienriched subcategory generated by $\mM$ under small colimits and the right action.
If $\rS=\emptyset$ or $\T=\emptyset$, we drop $\rS,\T$ from the notation, respectively.

\end{enumerate}	

\end{notation}

\begin{remark}
Remark \ref{reipp} implies that $\mP\L\Env(\mM)^\circledast_{\rS} \to \rS^{-1}\mP\Env(\mV)^\ot\times\mW^\ot$
is a presentably left tensored $\infty$-category and $\mP\R\Env(\mM)^\circledast_{\T} \to \mV^\ot\times\mP\Env(\mW)^\ot$ is a presentably right tensored $\infty$-category.

\end{remark}


\begin{lemma}\label{lemozj}
Let $(\mV^\ot \to \Ass, \rS)$ be a small localization pair, $\mW^\ot \to \Ass$ a small $\infty$-operad and $\kappa$ a small regular cardinal.
Let $\rho: \mN^\circledast \to \mV^\ot \times \mW^\ot$ be a small left $\rS$-enriched $\infty$-category that admits left tensors and $\kappa$-small conical colimits. 
There is a left linear $\kappa$-small colimits preserving embedding $\mN^\circledast \subset \mO^\circledast$ into a presentably bitensored $\infty$-category that lies over the embeddings $\mV^\ot \subset \rS^{-1}\mP\Env(\mV)^\ot$ and $\mW^\ot \subset \mP\Env(\mW)^\ot.$
In particular, by Proposition \ref{eqq} for every small left enriched $\infty$-category $\mN^\circledast \to \rS^{-1}\mP\Env(\mV)^\ot \times \mW^\ot$ there is a left $\rS^{-1}\mP\Env(\mV)$-linear $\kappa$-small colimits preserving embedding $\mN^\circledast \subset \mO^\circledast$ into a presentably bitensored $\infty$-category that lies over the embedding $\mW^\ot \subset \mP\Env(\mW)^\ot.$


\end{lemma}

\begin{proof}

Let $\iota: \mN^\circledast \subset \mP\B\Env(\mN)^\circledast_\rS $ be the enriched embedding.
Let $\mQ$ be the set of morphisms $\V \ot \iota(\X) \to \iota(\V\ot\X), \colim(\iota\circ\rH) \to \iota(\colim(\rH))$ for any $\V \in \mV, \X \in \mN$ and any functor $\rH: \K \to \mN,$ where $\K$ is $\kappa$-small.
Let $\mO^\circledast \to \rS^{-1}\mP\Env(\mV)^\ot \times \mP\Env(\mW)^\ot$ be the 
localization of $\mP\B\Env(\mN)_\rS^\circledast \to \rS^{-1}\mP\Env(\mV)^\ot \times \mP\Env(\mW)^\ot$ with respect to the set of morphisms $\V_1 \ot ... \ot \V_\n \ot \f \ot \W_1 \ot ... \ot \W_\m$ for $\f \in \mQ, \V_1,...,\V_\n \in \mV, \W_1,...,\W_\m \in \mW$ for $\n, \m \geq 0$.
Then $\mO^\circledast \to \rS^{-1}\mP\Env(\mV)^\ot \times \mP\Env(\mW)^\ot$ is a presentably bitensored $\infty$-category and the enriched embedding $\mO^\circledast \subset \mP\B\Env(\mN)_\rS^\circledast$ admits an enriched left adjoint. The left $\mV$-enriched embedding $\iota: \mN^\circledast \subset \mP\B\Env(\mN)^\circledast_\rS$ lands in $\mO^\circledast$ and the resulting left $\mV$-enriched embedding $\mN^\circledast \subset \mO^\circledast$ is left $\mV$-linear and preserves $\kappa$-small colimits by construction.

\end{proof}

\begin{proposition}\label{pseuso2}
Let $(\mV^\ot \to \Ass, \rS)$ be a small localization pair $\mW^\ot \to \Ass$ a small $\infty$-operad, $\mM^\circledast \to \mV^\ot \times \mW^\ot$ a small left $\rS$-enriched $\infty$-category and $\rho: \mN^\circledast \to \rS^{-1}\mP\Env(\mV)^\ot \times \mW^\ot$ a small left tensored $\infty$-category compatible with small colimits. 

\begin{enumerate}
\item The enriched embedding $\mM^\circledast \subset \mP\L\Env(\mM)^\circledast_{\rS}$ induces a functor
$$\Enr\Fun_{\rS^{-1}\mP\Env(\mV), \mW}(\mP\L\Env(\mM)_{\rS},\mN) \to \Enr\Fun_{\mV,\mW}(\mM,\mN)$$
that admits a fully faithful left adjoint.
The left adjoint lands in the full subcategory $$\L\LinFun^\L_{\rS^{-1}\mP\Env(\mV),\mW}(\mP\L\Env(\mM)_{\rS},\mN).$$

\item The following induced functor is an equivalence: $$\L\LinFun^\L_{\rS^{-1}\mP\Env(\mV),\mW}(\mP\L\Env(\mM)_{\rS},\mN) \to \Enr\Fun_{\mV,\mW}(\mM,\mN).$$

\end{enumerate}

\end{proposition}

\begin{proof}
(1) follows from (2) since the functor of (2) is conservative. We prove (1).
Since $\rho$ is locally small, $\rho$ is a left enriched $\infty$-category by the adjoint functor theorem.
By Corollary \ref{ihot} the left quasi-enriched $\infty$-category $\rho$ is the pullback of a unique left enriched $\infty$-category $\rho': \bar{\mN}^\circledast \to \rS^{-1}\widehat{\mP}\Env(\mV)^\ot \times \mW^\ot$. We apply Lemma \ref{lemozj} in a larger universe to $\rho'$ and $\kappa$ the large strongly inacessible cardinal corresponding to the small universe. We obtain a left linear small colimits preserving embedding $\mN^\circledast \subset \bar{\mN}^\circledast \subset \mO^\circledast$ into a bitensored $\infty$-category compatible with large colimits that lies over the embeddings $\rS^{-1}\mP\Env(\mV)^\ot \subset \rS^{-1}\widehat{\mP}\Env(\mV)^\ot$ and $\mW^\ot \subset \widehat{\mP}\Env(\mW)^\ot.$
By Proposition \ref{laan} the functors
$$\Enr\Fun_{\rS^{-1}\mP\Env(\mV), \mP\Env(\mW)}(\mP\B\Env(\mM)_{\rS},\mO) \xrightarrow{\gamma} \Enr\Fun_{\rS^{-1}\mP\Env(\mV), \mW}(\mP\L\Env(\mM)_{\rS},\mO), $$$$ \Enr\Fun_{\rS^{-1}\mP\Env(\mV), \mW}(\mP\L\Env(\mM)_{\rS},\mO) \to \Enr\Fun_{\mV,\mW}(\mM,\mO)$$ admit fully faithful left adjoints $\alpha, \beta$, respectively.
Thus $\beta$ factors as $\alpha \circ \beta$ followed by $\gamma.$
By Proposition \ref{pseuso} the composition $\alpha \circ \beta$ lands in $\LinFun^\L_{\rS^{-1}\mP\Env(\mV), \mP\Env(\mW)}(\mP\B\Env(\mM)_{\rS},\mO)$ so that $\beta$ lands in $\L\LinFun^\L_{\rS^{-1}\mP\Env(\mV), \mW}(\mP\L\Env(\mM)_{\rS},\mO).$
Since $\mP\L\Env(\mM)_{\rS}$ is generated by $\mM$ under small colimits and left tensors and the embedding $\mN \subset \mO$ preserves small colimits and left tensors, the adjunction $$ \Enr\Fun_{\mV,\mW}(\mM,\mO) \rightleftarrows \Enr\Fun_{\rS^{-1}\mP\Env(\mV), \mW}(\mP\L\Env(\mM)_{\rS},\mO)$$
restricts to an adjunction $ \Enr\Fun_{\mV,\mW}(\mM,\mN) \rightleftarrows \Enr\Fun_{\rS^{-1}\mP\Env(\mV), \mW}(\mP\L\Env(\mM)_{\rS},\mN)$
whose left adjoint lands in the full subcategory $\L\LinFun^\L_{\rS^{-1}\mP\Env(\mV), \mW}(\mP\L\Env(\mM)_{\rS},\mN).$	

\end{proof}

\subsection{Enriched presheaves}

In the following we specialize the theory of the latter subsection to construct
an enriched $\infty$-category of enriched presheaves and more generally for every small regular cardinal $\kappa$ a $\kappa$-enriched $\infty$-category of $\kappa$-enriched presheaves.

\begin{notation}\label{enrprrr}Let $\kappa$ be a small regular cardinal.

\begin{enumerate}

\item Let $\mV^\ot \to \Ass$ be a small monoidal $\infty$-category compatible with $\kappa$-small colimits, $\mW^\ot \to \Ass$ a small $\infty$-operad and $\mM^\circledast \to \mV^\ot \times \mW^\ot$ a small left $\kappa$-enriched $\infty$-category.
Let
$$\mP\B\Env_\kappa(\mM)_{\L\Enr}^\circledast \subset \mV^\ot \times_{\Ind_\kappa(\mV)^\ot} \mP\B\Env(\mM)_{\L\Enr_\kappa}^\circledast \times_{\mP\Env(\mW)^\ot} \mP_\kappa\Env(\mW)^\ot $$
be the full bitensored subcategory generated by $\mM$ under $\kappa$-small colimits and the $\mV, \mP_\kappa\Env(\mW)$-biaction.

\item Let $\mV^\ot \to \Ass$ be a small $\infty$-operad, $\mW^\ot \to \Ass$ a small monoidal $\infty$-category compatible with $\kappa$-small colimits and $\mM^\circledast \to \mV^\ot \times \mW^\ot$ a small right $\kappa$-enriched $\infty$-category.
Let
$$ \mP\B\Env_\kappa(\mM)_{\R\Enr}^\circledast \subset \mP_\kappa\Env(\mV)^\ot \times_{\mP\Env(\mV)^\ot} \mP\B\Env(\mM)_{\R\Enr_\kappa}^\circledast \times_{\Ind_\kappa(\mW)^\ot} \mW^\ot$$
be the full bitensored subcategory generated by $\mM$ under $\kappa$-small colimits and the $\mP_\kappa\Env(\mV), \mW$-biaction.

\item Let $\mV^\ot \to \Ass, \mW^\ot \to \Ass$ be small monoidal $\infty$-categories compatible with $\kappa$-small colimits and $\mM^\circledast \to \mV^\ot \times \mW^\ot$ a small $\kappa, \kappa$-bienriched $\infty$-category.
Let
$$\mP\B\Env_\kappa(\mM)_{\B\Enr}^\circledast \subset {\mV^\ot \times_{\Ind_\kappa(\mV)^\ot} \mP\B\Env(\mM)_{\B\Enr_\kappa}^\circledast\times_{\Ind_\kappa(\mW)^\ot} \mW^\ot}$$
be the full bitensored subcategory generated by $\mM$ under $\kappa$-small colimits and the $\mV,\mW$-biaction.	

\item Let $\mV^\ot \to \Ass$ be a small monoidal $\infty$-category compatible with $\kappa$-small colimits, $\mW^\ot \to \Ass$ a small $\infty$-operad and $\mM^\circledast \to \mV^\ot \times \mW^\ot$ a small left $\kappa$-enriched $\infty$-category. Let
$$\mP\L\Env_\kappa(\mM)_{\Enr}^\circledast \subset \mV^\ot \times_{\Ind_\kappa(\mV)^\ot} \mP\L\Env(\mM)_{\Enr_\kappa}^\circledast \to \mV^\ot \times\mW^\ot $$
be the full left tensored subcategory generated by $\mM$ under $\kappa$-small colimits and the left $\mV$-action.



\item Let $\mV^\ot \to \Ass$ be a small $\infty$-operad, $\mW^\ot \to \Ass$ a small monoidal $\infty$-category compatible with $\kappa$-small colimits and $\mM^\circledast \to \mV^\ot \times \mW^\ot$ a small right $\kappa$-enriched $\infty$-category.
Let
$$ \mP\R\Env_\kappa(\mM)_{\Enr}^\circledast \subset \mP\R\Env(\mM)_{\Enr_\kappa}^\circledast \times_{\Ind_\kappa(\mW)^\ot} \mW^\ot\to \mV^\ot \times \mW^\ot $$
be the full right tensored subcategory generated by $\mM$ under $\kappa$-small colimits and the right $ \mW$-action.





\end{enumerate}

\end{notation}

\begin{remark}\label{sizes}
Since the embeddings $\mV \subset \Ind_\kappa(\mV)$ and $\mP_\kappa\Env(\mV) \subset \mP\Env(\mV) $ preserve $\kappa$-small colimits,
$$\mP\B\Env_\kappa(\mM)_{\L\Enr}^\circledast \to \mV^\ot \times \mP_\kappa\Env(\mW)^\ot, \mP\B\Env_\kappa(\mM)_{\R\Enr}^\circledast \to \mP_\kappa\Env(\mV)^\ot \times \mW^\ot, \mP\B\Env_\kappa(\mM)_{\B\Enr}^\circledast \to \mV^\ot \times \mW^\ot$$
are bitensored $\infty$-categories compatible with $\kappa$-small colimits,
which are small because 
$\mM$ is small and the collection of all $\kappa$-small $\infty$-categories is small.
Similarly, $$\mP\L\Env_\kappa(\mM)_{\Enr}^\circledast \to \mV^\ot \times \mW^\ot, \mP\R\Env_\kappa(\mM)_{\Enr}^\circledast \to \mV^\ot \times \mW^\ot$$ are small left tensored, right tensored $\infty$-categories compatible with $\kappa$-small colimits, respectively.

Moreover note that $$\mP\B\Env_\kappa(\mM^\rev)_{\L\Enr}^\circledast \simeq (\mP\B\Env_\kappa(\mM)_{\R\Enr}^\rev)^\circledast, $$$$ \mP\L\Env_\kappa(\mM^\rev)_{\Enr}^\circledast \simeq (\mP\R\Env_\kappa(\mM)_{\Enr}^\rev)^\circledast,$$$$ \mP\B\Env_\kappa(\mM^\rev)_{\B\Enr}^\circledast \simeq (\mP\B\Env_\kappa(\mM)_{\B\Enr}^\rev)^\circledast.$$

\end{remark}


\begin{notation}\label{enrprrr2}Let $\sigma$ be the large strongly inaccessible cardinal corresponding to the small universe.

\begin{enumerate}
\item Let $\mV^\ot \to \Ass$ be a monoidal $\infty$-category compatible with small colimits, $\mW^\ot \to \Ass$ an $\infty$-operad and $\mM^\circledast \to \mV^\ot \times \mW^\ot$ a left quasi-enriched $\infty$-category.
Let
$$\mP\B\Env(\mM)_{\L\Enr}^\circledast := \widehat{\mP}\B\Env_\sigma(\mM)_{\L\Enr}^\circledast \to \mV^\ot \times \mP\Env(\mW)^\ot.$$

\item Let $\mV^\ot \to \Ass$ be an $\infty$-operad, $\mW^\ot \to \Ass$ a monoidal $\infty$-category compatible with small colimits and $\mM^\circledast \to \mV^\ot \times \mW^\ot$ a right quasi-enriched $\infty$-category.
Let
$$ \mP\B\Env(\mM)_{\R\Enr}^\circledast := \widehat{\mP}\B\Env_\sigma(\mM)_{\R\Enr}^\circledast \to \mP\Env(\mV)^\ot \times \mW^\ot.$$

\item Let $\mV^\ot \to \Ass, \mW^\ot \to \Ass$ be monoidal $\infty$-categories compatible with small colimits and $\mM^\circledast \to \mV^\ot \times \mW^\ot$ a biquasi-enriched $\infty$-category.
Let
$$\mP_{\mV, \mW}(\mM)^\circledast:= \mP\B\Env(\mM)_{\B\Enr}^\circledast := \widehat{\mP}\B\Env_\sigma(\mM)_{\B\Enr}^\circledast \to \mV^\ot \times \mW^\ot.$$

\item Let $\mV^\ot \to \Ass$ be a monoidal $\infty$-category compatible with small colimits, $\mW^\ot \to \Ass$ an $\infty$-operad and $\mM^\circledast \to \mV^\ot \times \mW^\ot$ a left quasi-enriched $\infty$-category.
Let
$$\mP_\mV(\mM)^\circledast:=\mP\L\Env(\mM)_{\Enr}^\circledast := \widehat{\mP}\L\Env_\sigma(\mM)_{\Enr}^\circledast \to \mV^\ot \times \mW^\ot.$$

\item Let $\mV^\ot \to \Ass$ be an $\infty$-operad, $\mW^\ot \to \Ass$ a monoidal $\infty$-category compatible with small colimits and $\mM^\circledast \to \mV^\ot \times \mW^\ot$ a right quasi-enriched $\infty$-category.
Let
$$ \mP\R\Env(\mM)_{\Enr}^\circledast := \widehat{\mP}\R\Env_\sigma(\mM)_{\Enr}^\circledast \to \mV^\ot \times \mW^\ot.$$

\end{enumerate}

\end{notation}

\begin{proposition}\label{unipor}Let $\kappa$ be a small regular cardinal.

\begin{enumerate}
\item Let $\mV^\ot \to \Ass$ be a small monoidal $\infty$-category compatible with $\kappa$-small colimits, $\mW^\ot \to \Ass$ a small $\infty$-operad, $\mM^\circledast \to \mV^\ot \times\mW^\ot $ a small left $\kappa$-enriched $\infty$-category and $\mN^\circledast \to \mV^\ot \times \mP_\kappa\Env(\mW)^\ot$ a small bitensored $\infty$-category compatible with $\kappa$-small colimits.
The functor 
$$\gamma_\mN: \Enr\Fun_{\mV,\mP_\kappa\Env(\mW)}(\mP\B\Env_\kappa(\mM)_{\L\Enr},\mN) \to \Enr\Fun_{\mV,\mW}(\mM,\mN)$$
admits a fully faithful left adjoint that lands in 
$\LinFun_{\mV,\mP_\kappa\Env(\mW)}^\kappa(\mP\B\Env_\kappa(\mM)_{\L\Enr},\mN)$
and so the following induced functor is an equivalence:
$$\LinFun^\kappa_{\mV,\mP_\kappa\Env(\mW)}(\mP\B\Env_\kappa(\mM)_{\L\Enr},\mN) \to \Enr\Fun_{\mV,\mW}(\mM,\mN).$$


\item Let $\mV^\ot \to \Ass$ be a small monoidal $\infty$-category compatible with $\kappa$-small colimits, $\mW^\ot \to \Ass$ a small $\infty$-operad, $\mM^\circledast \to \mV^\ot \times\mW^\ot $ a small left $\kappa$-enriched $\infty$-category and $\mN^\circledast \to \mV^\ot \times \mW^\ot$ a small left tensored $\infty$-category compatible with $\kappa$-small colimits.
The functor 
$$\gamma_\mN: \L\Enr\Fun_{\mV,\mW}(\mP\L\Env_\kappa(\mM)_{\Enr},\mN) \to \Enr\Fun_{\mV,\mW}(\mM,\mN)$$
admits a fully faithful left adjoint that lands in 
$\L\LinFun_{\mV,\mW}^\kappa(\mP\L\Env_\kappa(\mM)_{\Enr},\mN)$
and so the following induced functor is an equivalence:
$$\L\LinFun^\kappa_{\mV,\mW}(\mP\L\Env_\kappa(\mM)_{\Enr},\mN) \to \Enr\Fun_{\mV,\mW}(\mM,\mN).$$


\item Let $\mV^\ot \to \Ass, \mW^\ot \to \Ass$ be small monoidal $\infty$-categories compatible with $\kappa$-small colimits, $\mM^\circledast \to \mV^\ot \times \mW^\ot $ a small $\kappa, \kappa$-bienriched $\infty$-category and $\mN^\circledast \to \mV^\ot \times \mW^\ot$ a bitensored $\infty$-category compatible with $\kappa$-small colimits.
The functor 
$$\gamma_\mN: \Enr\Fun_{\mV,\mW}(\mP\B\Env_\kappa(\mM)_{\B\Enr},\mN) \to \Enr\Fun_{\mV,\mW}(\mM,\mN)$$
admits a fully faithful left adjoint that lands in 
$\LinFun_{\mV,\mW}^\kappa(\mP\B\Env_\kappa(\mM)_{\B\Enr},\mN)$
and so the following induced functor is an equivalence:
$$\LinFun^\kappa_{\mV,\mW}(\mP\B\Env_\kappa(\mM)_{\B\Enr},\mN) \to \Enr\Fun_{\mV,\mW}(\mM,\mN).$$

\end{enumerate}

\end{proposition}

\begin{proof}
(1): Since $\mP\B\Env_\kappa(\mM)_{\L\Enr}$ is generated by $\mM$ under $\kappa$-small colimits and the $\mV, \mP_\kappa\Env(\mW)$-biaction, the functor in the second part of (1) is conservative. So it remains to prove the first part. By Proposition \ref{cool} there is a $\kappa$-small colimits preserving linear embedding $\mN^\circledast \to \mO^\circledast:=  \Ind_\kappa(\mN)^\circledast$ into a bitensored $\infty$-category compatible with small colimits that lies over the monoidal embeddings $\mV^\ot \to \Ind_\kappa(\mV)^\ot, \mP_\kappa\Env(\mW)^\ot \to \Ind_\kappa(\mP_\kappa\Env(\mW))^\ot \simeq \mP\Env(\mW)^\ot$.
By Proposition \ref{laan} and Remark \ref{sizes} the induced functors $$\rho: \Enr\Fun_{\Ind_\kappa(\mV),\mP\Env(\mW)}(\mP\B\Env(\mM)_{\L\Enr_\kappa},\mO)\to \Enr\Fun_{\mV, \mP_\kappa\Env(\mW)}(\mP\B\Env_\kappa(\mM)_{\L\Enr},\mO),$$ $$\gamma_\mO: \Enr\Fun_{\mV,\mP_\kappa\Env(\mW)}(\mP\B\Env_\kappa(\mM)_{\L\Enr},\mO) \to \Enr\Fun_{\mV,\mW}(\mM,\mO)$$
admit fully faithful left adjoints.
Thus the left adjoint of $\gamma_\mO$ factors as the left adjoint of $\gamma_\mO \circ \rho$, which lands in $\LinFun_{\Ind_\kappa(\mV),\mP\Env(\mW)}^\L(\mP\B\Env(\mM)_{\L\Enr_\kappa},\mO)$ by Proposition \ref{pseuso}, followed by $\rho$ and so lands in $\LinFun_{\mV,\mP_\kappa\Env(\mW)}^\kappa(\mP\B\Env_\kappa(\mM)_{\L\Enr},\mO)$
as $\mP\B\Env_\kappa(\mM)_{\L\Enr}$ is closed in $\mP\B\Env(\mM)_{\L\Enr_\kappa}$
under $\kappa$-small colimits and the $\mV,\mP_\kappa\Env(\mW)$-biaction.
Since $\mP\B\Env_\kappa(\mM)_{\L\Enr}$ is generated by $\mM$ under $\kappa$-small colimits and bitensors and the linear embedding $\mN^\circledast \to \mO^\circledast$ preserves $\kappa$-small colimits, the left adjoint of $\gamma_{\mO}$ restricts to a left adjoint of $\gamma_\mN$ that lands in $\LinFun^\kappa_{\mV,\mP_\kappa\Env(\mW)}(\mP\B\Env_\kappa(\mM)_{\L\Enr},\mN).$

(2): Since $\mP\L\Env_\kappa(\mM)_{\Enr}$ is generated by $\mM$ under $\kappa$-small colimits and the left $\mV$-action, the functor in the second part of (2) is conservative so that it remains to prove the first part of (2).

By Lemma \ref{lemozj} By Proposition \ref{cool} there is a $\kappa$-small colimits preserving left linear embedding $\mN^\circledast \to \mO^\circledast$ into a presentably bitensored $\infty$-category that lies over the monoidal embeddings $\mV^\ot \to \Ind_\kappa(\mV)^\ot, \mW^\ot \to \mP\Env(\mW)^\ot$.
By Proposition \ref{laan} and Remark \ref{sizes} the induced functors $$\rho: \Enr\Fun_{\Ind_\kappa(\mV),\mW}(\mP\L\Env(\mM)_{\L\Enr_\kappa},\mO)\to \Enr\Fun_{\mV, \mW}(\mP\L\Env_\kappa(\mM)_{\L\Enr},\mO),$$ $$\gamma_\mO: \Enr\Fun_{\mV,\mW}(\mP\L\Env_\kappa(\mM)_{\L\Enr},\mO) \to \Enr\Fun_{\mV,\mW}(\mM,\mO)$$
admit fully faithful left adjoints.
Thus the left adjoint of $\gamma_\mO$ factors as the left adjoint of $\gamma_\mO \circ \rho$, which lands in $\L\LinFun_{\Ind_\kappa(\mV),\mW}^\L(\mP\L\Env(\mM)_{\L\Enr_\kappa},\mO)$ by Proposition \ref{pseuso2}, followed by $\rho$ and so lands in $\L\LinFun_{\mV,\mW}^\kappa(\mP\L\Env_\kappa(\mM)_{\L\Enr},\mO)$
because $\mP\L\Env_\kappa(\mM)_{\L\Enr}$ is closed in $\mP\L\Env(\mM)_{\L\Enr_\kappa}$
under $\kappa$-small colimits and the left $\mV$-action.
Since $\mP\L\Env_\kappa(\mM)_{\L\Enr}$ is generated by $\mM$ under $\kappa$-small colimits and left tensors and the left linear embedding $\mN^\circledast \to \mO^\circledast$ preserves $\kappa$-small colimits, the left adjoint of $\gamma_{\mO}$ restricts to a left adjoint of $\gamma_\mN$ that lands in $\L\LinFun^\kappa_{\mV, \mW}(\mP\L\Env_\kappa(\mM)_{\L\Enr},\mN).$

The proof of (3) is similar to the one of (1).
\end{proof}

Proposition \ref{unipor} specializes to the following theorem:

\begin{theorem}\label{unipor3}
\begin{enumerate}
\item Let $\mV^\ot \to \Ass$ be a monoidal $\infty$-category compatible with small colimits, $\mW^\ot \to \Ass$ an $\infty$-operad, $\mM^\circledast \to \mV^\ot \times \mW^\ot $ a left quasi-enriched $\infty$-category and $\mN^\circledast \to \mV^\ot \times \mP\Env(\mW)^\ot$ a bitensored $\infty$-category compatible with small colimits.
The induced functor is an equivalence:
$$\LinFun^\L_{\mV,\mP\Env(\mW)}(\mP\B\Env(\mM)_{\L\Enr},\mN) \to \Enr\Fun_{\mV,\mW}(\mM,\mN).$$

\item Let $\mV^\ot \to \Ass$ be a small $\infty$-operad, $\mW^\ot \to \Ass$ a small monoidal $\infty$-category compatible with small colimits, $\mM^\circledast \to \mV^\ot \times \mW^\ot $ a right quasi-enriched $\infty$-category and $\mN^\circledast \to \mP\Env(\mV)^\ot \times \mW^\ot$ a bitensored $\infty$-category compatible with small colimits. The induced functor is an equivalence:$$\LinFun_{\mP\Env(\mV),\mW}^\L(\mP\B\Env(\mM)_{\R\Enr},\mN) \to \Enr\Fun_{\mV,\mW}(\mM,\mN).$$



\item Let $\mV^\ot \to \Ass, \mW^\ot \to \Ass$ be monoidal $\infty$-categories compatible with small colimits, $\mM^\circledast \to \mV^\ot \times \mW^\ot $ a biquasi-enriched $\infty$-category and $\mN^\circledast \to \mV^\ot \times \mW^\ot$ a bitensored $\infty$-category compatible with small colimits.
The induced functor is an equivalence:
$$\LinFun^\L_{\mV,\mW}(\mP_{\mV,\mW}(\mM),\mN) \to \Enr\Fun_{\mV,\mW}(\mM,\mN).$$

\item Let $\mV^\ot \to \Ass$ be a monoidal $\infty$-category compatible with small colimits, $\mW^\ot \to \Ass$ an $\infty$-operad, $\mM^\circledast \to \mV^\ot \times \mW^\ot $ a left quasi-enriched $\infty$-category and $\mN^\circledast \to \mV^\ot \times \mW^\ot$ a left tensored $\infty$-category compatible with small colimits.
The induced functor is an equivalence:
$$\L\LinFun^\L_{\mV,\mW}(\mP\L\Env(\mM)_{\Enr},\mN) \to \Enr\Fun_{\mV,\mW}(\mM,\mN).$$

\item Let $\mV^\ot \to \Ass$ be a small $\infty$-operad, $\mW^\ot \to \Ass$ a small monoidal $\infty$-category compatible with small colimits, $\mM^\circledast \to \mV^\ot \times \mW^\ot $ a right quasi-enriched $\infty$-category and $\mN^\circledast \to\mV^\ot \times \mW^\ot$ a right tensored $\infty$-category compatible with small colimits. The induced functor is an equivalence:$$\R\LinFun_{\mV,\mW}^\L(\mP\R\Env(\mM)_{\Enr},\mN) \to \Enr\Fun_{\mV,\mW}(\mM,\mN).$$

\end{enumerate}

\end{theorem}

\begin{remark}

\cite[6.2.2.]{HINICH2020107129}, \cite{hinich2021colimits}, \cite[Notation 4.35., Theorem 5.1.]{HEINE2023108941} construct a different model for the $\infty$-category of enriched presheaves in Hinich's and Gepner-Haugseng's model of enriched $\infty$-categories. These constructions satisfy an analogous universal property in the framework of these models and are compared by \cite{heine2024equivalence}. Consequently, comparison of these models with our model \cite{HEINE2023108941}, \cite{MR4185309} identifies these constructions with the one of Theorem \ref{unipor3} up to an auto-equivalence.	

\end{remark}

\subsection{Presentability of enriched presheaves}

Next we prove that the $\infty$-category of presheaves of Notation \ref{enrprrr2}
enriched in a 
presentably monoidal $\infty$-category is presentable (Corollary \ref{Corfolsa}).

\begin{notation}
Let $\kappa, \tau$ be small regular cardinals and $\mM^\circledast \to \mV^\ot \times \mW^\ot$ a weakly bienriched $\infty$-category.
\begin{enumerate}
\item 
Let $\mM_\kappa^\circledast \to (\mV^\kappa)^\ot \times \mW^\ot$ be the pullback of $\mM^\circledast \to \mV^\ot \times \mW^\ot$ along the embedding of $\infty$-operads
$(\mV^\kappa)^\ot \subset \mV^\ot.$


\item 
Let $\mM_{\kappa,\tau}^\circledast \to (\mV^\kappa)^\ot \times (\mW^\tau)^\ot$ be the pullback of $\mM^\circledast \to \mV^\ot \times \mW^\ot$ along the embeddings of $\infty$-operads
$(\mV^\kappa)^\ot \subset \mV^\ot, (\mW^\tau)^\ot \subset \mW^\ot.$
\end{enumerate}

\end{notation}

\begin{proposition}\label{Corfols}
Let $\kappa, \tau$ be small regular cardinals. 
\begin{enumerate}
\item Let $\mV^\ot \to \Ass$ be a $\kappa$-compactly generated monoidal $\infty$-category, $ \mW^\ot \to \Ass$ a small $\infty$-operad and $\mM^\circledast \to \mV^\ot \times \mW^\ot$ a small left enriched $\infty$-category.
By Corollary \ref{cosqa} (1) the $\mV^\kappa, \mW$-enriched embedding $\mM_\kappa^\circledast \subset \mP\B\Env(\mM_\kappa)^\circledast_{\L\Enr_\kappa}$ extends to a $\mV, \mW$-enriched embedding $\mM^\circledast \subset \mP\B\Env(\mM_\kappa)^\circledast_{\L\Enr_\kappa}$.
For every bitensored $\infty$-category compatible with small colimits $\mN^\circledast \to \mV^\ot \times \mP\Env(\mW)^\ot$ the following functor is an equivalence:
$$\LinFun^\L_{\mV, \mP\Env(\mW)}(\mP\B\Env(\mM_\kappa)_{\L\Enr_\kappa},\mN) \to \Enr\Fun_{\mV, \mW}(\mM,\mN).$$

\item Let $\mV^\ot \to \Ass$ be a $\kappa$-compactly generated monoidal $\infty$-category, $ \mW^\ot \to \Ass$ a small $\infty$-operad and $\mM^\circledast \to \mV^\ot \times \mW^\ot$ a small left enriched $\infty$-category.
By Corollary \ref{cosqa} (1) the $\mV^\kappa, \mW$-enriched embedding $\mM_\kappa^\circledast \subset \mP\L\Env(\mM_\kappa)^\circledast_{\Enr_\kappa}$ extends to a $\mV, \mW$-enriched embedding $\mM^\circledast \subset \mP\L\Env(\mM_\kappa)^\circledast_{\Enr_\kappa}$.
For every left tensored $\infty$-category compatible with small colimits $\mN^\circledast \to \mV^\ot \times \mW^\ot$ the following functor is an equivalence:
$$\LinFun^\L_{\mV, \mW}(\mP\L\Env(\mM_\kappa)_{\Enr_\kappa},\mN) \to \Enr\Fun_{\mV, \mW}(\mM,\mN).$$




\item Let $\mV^\ot \to \Ass, \mW^\ot \to \Ass$ be $\kappa$-compactly generated, $\tau$-compactly generated monoidal $\infty$-categories, respectively, and $\mM^\circledast \to \mV^\ot \times \mW^\ot$ a small bienriched $\infty$-category.
By Corollary \ref{cosqa} (3) the $\mV^\kappa, \mW^\tau$-enriched embedding $\mM_{\kappa, \tau}^\circledast \subset \mP\B\Env(\mM_{\kappa, \tau})^\circledast_{\B\Enr_{\kappa, \tau}}$ extends to a $\mV, \mW$-enriched embedding $\mM^\circledast \subset \mP\B\Env(\mM_{\kappa, \tau})^\circledast_{\B\Enr_{\kappa, \tau}}$.
For every bitensored $\infty$-category compatible with small colimits $\mN^\circledast \to \mV^\ot \times \mW^\ot$ the following functor is an equivalence:
$$\LinFun^\L_{\mV, \mW}(\mP\B\Env(\mM_{\kappa, \tau})^\circledast_{\B\Enr_{\kappa, \tau}},\mN) \to \Enr\Fun_{\mV, \mW}(\mM,\mN).$$


\end{enumerate}

\end{proposition}

\begin{proof}

(1): By Corollary \ref{cosqa} (1) the functor $\Enr\Fun_{\mV, \mW}(\mM,\mN) \to \Enr\Fun_{\mV^\kappa, \mW}(\mM_\kappa,\mN_\kappa)$ is an equivalence.	
So we need to see that $\LinFun^\L_{\mV, \mP\Env(\mW)}(\mP\B\Env(\mM_\kappa)_{\L\Enr_\kappa},\mN) \to \Enr\Fun_{\mV^\kappa, \mW}(\mM_\kappa,\mN_\kappa)$ is an equivalence.
This holds by Proposition \ref{pseuso}.

(2): We need to see that the functor $$\L\LinFun^\L_{\mV, \mW}(\mP\L\Env(\mM_\kappa)_{\Enr_\kappa},\mN) \to \Enr\Fun_{\mV, \mW}(\mM,\mN) \simeq \Enr\Fun_{\mV^\kappa, \mW}(\mM_\kappa,\mN_\kappa)$$ is an equivalence. 
This holds by Proposition \ref{pseuso2}.

(3): By Corollary \ref{cosqa} (3) the functor $\Enr\Fun_{\mV, \mW}(\mM,\mN) \to \Enr\Fun_{\mV^\kappa, \mW^\tau}(\mM_{\kappa, \tau},\mN_{\kappa, \tau})$ is an equivalence.	
So it is enough to see that $\LinFun^\L_{\mV, \mW}(\mP\B\Env(\mM_{\kappa, \tau})_{\B\Enr_{\kappa, \tau}},\mN) \to \Enr\Fun_{\mV^\kappa, \mW^\tau}(\mM_{\kappa, \tau},\mN_{\kappa, \tau})$ is an equivalence. This holds by Proposition \ref{pseuso}.


\end{proof}

Theorem \ref{unipor3} and Proposition \ref{Corfols} give the following description of the $\infty$-category of presheaves enriched in a presentably monoidal $\infty$-category.

\begin{corollary}\label{Corfolsy}
Let $\kappa, \tau$ be small regular cardinals. 
\begin{enumerate}
\item Let $\mV^\ot \to \Ass$ be a $\kappa$-compactly generated monoidal $\infty$-category, $ \mW^\ot \to \Ass$ a small $\infty$-operad and $\mM^\circledast \to \mV^\ot \times \mW^\ot$ a small left enriched $\infty$-category.
The enriched embedding $\mM_\kappa^\circledast \subset \mM^\circledast$
induces enriched equivalences $$\mP\B\Env(\mM_\kappa)^\circledast_{\L\Enr_\kappa} \simeq \mP\B\Env(\mM)^\circledast_{\L\Enr},$$
$$\mP\L\Env(\mM_\kappa)^\circledast_{\Enr_\kappa} \simeq \mP\L\Env(\mM)^\circledast_{\Enr}.$$


\item Let $\mV^\ot \to \Ass, \mW^\ot \to \Ass$ be $\kappa$-compactly generated, $\tau$-compactly generated monoidal $\infty$-categories, respectively, and $\mM^\circledast \to \mV^\ot \times \mW^\ot$ a small bienriched $\infty$-category. 
The enriched embedding $\mM_{\kappa, \tau}^\circledast \subset \mM^\circledast$
induces an enriched equivalence $$\mP\B\Env(\mM_{\kappa,\tau})^\circledast_{\B\Enr_{\kappa, \tau}} \simeq \mP\B\Env(\mM)^\circledast_{\B\Enr}.$$

\end{enumerate}

\end{corollary}

\begin{corollary}\label{Corfolsa}
Let $\kappa, \tau$ be small regular cardinals. 
\begin{enumerate}
\item Let $\mV^\ot \to \Ass$ be a $\kappa$-compactly generated monoidal $\infty$-category, $ \mW^\ot \to \Ass$ a small $\infty$-operad and $\mM^\circledast \to \mV^\ot \times \mW^\ot$ a small left enriched $\infty$-category.
Then $\mP\B\Env(\mM)^\circledast_{\L\Enr} \to \mV^\ot \times \mP\Env(\mW)^\ot$ is a $\kappa$-compactly generated bitensored $\infty$-category.

\item Let $\mV^\ot \to \Ass$ be a $\kappa$-compactly generated monoidal $\infty$-category, $ \mW^\ot \to \Ass$ a small $\infty$-operad and $\mM^\circledast \to \mV^\ot \times \mW^\ot$ a small left enriched $\infty$-category. Then $\mP\L\Env(\mM)^\circledast_{\Enr} \to \mV^\ot \times \mW^\ot$ is a $\kappa$-compactly generated left tensored $\infty$-category.


\item Let $\mV^\ot \to \Ass$ be a $\kappa$-compactly generated monoidal $\infty$-category, $ \mW^\ot \to \Ass $ a $\tau$-compactly generated monoidal $\infty$-category and $\mM^\circledast \to \mV^\ot \times \mW^\ot$ a small bienriched $\infty$-category.
Then $\mP\B\Env(\mM)^\circledast_{\B\Enr} \to \mV^\ot \times \mW^\ot$ is a $\lambda$-compactly generated bitensored $\infty$-category for every regular cardinal $\lambda \geq \kappa, \tau.$

\end{enumerate}

\end{corollary}


\begin{proposition}\label{eqist}
Let $\mM^\circledast \to \mV^\ot \times \mW^\ot$ be an absolute small weakly  bienriched $\infty$-category.

\begin{enumerate}
\item Let $\bar{\mM}^\circledast \to \mP\Env(\mV)^\ot \times \mW^\ot$ be the unique extension to a left enriched $\infty$-category.
Then $\bar{\mM}^\circledast \subset \mP\B\Env(\mM)^\circledast$
induces an equivalence of $\infty$-categories bitensored over $\mP\Env(\mV), \mP\Env(\mW)$:
$$\mP\B\Env(\bar{\mM})_\mathrm{\L\Enr}^\circledast \simeq \mP\B\Env(\mM)^\circledast.$$


\item Let $\bar{\mM}^\circledast \to \mP\Env(\mV)^\ot \times \mP\Env(\mW)^\ot$ the unique extension to a bienriched $\infty$-category.
Then $\bar{\mM}^\circledast \subset \mP\B\Env(\mM)^\circledast$
induces an equivalence of $\infty$-categories bitensored over $\mP\Env(\mV), \mP\Env(\mW)$:
$$\mP\B\Env(\bar{\mM})_\mathrm{\B\Enr}^\circledast \simeq \mP\B\Env(\mM)^\circledast.$$
\end{enumerate}	

\end{proposition}

\begin{proof}
(1): Let $\mN^\circledast \to \mP\Env(\mV)^\ot \times \mP\Env(\mW)^\ot$ be a locally small bitensored $\infty$-category compatible with small colimits.
The $\mP\Env(\mV), \mP\Env(\mW)$-enriched embedding	
$\bar{\mM}^\circledast \subset \mP\B\Env(\bar{\mM})_{\L\Enr}^\circledast$
induces an equivalence 
$$\LinFun^\L_{\mP\Env(\mV), \mP\Env(\mW)}(\mP\B\Env(\bar{\mM})_{\L\Enr},\mN) \to \Enr\Fun_{\mP\Env(\mV), \mW}(\bar{\mM},\mN) \simeq \Enr\Fun_{\mV, \mW}(\mM,\mN)$$
by Theorem \ref{unipor3} (1), where the last equivalence is by Corollary \ref{coronn} (1). 
We apply Corollary \ref{Corfolsa}.
(2): Let $\mN^\circledast \to \mP\Env(\mV)^\ot \times \mP\Env(\mW)^\ot$ be a locally small bitensored $\infty$-category compatible with small colimits.
The $\mP\Env(\mV), \mP\Env(\mW)$-enriched embedding $\bar{\mM}^\circledast \subset \mP\B\Env(\bar{\mM})_{\B\Enr}^\circledast$ induces an equivalence 
$$\LinFun^\L_{\mP\Env(\mV), \mP\Env(\mW)}(\mP\B\Env(\bar{\mM})_{\B\Enr},\mN) \to \Enr\Fun_{\mP\Env(\mV), \mP\Env(\mW)}(\bar{\mM},\mN) \simeq \Enr\Fun_{\mV, \mW}(\mM,\mN)$$
by Theorem \ref{unipor3} (3), where the last equivalence is by Corollary \ref{coronn} (3). We apply Corollary \ref{Corfolsa}.

\end{proof}

Corollary \ref{Corfolsy}, Corollary \ref{envvcor} and Corollary \ref{cosqai} give the following corollary:

\begin{corollary}\label{eqistu}
Let $\kappa,\tau$ be small regular cardinals and $\mM^\circledast \to \mV^\ot \times \mW^\ot$ an absolute small weakly bienriched $\infty$-category.
\begin{enumerate}
\item Let $\bar{\mM}^\circledast \to \mP\Env(\mV)^\ot \times \mW^\ot$ be the unique extension to a left enriched $\infty$-category.
Then $\bar{\mM}_\kappa^\circledast \subset \mP\B\Env(\mM)^\circledast$
induces an equivalence of $\infty$-categories bitensored over $\mP\Env(\mV), \mP\Env(\mW)$:
$$\mP\B\Env(\bar{\mM}_\kappa)_\mathrm{\L\Enr_\kappa}^\circledast \simeq \mP\B\Env(\mM)^\circledast.$$


\item Let $\bar{\mM}^\circledast \to \mP\Env(\mV)^\ot \times \mP\Env(\mW)^\ot$ be the unique extension to a bienriched $\infty$-category.
Then $\bar{\mM}_{\kappa,\tau}^\circledast \subset \mP\B\Env(\mM)^\circledast$
induces an equivalence of $\infty$-categories bitensored over $\mP\Env(\mV), \mP\Env(\mW)$:
$$\mP\B\Env(\bar{\mM}_{\kappa,\tau})_\mathrm{\B\Enr_{\kappa,\tau}}^\circledast \simeq \mP\B\Env(\mM)^\circledast.$$
\end{enumerate}	

\end{corollary}

\subsection{Atomicity of representable enriched presheaves}

\begin{definition}
Let $\mV^\ot \to \Ass, \mW^\ot \to \Ass$ be a presentably monoidal $\infty$-category and $\mM^\circledast \to \mV^\ot \times \mW^\ot$ a locally small bienriched $\infty$-category. An object $\X \in \mM$ is atomic if
the $\mV,\mW$-enriched functor $$\Mor_\mM(\X,-): \mM^\circledast \to (\mV\ot\mW)^\circledast$$ is linear and preserves small colimits.
	
\end{definition}

In this subsection we prove the following theorem:

\begin{theorem}\label{corok}\label{compp}
Let $\mM^\circledast \to \mV^\ot \times \mW^\ot$ be an absolute small weakly bienriched $\infty$-category and $(\mV^\ot \to \Ass, \rS), (\mW^\ot \to \Ass, \T)$ small localization pairs.
The essential image of the functor $$\mM \subset \mP\B\Env(\mM) \xrightarrow{\L}\mP\B\Env(\mM)_{\rS,\T}$$
consists of atomic objects. 
	
	
\end{theorem}


\begin{lemma}\label{trasy}
Let $\mV^\ot \to \Ass, \mW^\ot \to \Ass$ be small monoidal $\infty$-categories,  
$\mM^\circledast \to \mP(\mV\times\mW)^\ot$ a bitensored $\infty$-category compatible with small colimits and $\mN^\circledast \to \mP(\mV)^\ot $ the underlying left tensored $\infty$-category. 
Let $\X, \Y \in \mM, \V \in \mV, \W \in \mW$ and $\F \in \mP(\mV)$. The canonical morphism $$\alpha: \F \ot \L\Mor_\mN(\X\ot \W,\Y) \to \L\Mor_\mN(\X\ot \W,\F \ot \Y) $$ in $\mP(\mV) $ evaluated at $\V$ is equivalent to the canonical morphism $$\beta: \F \ot \L\Mor_\mM(\X,\Y) \to \L\Mor_\mM(\X,\F \ot \Y) $$ in $\mP(\mV \times \mW) $ evaluated at $\V, \W.$

\end{lemma}

\begin{proof}
Let $\V \in \mV.$ Evaluating $\alpha$ at $\V $ gives the canonical map
$$\psi: \colim_{\V',\V'' \in \mV, \V \to \V' \ot \V''} \F(\V') \times \mM(\V'' \ot \X\ot \W,\Y) \to \mM(\V \ot \X\ot \W,\F \ot \Y) $$ in $\mP(\mV)$.
Evaluating $\beta$ at $(\V,\W) $ gives the canonical map
$$\colim_{\V',\V'' \in \mV, \W',\W'' \in \mW, \V \to \V' \ot \V'', \W \to \W'' \ot \W'} \F(\V') \times \mW(\W',\tu_\mW) \times \mM(\V'' \ot \X \ot \W'',\Y) \to \mM(\V \ot \X \ot \W,\F \ot \Y) $$ in $\mP(\mV \times \mW)$.
The latter map factors as the following map in $\mP(\mV \times \mW)$
and so identifies with $\psi:$
$$\colim_{\V',\V'' \in \mV, \V \to \V' \ot \V''} (\F(\V') \times \underbrace{\colim_{\W',\W'' \in \mW, \W \to \W'' \ot \W'} \mW(\W',\tu_\mW) \times \mM(\V'' \ot \X \ot \W'',\Y)}_{(\tu_\mW \ot \mM(\V'' \ot \X \ot (-),\Y))(\W)}) \to $$$$ \mM(\V \ot \X \ot \W,\F \ot \Y).$$


\end{proof}

\begin{proposition}\label{awas}
Let $\mM^\circledast \to \mV^\ot $ be a weakly left enriched $\infty$-category, $\V\in \mV, \X,\Y \in \mM$ and $\alpha \in \Mul_\mM(\V,\X,\Y)$ a multi-morphism that exhibits $\V$ as the left morphism object $\L\Mor_\mM(\X,\Y)\in \mV $ of $\X,\Y.$
For every $\V_1,...,\V_\n \in \mV$ for $\n \geq 0$
the induced morphism $$\V_1 \ot...\ot \V_\n \ot \V \ot \X \to \V_1 \ot...\ot \V_\n \ot \Y $$ in $\L\Env(\mM)$ exhibits
$\V_1 \ot...\ot \V_\n \ot \V$ as the left morphism object $\L\Mor_{\L\Env(\mM)}(\X,\V_1 \ot...\ot \V_\n \ot \Y) \in \Env(\mV)$ of $\X,\V_1 \ot...\ot \V_\n \ot \Y$.

In particular, the morphism object of $\X,\V_1 \ot...\ot \V_\n \ot \Y$
in $\L\Env(\mM)$ exists and the following canonical morphism in $\Env(\mV)$ is an equivalence:
$$\V_1 \ot...\ot \V_\n \ot\L\Mor_{\L\Env(\mM)}(\X, \Y)\to \L\Mor_{\L\Env(\mM)}(\X,\V_1 \ot...\ot \V_\n \ot \Y).$$

In particular, the following canonical morphism in $\mP\Env(\mV)$ is an equivalence:
$$\V_1 \ot...\ot \V_\n \ot\L\Mor_{\mP\L\Env(\mM)}(\X, \Y)\to \L\Mor_{\mP\L\Env(\mM)}(\X,\V_1 \ot...\ot \V_\n \ot \Y).$$

\end{proposition}

\begin{proof}
We like to prove that for any $\mV_1',...\mV_\ell' \in \mV$ for $\ell \geq 0$ the following induced map is an equivalence:
\begin{equation}\label{euk}
\Env(\mV)(\mV_1'\ot ...\ot \mV_\ell', \V_1 \ot...\ot \V_\n \ot\L\Mor_{\mM}(\X, \Y)) \to \L\Env(\mM)(\mV_1'\ot ...\ot \mV_\ell'\ot \X, \V_1 \ot...\ot \V_\n \ot \Y).\end{equation}

For any $[\n],[\m]\in \Ass$ let $$ \Ass_{}([\m],[\n]) \subset \Ass_{\mathrm{\min},\max}([\m],[\n])\subset \Ass_{\mathrm{\min}}([\m],[\n])$$ be the subsets of order preserving maps preserving the minimum and maximum, preserving the minimum, respectively. 
The map $\Ass([\ell],[\n+1])= \Delta([\n+1], [\ell])\to \Delta([\n], [\ell])= \Ass([\ell],[\n])$
restricting along the order preserving embedding $[\n]\cong\{0,...,\n\}\subset[\n+1]$
induces a bijection $\xi: \Ass_{\min,\max}([\ell],[\n+1])\to \Ass_{\min}([\ell],[\n]).$
The map (\ref{euk}) canonically covers $\xi$. Thus we have to prove that
for every order preserving map $\varphi:[\n+1]\to[\ell]$ preserving the minimum
the map induced by (\ref{euk}) on the fiber over $\varphi$ is an equivalence.
This map identifies with the map
\begin{equation*}\label{eukl}
\prod_{\bi=1}^\n\Env(\mV)(\V'_{\varphi(\bi-1)+1}\ot ...\ot \V'_{\varphi(\bi)},\V_\bi) \times \Env(\mV)(\V'_{\varphi(\n)+1}\ot ...\ot \V'_{\ell}, \L\Mor_{\L\Env(\mM)}(\X, \Y)) \to$$$$ \prod_{\bi=1}^\n\Env(\mV)(\V'_{\varphi(\bi-1)+1}\ot ...\ot \V'_{\varphi(\bi)},\V_\bi) \times \L\Env(\mM)(\V'_{\varphi(\n)+1} \ot...\ot \V'_\ell \ot \X, \Y).
\end{equation*}
induced by the identity of $\prod_{\bi=1}^\n\Env(\mV)(\V'_{\varphi(\bi-1)+1}\ot ...\ot \V'_{\varphi(\bi)},\V_\bi)$ and the canonical map
$$ \Env(\mV)(\V'_{\varphi(\n)+1}\ot ...\ot \V'_{\ell}, \L\Mor_{\L\Env(\mM)}(\X, \Y)) \to \L\Env(\mM)(\V'_{\varphi(\n)+1} \ot...\ot \V'_\ell \ot \X, \Y).$$
The latter map identifies with the map
$$ \Mul_\mV(\V'_{\varphi(\n)+1}, ..., \V'_{\ell}, \L\Mor_{\L\Env(\mM)}(\X, \Y)) \to \Mul_\mM(\V'_{\varphi(\n)+1}, ..., \V'_\ell, \X, \Y),$$
which is an equivalence.
Consequently, the morphism object of $\X,\V_1 \ot...\ot \V_\n \ot \Y$
in $\L\Env(\mM)$ exists and the following canonical morphism in $\Env(\mV)$ is an equivalence:
$$\V_1 \ot...\ot \V_\n \ot\L\Mor_{\mM}(\X, \Y)\to \L\Mor_{\L\Env(\mM)}(\X,\V_1 \ot...\ot \V_\n \ot \Y).$$
By Proposition \ref{morpre} the morphism object of $\X,\Y$ in $\L\Env(\mM)$ exists if the morphism object of $\X,\Y$ in $\mM$ exists, and the canonical morphism $\L\Mor_\mM(\X,\Y)\to \L\Mor_{\L\Env(\mM)}(\X,\Y)$ in $\Env(\mV)$ is an equivalence.
This proves the second part. 
The third part follows from the fact that the linear embedding
$\L\Env(\mM)^\circledast \subset \mP\L\Env(\mM)^\circledast$
lying over the monoidal embedding $\Env(\mV)^\ot \subset \mP\Env(\mV)^\ot$
preserves left morphism objects by Proposition \ref{morpre}.

\end{proof}



\begin{proof}[Proof of Theorem \ref{corok}]
We first reduce to the case that $\rS,\T$ are empty.
Consider the $\mP\Env(\mV), \mP\Env(\mW) $-enriched adjunction
$$ (-)\ot\X \ot (-): (\mP\Env(\mV) \ot \mP\Env(\mW))^\circledast \rightleftarrows \mP\B\Env(\mM)^\circledast: \L\Mor_{\mP\B\Env(\mM)}(\X,-).$$
The left adjoint preserves local equivalences
for the localizations $$\mP\B\Env(\mM)\to \mP\B\Env(\mM)_{\rS,\T},  \mP\Env(\mV) \ot \mP\Env(\mW) \to \rS^{-1}\mP\Env(\mV) \ot \T^{-1}\mP\Env(\mW).$$
If the statement holds for $\rS,\T$ empty, then the right adjoint is $\mP\Env(\mV), \mP\Env(\mW)$-linear and preserves small colimits.
This implies that also the right adjoint preserves local equivalences.
Consequently, in this case the latter $\mP\Env(\mV), \mP\Env(\mW) $-enriched adjunction induces a $\rS^{-1}\mP\Env(\mV), \T^{-1}\mP\Env(\mW) $-enriched adjunction
$$ (-)\ot\L(\X) \ot (-): (\rS^{-1}\mP\Env(\mV) \ot \T^{-1}\mP\Env(\mW))^\circledast \rightleftarrows \mP\B\Env(\mM)_{\rS,\T}^\circledast: \L\Mor_{\mP\B\Env(\mM)_{\rS,\T}}(\L(\X),-),$$
in which the right adjoint preserves small colimits and is $\rS^{-1}\mP\Env(\mV), \T^{-1}\mP\Env(\mW) $-linear.
So it suffices to prove the case that $\rS,\T$ are empty.

The $\mP\Env(\mV),\mP\Env(\mW)$-enriched functor $$\L\Mor_{\mP\B\Env(\mM)}(\X,-): 
\mP\B\Env(\mM)^\circledast \to \mP(\Env(\mV) \times \Env(\mW))^\circledast$$ preserves small colimits 
since for any $\V_1,..., \V_\n \in \mV, \W_1,...,\W_\m \in \mW$
the composition $$\mP\B\Env(\mM) \xrightarrow{\L\Mor_{\mP\B\Env(\mM)}(\X,-)} \mP(\Env(\mV) \times \Env(\mW)) \xrightarrow{\ev_{\V_1 \ot... \ot \V_\n, \W_1 \ot ... \ot \W_\m}} \mS$$
identifies with the map $\mP\B\Env(\mM)(\V \ot \X \ot \W,-):\mP\B\Env(\mM) \to \mS$ evaluating at $\V \ot \X \ot \W \in \B\Env(\mM).$
It remains to see that the $\mP\Env(\mV),\mP\Env(\mW)$-enriched functor $$\L\Mor_{\mP\B\Env(\mM)}(\X,-): \mP\B\Env(\mM)^\circledast \to \mP(\Env(\mV) \times \Env(\mW))^\circledast$$ is $\mP\Env(\mV),\mP\Env(\mW)$-linear.
We first reduce to the case that $\mW^\ot=\emptyset^\ot.$
Let $\mN^\circledast:= \R\Env(\mM)^\circledast \to \mV^\ot \times \Env(\mW)^\ot.$ By Remark \ref{envdecom} there is a canonical equivalence $\mP\B\Env(\mM)^\circledast \simeq \mP\L\Env(\mN)^\circledast $ of $\infty$-categories bitensored over $\mP\Env(\mV),\mP\Env(\mW)$.
In particular, the underlying left $\mP\Env(\mV)$-tensored $\infty$-category of
$\mP\B\Env(\mM)^\circledast\to \mP\Env(\mV)^\ot \times \mP\Env(\mW)^\ot$ is
$\mP\L\Env(\mN)^\circledast \to \mP\Env(\mV)^\ot.$
By Lemma \ref{trasy} the $\mP\Env(\mV),\mP\Env(\mW)$-enriched functor $$\L\Mor_{\mP\B\Env(\mM)}(\X,-): \mP\B\Env(\mM)^\circledast \to \mP(\Env(\mV) \times \Env(\mW))^\circledast$$ is $\mP\Env(\mV),\mP\Env(\mW)$-linear if for every $\W \in \mP\Env(\mW)$ the left $\mP\Env(\mV)$-enriched functor $$\L\Mor_{\mP\L\Env(\mN)}(\X \ot \W,-): \mP\L\Env(\mN)^\circledast \to \mP\Env(\mV)^\circledast$$ is left $\mP\Env(\mV)$-linear.
So have reduced to the case that $\mW^\ot=\emptyset^\ot$ and need to see that
for every $\X \in \mM$ the left $\mP\Env(\mV)$-enriched functor 
$\L\Mor_{\mP\L\Env(\mM)}(\X ,-): \mP\L\Env(\mM)^\circledast \to \mP\Env(\mV)^\circledast$ is left $\mP\Env(\mV)$-linear.
We reduce next to the case that $\mM^\circledast \to \mV^\ot$ is a left enriched $\infty$-category.
Let $\bar{\mM}^\circledast \subset \mP\L\Env(\mM)^\circledast \to \mP\Env(\mV)^\ot$ be the full left $\mP\Env(\mV)$-enriched $\infty$-category spanned by $\mM.$
Corollary \ref{eqistu} guarantees that the embedding $\bar{\mM}_\omega^\circledast \subset \mP\L\Env(\mM)^\circledast$
lying over the embedding of $\infty$-operads $(\mP\Env(\mV)^\omega)^\ot \subset \mP\Env(\mV)^\ot$ uniquely extends to a left adjoint $\mP\Env(\mV)$-linear equivalence
$\mP\L\Env(\bar{\mM}_\omega)_{\L\Enr_\omega}^\circledast \simeq \mP\L\Env(\mM)^\circledast$.
Thus the left $\mP\Env(\mV)$-enriched functor 
$\L\Mor_{\mP\L\Env(\mM)}(\X ,-): \mP\L\Env(\mM)^\circledast \to \mP\Env(\mV)^\circledast$ identifies with the left $\mP\Env(\mV)$-enriched functor 
$$\L\Mor_{\mP\L\Env(\bar{\mM})_{\L\Enr_\omega}}(\X ,-): \mP\L\Env(\bar{\mM}_\omega)_{\L\Enr}^\circledast \to \mP\Env(\mV)^\circledast \simeq \Ind_\omega(\mP\Env(\mV)^\omega)^\circledast.$$
So replacing $\mP\Env(\mV)^\omega$ by $\mV$
it is enough to show that for every absolute small left enriched $\infty$-category 
$\mM^\circledast \to \mV^\ot$ and $\X \in \mM$ the left $\mV$-enriched functor 
$\L\Mor_{\mP\L\Env(\mM)_{\L\Enr_\omega}}(\X ,-): \mP\L\Env(\mM)_{\L\Enr_\omega}^\circledast \to \Ind_\omega(\mV)^\circledast$ is left $\mV$-linear. 
By the first part of the proof this holds if the left $\mP\Env(\mV)$-enriched functor 
$\L\Mor_{\mP\L\Env(\mM)}(\X ,-): \mP\L\Env(\mM)^\circledast \to \mP\Env(\mV)^\circledast$ is left $\mP\Env(\mV)$-linear.
This way we reduced to the case that $\mM^\circledast \to \mV^\ot$ is a left enriched $\infty$-category.

So it remains to see that for every $\X \in \mM, \Y \in \mP\L\Env(\mM)$ and $\V \in \mP\Env(\mV)$
the following canonical morphism in $\mP\Env(\mV)$ is an equivalence:
$$\V \ot\L\Mor_{\mP\L\Env(\mM)}(\X, \Y)\to \L\Mor_{\mP\L\Env(\mM)}(\X,\V \ot \Y).$$ 

As $\mP\L\Env(\mM)^\circledast \to \mP\Env(\mV)^\ot$ is a left tensored $\infty$-category compatible with small colimits and $\mP\L\Env(\mM)$ is generated under small colimits by objects of the form $\V_1 \ot...\ot \V_\n \ot \Y$
for $\V_1,...,\V_\n \in \mV$ for $\n \geq 0$ and $\Y \in \mM$, it is enough to see that the canonical morphism 
$$\sigma: \V_1 \ot...\ot \V_\n \ot\L\Mor_{\mP\L\Env(\mM)}(\X, \W_1 \ot...\ot \W_\m \ot \Y)\to \L\Mor_{\mP\L\Env(\mM)}(\X,\V_1 \ot...\ot \V_\n \ot \W_1 \ot...\ot \W_\m \ot \Y)$$ 
is an equivalence.
The composition
$$\V_1 \ot...\ot \V_\n \ot \W_1 \ot...\ot \W_\m \ot \L\Mor_{\mP\L\Env(\mM)}(\X, \Y) \to \V_1 \ot...\ot \V_\n \ot\L\Mor_{\mP\L\Env(\mM)}(\X, \W_1 \ot...\ot \W_\m \ot \Y)$$$$\xrightarrow{\sigma} \L\Mor_{\mP\L\Env(\mM)}(\X,\V_1 \ot...\ot \V_\n \ot \W_1 \ot...\ot \W_\m \ot \Y)$$
is the canonical morphism. Consequently, it suffices to see that for every $\V_1,...,\V_\n \in \mV$ for $\n\geq 0$ and $\Y \in \mM$ the canonical morphism
$$\V_1 \ot...\ot \V_\n \ot \L\Mor_{\mP\L\Env(\mM)}(\X, \Y) \to \L\Mor_{\mP\L\Env(\mM)}(\X,\V_1 \ot...\ot \V_\n \ot \Y)$$ is an equivalence.
This holds by Proposition \ref{awas}.

\end{proof}

\begin{corollary}\label{MorFu} Let $\mM^\circledast \to \mV^\ot\times \mW^\ot $ be a weakly bienriched $\infty$-category.
The functor $$\xi: \mM^\op \to \Enr\Fun_{\mV,\mW}(\mM, \mP\Env(\mV)\ot \mP\Env(\mW))$$
sending $\X \in \mM$ to $\L\Mor_{\bar{\mM}}(\X,-)$
corresponding to the functor $\Gamma_\mM$ of Construction \ref{Enros}
is an embedding.
\end{corollary}

\begin{proof}

By construction \ref{Enros} the functor $\xi$ factors as an embedding 
$$\alpha: \mM^\op \to \Enr\Fun_{\mP\Env(\mV),\mP\Env(\mW)}(\mP\B\Env(\mM), \mP\Env(\mV)\ot \mP\Env(\mW))$$ sending $\X \in \mM$ to $\L\Mor_{\mP\B\Env(\mM)}(\X,-)$
followed by restriction $$\Enr\Fun_{\mP\Env(\mV),\mP\Env(\mW)}(\mP\B\Env(\mM), \mP\Env(\mV)\ot \mP\Env(\mW)) \to \Enr\Fun_{\mV,\mW}(\mM, \mP\Env(\mV)\ot \mP\Env(\mW)).$$
Proposition \ref{corok} for $\rS,\T$ empty guarantees that $\alpha$ lands in the full subcategory $$\Enr\Fun^\L_{\mP\Env(\mV),\mP\Env(\mW)}(\mP\B\Env(\mM), \mP\Env(\mV)\ot \mP\Env(\mW)).$$
By Corollary \ref{envvcor} following restriction is an equivalence: $$ \Enr\Fun^\L_{\mP\Env(\mV),\mP\Env(\mW)}(\mP\B\Env(\mM), \mP\Env(\mV)\ot \mP\Env(\mW)) \to \Enr\Fun_{\mV,\mW}(\mM, \mP\Env(\mV)\ot \mP\Env(\mW)).$$


\end{proof}

\begin{corollary}\label{Univ}
	
Let $(\mV^\ot \to \Ass, \rS), (\mW^\ot \to \Ass, \T)$ be small localization pairs, $\mM^\circledast \to \mV^\ot \times \mW^\ot$ an absolute small $\rS,\T$-bienriched $\infty$-category, $\mN^\circledast \to \rS^{-1}\mP\Env(\mV)^\ot \times \T^{-1}\mP\Env(\mW)^\ot$ a locally small bitensored $\infty$-category compatible with small colimits and $\phi: \mP\B\Env(\mM)_{\rS, \T}^\circledast \to \mN^\circledast$ a small colimits preserving $\rS^{-1}\mP\Env(\mV), \T^{-1}\mP\Env(\mW)$-linear functor.

\begin{enumerate}
\item Then $\phi$ 
is an embedding if the restriction $\phi_{\mid \mM^\circledast}: \mM^\circledast \to \mN^\circledast$ is an embedding and for every $\X \in \mM$ the
image $\phi(\X)$ is atomic. In this case $\mN$ is generated by the essential image of $\phi_{\mid \mM}$ under small colimits and the biaction of $\rS^{-1}\mP\Env(\mV), \T^{-1}\mP\Env(\mW)$.

\item Then $\phi$ 
is an equivalence if and only if the restriction $\phi_{\mid \mM^\circledast}: \mM^\circledast \to \mN^\circledast$ is an embedding, for every $\X \in \mM$ the
image $\phi(\X)$ is atomic 
and $\mN$ is generated under small colimits and left and right tensors by the essential image of $\phi_{\mid \mM^\circledast}.$

\end{enumerate}
	
\end{corollary}

\begin{proof}

(1) We like to see that for every $\X,\Y\in \mP\B\Env(\mM)_{\rS, \T}$ the induced morphism
$$\alpha_{\X,\Y}: \Mor_{\mP\B\Env(\mM)_{\rS, \T}}(\X,\Y)\to \Mor_{\mN}(\phi(\X),\phi(\Y))$$
in $\rS^{-1}\mP\Env(\mV)\ot \T^{-1}\mP\Env(\mW) $ is an equivalence.
By assumption $\alpha_{\X,\Y}$ is an equivalence if $\X,\Y\in \mM$. 
We prove first that for every $\Y \in \mP\B\Env(\mM)_{\rS, \T}$ the full subcategory $\Theta_\Y$ of $ \mP\B\Env(\mM)_{\rS, \T}$ spanned by those $\X$ such that $\alpha_{\X,\Y}$ is an equivalence, is closed under small colimits and the biaction.
This implies the claim if we have shown that $\mM \subset \Theta_\Y$ since
$ \mP\B\Env(\mM)_{\rS, \T}$ is generated by $\mM$ under small colimits and the biaction as a consequence of Proposition \ref{bitte}.
The full subcategory $\Theta_\Y$ is closed under small colimits since $\phi$ preserves small colimits and for every $\Z \in \mN$
the functor $\Mor_\mN(-,\Z): \mN^\op \to \mP\Env(\mV)\otimes \mP\Env(\mW)$
preserves small limits as $\mN^\circledast \to \mP\Env(\mV)^\ot \times \mP\Env(\mW)^\ot$ is a bitensored $\infty$-category compatible with small colimits,
and similar for $ \mP\B\Env(\mM)_{\rS, \T}$.
The full subcategory $\Theta_\Y$ is closed under the biaction since for every $\V \in \rS^{-1}\mP\Env(\mV), \W \in \T^{-1}\mP\Env(\mW)$ the induced morphism
$\Mor_{\mP\B\Env(\mM)_{\rS, \T}}(\V \ot \X\ot\W,\Y)\to \Mor_{\mN}(\phi(\V \ot \X\ot \W),\phi(\Y)) \simeq \Mor_{\mN}(\V \ot \phi(\X)\ot \W,\phi(\Y))$ 
identifies with the induced morphism 
$\rS^{-1}\mP\Env(\mV)\ot \T^{-1}\mP\Env(\mW)(\V\ot\W, \Mor_{\mP\B\Env(\mM)_{\rS, \T}}(\X,\Y))\to \rS^{-1}\mP\Env(\mV)\ot \T^{-1}\mP\Env(\mW)(\V\ot\W, \Mor_{\mN}(\phi(\X),\phi(\Y))).$
It remains to see that $\mM \subset \Theta_\Y$.
Because $ \mP\B\Env(\mM)_{\rS, \T}$ is generated by $\mM$ under small colimits and the biaction, it is enough to prove that for every $\X \in \mM$ the full subcategory of $ \mP\B\Env(\mM)_{\rS, \T}$ spanned by those $\Y$ such that $\alpha_{\X,\Y}$ is an equivalence, is closed under small colimits and the biaction.
This follows from Proposition \ref{corok}.

(2) follows immediately from (1), Remark \ref{embil}, Proposition \ref{corok} and Proposition \ref{bitte}.	
	
\end{proof}

\begin{corollary}
Let $\mV^\ot \to \Ass$ be an $\infty$-operad and $\A$ an associative algebra in $\mV.$ Let $B(\A)^\circledast \subset \RMod_\A(\mV)^\circledast$ be the full weakly left enriched subcategory spanned by $\A.$
	
\end{corollary}

\begin{corollary}\label{corz}
Let $\mV^\ot \to \Ass$ be a presentably monoidal $\infty$-category and $\A$ an associative algebra in $\mV.$ The left $\mV$-enriched embedding $B(\A)^\circledast \subset \RMod_\A(\mV)^\circledast$ induces a left $\mV$-linear equivalence $$\mP\B\Env(B(\A))_{\L\Enr}^\circledast \subset \RMod_\A(\mV)^\circledast .$$
	
\end{corollary}

\begin{proof}
We apply Corollary \ref{Univ} and observe that $\R\Mod_\A(\mV)$ is generated by $\A$ under left tensors and that $\A$ is atomic in $\R\Mod_\A(\mV)$.
	
\end{proof}

\begin{proposition}\label{maimai}
Let $(\mV^\ot \to \Ass, \rS)$ be a small localization pair, $\mW^\ot \to \Ass$ a small $\infty$-operad and $\mM^\circledast \to \mV^\ot \times \mW^\ot$ an absolute small weakly  bienriched $\infty$-category. Let $\mN^\circledast:= \mM^\circledast \times_{\mW^\ot} \emptyset^\ot \to \mV^\ot$ the underlying weakly left enriched $\infty$-category. The enriched embedding $\mM^\circledast \subset \mP\L\Env(\mM)_\rS^\circledast$
induces a $\rS^{-1}\mP\Env(\mV)$-linear equivalence 
$$\mP\L\Env(\mN)_\rS^\circledast \simeq \mP\L\Env(\mM)_\rS^\circledast\times_{\mW^\ot} \emptyset^\ot.$$

\end{proposition}

\begin{proof}
By Proposition \ref{pseuso2} the enriched embedding $\mM^\circledast \subset \mP\L\Env(\mM)_\rS^\circledast$
induces a $\rS^{-1}\mP\Env(\mV)$-linear small colimits preserving functor $$\mP\L\Env(\mN)_\rS^\circledast \to \mM'^\circledast:= \mP\L\Env(\mM)_\rS^\circledast\times_{\mW^\ot} \emptyset^\ot.$$
By Corollary \ref{Univ} it is enough to see that for every $\X \in \mM$ the image of $\X$ in
$\mM'=\mP\L\Env(\mM)_\rS$ is atomic, i.e. the following left $\rS^{-1}\mP\Env(\mV)$-enriched functor is linear and preserves small colimits: \begin{equation}\label{uhuu}
\Mor_{\mM'}(\X,-): \mM'^\circledast \to \rS^{-1}\mP\Env(\mV)^\circledast.\end{equation}
The unique map of $\infty$-operads $\emptyset^\ot \to \mW^\ot$ gives rise to a left adjoint monoidal functor $\mS^\times \simeq \mP\Env(\emptyset)^\ot \to \mP\Env(\mW)^\ot$
whose right adjoint, denoted by $\alpha$, evaluates at the tensor unit of $\Env(\mW)$ and so preserves colimits, too.
The left $\rS^{-1}\mP\Env(\mV)$-enriched functor (\ref{uhuu}) factors as
the left $\rS^{-1}\mP\Env(\mV)$-linear small colimits preserving embedding $\mP\L\Env(\mM)_\rS^\circledast\times_{\mW^\ot} \emptyset^\ot \subset \mP\B\Env(\mM)_\rS^\circledast\times_{\mW^\ot} \emptyset^\ot$ followed by the underlying left $\rS^{-1}\mP\Env(\mV)$-enriched functor of the
$\rS^{-1}\mP\Env(\mV), \mP\Env(\mW)$-enriched functor $$\Mor_{\mP\B\Env(\mM)_\rS}(\X,-): \mP\B\Env(\mM)_\rS^\circledast \to (\rS^{-1}\mP\Env(\mV) \otimes \mP\Env(\mW))^\circledast,$$
which is linear and preserves small colimits by Theorem \ref{corok},
followed by the free $\rS^{-1}\mP\Env(\mV)$-linear small colimits preserving functor
on the small colimits preserving functor $\alpha.$

\end{proof}

\begin{corollary}\label{maimait}
Let $\mV^\ot \to \Ass$ be a presentably monoidal $\infty$-category, $\mW^\ot \to \Ass$ a small $\infty$-operad and $\mM^\circledast \to \mV^\ot \times \mW^\ot$ a small left enriched $\infty$-category. Let $\mN^\circledast:= \mM^\circledast \times_{\mW^\ot} \emptyset^\ot \to \mV^\ot$ the underlying weakly left enriched $\infty$-category. The enriched embedding $\mM^\circledast \subset \mP\L\Env(\mM)_{\L\Enr}^\circledast$
induces a $\mV$-linear equivalence $$\mP\L\Env(\mN)_{\L\Enr}^\circledast \simeq \mP\L\Env(\mM)_{\L\Enr}^\circledast\times_{\mW^\ot} \emptyset^\ot.$$
	
\end{corollary}

\begin{notation}\label{psss}
	
Let $(\mV^\ot \to \Ass, \rS), (\mW^\ot \to \Ass, \T)$ be small localization pairs and $\mM^\circledast \to \mV^\ot \times \mW^\ot$ an absolute small weakly  bienriched $\infty$-category.
Let $$\mM_{\rS,\T}^\circledast \subset \mV^\ot \times_{ \rS^{-1}\mP\Env(\mV)^\ot}\mP\B\Env(\mM)_{\rS, \T}^\circledast \times_{\T^{-1}\mP\Env(\mW)^\ot} \mW^\ot \to \mV^\ot \times \mW^\ot $$
be the full weakly bienriched subcategory spanned the essential image of the functor $\mM \subset \mP\B\Env(\mM) \xrightarrow{\L} \mP\B\Env(\mM)_{\rS, \T}$.
	
\end{notation}

\begin{proposition}\label{pseusor}
Let $(\mV^\ot \to \Ass, \rS), (\mW^\ot \to \Ass, \T)$ be small localization pairs and $\mM^\circledast \to \mV^\ot \times \mW^\ot$ an absolute small weakly  bienriched $\infty$-category.

\begin{enumerate}
\item The weakly bienriched $\infty$-category $\mM_{\rS,\T}^\circledast \to \mV^\ot \times \mW^\ot$ exhibits $\mM_{\rS,\T}$ as $\rS, \T$-bienriched in $\mV,\mW.$

\item For every
$\X,\Y \in \mM$ the induced morphism $$\Mor_{\mP\B\Env(\mM)}(\X,\Y) \to \Mor_{\mP\B\Env(\mM_{\rS,\T})}(\L(\X),\L(\Y)) $$
in $\mP\Env(\mV) \ot \mP\Env(\mW)$ is a local equivalence for the localization $$\mP\Env(\mV) \ot \mP\Env(\mW) \to \rS^{-1}\mP\Env(\mV) \ot \T^{-1}\mP\Env(\mW).$$

\item For every weakly bienriched $\infty$-category $\mN^\circledast \to \mV^\ot \times \mW^\ot$ that exhibits $\mN$ as $\rS, \T$-bienriched in $\mV, \mW$
the induced $\mV, \mW$-enriched functor $\mM^\circledast \to \mM_{\rS,\T}^\circledast$ induces an equivalence
$$\rho_\mN: \Enr\Fun_{\mV, \mW}(\mM_{\rS,\T},\mN) \to \Enr\Fun_{\mV,\mW}(\mM,\mN).$$

\item The left adjoint $\rS^{-1}\mP\Env(\mV), \T^{-1}\mP\Env(\mW)$-linear functor $$\mP\B\Env(\mM)^\circledast \to \mP\B\Env(\mM_{\rS,\T})^\circledast \to \mP\B\Env(\mM_{\rS,\T})_{\rS,\T}^\circledast$$ induces a $\rS^{-1}\mP\Env(\mV), \T^{-1}\mP\Env(\mW)$-linear equivalence $$ (\rS^{-1}\mP\Env(\mV) \otimes_{\mP\Env(\mV)} \mP\B\Env(\mM) \otimes_{\mP\Env(\mW)} \T^{-1}\mP\Env(\mW))^\circledast \simeq \mP\B\Env(\mM_{\rS,\T})^\circledast_{\rS,\T}.$$

\end{enumerate}

\end{proposition}

\begin{proof}

(1): 
Let $\hat{\mM}_{\rS,\T}^\circledast \subset \mP\B\Env(\mM)^\circledast$
be the full weakly bienriched subcategory spanned $\mM_{\rS,\T}$. 
Then $\mM_{\rS,\T}^\circledast \to \mV^\ot \times \mW^\ot $ is the pullback of the bienriched $\infty$-category $\hat{\mM}_{\rS,\T}^\circledast \to \mP\Env(\mV)^\ot \times \mP\Env(\mW)^\ot$.
Let $\bar{\mM}_{\rS,\T}^\circledast \subset \mP\B\Env(\mM_{\rS,\T})^\circledast$
be the full weakly bienriched subcategory spanned $\mM_{\rS,\T}$. Then $\mM_{\rS,\T}^\circledast \to \mV^\ot \times \mW^\ot$ is the pullback of the bienriched $\infty$-category $\bar{\mM}_{\rS,\T}^\circledast \to \mP\Env(\mV)^\ot \times \mP\Env(\mW)^\ot$.
By Proposition \ref{eqq} the induced $\mP\Env(\mV), \mP\Env(\mW)$-linear functor $\mP\B\Env(\mM)^\circledast \to \mP\B\Env(\mM_{\rS,\T})^\circledast$ restricts to a $\mP\Env(\mV), \mP\Env(\mW)$-enriched equivalence $\hat{\mM}_{\rS,\T}^\circledast \to \bar{\mM}_{\rS,\T}^\circledast$
since the pullback of the latter enriched functor to $\mV^\ot, \mW^\ot$ is the identity of $\mM_{\rS,\T}^\circledast.$
Hence for every $\X, \Y \in \mM, \V \in \mP\Env(\mV), \W \in \mP\Env(\mW)$ the following map is an equivalence:
\begin{equation}\label{ravv}
\mP\B\Env(\mM)(\V \ot \L(\X)\ot \W, \L(\Y)) \to \mP\B\Env(\mM_{\rS,\T})(\V \ot \L(\X) \ot \W, \L(\Y)).\end{equation}

Since $\L_{\mid \mM}: \mM \to \mM_{\rS,\T}$ is essentially surjective, we need to see that for every $\X, \Y \in \mM, \f \in \rS, \g \in \T$ the map $\mP\B\Env(\mM_{\rS,\T})(\f \ot \L(\X) \ot \g, \L(\Y))$ is an equivalence.
By equivalence (\ref{ravv}) the latter map identifies with the map
$\mP\B\Env(\mM)(\f \ot \L(\X) \ot \g, \L(\Y))$, which is an equivalence by Lemma \ref{Leom}
since $\L(\Y) \in \mM_{\rS,\T}.$

(2): 
For every $\X,\Y \in \mM$ the induced morphism $\Mor_{\mP\B\Env(\mM)}(\X,\Y) \to \Mor_{\mP\B\Env(\mM_{\rS,\T})}(\L(\X),\L(\Y)) $ in $\mP\Env(\mV) \ot \mP\Env(\mW)$ identifies with the following morphism via equivalence (\ref{ravv}):
$$\Mor_{\mP\B\Env(\mM)}(\X,\Y) \to \Mor_{\mP\B\Env(\mM)}(\X,\L(\Y)) \simeq \Mor_{\mP\B\Env(\mM)}(\L(\X),\L(\Y)) $$$$\simeq \Mor_{\mP\B\Env(\mM_{\rS,\T})}(\L(\X),\L(\Y)).$$
By Theorem \ref{corok} the enriched functor $\Mor_{\mP\B\Env(\mM)}(\X,-): \mP\B\Env(\mM)^\circledast \to (\mP\Env(\mV) \ot \mP\Env(\mW))^\circledast $ 
is linear and preserves small colimits and therefore sends (generating) local equivalences for the localization $ \mP\B\Env(\mM) \to \mP\B\Env(\mM)_{\rS,\T}$ 
to local equivalences for the localization $\mP\Env(\mV) \ot \mP\Env(\mW) \to \rS^{-1}\mP\Env(\mV) \ot \T^{-1}\mP\Env(\mW).$

(3): The functor $\rho_\mN$ is the pullback of the functor $\rho_{\mP\B\Env(\mN)_{\rS, \T}}.$ So it is enough to see that the latter functor is an equivalence.
Consider the commutative triangle:
\begin{equation*}
\begin{xy}
\xymatrix{
\LinFun^\L_{\rS^{-1}\mP\Env(\mV), \T^{-1}\mP\Env(\mW)}(\mP\B\Env(\mM)_{\rS, \T},\mP\B\Env(\mN)_{\rS, \T}) \ar[dd] \ar[rd]^\alpha
\\ 
& \Enr\Fun_{\mV, \mW}(\mM,\mP\B\Env(\mN)_{\rS, \T}) 
\\ \Enr\Fun_{\mV, \mW}(\mM_{\rS,\T},\mP\B\Env(\mN)_{\rS, \T}) \ar[ru]^{\rho_{\mP\B\Env(\mN)_{\rS, \T}}}
}
\end{xy} 
\end{equation*} 

By Proposition \ref{pseuso} the functor $\alpha$ is an equivalence.
Consequently, it is enough to see that the vertical functor is an equivalence.
By Proposition \ref{pseuso} it is enough to prove that the induced left adjoint $\rS^{-1}\mP\Env(\mV), \T^{-1}\mP\Env(\mW)$-linear functor $ \theta: \mP\B\Env(\mM_{\rS,\T})_{\rS, \T}^\circledast  \to \mP\B\Env(\mM)_{\rS, \T}^\circledast$
is an equivalence.
The $\infty$-category $ \mP\B\Env(\mM)$ is generated by $\mM$ under small colimits and the $\mP\Env(\mV), \mP\Env(\mW)$-biaction. Hence the localization 
$ \mP\B\Env(\mM)_{\rS, \T}$ is generated by $\mM_{\rS,\T}$ under small colimits and the $\rS^{-1}\mP\Env(\mV), \T^{-1}\mP\Env(\mW)$-biaction.
Therefore by Corollary \ref{Univ} the functor $\theta$ is an equivalence if 
for every $\X \in \mM$ the $\rS^{-1}\mP\Env(\mV), \T^{-1}\mP\Env(\mW) $-enriched functor
$$\Mor_{\mP\B\Env(\mM)_{\rS,\T}}(\L(\X),-): \mP\B\Env(\mM)_{\rS,\T}^\circledast \to (\rS^{-1}\mP\Env(\mV) \ot \T^{-1}\mP\Env(\mW))^\circledast $$ 
is linear and preserves small colimits. This holds by Theorem \ref{corok}.

(4): Let $\mN^\circledast \to \rS^{-1}\mP\Env(\mV)^\ot \times \T^{-1}\mP\Env(\mW)^\ot$
be a presentably bitensored $\infty$-category. The equivalence of (4) is represented by the following canonical equivalence:
$$ \LinFun^\L_{\rS^{-1}\mP\Env(\mV), \T^{-1}\mP\Env(\mW)}(\mP\B\Env(\mM_{\rS,\T})_{\rS,\T},\mN)
\simeq \Enr\Fun_{\mV, \mW}(\mM_{\rS,\T},\mN) \simeq \Enr\Fun_{\mV, \mW}(\mM,\mN) \simeq $$$$
\LinFun^\L_{\mP\Env(\mV), \mP\Env(\mW)}(\mP\B\Env(\mM),\mN) \simeq $$$$ \LinFun^\L_{\rS^{-1}\mP\Env(\mV), \T^{-1}\mP\Env(\mW)}(\rS^{-1}\mP\Env(\mV) \otimes_{\mP\Env(\mV)} \mP\B\Env(\mM) \otimes_{\mP\Env(\mW)} \T^{-1}\mP\Env(\mW),\mN).$$


\end{proof}

\begin{corollary}\label{uuulp}
Let $(\mV^\ot \to \Ass, \rS), (\mW^\ot \to \Ass, \T)$ be small localization pairs.
The embedding $${^\rS_\mV\B\Enr^\T_{\mW}} \subset {_\mV\omega\B\Enr}_{\mW}$$ admits a left adjoint.
A $\mV,\mW$-enriched functor $\phi: \mM^\circledast \to \mN^\circledast$ is a local equivalence if and only if 
\begin{itemize}
\item the underlying functor $\mM \to \mN$ is essentially surjective,
\item for every objects $\X, \Y \in \mM$ the induced morphism $$\Mor_{\mP\B\Env(\mM)}(\X,\Y) \to \Mor_{\mP\B\Env(\mN)}(\phi(\X),\phi(\Y))$$
in $\mP\Env(\mV) \ot \mP\Env(\mW)$ is a local equivalence for the localization $$\mP\Env(\mV) \ot \mP\Env(\mW) \to \rS^{-1}\mP\Env(\mV) \ot \T^{-1}\mP\Env(\mW).$$
\end{itemize}

\end{corollary}


\begin{corollary}\label{prcor}

\begin{enumerate}
\item Let $(\mV^\ot \to \Ass, \rS), (\mW^\ot \to \Ass, \T)$ be small localization pairs. The $\infty$-category $_\mV^\rS\B\Enr^{\T}_{\mW}$ is compactly generated.

\vspace{1mm}
\item For every presentably monoidal $\infty$-categories
$\mV^\ot \to \Ass, \mW^\ot \to \Ass$ the $\infty$-category $_\mV\B\Enr_{\mW}$ is compactly generated.	
\end{enumerate}

\end{corollary}

\begin{proof}
(1): By Proposition \ref{pr} the $\infty$-category $_\mV\omega\B\Enr_{\mW}$
is compactly generated. By Proposition \ref{pseusor} the embedding 
$_\mV^\rS\B\Enr^{\T}_{\mW} \subset {_\mV\omega\B\Enr}_{\mW}$, which preserves filtered colimits, admits a left adjoint. So the claim follows.

(2): Let $\kappa, \tau$ be small regular cardinals such that $\mV^\ot \to \Ass$ is 
$\kappa$-compactly generated and $\mW^\ot \to \Ass$ is $\tau$-compactly generated.
By Corollary \ref{cosqa} (3) the canonical functor  
$_\mV\B\Enr_{\mW} \to {^\kappa_{\mV^\kappa}\B\Enr^\tau_{\mW^\tau}}$ is an 
equivalence.
So the claim follows from (1) and Example \ref{Exaso}. 

\end{proof}

\subsection{Transfer of enrichment via scalar extension}

In this subsection we describe and study the process of transfering (weak) enrichment along maps of $\infty$-operads (Theorem \ref{bica}).



\begin{notation}Let $\mM^\circledast \to \mU^\ot \times \mV^\ot$, $\mN^\circledast \to \mV^\ot \times \mW^\ot$, $\mO^\circledast \to \mW^\ot \times \mQ^\ot$ be presentably bitensored $\infty$-categories.
We write $ (\mM \ot_\mV \mN \ot_\mW \mO)^\circledast \to \mU^\ot \times \mQ^\ot $ for the relative tensor product of \cite[Definition 4.4.2.10.]{lurie.higheralgebra}.
\end{notation}

\begin{lemma}\label{remlo}

Let $\mM^\circledast \to \mV^\ot$ be a presentably right tensored $\infty$-category, $\mN^\circledast \to \mV^\ot \times \mW^\ot$ a presentably bitensored $\infty$-category and $\mO^\circledast \to \mW^\ot$ a presentably left tensored $\infty$-category.
The $\infty$-category $\mM \ot_\mV \mN \ot_\mW \mO $ is generated under small colimits by the essential image of the functor $\mM \times \mN \times \mO \to \mM \ot_\mV \mN \ot_\mW \mO.$	

\end{lemma}

\begin{proof}

By \cite[Example 4.7.2.7.]{lurie.higheralgebra} there is a canonical equivalence $$ \mM \ot_\mV \mN \ot_\mW \mO\simeq \colim_{[\n] \in \Ass}(\mM \ot \mV^{\ot \n} \ot \mN \ot \mW^{\ot \n} \ot \mO).$$
So every object of $ \mM \ot_\mV \mN \ot_\mW \mO $ lies in the essential image of
the canonical small colimits preserving functor $\lambda_\n: \mM \ot \mV^{\ot \n} \ot \mN \ot \mW^{\ot \n} \ot \mO \to  \mM \ot_\mV \mN \ot_\mW \mO$ for some $[\n] \in \Ass.$
The map $[0] \simeq \{0\} \subset [\n]$ induces a map
$\mM \ot \mV^{\ot \n} \ot \mN \ot \mW^{\ot \n} \ot \mO \to \mM \ot \mN \ot \mO$ over $\mM \ot_\mV \mN \ot_\mW \mO$ so that the functor $\lambda_0: \mM \ot \mN \ot \mO \to \mM \ot_\mV \mN \ot_\mW \mO$ is essentially surjective.
There is a universal functor $ \alpha: \mM \times \mN \times \mO \to \mM \ot \mN \ot \mO$ that preserves small colimits component-wise.
We will prove that $\mM \ot \mN \ot \mO$ is generated under small colimits by the essential image of $\alpha.$
This will imply that $\mM \ot_\mV \mN \ot_\mW \mO$ is generated under small colimits by the essential image of $\lambda_0 \circ \alpha$.

Since $ \mM, \mN, \mO$ are presentable, there are small $\infty$-categories $\mA, \mB, \mC$ and localizations $\mP(\mA) \leftrightarrows \mM, \mP(\mB) \leftrightarrows \mN, \mP(\mC) \leftrightarrows \mO .$
By \cite[Proposition 4.8.1.17.]{lurie.higheralgebra} we obtain an induced localization
$\mP(\mA) \ot \mP(\mB) \ot \mP(\mC) \leftrightarrows \mM \ot \mN \ot \mO$.
Since there is a commutative square
\begin{equation*} 
\begin{xy}
\xymatrix{
\mP(\mA) \times \mP(\mB) \times \mP(\mC) \ar[d] \ar[rr]^\alpha
&&\mP(\mA) \ot \mP(\mB) \ot \mP(\mC) \ar[d] 
\\  \mM \times \mN  \times \mO
\ar[rr]^\alpha  &&  \mM \ot \mN \ot \mO,
}
\end{xy} 
\end{equation*}
it is enough to see that $\mP(\mA) \ot \mP(\mB) \ot \mP(\mC)$ is generated under small colimits by the essential image of the functor
$\alpha: \mP(\mA) \times \mP(\mB)\times \mP(\mC)\to \mP(\mA) \ot \mP(\mB)\ot \mP(\mC).$
By \cite[Proposition 4.8.1.17.]{lurie.higheralgebra} there is a canonical equivalence $\mP(\mA) \ot \mP(\mB)\ot \mP(\mC) \simeq \mP(\mA \times \mB \times \mC)$ and the left adjoint functor
$\alpha$ identifies with the canonical left adjoint functor $\mP(\mA) \times \mP(\mB) \times \mP(\mC)\to \mP(\mA \times \mB \times \mC)$ whose restriction to $\mA \times \mB \times \mC$ is the Yoneda-embedding.
So the result follows from the fact that $\mP(\mA \times \mB \times \mC)$ is generated under small colimits by $\mA \times \mB \times \mC$ \cite[Corollary 5.1.5.8.]{lurie.HTT}.

\end{proof}

\begin{corollary}\label{remlop}

Let $\mM^\circledast \to \mV^\ot \times \mW^\ot$ be a presentably bitensored $\infty$-category and $\mV^\ot \to \mV'^\ot, \mW^\ot \to \mW'^\ot$ small colimits preserving monoidal functors of presentably monoidal $\infty$-categories.
Then $\mV' \ot_\mV \mM \ot_\mW \mW' $ is generated under small colimits and the $\mV', \mW'$-biaction by the essential image of the functor
$\mM \to \mV' \ot_\mV \mM \ot_\mW \mW'.$	

\end{corollary}

\begin{proof}
This follows immediately from Lemma \ref{remlo} since the functor
$\mM \to \mV' \ot_\mV \mM \ot_\mW \mW'$ refines to a 
$\mV, \mW$-linear functor.	

\end{proof}

\begin{corollary}\label{corazon}

Let $\mM^\circledast \to \mV^\ot \times \mW^\ot, \mN^\circledast \to \mV^\ot \times \mW^\ot$ be presentably bitensored $\infty$-categories,
$\G: \mN^\circledast \to \mM^\circledast$ a $\mV, \mW$-linear functor whose underlying functor preserves small colimits and is monadic, that admits a 
$\mV, \mW$-linear left adjoint $\F$.
Let $\mO^\circledast \to \mV^\ot$ be a presentably right tensored $\infty$-category and $\mP^\circledast \to \mW^\ot$ a presentably left tensored $\infty$-category.
The functor $\mO \ot_\mV \mN \ot_\mW \mP \to \mO \ot_\mV \mM \ot_\mW \mP$ is monadic.

\end{corollary}

\begin{proof}

The $\mV, \mW$-enriched adjunction $\F: \mM^\circledast \rightleftarrows \mN^\circledast:\G $,
where both adjoints preserve small colimits and are $\mV,\mW$-linear, 
gives rise to an adjunction $\mO \ot_\mV \F \ot_\mW \mP: \mO \ot_\mV \mM \ot_\mW \mP \rightleftarrows \mO \ot_\mV \mN \ot_\mW \mP: \mO \ot_\mV \G \ot_\mW \mP$, where both adjoints preserve small colimits.
So by the monadicity theorem \cite[Theorem 4.7.3.5.]{lurie.higheralgebra}
the right adjoint is monadic if it is conservative.
We observe that the right adjoint is conservative if $\mO \ot_\mV \mN \ot_\mW \mP$ is generated under small colimits by the essential image of the left adjoint.
We will prove that $\mO \ot_\mV \mN \ot_\mW \mP$ is generated under small colimits by the essential image of the left adjoint.
Let $\mQ \subset \mO \ot_\mV \mN \ot_\mW \mP$ be a full subcategory containing
the essential image of $\mO \ot_\mV \F \ot_\mW \mP$ and closed under small colimits. For every $\X \in \mO, \Y \in \mP$ let $\mQ_{\X,\Y} \subset \mN$
the full subcategory spanned by all $\Z \in \mN$ whose image under the canonical small colimits preserving functor $\{\X\}\times \mN \times \{\Y\} \to \mO \times \mN \times \mP \to \mO \otimes \mN \otimes \mP \to \mO \ot_\mV \mN \ot_\mW \mP$ belongs to $\mQ.$
Then $\mQ_{\X,\Y} $ is closed in $\mN$ under small colimits and contains the essential image of $\F: \mM \to \mN.$
Since $\G$ is monadic, $\mN$ is generated under small colimits by the essential image of $\F$ \cite[Theorem 4.7.3.5.]{lurie.higheralgebra}.
Hence $\mQ_{\X,\Y}=\mN$ so that $\mQ$ contains the essential image of the functor
$ \mO \times \mN \times \mP \to \mO \otimes \mN \otimes \mP \to \mO \ot_\mV \mN \ot_\mW \mP$. Hence $\mQ= \mO \ot_\mV \mN \ot_\mW \mP$ by Lemma \ref{remlo}.

\end{proof}

\begin{theorem}\label{bica}
The forgetful functor $\gamma: \omega\B\Enr \to \Op_\infty \times \Op_\infty$ 
is a cocartesian fibration.
Let $\psi: \mM^\circledast \to \mN^\circledast$ be an enriched functor lying over maps of $\infty$-operads $\alpha: \mV^\ot \to \mV'^\ot, \alpha': \mW^\ot \to \mW'^\ot$.
The following conditions are equivalent:

\begin{enumerate}
\item The map $\psi$ is $\gamma$-cocartesian.

\item The induced linear functor $\B\Env(\mM)^\circledast \to \B\Env(\mN)^\circledast$
exhibits $\B\Env(\mN)$
as $$\Env(\mV') \ot_{\Env(\mV)} \B\Env(\mM) \ot_{\Env(\mW)} \Env(\mW').$$

\item The functor $\mM \to \mN$ is essentially surjective and the induced linear functor $\mP\B\Env(\mM)^\circledast \to \mP\B\Env(\mN)^\circledast$ exhibits
$\mP\B\Env(\mN)$ as $$\mP\Env(\mV') \ot_{\mP\Env(\mV)} \mP\B\Env(\mM) \ot_{\mP\Env(\mW)} \mP\Env(\mW').$$

\item The functor $\mM \to \mN$ is essentially surjective and for every $\X, \Y \in \mM$ the induced morphism 
$$ (\alpha, \beta)_!(\Gamma_{\mM}(\X,\Y)) \to \Gamma_{\mN}(\psi(\X),\psi(\Y))$$ in $\mP(\B\Env(\mV') \times \B\Env(\mW')) $ is an equivalence.

\end{enumerate}

\end{theorem}

\begin{proof}
We prove that (1) implies (2): 
Let $\psi: \mM^\circledast \to \mN^\circledast$ be an enriched functor lying over maps of $\infty$-operads $\alpha: \mV^\ot \to \mV'^\ot, \beta: \mW^\ot \to \mW'^\ot$.
For every bitensored $\infty$-category $ \mQ^\circledast \to \mU^\ot \times \mT^\ot$ the square
\begin{equation*}\label{eqapoc}\begin{xy}
\xymatrix{
\BMod(\mB\Env(\mN), \mQ) \ar[d]
\ar[rr]
&&\BMod(\mB\Env(\mM), \mQ) \ar[d]
\\
\Mon(\Env(\mV'), \mU) \times \Mon(\Env(\mW'), \mT) \ar[rr] && \Mon(\Env(\mV), \mU) \times \Mon(\Env(\mW), \mT)}
\end{xy}
\end{equation*}
identifies with the commutative square
\begin{equation}\label{eqapo}
\begin{xy}
\xymatrix{
\omega\B\Enr(\mN, \mQ) \ar[d]
\ar[rr]
&&\omega\B\Enr(\mM, \mQ) \ar[d]
\\
\Op_\infty(\mV', \mU) \times \Op_\infty(\mW', \mT) \ar[rr] && \Op_\infty(\mV, \mU) \times \Op_\infty(\mW, \mT).}
\end{xy}\end{equation}	

Next we show that (3) implies (1).
Let $ \mQ^\circledast \to \mU^\ot \times \mT^\ot$ be a weakly bienriched $\infty$-category.
We like to see that square (\ref{eqapo}) is a pullback square.
If the functor $\mM \to \mN$ is essentially surjective,
square (\ref{eqapo}) is a pullback square if square (\ref{eqapo})
for $ \mQ^\circledast $ replaced by $ \mP\mB\Env(\mQ)^\circledast $
is a pullback square.
But this square (after rotation) identifies with the commutative square
\begin{equation*}\label{eqapot}\begin{xy}
\xymatrix{
\cc\cc\BMod(\mP\mB\Env(\mN), \mP\mB\Env(\mQ)) \ar[d]
\ar[r]
&\L\Mon(\mP\Env(\mV'), \mP\Env(\mU)) \times \L\Mon(\mP\Env(\mW'), \mP\Env(\mT))  \ar[d]
\\
\cc\cc\BMod(\mP\mB\Env(\mM), \mP\mB\Env(\mQ))\ar[r] & \L\Mon(\mP\Env(\mV), \mP\Env(\mU)) \times \L\Mon(\mP\Env(\mW), \mP\Env(\mT)).}
\end{xy}
\end{equation*}

Next we prove that (3) implies (4).
By Example \ref{Line} for every $\X \in \mM$ there is a unique left adjoint $\mP\Env(\mV), \mP\Env(\mW)$-linear functor
$\mP(\Env(\mV) \times \Env(\mW))^\circledast \to \mP\B\Env(\mM)^\circledast$
sending the tensor unit to $\X$, which by Corollary \ref{corok} admits a $\mP\Env(\mV), \mP\Env(\mW)$-linear, small colimits preserving right adjoint sending $\Z$ to $\L\Mor_{\mP\B\Env(\mM)}(\X,\Z).$
Applying the functor $\mP\Env(\mV') \ot_{\mP\Env(\mV)} (-) \ot_{\mP\Env(\mW)} \mP\Env(\mW')$ the induced left adjoint $\mP\Env(\mV'), \mP\Env(\mW')$-linear functor
$$ \mP(\Env(\mV') \times \Env(\mW'))^\circledast \to (\mP\Env(\mV') \ot_{\mP\Env(\mV)} \mP\B\Env(\mM) \ot_{\mP\Env(\mW)} \mP\Env(\mW'))^\circledast$$ is right adjoint to the $\mP\Env(\mV'), \mP\Env(\mW')$-linear, small colimits preserving functor $$\L\Mor_{\mP\B\Env(\mN)}(\psi(\X),-).$$
Consequently, by uniqueness of right adjoints for every $\Y \in \mM$ there is a canonical equivalence
$$ (\alpha, \beta)_!(\L\Mor_{\mP\B\Env(\mM)}(\X,\Y)) \simeq \L\Mor_{\mP\B\Env(\mN)}(\psi(\X),\psi(\Y)).$$
We will complete the proof by showing that $\gamma$ is a cocartesian fibration
whose cocartesian morphisms satisfy (3) and so conditions (1), (2) and (4).
This will imply that condition (3) follows from (2) and that condition (1) follows from (4):
the enriched functors satisfying conditions (2), (4), respectively, satisfy the (2) out of (3)-property.
By the existence of $\gamma$-cocartesian lifts (that satify conditions (1)-(4))
we are reduced to show that every $\mV,\mW$-enriched functor
$\psi: \mM^\circledast \to \mN^\circledast$ satisfying (2), (4), respectively, is an equivalence.
The first fact follows from the fact that $\psi$ is the pullback of the functor $\mB\Env(\mM)^\circledast \to \mB\Env(\mN)^\circledast$.
If a $\mV,\mW$-enriched functor $\psi: \mM^\circledast \to \mN^\circledast$ satisfies (4), the induced functor $\mM \to \mN$ on underlying $\infty$-categories is essentially surjective and for every $\X, \Y \in \mM$ the morphism 
$\L\Mor_{\mP\B\Env(\mM)}(\X,\Y) \to \L\Mor_{\mP\B\Env(\mN)}(\psi(\X),\psi(\Y))$ is an equivalence.
So for every $\V_1,..., \V_\n \in \mV$ for $\n \geq0$ the following map is an equivalence: $$ \mP\Env(\mV)(\V_1 \ot ... \ot \V_\n, \L\Mor_{\mP\B\Env(\mM)}(\X,\Y))
\to \mP\Env(\mV)(\V_1 \ot ... \ot \V_\n, \L\Mor_{\mP\B\Env(\mN)}(\psi(\X),\psi(\Y))).$$
The latter map identifies with the map
$ \Mul_{\mM}(\V_1, ..., \V_\n,\X,\Y) \to \Mul_{\mN}(\V_1 ... \V_\n, \psi(\X),\psi(\Y)).$
So if (4) holds, $\psi$ is an equivalence.

Hence it remains to prove that $\gamma$ is a cocartesian fibration whose cocartesian morphisms satisfy (3).
Let $\mM^\circledast \to \mV^\ot \times \mW^\ot$ be a weakly bienriched $\infty$-category and $\alpha: \mV^\ot \to \mV'^\ot, \beta: \mW^\ot \to \mW'^\ot$ be maps of $\infty$-operads.
Let $(\alpha, \beta)_!(\mM)^\circledast \to \mV'^\ot \times \mW'^\ot$
be the full weakly bienriched subcategory of the presentably bitensored $\infty$-category $(\mP\Env(\mV') \ot_{\mP\Env(\mV)} \mP\B\Env(\mM) \ot_{\mP\Env(\mW)} \mP\Env(\mW'))^\circledast \to (\mP\Env(\mV') \ot \mP\Env(\mW'))^\ot$
spanned by the essential image of the canonical functor
$$\tau: \mM \subset \mP\B\Env(\mM) \to \mP\Env(\mV') \ot_{\mP\Env(\mV)} \mP\B\Env(\mM) \ot_{\mP\Env(\mW)} \mP\Env(\mW')$$
and let $\psi: \mM^\circledast \to (\alpha, \beta)_!(\mM)^\circledast$ be the canonical enriched functor. We will prove that the embedding
$(\alpha, \beta)_!(\mM)^\circledast \subset (\mP\Env(\mV') \ot_{\mP\Env(\mV)} \mP\B\Env(\mM) \ot_{\mP\Env(\mW)} \mP\Env(\mW'))^\circledast$
induces an equivalence: $$\kappa: \mP\B\Env((\alpha, \beta)_!(\mM))^\circledast \simeq (\mP\Env(\mV') \ot_{\mP\Env(\mV)} \mP\B\Env(\mM) \ot_{\mP\Env(\mW)} \mP\Env(\mW'))^\circledast.$$
This will imply that the induced left adjoint linear functor 
$$ \mP\B\Env(\mM)^\circledast \to \mP\B\Env((\alpha, \beta)_!(\mM))^\circledast \simeq (\mP\Env(\mV') \ot_{\mP\Env(\mV)} \mP\B\Env(\mM) \ot_{\mP\Env(\mW)} \mP\Env(\mW'))^\circledast$$
is the canonical one so that $\psi$ satisfies condition (3).
To see that the functor $\kappa$ is fully faithful, it is enough to see that
for every $\X \in \mM$ the $\mP\Env(\mV'), \mP\Env(\mW')$-enriched functor
$$\L\Mor_{\mP\Env(\mV') \ot_{\mP\Env(\mV)} \mP\B\Env(\mM) \ot_{\mP\Env(\mW)} \mP\Env(\mW')}(\tau(\X),-): $$$$ (\mP\Env(\mV') \ot_{\mP\Env(\mV)} \mP\B\Env(\mM) \ot_{\mP\Env(\mW)} \mP\Env(\mW'))^\circledast \to (\mP\Env(\mV') \otimes \mP\Env(\mW'))^\circledast$$
is $\mP\Env(\mV'), \mP\Env(\mW')$-linear and preserves small colimits.
The $\mP\Env(\mV), \mP\Env(\mW)$-linear adjunction $$(-)\ot \X \ot (-): (\mP\Env(\mV) \ot \mP\Env(\mW))^\circledast \rightleftarrows \mP\B\Env(\mM)^\circledast: \L\Mor_{\mP\B\Env(\mM)}(\X,-),$$
where the right adjoint is $\mP\Env(\mV), \mP\Env(\mW)$-linear and preserves small colimits by Corollary \ref{corok}, is sent by the $\Cat_\infty$-linear functor $$\mP\Env(\mV') \ot_{\mP\Env(\mV)}(-) \ot_{\mP\Env(\mW)} \mP\Env(\mW'): {_{\mP\Env(\mV)}\BMod}_{\mP\Env(\mW)} \to {_{\mP\Env(\mV')}\BMod}_{\mP\Env(\mW)'}$$ to a $\mP\Env(\mV'), \mP\Env(\mW')$-linear adjunction 
$$(-)\ot \tau(\X) \ot (-): (\mP\Env(\mV') \otimes \mP\Env(\mW'))^\circledast \rightleftarrows (\mP\Env(\mV') \ot_{\mP\Env(\mV)} \mP\B\Env(\mM) \ot_{\mP\Env(\mW)} \mP\Env(\mW'))^\circledast:$$$$ \mP\Env(\mV') \ot_{\mP\Env(\mV)}\L\Mor_{\mP\B\Env(\mM)}(\X,-) \ot_{\mP\Env(\mW)} \mP\Env(\mW'),$$ where the right adjoint is $\mP\Env(\mV'), \mP\Env(\mW')$-linear and preserves small colimits.
Hence the latter right adjoint identifies with the $\mP\Env(\mV'), \mP\Env(\mW')$-enriched functor
$$\L\Mor_{\mP\Env(\mV') \ot_{\mP\Env(\mV)} \mP\B\Env(\mM) \ot_{\mP\Env(\mW)} \mP\Env(\mW')}(\tau(\X),-),$$ which therefore becomes linear and preserves small colimits. Moreover $\kappa$ is essentially surjective since the $\infty$-category
$\mP\Env(\mV') \ot_{\mP\Env(\mV)} \mP\B\Env(\mM) \ot_{\mP\Env(\mW)} \mP\Env(\mW')$
is generated by the essential image of $\tau$ under small colimits and the $\mP\Env(\mV'), \mP\Env(\mW')$-biaction according to Corollary \ref{remlop}.

\end{proof}

\begin{corollary}
Let $\phi: \mV^\ot \to \mV'^\ot, \psi: \mW^\ot \to \mW'^\ot$
be maps of small $\infty$-operads.
The functor $(\phi, \psi)^*:{_{\mV'} \omega\B\Enr}_{\mW'} \to {_{\mV} \omega\B\Enr}_{\mW}$ admits a left adjoint $(\phi, \psi)_!: {_{\mV} \omega\B\Enr}_{\mW} \to {_{\mV'} \omega\B\Enr}_{\mW'}.$	

\end{corollary}

\begin{proposition}\label{transpo}Let $\phi: \mV^\ot \to \mV'^\ot, \psi: \mW^\ot \to \mW'^\ot$ be maps of small $\infty$-operads.

\begin{enumerate}\item Let $\kappa,\tau$ be small regular cardinals.	

\begin{enumerate}

\item Let $\phi: \mV^\ot \to \mV'^\ot$ be a lax monoidal functor between monoidal $\infty$-categories compatible with $\kappa$-small colimits that preserves $\kappa$-small colimits.
The induced functor $(\phi, \psi)_!:{ _\mV \omega\B\Enr}_\mW \to {_{\mV'} \omega\B\Enr}_{\mW'}$
of Proposition \ref{bica} restricts to a functor
$$_\mV \L\P\Enr^\kappa_\mW \to {_{\mV'} \L\P\Enr}^\kappa_{\mW'}.$$

\item Let $\psi: \mW^\ot \to \mW'^\ot$ be a lax monoidal functor between monoidal $\infty$-categories compatible with $\tau$-small colimits that preserves $\tau$-small colimits.
The induced functor $(\phi, \psi)_!:{ _\mV \omega\B\Enr}_\mW \to {_{\mV'} \omega\B\Enr}_{\mW'}$
of Proposition \ref{bica} restricts to a functor
$$_\mV \R\Enr^\tau_\mW \to {_{\mV'} \R\P\Enr}^\tau_{\mW'}.$$

\item Let $\phi: \mV^\ot \to \mV'^\ot$ be a lax monoidal functor between monoidal $\infty$-categories compatible with $\kappa$-small colimits that preserves $\kappa$-small colimits and $\psi: \mW^\ot \to \mW'^\ot$ a lax monoidal functor between monoidal $\infty$-categories compatible with $\tau$-small colimits that preserves $\tau$-small colimits.
The induced functor $(\phi, \psi)_!:{ _\mV \omega\B\Enr}_\mW \to {_{\mV'} \omega\B\Enr}_{\mW'}$
of Proposition \ref{bica} restricts to a functor
$$_\mV ^\kappa\B\P\Enr^{\tau}_\mW \to {_{\mV'}^\kappa \B\P\Enr}^{\tau}_{\mW'}.$$
\end{enumerate}

\item The induced functor $(\phi, \psi)_!:{ _\mV \omega\B\Enr}_\mW \to {_{\mV'} \omega\B\Enr}_{\mW'}$ of Proposition \ref{bica} restricts to functors
$$_\mV \L\Enr_\mW \to {_{\mV'} \L\Enr}_{\mW'},$$
$$_\mV \R\Enr_\mW \to {_{\mV'} \R\Enr}_{\mW'},$$
$$_\mV \B\Enr_\mW \to {_{\mV'} \B\Enr}_{\mW'},$$
respectively, if $\psi$, $\phi$, both $\phi$ and $\psi$ admit a left adjoint relative to $\Ass$, respectively.

\item We have the following refinement of (2):
\begin{enumerate}
\item If $\psi: \mW^\ot \to \mW'^\ot$ is a lax monoidal functor between monoidal
$\infty$-categories whose underlying functor admits a left adjoint, the induced functor $(\phi, \psi)_!:{ _\mV \omega\B\Enr}_\mW \to {_{\mV'} \omega\B\Enr}_{\mW'}$ of Proposition \ref{bica} restricts to a functor
$$_\mV \L\Enr_\mW \cap {_\mV \R\P\Enr_\mW} \to {_{\mV'} \L\Enr}_{\mW'} \cap {_{\mV'} \R\P\Enr_{\mW'}}.$$

\item If $\phi: \mV^\ot \to \mV'^\ot$ is a lax monoidal functor between monoidal
$\infty$-categories whose underlying functor admits a left adjoint, the induced functor $(\phi, \psi)_!:{ _\mV \omega\B\Enr}_\mW \to {_{\mV'} \omega\B\Enr}_{\mW'}$ of Proposition \ref{bica} restricts to a functor
$$_\mV \L\Enr_\mW \cap {_\mV \R\P\Enr_\mW} \to {_{\mV'} \L\Enr}_{\mW'} \cap {_{\mV'} \R\P\Enr_{\mW'}}.$$

\item If $\phi, \psi$ are lax monoidal functors between monoidal
$\infty$-categories whose underlying functors admit a left adjoint, the induced functor $(\phi, \psi)_!:{ _\mV \omega\B\Enr}_\mW \to {_{\mV'} \omega\B\Enr}_{\mW'}$ of Proposition \ref{bica} restricts to a functor
$$_\mV \B\Enr_\mW \to {_{\mV'} \B\Enr}_{\mW'}.$$
\end{enumerate}
\end{enumerate}

\end{proposition}

\begin{proof}
(1) follows from Proposition \ref{bica} (4) in view of the fact that 
the induced left adjoint functor $\phi_!: \mP(\mV) \to \mP(\mV')$
restricts to a functor $\Ind_\kappa(\mV) \to \Ind_\kappa(\mV')$ and similar for $\psi$. 
The latter holds since $\Ind_\kappa(\mV)$ is generated by $\mV$ as full subcategory of $\mP(\mV) $ under small $\kappa$-filtered colimits \cite[Corollary 5.3.5.4.]{lurie.HTT}.

We prove (2) for the first functor of (2), the second case of (2) is dual and the third case of (2) follows from the first and second one.
If $\psi$ admits a left adjoint $\alpha$ relative to $\Ass$,
also $\Env(\alpha)$ is left adjoint to $\Env(\psi)$. 
So there is an adjunction $$ \Fun(\Env(\alpha)^\op, \mP\Env(\mV')): \Fun(\Env(\mW)^\op, \mP\Env(\mV')) \rightleftarrows $$$$ \Fun(\Env(\mW')^\op, \mP\Env(\mV')): \Fun(\Env(\psi)^\op, \mP\Env(\mV')).$$ 
Thus the functor $(\Env(\phi) \times \Env(\psi))_!: \mP(\Env(\mV)\times \Env(\mW)) \to \mP(\Env(\mV')\times \Env(\mW'))$
is equivalent to the functor
$$\Fun(\Env(\mW)^\op, \mP\Env(\mV)) \xrightarrow{\Fun(\Env(\mW)^\op,\phi_!)} \Fun(\Env(\mW)^\op, \mP\Env(\mV')) $$$$\xrightarrow{\Fun(\Env(\alpha)^\op, \mP\Env(\mV'))} \Fun(\Env(\mW')^\op, \mP\Env(\mV'))$$
since the right adjoints are equivalent.
The latter functor restricts to a functor
$$\Fun(\Env(\mW)^\op, \mV) \xrightarrow{\Fun(\Env(\mW)^\op,\phi)} \Fun(\Env(\mW)^\op, \mV') \xrightarrow{\Fun(\Env(\alpha)^\op, \mV')} \Fun(\Env(\mW')^\op, \mV').$$

The proof of (3) is similar to (2) but easier. We prove 3 (a). The proof of (3) (b) is dual and (3) (c) follows from 3 (a) and 3 (b). If $\psi$ admits a left adjoint $\alpha$, there is an adjunction $$ \Fun(\alpha^\op, \mP\Env(\mV')): \Fun(\mW^\op, \mP\Env(\mV')) \rightleftarrows \Fun(\mW'^\op, \mP\Env(\mV')): \Fun(\psi^\op, \mP\Env(\mV')).$$ 
Thus the functor $(\Env(\phi) \times \psi)_!: \mP(\Env(\mV)\times \mW) \to \mP(\Env(\mV')\times \mW')$
is equivalent to the functor
$$\Fun(\mW^\op, \mP\Env(\mV)) \xrightarrow{\Fun(\mW^\op,\phi_!)} \Fun(\mW^\op, \mP\Env(\mV')) \xrightarrow{\Fun(\alpha^\op, \mP\Env(\mV'))} \Fun(\mW'^\op, \mP\Env(\mV'))$$
since the right adjoints are equivalent.
The latter functor restricts to a functor
$$\Fun(\mW^\op, \mV) \xrightarrow{\Fun(\mW^\op,\phi)} \Fun(\mW^\op, \mV') \xrightarrow{\Fun(\alpha^\op, \mV')} \Fun(\mW'^\op, \mV').$$

\end{proof}

\begin{proposition}\label{colttp}
Let $\alpha: \mV^\ot \to \mV'^\ot, \beta: \mW^\ot \to \mW'^\ot$ be maps of small $\infty$-operads that admit right adjoints $\phi,\psi$ relative to $\Ass$, respectively.
There is an adjunction $$(\alpha, \beta)^*: {_{\mV'} \omega\B\Enr}_{\mW'} \rightleftarrows {_{\mV} \omega\B\Enr}_{\mW}: (\phi, \psi)^*.$$
In particular, there is a canonical equivalence $(\phi, \psi)_! \simeq (\alpha, \beta)^*$
and so an adjunction $$(\alpha, \beta)_!:{_{\mV} \omega\B\Enr}_{\mW} \rightleftarrows {_{\mV'} \omega\B\Enr}_{\mW'}: (\phi, \psi)_!.$$

\end{proposition}

\begin{proof}
Let $\rho: \mA \to \Alg_\mV(\mV'), \rho': \mB \to \Alg_\mW(\mW'), \theta: \mC \to {_{\mV'} \omega\B\Enr}_{\mW'}$
be functors. The functors $\rho, \rho'$ correspond to maps $\mA \times \mV^\ot \to \mV'^\ot, \mW^\ot \times \mB \to \mW'^\ot $ of cocartesian fibrations relative to the collection of inert morphisms of $\Ass$, and $\theta$ is classified by a map $ \mM^\circledast \to \mV'^\ot \times \mC \times \mW'^\ot$ of cocartesian fibrations relative to the collection of inert morphisms of $\Ass \times \mC \times \Ass$ that preserve the minimum in the first component
and preserve the maximum in the last component.
Thus the pullback of $\mA \times \mV^\ot \times_{\mV'^\ot} \mM^\circledast \times_{\mW'^\ot} \mW^\ot \times \mB \to \mA \times \mV^\ot \times \mC \times \mW^\ot\times \mB $ classifies a functor $\mA \times \mC \times \mB \to {_{\mV} \omega\B\Enr}_{\mW}.$
For $\rho, \rho', \theta$ the identities we obtain a functor $\Alg_\mV(\mV') \times {_{\mV'} \omega\B\Enr}_{\mW'} \times \Alg_\mW(\mW')\to {_{\mV} \omega\B\Enr}_{\mW}$
corresponding to a functor $\T: \Alg_\mV(\mV') \times \Alg_\mW(\mW')\to \Fun({_{\mV'} \omega\B\Enr}_{\mW'}, {_{\mV} \omega\B\Enr}_{\mW})$ that sends $(\kappa,\tau)$ to $(\kappa,\tau)^*.$
Moreover by construction for every maps of small $\infty$-operads $\mU^\ot \to \mV^\ot, \mV'^\ot \to \mU'^\ot, \mQ^\ot \to \mW^\ot, \mW'^\ot \to \mQ'^\ot$ there is a commutative square:
\begin{equation}\label{eol}
\begin{xy}
\xymatrix{
\Alg_\mV(\mV') \times \Alg_\mW(\mW') \ar[d]
\ar[rr]^\T
&& \Fun({_{\mV'} \omega\B\Enr}_{\mW'}, {_{\mV} \omega\B\Enr}_{\mW}) \ar[d]
\\
\Alg_\mU(\mU') \times \Alg_\mQ(\mQ') \ar[rr]^\T &&  \Fun({_{\mU'} \omega\B\Enr}_{\mQ'}, {_{\mU} \omega\B\Enr}_{\mQ}).}
\end{xy}\end{equation}	

By Remark \ref{enrra} there are monoidal natural transformations $\eta: \id \to \phi \circ \alpha,
\epsilon: \alpha \circ \phi \to \id$ such that $\epsilon \alpha \circ \alpha \eta: \alpha \to \alpha $ and $\phi \epsilon \circ \eta \phi: \phi \to \phi $ are the identities
and 
monoidal natural transformations $\eta': \id \to \psi \circ \beta,
\epsilon': \beta \circ \psi \to \id$ such that $\epsilon' \beta \circ \beta \eta': \beta \to \beta $ and $\psi \epsilon' \circ \eta' \psi: \psi \to \psi $ are the identities.
Consequently, by the commutative square (\ref{eol}) we find that $\T(\epsilon, \epsilon') \T(\alpha,\beta) \circ \T(\alpha,\beta) \T(\eta, \eta'): \T(\alpha,\beta) \to \T(\alpha,\beta) $ and $\T(\phi,\psi) \T(\epsilon,\epsilon') \circ \T(\eta,\eta') \T(\phi,\psi): \T(\phi,\psi) \to \T(\phi,\psi) $ are the identities.
So 
we obtain an adjunction $ \T(\alpha,\beta)= (\alpha, \beta)^*: {_{\mV'} \omega\B\Enr}_{\mW'} \rightleftarrows {_{\mV} \omega\B\Enr}_{\mW}: (\phi, \psi)^*=\T(\phi,\psi).$

	
\end{proof}

Proposition \ref{transpo} and \ref{colttp} give the following corollary:

\begin{corollary}Let $\alpha: \mV^\ot \to \mV'^\ot, \beta: \mW^\ot \to \mW'^\ot$ be maps of small $\infty$-operads that admit right adjoints $\phi, \psi$ relative to $\Ass$, respectively.
	
\begin{enumerate}
\item The functor $(\alpha, \beta)^*:{_{\mV'} \L\Enr}_{\mW'} \to {_{\mV} \L\Enr}_{\mW}$ 
is canonically equivalent to $(\phi, \psi)_!.$

\item The functor $(\alpha, \beta)^*:{_{\mV'} \R\Enr}_{\mW'} \to {_{\mV} \R\Enr}_{\mW}$ 
is canonically equivalent to $(\phi, \psi)_!.$

\item The functor $(\alpha, \beta)^*:{_{\mV'} \B\Enr}_{\mW'} \to {_{\mV} \B\Enr}_{\mW}$ is canonically equivalent to $(\phi, \psi)_!.$
\end{enumerate}		

\end{corollary}


For the next proposition we use Notation \ref{psss}:


\begin{proposition}\label{bicay}
Let $ \mM^\circledast \to \mV^\ot \times \mW^\ot$ be a small $\infty$-category bienriched
in presentably monoidal $\infty$-categories and $\alpha: \mV^\ot \to \mV'^\ot, \beta: \mW^\ot \to \mW'^\ot$ left adjoint monoidal functors between presentably monoidal $\infty$-categories.
The enriched functor $\rho: \mM^\circledast \to (\alpha,\beta)_!(\mM)^\circledast \to (\alpha,\beta)_!(\mM)^\circledast_{\B\Enr}$ satisfies the following properties:

\begin{enumerate}
\item The enriched functor $\rho$ is essentially surjective.
\item The enriched functor $\rho$ induces for every $\X,\Y \in \mM$ a local equivalence
$$ (\alpha,\beta)_!(\Mor_\mM(\X,\Y)) \to \Mor_{(\alpha,\beta)_!(\mM)_{\B\Enr}}(\rho(\X),\rho(\Y))$$
for the localization $\widehat{\mP}(\Env(\mV)\times\Env(\mW)) \to \widehat{\mP}_\rep(\mV\times\mW)$ to presheaves representable in both variables.

\item The enriched functor $\rho$ induces for every small $\infty$-category
$\mN^\circledast \to \mV''^\ot \times \mW''^\ot$ bienriched
in presentably monoidal $\infty$-categories and any left adjoint monoidal functors $\alpha': \mV'^\ot \to \mV''^\ot, \beta': \mW'^\ot \to \mW''^\ot$ an equivalence
$$ \Enr\Fun_{\mV',\mW'}((\alpha,\beta)_!(\mM)_{\B\Enr},(\alpha',\beta')^*(\mN)) \to \Enr\Fun_{\mV,\mW}(\mM,(\alpha' \circ \alpha,\beta'\circ \beta)^*(\mN)).$$ 
\item The enriched functor $\rho$ exhibits
$\mP_{\mV',\mW'}((\alpha,\beta)_!(\mM)_{\B\Enr})$ as $\mV' \ot_{\mV} \mP_{\mV,\mW}(\mM) \ot_{\mW} \mW'.$

\item Let $ \mM'^\circledast \to \mV'^\ot \times \mW'^\ot$ be a small bienriched $\infty$-category and $\kappa: \mM^\circledast \to \mM'^\circledast$ an enriched functor lying over $\alpha,\beta$.
Then $\kappa$ is transfer of enrichment, i.e. induces an equivalence
$(\alpha,\beta)_!(\mM)_{\B\Enr}^\circledast \simeq \mM'^\circledast$ if and only if $\kappa$ is essentially surjective and one of the following conditions holds:
\begin{enumerate}
\item The enriched functor $\rho$ induces for every $\X,\Y \in \mM$ a local equivalence
$$ (\alpha,\beta)_!(\Mor_\mM(\X,\Y)) \to \Mor_{\mM'}(\kappa(\X),\kappa(\Y)).$$

\item The enriched functor $\kappa$ exhibits
$\mP_{\mV',\mW'}(\mM')$ as $\mV' \ot_{\mV} \mP_{\mV,\mW}(\mM) \ot_{\mW} \mW'.$
\end{enumerate}

\end{enumerate}
	
	
	

\end{proposition}

\begin{proof}(1), (2), (3) follows immediately from Theorem \ref{bica} and Proposition \ref{pseusor}. Here we use that if $\sigma$ is the strongly inaccessible cardinal corresponding to the small universe, then $\widehat{\Ind}_\sigma(\mV\ot \mW) \simeq \widehat{\mP}_\rep(\mV\times\mW).$ 
(5): Every transfer of enrichment satisfies the conditions of (5) by (1)-(4).
So we prove the converse: by (3) the enriched functor $\kappa$ factors as enriched functors $ \mM^\circledast \to  (\alpha,\beta)_!(\mM)_{\B\Enr}^\circledast \xrightarrow{\zeta} \mM'^\circledast,$ where $\zeta$ is essentially surjective since $\kappa$ is. 
The $\mV',\mW'$-enriched functor $\zeta$ is an embedding and so an equivalence
by the assumptions on $\kappa$ since the enriched Yoneda-embedding is an embedding.

It remains to prove (4).
We prove first that the enriched small colimits preserving embedding $$\mP_{\mV,\mW}(\mM)^\circledast \subset \widehat{\mP}_{\mV,\mW}(\mM)^\circledast $$ induces an $\widehat{\Ind}_\sigma(\mV), \widehat{\Ind}_\sigma(\mW) $-linear equivalence
$\theta: \widehat{\Ind}_\sigma(\mP_{\mV,\mW}(\mM))^\circledast \to \widehat{\mP}_{\mV,\mW}(\mM)^\circledast.$
The right adjoint of $\theta$ is conservative since $\widehat{\mP}_{\mV,\mW}(\mM)$ is generated under large colimits by bitensors
$\V \ot \X \ot \W$ for $\X \in \mM, \V \in \mV, \W \in \mW$, which all belong to
$\mP_{\mV,\mW}(\mM).$
Therefore $\theta$ is an equivalence if it is an embedding. This holds by \cite[Proposition 5.3.5.11. (1)]{lurie.HTT} since for every object $\X\in \mM,\V \in \mV, \W \in \mW$ the bitensor $\V \ot \X \ot \W$ is $\sigma$-compact in $\widehat{\mP}_{\mV,\mW}(\mM)$ by Theorem \ref{corok} and the full subcategory of $\widehat{\mP}_{\mV,\mW}(\mM)$ of $\sigma$-compact objects is closed under small colimits so that the full subcategory $ \mP_{\mV,\mW}(\mM)$ of $ \widehat{\mP}_{\mV,\mW}(\mM) $ consists of $\sigma$-compact objects.

Consequently, it is enough to see that the functor $$\widehat{\mP}_{\mV,\mW}(\mM) \simeq \widehat{\Ind}_\sigma(\mP_{\mV,\mW}(\mM)) \to \widehat{\mP}_{\mV',\mW'}(\mN) \simeq  \widehat{\Ind}_\sigma(\mP_{\mV',\mW'}((\alpha,\beta)_!(\mM)_{\B\Enr}))$$
exhibits
$\widehat{\mP}_{\mV',\mW'}((\alpha,\beta)_!(\mM)_{\B\Enr})$ as $$\widehat{\Ind}_\sigma(\mV') \ot_{\widehat{\Ind}_\sigma(\mV)} \widehat{\mP}_{\mV,\mW}(\mM) \ot_{\widehat{\Ind}_\sigma(\mW)} \widehat{\Ind}_\sigma(\mW') \simeq \widehat{\Ind}_\sigma(\mV' \ot_{\mV} \mP_{\mV,\mW}(\mM) \ot_{\mW} \mW').$$
By Proposition \ref{pseuso} (3) there is a canonical equivalence
$$\widehat{\mP}_{\mV,\mW}(\mM) \simeq \widehat{\Ind}_\sigma(\mV) \ot_{\widehat{\mP}\Env(\mV)} \widehat{\mP}\B\Env(\mM) \ot_{\widehat{\mP}\Env(\mW)} \widehat{\Ind}_\sigma(\mW).$$
By Proposition \ref{pseusor} (4) there is a canonical equivalence 
$$\widehat{\mP}_{\mV',\mW'}((\alpha,\beta)_!(\mM)_{\B\Enr})\simeq \widehat{\Ind}_\sigma(\mV') \ot_{\widehat{\mP}\Env(\mV')} \widehat{\mP}\B\Env((\alpha,\beta)_!(\mM)) \ot_{\widehat{\mP}\Env(\mW')} \widehat{\Ind}_\sigma(\mW').$$
By Theorem \ref{bica} there is a canonical equivalence  
$$\widehat{\mP}\B\Env((\alpha,\beta)_!(\mM)) \simeq \widehat{\mP}\Env(\mV') \ot_{\widehat{\mP}\Env(\mV)} \widehat{\mP}\B\Env(\mM) \ot_{\widehat{\mP}\Env(\mW)} \widehat{\mP}\Env(\mW').$$

We obtain a canonical equivalence
$$ \widehat{\Ind}_\sigma(\mV') \ot_{\widehat{\Ind}_\sigma(\mV)} \widehat{\mP}_{\mV,\mW}(\mM) \ot_{\widehat{\Ind}_\sigma(\mW)} \widehat{\Ind}_\sigma(\mW') \simeq \widehat{\Ind}_\sigma(\mV') \ot_{\widehat{\mP}\Env(\mV)} \widehat{\mP}\B\Env(\mM) \ot_{\widehat{\mP}\Env(\mW)} \widehat{\Ind}_\sigma(\mW')$$
$$ \simeq \widehat{\Ind}_\sigma(\mV') \ot_{\widehat{\mP}\Env(\mV')} \widehat{\mP}\B\Env((\alpha,\beta)_!(\mM)) \ot_{\widehat{\mP}\Env(\mW')} \widehat{\Ind}_\sigma(\mW') \simeq \widehat{\mP}_{\mV',\mW'}((\alpha,\beta)_!(\mM)_{\B\Enr}).$$
\end{proof}

\begin{corollary}Let $\alpha: \mV^\ot \to \mV'^\ot, \beta: \mW^\ot \to \mW'^\ot$ be left adjoint monoidal functors between presentably monoidal $\infty$-categories.
The functor $(\alpha, \beta)^*:{_{\mV'} \B\Enr}_{\mW'} \to {_{\mV} \B\Enr}_{\mW}$ admits a left adjoint $ {_{\mV} \B\Enr}_{\mW} \to {_{\mV'} \B\Enr}_{\mW'}$	
that sends $ \mM^\circledast \to \mV^\ot \times \mW^\ot$ to $ (\alpha,\beta)_!(\mM)^\circledast_{\B\Enr} \to \mV'^\ot \times \mW'^\ot $, where the enriched functor $\rho: \mM^\circledast \to (\alpha,\beta)_!(\mM)^\circledast_{\B\Enr}$ lying over $\alpha, \beta$ corresponds to the unit component
$\mM^\circledast \to (\alpha,\beta)^*((\alpha,\beta)_!(\mM)_{\B\Enr})^\circledast.$

\end{corollary}

Theorem \ref{bica} implies the following corollary:

\begin{corollary}\label{embeto}
Let $\mW^\ot \to \Ass$ be a small $\infty$-operad and $\phi: \mV^\ot \to \mV'^\ot$
an embeddings of small $\infty$-operads.
The induced functor $\phi_!: {_\mV \L\Enr_\mW} \to {_{\mV'} \L\Enr_{\mW}}$ is fully faithful and the essential image precisely consists of the bienriched $\infty$-categories whose left morphism objects lie in the essential image of $\phi.$	

\end{corollary}

\begin{proof}

By Proposition \ref{bica} there is an induced adjunction
$\phi_!: {_\mV \omega\B\Enr_\mW} \rightleftarrows {_{\mV'} \omega\B\Enr_{\mW}}:\phi^* $.
By Proposition \ref{bica} (4) the left adjoint sends ${_\mV\B\Enr_\mW}$ to the full subcategory $\Xi $ of ${_{\mV'} \omega\B\Enr_{\mW}}$ spanned by the bienriched $\infty$-categories whose left multi-morphism objects lie in the essential image of $\phi.$ We first prove that the right adjoint $\phi^*$ sends $\Xi$ to ${_\mV\B\Enr_\mW}$. Let $\mM^\circledast \to \mV'^\ot \times \mW^\ot$ be a bienriched $\infty$-category that belongs to $\Xi$.
Then for every $\X,\Y \in \mM, \W_1,...,\W_\m \in \mW$
for $\m \geq0$ the left multi-morphism object $\L\Mul\Mor_{\mM}(\X,\W_1,...,\W_\m;\Y)$ belongs to the essential image of $\phi$ and
so corresponds to an object of $\mV$ denoted by $\phi^{-1}(\L\Mul\Mor_{\mM}(\X,\W_1,...,\W_\m;\Y)).$
The universal morphism $$\L\Mul\Mor_{\mM}(\X,\W_1,...,\W_\m;\Y), \X, \W_1,...,\W_\m \to \Y$$ in $\mM^\circledast$ corresponds to a morphism
$\phi^{-1}(\L\Mul\Mor_{\mM}(\X,\W_1,...,\W_\m;\Y)), \X, \W_1,...,\W_\m \to \Y$ in $\phi^*(\mM)^\circledast$. For every $\V_1,...,\V_\n \in \mV$ for $\n \geq0$ the induced map 
$$\Mul_\mV(\V_1,...,\V_\n; \phi^{-1}(\L\Mul\Mor_{\mM}(\X,\W_1,...,\W_\m;\Y))) \to \Mul_{\phi^*(\mM)}(\V_1,...,\V_\n,\X,\W_1,...,\W_\m;\Y)$$
identifies with the map 
$$\Mul_\mV(\V_1,...,\V_\n; \phi^{-1}(\L\Mul\Mor_{\mM}(\X,\W_1,...,\W_\m;\Y))) \simeq$$$$\Mul_{\mV'}(\phi(\V_1),...,\phi(\V_\n); \L\Mul\Mor_{\mM}(\X,\W_1,...,\W_\m;\Y)) \simeq \Mul_\mM(\phi(\V_1),...,\phi(\V_\n),\X,\W_1,...,\W_\m;\Y)$$$$ \simeq \Mul_{\phi^*(\mM)}(\V_1,...,\V_\n,\X,\W_1,...,\W_\m;\Y).$$

So the latter adjunction restricts to an adjunction
${_\mV\B\Enr}_\mW \rightleftarrows \Xi$. To see that the adjunction is an equivalence we prove that the right adjoint is conservative and the unit of the adjunction is an equivalence. A $\mV', \mW$-enriched functor $\F: \mM^\circledast \to \mN^\circledast$
that is a morphism of $\Xi$ is an equivalence if it induces an equivalence on underlying $\infty$-categories and for every $\X, \Y \in \mM$ and $\W_1,...,\W_\m \in \mW$ for $\m \geq 0$ the induced morphism $\L\Mul\Mor_\mM(\X,\W_1,...,\W_\m; \Y) \to \L\Mul\Mor_\mN(\F(\X),\W_1,...,\W_\m; \F(\Y))$ in $\mV'$ is an equivalence. The latter morphism belongs to the essential image of $\phi$ and so
is an equivalence if for every $\V \in \mV$ the induced map
$$\mW(\phi(\V), \L\Mul\Mor_\mM(\X,\W_1,...,\W_\m; \Y)) \to \mW(\phi(\V), \L\Mul\Mor_\mN(\F(\X),\W_1,...,\W_\m; \F(\Y))) $$ is an equivalence. The latter map identifies with the map
$$\Mul_\mM(\phi(\V),\X,\W_1,...,\W_\m;\Y) \to \Mul_\mN(\phi(\V), \F(\X),\W_1,...,\W_\m; \F(\Y)),$$ which identifies with the map
$\Mul_{\phi^*(\mM)}(\V,\X,\W_1,...,\W_\m;\Y) \to \Mul_{\phi^*(\mN)}(\V, \F(\X),\W_1,...,\W_\m; \F(\Y)) $. This proves conservativity of $\phi^*.$

It remains to see that the unit is an equivalence.
For every bienriched $\infty$-category $\mM^\circledast \to \mV^\ot \times \mW^\ot$
the unit $\alpha: \mM^\circledast \to \phi^*(\phi_!(\mM))^\circledast $ induces an essentially surjective functor $\mM \to \phi_!(\mM) $ on underlying $\infty$-categories by Proposition \ref{bica}. So it remains to see that the unit induces an equivalence on left morphism objects. 
By Proposition \ref{bica} for every $ \W_1,...,\W_\m \in \mW$ for $\m \geq0$ and $\X, \Y \in \mM$ the induced morphism
$$ \phi(\L\Mul\Mor_{\mM}(\X,\W_1,...,\W_\m; \Y)) \to \L\Mul\Mor_{\phi_!(\mM)}(\alpha(\X),\W_1,...,\W_\m; \alpha(\Y))$$
in $\mV'$ is an equivalence.
For every $\V$ the unit $\alpha $ induces a map $$ \mV(\V, \L\Mul\Mor_\mM(\X,\W_1,...,\W_\m;\Y)) \to \mV(\V, \L\Mul\Mor_{\phi^*(\phi_!(\mM))}(\alpha(\X),\W_1,...,\W_\m; \alpha(\Y))) \simeq$$$$  \mV'(\phi(\V), \L\Mul\Mor_{\phi_!(\mM)}(\alpha(\X),\W_1,...,\W_\m; \alpha(\Y))) \simeq \mV'(\phi(\V), \phi(\L\Mul\Mor_{\mM}(\X,\W_1,...,\W_\m; \Y))),$$
which is the canonical map induced by $\phi$ and so is an equivalence by assumption.

\end{proof}

\subsection{Folding bienrichment into left enrichment}

It is well-known that for any monoidal $\infty$-category $\mO^\ot \to \Ass$ and associative algebras $\A,\B $ in $\mO$ there is an equivalence between the $\infty$-category of $\A,\B$-bimodules in $\mO$ and the $\infty$-category of left $\A \ot \B^\rev$-modules in $\mO$ (see for example \cite[Proposition 3.6.7., Corollary 3.6.8.]{HINICH2020107129} or \cite[Theorem 4.3.2.7.]{lurie.higheralgebra}).
Specializing to $\mO=\Cat_\infty$ one obtains for any small monoidal $\infty$-categories $\mV^\ot \to \Ass, \mW^\ot\to \Ass$ an equivalence 
\begin{equation}\label{ujjj}
_\mV\BMod_\mW\simeq{_{\mV \times \mW^\rev}\LMod_\emptyset}
\end{equation} between the $\infty$-category of small $\infty$-categories bitensored over $\mV,\mW$ and the $\infty$-category of small $\infty$-categories left tensored over $\mV \times \mW^\rev$ (see also Lemma \ref{basi}).
In this section we extend equivalence (\ref{ujjj}) to enriched and pseudo-enriched $\infty$-categories (Theorem \ref{biii}, Corollary \ref{kilio}, Corollary \ref{abcd}).
Moreover we prove that under this equivalence the respective $\infty$-categories of enriched presheaves correspond (Proposition \ref{bienve}).


We start with transforming bipseudo-enrichment into left pseudo-enrichment
(Theorem \ref{biii}).
For that we first transform bitensored $\infty$-categories into left tensored $\infty$-categories. We give a proof of that in our setting for the reader's convenience (Lemma \ref{basi}).

\begin{notation}
Let $\tau:=(\id, (-)^\op) : \Ass \xrightarrow{} \Ass \times \Ass$
and $\kappa$ the functor $\Op_{\infty} \times \Op_{\infty} \to \Op_{\infty},
(\mV^\ot \to \Ass, \mW^\ot \to \Ass) \mapsto \mV^\ot \times_\Ass (\mW^\ot)^\rev.$
Taking pullback along $\tau$ defines a commutative square
\begin{equation*} 
\begin{xy}
\xymatrix{
\omega\B\Enr \ar[d] \ar[rr]
&&\omega\B\Enr_\emptyset  \ar[d] 
\\ \Op_{\infty} \times \Op_{\infty}
\ar[rr]^\kappa && \Op_{\infty}
}
\end{xy} 
\end{equation*}
corresponding to a functor
$$\theta: \omega\B\Enr \to \kappa^*(\omega\B\Enr_\emptyset)$$ 
over $\Op_{\infty} \times \Op_{\infty}.$

\end{notation}

\begin{remark}\label{embrace}
The functor $\tau$ sends the map $\{0\} \subset [\n]$ in $\Delta$ to the map $(\{0\} \subset [\n], \{\n\} \subset [\n])$ in $\Delta \times \Delta.$
So for every weakly bienriched $\infty$-category $\mM^\circledast \to \mV^\ot \times \mW^\ot$ and $\V_1,...,\V_\n \in \mV, \W_1,...,\W_\n \in \mW, \X,\Y \in \mM$ there is a canonical equivalence
$$ \Mul_{\theta(\mM)}((\V_1,\W_1), ..., (\V_\n,\W_\n), \X; \Y) \simeq \Mul_{\mM}(\V_1,...,\V_\n, \X, \W_1,...,\W_\n; \Y) .$$

So if $\mM^\circledast \to \mV^\ot \times \mW^\ot $ is a bitensored $\infty$-category, 
$\theta(\mM)^\circledast \to \mV^\ot \times_\Ass (\mW^\ot)^\rev$ provides a left $\mV \times \mW^\rev$-action on $\mM$ that sends $(\V,\W) \in \mV \times \mW, \X \in \mM$ to $\V \ot \X \ot \W.$

\end{remark}


\begin{lemma}\label{basi}
The functor $\theta: \omega\B\Enr \to \kappa^*(\omega\B\Enr_\emptyset)$ 	
restricts to an equivalence
$$ \BMod \simeq \kappa^*(\LMod).$$	

\end{lemma}

\begin{proof}
By Remark \ref{embrace} the functor $\theta: \omega\B\Enr \to \kappa^*(\omega\B\Enr_\emptyset)$ restricts to a functor $ \BMod \simeq \kappa^*(\LMod),$
which is a map of cartesian fibrations over $\Mon \times \Mon$
by the description of cartesian morphisms.
Consequently, it is enough to see that for any monoidal $\infty$-categories $\mV^\ot \to \Ass,\mW^\ot \to \Ass $ the induced functor 
$\theta :{_\mV\omega\B\Enr_\mW} \to {_{\mV \times \mW^\rev}\omega\B\Enr_\emptyset}$
restricts to an equivalence 
$ {_\mV\BMod_\mW} \simeq {_{\mV \times \mW^\rev}\LMod_\emptyset}.$	
By Proposition \ref{Line} the forgetful functors $ {_{\mV \times \mW^\rev}\LMod_\emptyset} \to \Cat_\infty,  {_\mV\BMod_\mW} \to \Cat_\infty$ admit left adjoints sending an $\infty$-category $\K$ to $\mV^\circledast \times_\Ass (\mW^\circledast)^\rev \times \K \to \Ass$, $\mV^\circledast \times \K \times \mW^\circledast \to \Ass \times \Ass$, respectively, and preserves small colimits by Lemma \ref{colas}.
Thus these functors are monadic by \cite[Theorem 4.7.3.5]{lurie.higheralgebra}.
So by \cite[Corollary 4.7.3.16.]{lurie.higheralgebra} 
the claim follows from the fact that $\theta$ preserves the left adjoints, i.e. the canonical functor
$\mV^\circledast \times_\Ass (\mW^\circledast)^\rev \times \K \to \theta(\mV^\circledast  \times \K \times \mW^\circledast) $ is an equivalence.
\end{proof}

\begin{notation}\label{noii}
	
Let $\mM^\circledast \to \mV^\ot \times \mW^\ot$ be a bipseudo-enriched $\infty$-category.
Let $$\widetilde{\B\Env}(\mM)^\circledast \subset \mV^\ot \times_{\mP(\mV)^\ot} \mP\B\Env_{\B\Pseu}(\mM)^\circledast \times_{\mP(\mW)^\ot} \mW^\ot$$ be the full subcategory with $\mV,\mW$-biaction
generated by $\mM$, i.e. the full weakly bitensored subcategory spanned by the objects $\V \ot \X \ot \W$ for $\V \in \mV, \W \in \mW$, $\X \in \mM.$
	
	
\end{notation}

\begin{lemma}\label{tenno}
Let $\mM^\circledast \to \mV^\ot \times \mW^\ot$ be a bipseudo-enriched $\infty$-category. The $\mV, \mW$-linear embedding $\widetilde{\B\Env}(\mM)^\circledast \subset \mV^\ot \times_{\mP(\mV)^\ot} \mP\B\Env_{\B\Pseu}(\mM)^\circledast \times_{\mP(\mW)^\ot} \mW^\ot$
induces a $\mP(\mV), \mP(\mW)$-linear equivalence $$ \mP(\widetilde{\B\Env}(\mM))^\circledast \simeq \mP\B\Env_{\B\Pseu}(\mM)^\circledast.$$		
	
\end{lemma}

\begin{proof}
By Lemma \ref{compp} for every $\V \in \mV, \W \in \mW$, $\X \in \mM$ 
the object $\V \ot \X \ot \W$ is $\emptyset$-compact in $\mP\B\Env_{\B\Pseu}(\mM).$ So by \cite[Corollary 5.1.6.11.]{lurie.HTT} the result follows from the fact that  $\mP\B\Env_{\B\Pseu}(\mM)$ is generated under small colimits by the objects $\V \ot \X \ot \W$ for $\V \in \mV, \W \in \mW$, $\X \in \mM.$
\end{proof}

\begin{proposition}\label{mcha}
Let $\mM^\circledast \to \mV^\ot \times \mW^\ot$ be a bipseudo-enriched $\infty$-category and $\mN^\circledast \to \mV^\ot \times \mW^\ot$ a bitensored $\infty$-category.
The lax $\mV,\mW$-linear embedding $\mM^\circledast \subset
\widetilde{\B\Env}(\mM)^\circledast$ induces an equivalence $$\psi_\mN: \LinFun_{\mV,\mW}(\widetilde{\B\Env}(\mM),\mN) \to \Enr\Fun_{\mV,\mW}(\mM,\mN).$$

\end{proposition}

\begin{proof}
	
By Proposition \ref{presta} the $\infty$-category $\mN^\circledast \to \mV^\ot\times \mW^\ot$ bitensored over $\mV,\mW$
embeds into an $\infty$-category $\mP(\mN)^\circledast \to \mP(\mV)^\ot \times \mP(\mW)^\ot$ bitensored over $\mP(\mV), \mP(\mW)$ compatible with small colimits, and $\psi_\mN$ is the pullback of 
$\psi_{\mP(\mN)}.$ So it is enough to check that $\psi_{\mP(\mN)}$ is an equivalence.
Consider the commutative triangle:
\begin{equation*}
\begin{xy}
\xymatrix{
\LinFun^\L_{\mP(\mV), \mP(\mW)}(\mP\B\Env_{\B\Pseu}(\mM),\mP(\mN)) \ar[dd] \ar[rd]^\rho
\\ 
& \Enr\Fun_{\mV, \mW}(\mM,\mP(\mN))
\\
\LinFun_{\mV, \mW}(\widetilde{\B\Env}(\mM)^\circledast,\mP(\mN)) \ar[ru]^{\psi_{\mP(\mN)}}
}
\end{xy} 
\end{equation*} 
	
By Proposition \ref{pseuso} the functor $\rho$ is an equivalence.
By Lemma \ref{tenno} the vertical functor is an equivalence.
	
\end{proof}

\begin{corollary}
Let $\mM^\circledast \to \mV^\ot \times \mW^\ot$ be a weakly
bitensored $\infty$-category and $\bar{\mM}^\circledast\to \Env(\mV)^\ot \times \Env(\mW)^\ot$ the unique extension to a bipseudo-enriched $\infty$-category.
The embedding $\bar{\mM}^\circledast \subset \B\Env(\mM)^\circledast$
induces an $\Env(\mV), \Env(\mW)$-linear equivalence
$$\widetilde{\B\Env}(\bar{\mM})^\circledast \simeq \B\Env(\mM)^\circledast.$$
	
\end{corollary}

\begin{proof}
For every bitensored $\infty$-category $\mN^\circledast \to \Env(\mV)^\ot \times \Env(\mW)^\ot$ by Proposition \ref{mcha} the functor
$$ \LinFun_{\Env(\mV), \Env(\mW)}(\widetilde{\B\Env}(\bar{\mM}),\mN)
\to \Enr\Fun_{\Env(\mV), \Env(\mW)}(\bar{\mM},\mN) \simeq \Enr\Fun_{\mV, \mW}(\mM,\mN) $$ is an equivalence, where the latter equivalence is by Corollary \ref{coronn}. 
	
\end{proof}

\begin{theorem}\label{biii}

The functor $\theta:\B\P\Enr \to \kappa^*(\L\P\Enr_\emptyset)$ is an equivalence.

\end{theorem}

\begin{proof}

The functor $\theta:\B\P\Enr \to \kappa^*(\L\P\Enr_\emptyset)$ is a map of cocartesian fibrations over $\Op_\infty \times \Op_\infty$ using Proposition \ref{bica} (4).
So it is enough to see that for any $\infty$-operads $\mV^\ot \to \Ass,\mW^\ot \to \Ass $ the induced functor 
$\theta: {_\mV\B\P\Enr_{\mW}} \to {_{\mV \times \mW^\rev}\L\P\Enr_\emptyset}$ on the fiber is an equivalence. We construct an inverse $\gamma$ of $\theta$.
By Proposition \ref{mcha} and Corollary \ref{coronn} the inclusion $ _{\mV \times \mW^\rev} \LMod_\emptyset \subset {_{\mV \times \mW^\rev}\L\P\Enr_\emptyset} $ admits a left adjoint $\L$.
By Lemma \ref{basi} the functor $\theta : {_\mV\B\P\Enr_{\mW}} \to {_{\mV \times \mW^\rev}\L\P\Enr_\emptyset}$ restricts to an equivalence $ {_\mV \BMod_{ \mW}} \simeq {_{\mV \times \mW^\rev}\LMod},$
whose inverse we call $\Lambda$.
For every bipseudo-enriched $\infty$-category $\mM^\circledast \to \mV^\ot \times (\mW^\rev)^\ot$ let $\gamma(\mM)^\circledast \subset \Lambda(\L(\mM))^\circledast$
be the full weakly bienriched subcategory spanned by $\mM \subset \L(\mM).$
Since $\Lambda(\L(\mM))^\circledast \to \mV^\ot \times \mW^\ot$ is a bitensored $\infty$-category, $\gamma(\mM)^\circledast \to \mV^\ot \times \mW^\ot$ exhibits
$\mM$ as bipseudo-enriched in $\mV,\mW.$
Thus the functor $\Lambda \circ \L: {_{\mV \times \mW^\rev}\L\P\Enr_\emptyset} \xrightarrow{} {_{\mV \times \mW^\rev}\LMod_\emptyset} \xrightarrow{} {_\mV \BMod}_{\mW}$ induces a functor $\gamma: {_{\mV \times \mW^\rev}\L\P\Enr_\emptyset} \to {_\mV\B\P\Enr}_{\mW} $. We will prove that $\gamma$ is inverse to $\theta.$

The tautological natural transformation
$ \gamma \subset \Lambda \circ \L$ of functors ${_{\mV \times \mW^\rev} \L\P\Enr_\emptyset}
\to {_\mV\B\P\Enr}_{\mW}$
gives rise to a natural transformation
$\theta \circ \gamma \subset \theta \circ \Lambda \circ \L \simeq \L$ of functors ${_{\mV \times \mW^\rev}\L\P\Enr_\emptyset} \to {_{\mV \times \mW^\rev} \L\P\Enr_\emptyset}, $
which induces an equivalence 
$\alpha: \theta \circ \gamma \simeq \id$ of functors $ {_{\mV \times \mW^\rev}\L\P\Enr_\emptyset} \to {_{\mV \times \mW^\rev} \L\P\Enr_\emptyset}.$
By Proposition \ref{mcha} the inclusion $ _\mV\BMod_{\mW} \subset {_\mV\B\P\Enr}_{\mW} $ admits a left adjoint $\L'$.
The composite
$ \L \circ \theta \to \L \circ \theta \circ \L' \simeq \theta \circ \L'
$ (induced by the unit $\id \to \L'$) 
of natural transformations of functors ${_\mV\B\P\Enr}_{\mW} \to {_{\mV \times \mW^\rev}\LMod_\emptyset}$ gives rise to a natural transformation
$ \Lambda \circ \L \circ \theta \to \Lambda \circ \theta \circ \L' \simeq \L'$
of functors $_\mV\B\P\Enr_{\mW} \to {_\mV\BMod}_{\mW}$.
Let $\beta$ be the composition $\gamma \circ \theta \subset \Lambda \circ \L \circ \theta \to \L'$ of natural transformations of functors $_\mV\B\P\Enr_{\mW} \to {_\mV\B\P\Enr}_{\mW}$.
There is a commutative triangle:
\begin{equation*}
\begin{xy}
\xymatrix{
\theta\circ \gamma \circ \theta \ar[rr]^{\theta\circ \beta} \ar[rd] ^{\alpha \circ \theta}
&&\theta\circ  \L'
\\ 
& \theta. \ar[ru].
}
\end{xy} 
\end{equation*} 
Since $\alpha$ is an equivalence and the unit $\id \to \L'$ is component-wise an embedding, $\theta \circ \beta$ and so $\beta$ are component-wise embeddings.
Hence $\beta: \gamma \circ \theta \to \L'$ induces an equivalence $\gamma \circ \theta \simeq \id$ of functors $_\mV\B\P\Enr_{\mW} \to {_\mV\B\P\Enr}_{\mW}$.

\end{proof}	

\begin{corollary}
Let $\mV^\ot \to \Ass, \mW^\ot\to \Ass$ be small monoidal $\infty$-categories.
There is a canonical equivalence 
\begin{equation*}\label{ujjj}
_\mV\P\B\Enr_\mW\simeq {_{\mV \times \mW^\rev}\P\L\Enr_\emptyset}.
\end{equation*}
	
\end{corollary}

\begin{corollary}\label{bienr}
Let $\mV^\ot \to \Ass, \mW^\ot \to \Ass$ be small $\infty$-operads.
There is a canonical equivalence
$$ _\mV\omega\B\Enr_{\mW} \simeq {_{\mP\Env(\mV)\otimes \mP\Env(\mW^\rev)}\L\Enr_\emptyset}.$$
	
\end{corollary}

\begin{proof}
	
By Theorem \ref{biii} there is a canonical equivalence
$$ {_{\Env(\mV)}\B\P\Enr}_{\Env(\mW)} \simeq {_{\Env(\mV) \times \Env(\mW)^\rev}\L\P\Enr_\emptyset}.$$

By Corollary \ref{coronn} there are canonical equivalences
$$ _\mV\omega\B\Enr_{\mW} \simeq {_{\Env(\mV)}\B\P\Enr}_{\Env(\mW)}, \ _{\mP(\Env(\mV)\times \Env(\mW)^\rev)}\L\Enr_\emptyset \simeq {_{\Env(\mV) \times \Env(\mW)^\rev}\L\P\Enr_\emptyset}.$$
	
\end{proof}

\begin{corollary}

For every bipseudo-enriched $\infty$-category $\mM^\circledast \to \mV^\ot \times \mW^\ot$ and weakly bienriched $\infty$-category $ \mN^\circledast \to \mU^\ot \times \mQ^\ot$ the following induced commutative square is a pullback square:
\begin{equation}\label{ewjk}
\begin{xy}
\xymatrix{
\Enr\Fun(\mM,\mN) \ar[d]
\ar[rrr]
&&& \Enr\Fun(\theta(\mM), \theta(\mN)) \ar[d] 
\\
\Alg_\mV(\mU) \times \Alg_\mW(\mQ) \ar[rrr]  &&& \Alg_{\mV \times \mW^\rev}(\mU \times \mQ^\rev).}
\end{xy}\end{equation}

\end{corollary}

\begin{proof}
The commutative square (\ref{ewjk}) is the pullback of the commutative square (\ref{ewjk}), where we replace $ \mN^\circledast \to \mU^\ot \times \mQ^\ot$ by the enveloping bitensored $\infty$-category $\mB\Env(\mN)^\circledast \to \Env(\mU)^\ot \times \Env(\mQ)^\ot$.
Consequently, we can assume that $ \mN^\circledast \to \mU^\ot \times \mQ^\ot$
is a bipseudo-enriched $\infty$-category.
For any small $\infty$-category $\K$ applying the functor $\Cat_\infty(\K,-):
\Cat_\infty \to \mS$ to square (\ref{ewjk}) gives 
the induced commutative square 
\begin{equation*}
\begin{xy}
\xymatrix{
\omega\B\Enr(\mM,\mN^\K) \ar[d]
\ar[rrr]
&&& \omega\B\Enr_\emptyset(\theta(\mM), \theta(\mN^\K)) \ar[d] 
\\
\Op_{\infty}(\mV, \mU^\K) \times \Op_\infty(\mW, \mQ^\K) \ar[rrr]  &&& \Op_\infty(\mV \times \mW^\rev, \mU^\K \times (\mQ^\K)^\rev).}
\end{xy}\end{equation*}
So the result follows from Theorem \ref{biii} and that $ (\mM^\circledast)^\K \to (\mV^\ot)^\K \times (\mW^\ot)^\K$ is a bipseudo-enriched $\infty$-category.

\end{proof}	

\begin{remark}
The functor $\theta: \omega\B\Enr \to \kappa^*(\omega\B\Enr_\emptyset)$ is not an equivalence since it is not conservative.
There is a canonical commutative square:
$$\begin{xy}
\xymatrix{
{_{\Env(\mV)}\B\P\Enr}_{\Env(\mW)} \ar[d]^\simeq
\ar[rrr]_\simeq^{\theta}
&&& {_{\Env(\mV) \times \Env(\mW)^\rev}} \L\P\Enr_\emptyset\ar[d] 
\\
_\mV\omega\B\Enr_{\mW} \ar[rrr]^{\theta}  &&& {_{\mV \times \mW^\rev}}\omega\B\Enr_\emptyset \simeq {_{\Env(\mV \times \mW^\rev)}}\L\P\Enr_\emptyset,}
\end{xy}$$
which identifies the functor $\theta: {_\mV\omega\B\Enr}_{\mW} \to {_{\mV \times \mW^\rev}\omega\B\Enr_\emptyset}$
with the right vertical functor $${_{\Env(\mV) \times \Env(\mW)^\rev}\L\P\Enr_\emptyset} \to {_{\Env(\mV \times \mW^\rev)}{\L\P\Enr_\emptyset}}$$
taking pullback along the monoidal functor $ \Env(\mV \times \mW^\rev)^\ot \to \Env(\mV)^\ot \times_\Ass \Env(\mW^\rev)^\ot.$
\end{remark}


We apply Theorem \ref{biii} to transform bi-enrichment into left enrichment (Corollary \ref{kilio}).

\begin{notation}
Let $\mV^\ot \to \Ass, \mW^\ot \to \Ass $ be small $\infty$-operads. 

\begin{enumerate}
\item Let $\Fun(\Env(\mV)^\op, \mW)^\ot \subset \mP(\Env(\mV) \times \Env(\mW))^\ot$ be the full suboperad whose colors are the objects of
$\Fun(\Env(\mV)^\op, \mW) \subset \Fun(\Env(\mV)^\op, \mP\Env(\mW)) \simeq \mP(\Env(\mV) \times \Env(\mW)).$

\vspace{1mm}

\item Let $\langle \mV, \mW \rangle^\ot \subset \mP(\Env(\mV) \times \Env(\mW))^\ot $ be the full suboperad whose colors are the presheaves on $\Env(\mV) \times \Env(\mW)$
representable in both variables such that for the corresponding adjunction
$\Env(\mV) \rightleftarrows \Env(\mW)^\op$ the left adjoint lands in $\mW^\op$
and the right adjoint lands in $\mV.$
\end{enumerate}	
\end{notation}

\begin{corollary}\label{kilio}
Let $\mV^\ot \to \Ass, \mW^\ot \to \Ass $ be small $\infty$-operads. 
There are canonical equivalences preserving the underlying $\infty$-category and morphism objects:
\begin{enumerate}
\item $${_\mV\L\Enr}_{\mW} \simeq {_{\Fun(\Env(\mW^\rev)^\op, \mV)}\L\Enr_\emptyset},$$
\item $${_\mV\R\Enr}_{\mW} \simeq {_{\Fun(\Env(\mV)^\op, \mW^\rev)}\L\Enr_\emptyset},$$
\item $${_\mV\B\Enr}_{\mW} \simeq {_{_{\langle \mV, \mW^\rev \rangle}}\L\Enr_\emptyset}.$$

\end{enumerate}
	
\end{corollary}

\begin{proof}
(1): Under the equivalence $$ {_\mV\omega\B\Enr}_{\mW} \simeq {_{\mP(\Env(\mV)\times \Env(\mW)^\rev)}}\L\Enr_\emptyset $$ of Corollary \ref{bienr} the full subcategory ${_\mV\L\Enr_{\mW}}$
corresponds to the full subcategory $$\Xi \subset{_{\mP(\Env(\mV)\times \Env(\mW)^\rev)}} \L\Enr_\emptyset$$ spanned by the left enriched $\infty$-categories $\mM^\circledast \to \mP(\Env(\mV)\times \Env(\mW)^\rev)^\ot$ such that for every $\X, \Y \in \mM$
the left morphism object $\L\Mor_\mM(\X,\Y) \in \mP(\Env(\mV)\times \Env(\mW))$ belongs to
$$\Fun(\Env(\mW)^\op, \mV) \subset \Fun(\Env(\mW)^\op, \mP\Env(\mV)) \simeq \mP(\Env(\mV)\times \Env(\mW)).$$
The induced functor 
$_{\mP(\Env(\mV)\times \Env(\mW)^\rev)}\omega\B\Enr_\emptyset \to {_{\Fun(\Env(\mW^\rev)^\op, \mV)}\omega\B\Enr_\emptyset} $
restricts to a functor $$\Xi \to {_{\Fun(\Env(\mW^\rev)^\op, \mV)}\L\Enr_\emptyset},$$
which is an equivalence by Corollary \ref{embeto}.
(2) and (3) are similar to (1).	
	
\end{proof}
Corollary \ref{kilio} gives the following theorem as a corollary:

\begin{theorem}\label{Enrbi}
Let $\mV^\ot \to \Ass, \mW^\ot \to \Ass $ be presentably monoidal $\infty$-categories. There is a canonical equivalence preserving the underlying $\infty$-category and morphism objects:
$${_\mV\B\Enr}_{\mW} \simeq {_{\mV\ot \mW^\rev}\L\Enr}_\emptyset.$$
	
\end{theorem}


Next we apply Theorem \ref{biii} to tranform bimodules into left modules
(Corollary \ref{juha}).

\begin{notation}\label{notat}
Let $\mM^\circledast \to\mV^\ot$ be a weakly left enriched $\infty$-category, $\mO^\circledast \to\mW^\ot$ a weakly right enriched $\infty$-category,
$\mN^\circledast \to\mV^\ot \times \mW^\ot$ a weakly bienriched $\infty$-category and $\A,\B$ associative algebras in $\mV,\mW$, respectively.
Let $${_\A \mM^\circledast} \to \Ass, \mO^\circledast_\B\to \Ass, {_\A \mN^\circledast_\B}\to \Ass \times \Ass$$ be the pullbacks of $\mM^\circledast \to\mV^\ot$ along $\A: \Ass \to \mV^\ot$, of $ \mO^\circledast \to \mW^\ot$ along $\B: \Ass \to \mW^\ot$ and of $\mN^\circledast \to\mV^\ot \times \mW^\ot$ along $\A \times \B: \Ass \times \Ass \to \mV^\ot\times \mW^\ot,$ respectively. 
Let $$_{\A}\LMod(\mM):=  \Enr\Fun_{[0]}([0], {_\A \mM}), {\LMod(\mO)_\B}:=  \Enr\Fun_{[0]}([0], {\mO_\B}), $$$$ _{\A}\BMod_{\B}(\mN):= \Enr\Fun_{[0],[0]}([0], {_\A \mN_\B}).$$ 
\end{notation}

We apply Notation \ref{notat} to $\mV^\circledast \to \mV^\ot \times \mV^\ot$ and $\emptyset^\ot \times_{\mV^\ot} \mV^\circledast \to \mV^\ot$ for any $\infty$-operad $\mV^\ot \to \Ass.$

\begin{remark}\label{bimo}
Let $\mN^\circledast \to\mV^\ot \times \mW^\ot$ be a weakly bienriched $\infty$-category and $\A,\B$ associative algebras in $\mV,\mW$, respectively.
There is a canonical equivalence
$$_{\A}\LMod(\RMod_\B(\mN)) =  \Enr\Fun_{[0]}([0], {_\A \Enr\Fun_{[0]}([0], {\mN_\B}) }) \simeq $$$$ \Enr\Fun_{[0]}([0], \Enr\Fun_{[0]}([0], {_\A\mN_\B})) \simeq  \Enr\Fun_{[0],[0]}([0], {_\A \mN_\B})= {_{\A}\BMod_{\B}(\mN)}.$$ 	

\end{remark}

\begin{corollary}\label{juha}

Let $\mN^\circledast \to \mV^\ot \times \mW^\ot$ be a weakly bienriched $\infty$-category, $\A$ an associative algebra in $\mV$ and
$\B$ an associative algebra in $\mW.$
The following canonical functor is an equivalence: $$ _\A\BMod_{\B}(\mN)\to {_{(\A,\B^\rev)}\LMod}(\theta(\mN)).$$
\end{corollary}

Next we apply Theorem \ref{biii} to show that $(\kappa,\kappa)$-bipseudo-enriched
$\infty$-categories can be transformed into left $\kappa$-pseudo-enriched $\infty$-categories for any regular cardinal $\kappa$ (Corollary \ref{abcd}).
\begin{notation}
Let $\n,\m \geq 0$ and for every $1 \leq \bi \leq \n, 1 \leq \bj \leq \m$ let $\mV_\bi^\ot \to \Ass,\mW_\bj^\ot \to \Ass $ be small monoidal $\infty$-categories compatible with $\kappa_\bi$-small colimits, $\tau_\bj$-small colimits, respectively,
for small regular cardinals $\kappa_\bi, \tau_\bj.$
Let $${_{\mV_1 \times ... \times \mV_\n}^{\kappa_1,...,\kappa_\n}\B\Enr^{\tau_1,...,\tau_\m}_{\mW_1 \times ... \times \mW_\m}} \subset {_{\mV_1 \times ... \times \mV_\n}\B\P\Enr}_{\mW_1 \times ... \times \mW_\m}$$
be the full subcategory spanned by the bipseudo-enriched $\infty$-categories whose
morphism objects are presheaves on $\mV_1 \times ... \times \mV_\n \times \mW_1 \times ... \times \mW_\m$ preserving $\kappa_\bi$-small colimits in the $\bi$-th
component for $1 \leq \bi \leq \n$ and preserving $\tau_\bj$-small colimits in the
$\n+\bj$-th component for $1 \leq \bj \leq \m.$
	
\end{notation}

\begin{lemma}\label{basimo}Let $\kappa,\tau$ be small regular cardinals,
$\n,\m \geq 0$ and for every $1 \leq \bi \leq \n, 1 \leq \bj \leq \m$ let $\mV_\bi^\ot \to \Ass,\mW_\bj^\ot \to \Ass $ be small monoidal $\infty$-categories compatible with $\kappa$-small colimits, $\tau$-small colimits, respectively. 
The functor 
$$_{\mV_1 \otimes_\kappa ... \ot_\kappa \mV_\n}\B\Enr_{\mW_1 \otimes_\tau ... \ot_\tau \mW_\m} \to {_{\mV_1 \times ... \times \mV_\n}\B\P\Enr}_{\mW_1 \times ... \times \mW_\m}$$
taking pullback along the universal monoidal functors preserving $\kappa$-small colimits, $\tau$-small colimits componentwise, respectively, induces an equivalence
$$_{\mV_1 \otimes_\kappa ... \ot_\kappa \mV_\n}^\kappa\B\Enr^{\tau}_{\mW_1 \otimes_\tau ... \ot_\tau \mW_\m} \to {_{\mV_1 \times ... \times \mV_\n}^{\kappa,..., \kappa}\B\Enr}^{\tau,...,\tau}_{\mW_1 \times ... \times \mW_\m}.$$

\end{lemma}

\begin{proof}
The universal monoidal functor $\alpha: \mV_1^\ot \times_\Ass ... \times_\Ass \mV^\ot_\n \to (\mV_1 \otimes_\kappa ... \ot_\kappa \mV_\n)^\ot$
induces an adjunction $$\alpha_!:(\mP(\mV_1) \ot ... \ot \mP(\mV_\n))^\ot \simeq \mP(\mV_1 \times ... \times \mV_\n)^\ot \to \mP(\mV_1 \ot_\kappa ... \ot_\kappa \mV_\n)^\ot: \alpha^*$$ relative to $\Ass.$
By universal property of $\alpha$ the latter adjunction restricts to a monoidal equivalence
$$\alpha_!: (\Ind_\kappa(\mV_1) \ot ... \ot \Ind_\kappa(\mV_\n))^\ot \simeq \Ind_\kappa(\mV_1 \ot_\kappa ... \ot_\kappa \mV_\n)^\ot: \alpha^*, $$
where $(\Ind_\kappa(\mV_1) \ot ... \ot \Ind_\kappa(\mV_\n))^\ot \subset \mP(\mV_1 \times ... \times \mV_\n)^\ot$ is the full suboperad spanned by the presheaves preserving component-wise $\kappa$-small colimits, which is a localization relative to $\Ass$ and so a presentably monoidal $\infty$-category.
By Corollary \ref{cosqa} the canonical functor
$$_{\Ind_\kappa(\mV_1 \otimes_\kappa ... \ot_\kappa \mV_\n)}\B\Enr_{\Ind_\tau(\mW_1 \otimes_\tau ... \ot_\tau \mW_\m)} \to {_{\mV_1 \otimes_\kappa ... \ot_\kappa \mV_\n}^\kappa\B\Enr^{\tau}_{\mW_1 \otimes_\tau ... \ot_\tau \mW_\m}} $$ is an equivalence.
So it is enough to see that the composition
$$ _{\Ind_\kappa(\mV_1 \otimes_\kappa ... \ot_\kappa \mV_\n)}\B\Enr_{\Ind_\tau(\mW_1 \otimes_\tau ... \ot_\tau \mW_\m)} \to {_{\mV_1 \otimes_\kappa ... \ot_\kappa \mV_\n}^\kappa\B\Enr^{\tau}_{\mW_1 \otimes_\tau ... \ot_\tau \mW_\m}} $$$$ \to {_{\mV_1 \times ... \times \mV_\n}^{\kappa,..., \kappa}\B\Enr}^{\tau,...,\tau}_{\mW_1 \times ... \times \mW_\m} $$ is an equivalence. The latter composition factors as 
$$ _{\Ind_\kappa(\mV_1 \otimes_\kappa ... \ot_\kappa \mV_\n)}\B\Enr_{\Ind_\tau(\mW_1 \otimes_\tau ... \ot_\tau \mW_\m)} \simeq {_{\Ind_\kappa(\mV_1) \ot ... \ot \Ind_\kappa(\mV_\n)}\B\Enr_{\Ind_\tau(\mW_1) \ot ... \ot \Ind_\tau(\mW_\m)}}
\to $$$$ {{_{\mV_1 \times ... \times \mV_\n}^{\kappa,..., \kappa}\B\Enr}^{\tau,...,\tau}_{\mW_1 \times ... \times \mW_\m}}.$$	
Let $\Theta \subset {_{\mP(\mV_1 \times ... \times \mV_\n)}\B\Enr_{\mP(\W_1 \times ... \times \mW_\m)}}$ be the full subcategory of bienriched $\infty$-categories whose morphism objects belong to $$\Ind_\kappa(\mV_1) \ot ... \ot \Ind_\kappa(\mV_\n) \otimes \Ind_\tau(\mW_1) \ot ... \ot \Ind_\tau(\mW_\m) \subset \mP(\mV_1 \times ... \times \mV_\n \times \mW_1 \times ... \times \mW_\m).$$
By Proposition \ref{embeto} the functor $$\Theta \subset {_{\mP(\mV_1 \times ... \times \mV_\n)}\B\Enr_{\mP(\W_1 \times ... \times \mW_\m)}} \to {_{\Ind_\kappa(\mV_1) \ot ... \ot \Ind_\kappa(\mV_\n)}\B\P\Enr_{\Ind_\tau(\mW_1) \ot ... \ot \Ind_\tau(\mW_\m)}}$$
induces an equivalence $\Theta \to {_{\Ind_\kappa(\mV_1) \ot ... \ot \Ind_\kappa(\mV_\n)}\B\Enr_{\Ind_\tau(\mW_1) \ot ... \ot \Ind_\tau(\mW_\m)}}.$
So it is enough to see that the functor ${_{\mP(\mV_1 \times ... \times \mV_\n)}\B\Enr_{\mP(\W_1 \times ... \times \mW_\m)}} \to {_{\mV_1 \times ... \times \mV_\n}\B\P\Enr}_{\mW_1 \times ... \times \mW_\m}$
is an equivalence that restricts to an equivalence $\Theta \simeq {_{\mV_1 \times ... \times \mV_\n}^{\kappa,..., \kappa}\B\Enr}^{\tau,...,\tau}_{\mW_1 \times ... \times \mW_\m}.$
This holds by Corollary \ref{coronn}.


\end{proof}

\begin{corollary}\label{imbos}
	
Let $\kappa$ be a small regular cardinal and $\mV^\ot \to \Ass,\mW^\ot \to \Ass $ small monoidal $\infty$-categories compatible with $\kappa$-small colimits.
The following induced functor is an equivalence:
$$_{\mV \otimes_\kappa \mW}^\kappa\L\Enr_\emptyset \to {_{\mV \times \mW}^{\kappa,\kappa}\L\Enr_\emptyset}.$$
	
\end{corollary}

Theorem \ref{biii} gives the following corollary:

\begin{corollary}\label{abcd}
	
Let $\kappa$ be a small regular cardinal and $\mV^\ot \to \Ass,\mW^\ot \to \Ass $ small monoidal $\infty$-categories compatible with $\kappa$-small colimits.
The equivalence
$$_{\mV}\B\P\Enr_\mW \simeq {_{\mV \times \mW^\rev}\L\P\Enr_\emptyset}$$
of Theorem \ref{biii} taking pullback along the projection $\mV^\ot \times_\Ass (\mW^\ot)^\rev \to \mV^\ot \times (\mW^\ot)^\rev \simeq \mV^\ot \times \mW^\ot$
induces an equivalence
$$\theta: {^\kappa_{\mV}\B\Enr^{\kappa}_\mW} \to {_{\mV \times \mW^\rev}^{\kappa,\kappa}\L\Enr_\emptyset}
\simeq {_{\mV \otimes_\kappa \mW^\rev}^\kappa\L\Enr_\emptyset}.$$

\end{corollary}
\begin{corollary}\label{abcde}
	
Let $\mV^\ot \to \Ass,\mW^\ot \to \Ass $ be presentably monoidal $\infty$-categories.
There is a canonical equivalence
$${_{\mV}\B\Enr_\mW} \simeq {_{\mV \otimes \mW^\rev}\L\Enr_\emptyset}.$$
	
\end{corollary}


Next we prove that the $\infty$-category of enriched presheaves is invariant under transforming bi-enrichment into left enrichment
(Corollary \ref{chan}).
To state the following results we need the next remark:
\begin{remark}
Let $\kappa < \kappa'$ be small regular cardinals, $\mU^\ot \to \Ass$ a small $\infty$-operad, $\mV^\ot \to \Ass, \mW^\ot \to \Ass$ small monoidal $\infty$-categories compatible with $\kappa$-small colimits,
$\mV'^\ot \to \Ass, \mW'^\ot \to \Ass$ small monoidal $\infty$-categories compatible with $\kappa'$-small colimits and $\alpha: \mV^\ot \to \mV'^\ot, \beta: \mW^\ot \to \mW'^\ot$
monoidal functors preserving $\kappa$-small colimits. Let $\gamma: (\mV \ot_{\kappa} \mW)^\ot \to (\mV' \ot_{\kappa} \mW')^\ot \to (\mV' \ot_{\kappa'} \mW')^\ot$ be the induced monoidal functor preserving $\kappa$-small colimits. There is a commutative square:
\begin{equation*}
\begin{xy}
\xymatrix{
{_{\mV' \ot_{\kappa'} \mW'}\L\Enr_\mU} \ar[d]^\gamma
\ar[rrr]
&&& {_{\mV' \times \mW'}\L\Enr_\mU} \ar[d]^{(\alpha \times \beta)^*} 
\\
{_{\mV \ot_{\kappa} \mW}\L\Enr_\mU} \ar[rrr]  &&& {_{\mV \times \mW}\L\Enr_\mU}.}
\end{xy}\end{equation*}
This follows from naturality and the fact that the monoidal functor
$\mV'^\ot \times_\Ass \mW'^\ot \to (\mV' \ot_{\kappa'} \mW')^\ot$
factors as $\mV'^\ot \times_\Ass \mW'^\ot \to (\mV' \ot_{\kappa} \mW')^\ot \to (\mV' \ot_{\kappa'} \mW')^\ot$.
Hence there is also a commutative square:
\begin{equation*}
\begin{xy}
\xymatrix{
{_{\mV'}^{\kappa'}\B\Enr^{\kappa'}_{\mW'}} \ar[d]^{(\alpha,\beta)^*}
\ar[rrr]
&&& {_{\mV' \otimes \mW'^\rev}^{\kappa'}\L\Enr_\emptyset} \ar[d]^{\gamma^*} 
\\
{_{\mV}^\kappa\B\Enr^{\kappa}_\mW} \ar[rrr] &&& {_{\mV \otimes \mW^\rev}^\kappa\L\Enr_\emptyset}.}
\end{xy}\end{equation*}

\end{remark}


\begin{proposition}\label{bienve}
Let $\kappa$ be a small regular cardinal and $\mM^\circledast \to \mV^\ot \times \mW^\ot$ a small $\kappa, \kappa$-bienriched $\infty$-category.
The $\mV \ot_\kappa \mW^\rev$-enriched embedding $$\theta(\mM)^\circledast \to \theta(\mP\B\Env(\mM)_{\B\Enr_{\kappa,\kappa}})^\circledast$$
induces a left $\Ind_\kappa(\mV) \otimes \Ind_\kappa(\mW)^\rev \simeq \Ind_\kappa(\mV \otimes_\kappa \mW^\rev)$-linear equivalence
$$\mP\L\Env(\theta(\mM))_{\L\Enr_\kappa}^\circledast \to \theta(\mP\B\Env(\mM)_{\B\Enr_{\kappa,\kappa}})^\circledast.$$

\end{proposition}

\begin{proof}
The left $\Ind_\kappa(\mV \otimes_\kappa \mW^\rev)$-linear functor
$\gamma: \mP\L\Env(\theta(\mM))_{\L\Enr_\kappa}^\circledast \to \theta(\mP\B\Env(\mM)_{\B\Enr_{\kappa,\kappa}})^\circledast$
is an equivalence if the pullback of $\gamma$ to $(\mV \otimes_\kappa \mW^\rev)^\ot$ is an equivalence because source and target of $\gamma$ are left $\mV \times \mW^\rev$-enriched $\infty$-categories.
By Corollary \ref{abcd} the functor
$$\theta: {_{\mV}^\kappa\B\Enr^{\kappa}_\mW} \to {_{\mV \times \mW^\rev}^{\kappa,\kappa}\L\Enr_\emptyset} \simeq {_{\mV \otimes \mW^\rev}^{\kappa}\L\Enr}_\emptyset$$
taking pullback along the projection $\mV^\ot \times_\Ass (\mW^\ot)^\rev \to \mV^\ot \times (\mW^\ot)^\rev \simeq \mV^\ot \times \mW^\ot$ is an equivalence.
The pullback of $\gamma$ to $(\mV \otimes_\kappa \mW^\rev)^\ot$ is an
embedding $\gamma': \mP\widetilde{\L\Env}_{\L\Enr_\kappa}(\theta(\mM))^\circledast \to \theta(\mP\widetilde{\B\Env}(\mM)_{\B\Enr_{\kappa,\kappa}})^\circledast$.
Let $\mN^\circledast \to \mV^\ot \times_\Ass (\mW^\ot)^\rev$ be a left tensored $\infty$-category that admits small conical colimits.
Consider the commutative square:
$$\begin{xy}
\xymatrix{\LinFun^\L_{\mV, \mW}(\mP\widetilde{\B\Env}(\mM)_{\B\Enr_{\kappa,\kappa}}, \theta^{-1}(\mN)) \ar[d]^\simeq\ar[rr]^{\simeq}
&& \Enr\Fun_{\mV, \mW}(\mM, \theta^{-1}(\mN)) \ar[d]^{\simeq} 
\\
\LinFun^\L_{\mV\ot_\kappa\mW^\rev}(\theta(\mP\widetilde{\B\Env}(\mM)_{\B\Enr_{\kappa,\kappa}}), \mN) \ar[rr]^{}  && \Enr\Fun_{\mV \otimes_\kappa \mW^\rev}(\theta(\mM), \mN).}
\end{xy}$$



\end{proof}

\begin{corollary}
Let $\kappa$ be a small regular cardinal and $\mM^\circledast \to \mV^\ot \times \mW^\ot$ a small $\kappa, \kappa$-bienriched $\infty$-category.
The left $\mV \otimes_\kappa \mW^\rev$-enriched embedding $\theta(\mM)^\circledast \to \theta(\mP\B\Env_\kappa(\mM)_{\B\Enr})^\circledast$
induces a left $\mV \otimes_\kappa \mW^\rev$-linear equivalence
$$\mP\L\Env_{\kappa}(\theta(\mM))_{\L\Enr}^\circledast \to \theta(\mP\B\Env_\kappa(\mM)_{\B\Enr})^\circledast.$$	
	
\end{corollary}

\begin{corollary}\label{chan}
Let $\mM^\circledast \to \mV^\ot \times \mW^\ot$ be a small bienriched $\infty$-category.
The left $\mV \otimes \mW^\rev$-enriched embedding $\theta(\mM)^\circledast \to \theta(\mP\B\Env(\mM)_{\B\Enr})^\circledast$
induces a left $\mV \otimes \mW^\rev$-linear equivalence
$$\mP\L\Env(\theta(\mM))_{\L\Enr}^\circledast \to \theta(\mP\B\Env(\mM)_{\B\Enr})^\circledast.$$	
	
\end{corollary}

\begin{corollary}\label{bienve}
Let $\kappa$ be a small regular cardinal, $\mM^\circledast \to \mV^\ot \times \mW^\ot$ an absolute small weakly  bienriched $\infty$-category and $\bar{\mM}^\circledast\to (\mP\Env(\mV)^\kappa)^\ot \times (\mP\Env(\mW)^\kappa)^\ot$ the unique extension to a $\kappa,\kappa$-bienriched $\infty$-category.
The left $\mP\Env(\mV)^\kappa \otimes_\kappa (\mP\Env(\mW)^\kappa)^\rev$-enriched embedding
$$\theta(\bar{\mM})^\circledast \to \theta(\mP\B\Env(\mM))^\circledast$$
induces a left $\Ind_\kappa(\mP\Env(\mV)^\kappa \otimes_\kappa (\mP\Env(\mW)^\kappa)^\rev)\simeq \mP\Env(\mV) \otimes \mP\Env(\mW)^\rev$-linear equivalence
$$\mP\L\Env(\theta(\bar{\mM}))_{\L\Enr_\kappa}^\circledast \to \theta(\mP\B\Env(\mM))^\circledast.$$


\end{corollary}

\section{Enriched $\infty$-categories of enriched functors}


In this section we introduce and study enriched $\infty$-categories of enriched functors. 

\subsection{Left and right enriched $\infty$-categories of enriched functors}

\begin{notation}
Let $\mV^\ot \to \Ass,\mW^\ot \to \Ass$ be small $\infty$-operads.
The functor $$\Cat_{\infty / \mV^\ot} \times \Cat_{\infty / \mW^\ot} \to \Cat_{\infty / \mV^\ot \times \mW^\ot}, \ (\mM^\circledast \to \mV^\ot, \mN^\circledast \to \mW^\ot) \mapsto \mM^\circledast \times \mN^\circledast \to \mV^\ot \times \mW^\ot$$
restricts to a functor
\begin{equation}\label{kkk}
\alpha: {_\mV\omega\B\Enr}_\emptyset \times {_\emptyset \omega\B\Enr}_{\mW} \to {_\mV}\omega\B\Enr_{\mW}. \end{equation}

\end{notation}

\begin{remark}Let $\mV^\ot \to \Ass,\mW^\ot \to \Ass$ be small monoidal $\infty$-categories. The functor (\ref{kkk}) restricts to functors \begin{equation}\label{kkkk}
{_\mV\P\L\Enr}_\emptyset \times {_\emptyset \P\R\Enr}_{\mW} \to {_\mV\P\B\Enr_{\mW}}, \end{equation} 
$$ {_\mV\L\Mod}_\emptyset \times {_\emptyset \R\Mod}_{\mW} \to {_\mV\B\Mod_{\mW}}. $$
\end{remark}



We have the following proposition, which follows immediately from Proposition \ref{lehmmm}:
\begin{proposition}

For every small $\infty$-operads $\mV^\ot \to \Ass,\mW^\ot \to \Ass$ the functor
(\ref{kkk}) admits componentwise right adjoints.

\end{proposition}

We will prove later that also the functor (\ref{kkkk}) admits componentwise right adjoints (Corollary \ref{ps}).
To prove Proposition \ref{lehmmm} we use the following lemma:

\begin{lemma}\label{lemist}
Let $\mC \to \T, \T \to \rS$ be functors such that the composition $\mC \to \T \to \rS$ is a cocartesian fibration. The functor 
$(-)\times_\rS \mC:\Cat_{\infty / \rS} \xrightarrow{ } \Cat_{\infty / \mC} \to \Cat_{\infty / \T}$ admits a right adjoint that we denote by
$ \Fun_\T^{\rS}(\mC,-).$ 
If $\rS, \T $ are contractible, we drop $\rS$, $\T$ from the notation.	
\end{lemma}

\begin{proof}
It is enough to see that the functor $ \mC \times_\rS (-): \Cat_{\infty / \rS} \to \Cat_{\infty / \mC}$ admits a right adjoint. 	
This follows from \cite[Example B.3.11.]{lurie.higheralgebra} and \cite[Corollary B.3.15.]{lurie.higheralgebra}.

\end{proof}

The following remark is \cite[Remark 3.23.]{heine2024localglobalprincipleparametrizedinftycategories}:

\begin{remark}\label{puas}
For every functors $\mD \to \T$ and $\rS' \to \rS$ there is a canonical equivalence 
$ \rS' \times_\rS  \Fun_\T^{\rS}(\mC,\mD) \simeq \Fun_{\rS' \times_\rS \T}^{\rS'}( \rS' \times_\rS \mC, \rS' \times_\rS \mD)$ specifying the fibers of the functor $\Fun_\T^{\rS}(\mC,\mD) \to \rS.$
\end{remark}

\begin{remark}\label{reuil}
Let $\phi: \T \to \rS$ be a functor, $\mE \subset \Fun([1],\rS)$ a full subcategory
and $\mC \to \T$ a cartesian fibration relative to the collection of 
$\phi$-cocartesian lifts of morphisms of $\rS$ that belong to $\mE$ and 
$\mD \to \T$ a cocartesian fibration relative to the collection of 
$\phi$-cocartesian lifts of morphisms of $\rS$ that belong to $\mE$.
By \cite[Proposition B.4.1.]{lurie.higheralgebra} the functor $\Fun_\T^{\rS}(\mC,\mD) \to \rS$ is a cocartesian fibration relative to $\mE.$
	
\end{remark}

\begin{lemma}\label{faith}
Let $\mC \to \T, \T \to \rS$ be functors such that the composition $\mC \to \T \to \rS$ is a cocartesian fibration.
Let $\mD \to \mE$ be a fully faithful functor over $\T$. Then $\Fun^\rS_\T(\mC,\mD) \to \Fun^\rS_\T(\mC,\mE) $ is fully faithful.

\end{lemma}

\begin{proof}
Note that the monomorphisms in $\Cat_{\infty / \T}$ are detected by the forgetful functor to $\Cat_\infty$ and so are the functors over $\T$ whose image in $\Cat_\infty$ is an inclusion, i.e. induces embeddings on mapping spaces and the maximal subspace.	
Because the functor $\Fun^\rS_\T(\mC,-): \Cat_{\infty / \T} \to \Cat_{\infty / \rS}$ is a right adjoint, it preserves monomorphisms.
So it remains to see that the functor $\Fun^\rS_\T(\mC,\mD) \to \Fun^\rS_\T(\mC,\mE) $ is full, i.e. essentially surjective on mapping spaces. Let $\phi: [1] \to \Fun^\rS_\T(\mC,\mE)$ be a functor over $\rS$
whose restrictions to $\{0\} \subset [1], \{1\} \subset [1]$ belong to
$\Fun^\rS_\T(\mC,\mD).$ Then $\phi$ corresponds to a
functor $\psi: [1]\times_\rS \mC \to [1]\times_\rS \mE$ over $[1]$
whose fibers over $0,1$ factor through $\mD.$ But then $\psi$ factors through $\mD$ so that $\phi$ factors through $\Fun^\rS_\T(\mC,\mD).$

\end{proof}

The following lemma is \cite[Lemma 3.73.]{HEINE2023108941} whose proof we include for the reader's convenience:

\begin{lemma}\label{innerho}
Let $\mM^\circledast \to \mV^\ot, \mO^\circledast \to \mW^\ot, \mN^\circledast \to \mV^\ot \times \mW^\ot$ be weakly left, weakly right, weakly bienriched $\infty$-categories. 

\begin{enumerate}
\item The functor $\alpha: \Fun^{\mW^\ot}_{\mV^\ot \times \mW^\ot}(\mM^\circledast \times \mW^\ot, \mN^\circledast) \to \mW^\ot $ is a
weakly right enriched $\infty$-category.

\vspace{1mm}\item The functor $\beta: \Fun^{\mV^\ot}_{\mV^\ot \times \mW^\ot}(\mV^\ot \times \mO^\circledast, \mN^\circledast) \to \mV^\ot $ is a weakly left enriched $\infty$-category. \vspace{1mm}

\item If $\mN^\circledast \to\mV^\ot \times \mW^\ot$ exhibits $\mN$ as right tensored over $\mW$, the functor $\alpha$ is a right tensored $\infty$-category 
and for every $\X \in \mM$ the following canonical functor is right $\mW$-linear:
$$ \Fun^{\mW^\ot}_{\mV^\ot \times \mW^\ot}(\mM^\circledast \times \mW^\ot, \mN^\circledast) \to \mN^\circledast.$$

\item If $\mN^\circledast \to\mV^\ot \times \mW^\ot$ exhibits $\mN$ as left tensored over $\mV$, the functor $\beta$ is a left tensored $\infty$-category 
and for every $\X \in \mO$ the following canonical functor is left $\mV$-linear:
$$ \Fun^{\mV^\ot}_{\mV^\ot \times \mW^\ot}(\mV^\ot \times \mO^\circledast, \mN^\circledast) \to \mN^\circledast.$$

\end{enumerate}
\end{lemma}

\begin{proof}
	
We prove (1) and (3). (2) and (4) are dual.
By Remark \ref{reuil} the functor $\alpha$ and for every $\X \in \mM$ the functor
$ \Fun^{\mW^\ot}_{\mV^\ot \times \mW^\ot}(\mM^\circledast \times \mW^\ot, \mN^\circledast) \to \mN^\circledast $ are maps of cocartesian fibrations
relative to the collection of inert morphisms of $\Ass$ preserving the maximum.
Moreover these functors are maps of cocartesian fibrations over $\Ass$
if $\mN^\circledast \to\mV^\ot \times \mW^\ot$ exhibits $\mN$ as right tensored over $\mW$.
For the general case every cocartesian lift $\W \to \W'$ in $\mW^\ot$ of the map $\{\n\} \subset [\n]$ in $\Ass^\op$ induces the functor
$$ \Fun^{\mW^\ot}_{\mV^\ot \times \mW^\ot}(\mM^\circledast \times \mW^\ot, \mN^\circledast)_{\W} \simeq \Fun_{\mV^\ot}(\mM^\circledast, \mN_{\W}^\circledast) \to \Fun_{\mV^\ot}(\mM^\circledast, \mN_{\W'}^\circledast) \simeq \Fun^{\mW^\ot}_{\mV^\ot \times \mW^\ot}(\mM^\circledast \times \mW^\ot, \mN^\circledast)_{\W'}$$
induced by the equivalence $\mN_{\W}^\circledast \to \mN_{\W'}^\circledast.$
This proves (2).
To see (1) it remains to verify condition (2) of Definition \ref{bla}.
By (2) the functor $ \Fun^{\mW^\ot}_{\mV^\ot \times \mW^\ot}(\mM^\circledast \times \mW^\ot, \R\Env(\mN)^\circledast) \to \Env(\mW)^\ot$ is a right tensored $\infty$-category. So the pullback $ \mW^\ot \times_{\Env(\mW)^\ot} \Fun^{\mW^\ot}_{\mV^\ot \times \mW^\ot}(\mM^\circledast \times \mW^\ot, \R\Env(\mN)^\circledast) \to \mW^\ot$ is a weakly right enriched $\infty$-category.
By Lemma \ref{faith} the embedding $\mN^\circledast \to \R\Env(\mN)^\circledast$ gives rise to an embedding $$ \Fun^{\mW^\ot}_{\mV^\ot \times \mW^\ot}(\mM^\circledast \times \mW^\ot, \mN^\circledast) \hookrightarrow \mW^\ot \times_{\Env(\mW)^\ot} \Fun^{\mW^\ot}_{\mV^\ot \times \mW^\ot}(\mM^\circledast \times \mW^\ot, \R\Env(\mN)^\circledast)$$ over $\mW^\ot$ that is a map of cocartesian fibrations 
relative to the collection of inert morphisms of $\Ass$ preserving the maximum.
This implies that $ \Fun^{\mW^\ot}_{\mV^\ot \times \mW^\ot}(\mM^\circledast \times \mW^\ot, \mN^\circledast) \to \mW^\ot $ satisfies condition (2) of Definition \ref{bla} and so is a right enriched $\infty$-category.
\end{proof}

The next notation is \cite[Notation 3.74.]{HEINE2023108941}:

\begin{notation}\label{bbbb}
Let $\mM^\circledast \to \mV^\ot, \mO^\circledast \to \mW^\ot,
\mN^\circledast \to \mV^\ot \times \mW^\ot$ be weakly left, weakly right and weakly bienriched $\infty$-categories, respectively.
\begin{enumerate}
\item Let $$\Enr\Fun_{\mV, \emptyset}(\mM,\mN)^\circledast \subset \Fun^{\mW^\ot}_{\mV^\ot \times \mW^\ot}(\mM^\circledast \times \mW^\ot, \mN^\circledast) $$ be the full weakly right $\mW$-enriched subcategory spanned by 
$$\Enr\Fun_{\mV, \emptyset}(\mM, \mN) \subset \Fun_{\mV^\ot}(\mM^\circledast, \mN_{[0]}^\circledast) \simeq \Fun^{\mW^\ot}_{\mV^\ot \times \mW^\ot}(\mM^\circledast \times \mW^\ot, \mN^\circledast)_{[0]}.$$

\item Let $$\Enr\Fun_{\emptyset,\mW}(\mO,\mN)^\circledast \subset \Fun^{\mV^\ot}_{\mV^\ot \times \mW^\ot}(\mV^\ot \times \mO^\circledast, \mN^\circledast)$$
be the full weakly left $\mV$-enriched subcategory spanned by 
$$\Enr\Fun_{\emptyset,\mW}(\mO,\mN) \subset \Fun_{\mV^\ot}(\mM^\circledast, \mN_{[0]}^\circledast) \simeq \Fun^{\mW^\ot}_{\mV^\ot \times \mW^\ot}(\mM^\circledast \times \mW^\ot, \mN^\circledast)_{[0]}.$$

\item Let $$\L\LinFun_{\mV, \emptyset}(\mM, \mN)^\circledast \subset \Enr\Fun_{\mV, \emptyset}(\mM,\mN)^\circledast$$ be the full weakly right $\mW$-enriched subcategory spanned by 
$$\L\LinFun_{\mV, \emptyset}(\mM, \mN) \subset \Enr\Fun_{\mV, \emptyset}(\mM, \mN).$$

\item Let $$\R\LinFun_{\emptyset,\mW}(\mO,\mN)^\circledast \subset\Enr\Fun_{\emptyset,\mW}(\mO,\mN)^\circledast$$
be the full weakly left $\mV$-enriched subcategory spanned by 
$$\R\LinFun_{\emptyset,\mW}(\mO,\mN) \subset \Enr\Fun_{\emptyset,\mW}(\mO,\mN).$$

\end{enumerate}
\end{notation}

\begin{remark}
There are canonical left $\mV$-enriched equivalences:
$$\Enr\Fun_{\emptyset,\mW}(\mO,\mN)^\circledast:= (\Enr\Fun_{\mW, \emptyset}(\mO^\rev, \mN^\rev)^\rev)^\circledast,$$
$$\LinFun_{\emptyset,\mW}(\mO,\mN)^\circledast:= (\LinFun_{\mW, \emptyset}(\mO^\rev, \mN^\rev)^\rev)^\circledast.$$	
\end{remark}

Remark \ref{reuil} implies the following proposition, which is \cite[Proposition 3.78.]{HEINE2023108941}: 

\begin{proposition}\label{lehmmm} 
Let $\mM^\circledast \to \mV^\ot, \mO^\circledast \to \mW^\ot,
\mN^\circledast \to \mV^\ot \times \mW^\ot$ be weakly left, weakly right and weakly bienriched $\infty$-categories, respectively.
\begin{enumerate}
\item The canonical equivalence
$$ \Fun_{\mW^\ot}(\mO^\circledast, \Fun^{\mW^\ot}_{\mV^\ot \times \mW^\ot}(\mM^\circledast \times \mW^\ot, \mN^\circledast)) \simeq \Fun_{\mV^\ot \times \mW^\ot}(\mM^\circledast \times \mO^\ot, \mN^\circledast)$$
restricts to an equivalence
\begin{equation*}\label{equino}
\Enr\Fun_{\emptyset, \mW}(\mO, \Enr\Fun_{\mV, \emptyset}(\mM,\mN)) \simeq \Enr\Fun_{\mV, \mW}(\mM \times \mO, \mN).
\end{equation*} 

\item The canonical equivalence $$ \Fun_{\mV^\ot}(\mM^\circledast,\Fun^{\mV^\ot}_{\mV^\ot \times \mW^\ot}(\mV^\ot \times \mO^\circledast, \mN^\circledast)) \simeq \Fun_{\mV^\ot \times \mW^\ot}(\mM^\circledast \times \mO^\ot, \mN^\circledast)$$ restricts to an equivalence $$ \Enr\Fun_{\mV, \emptyset}(\mM, \Enr\Fun_{\emptyset, \mW}(\mO,\mN)) \simeq \Enr\Fun_{\mV,\mW}(\mM \times \mO, \mN).$$
\end{enumerate}
\end{proposition}


Remark \ref{puas} implies the following remark:
\begin{remark}\label{leman}
Let $\mM^\circledast \to \mV^\ot, \mN^\circledast \to \mV^\ot \times \mW^\ot$ be weakly left, weakly bienriched $\infty$-categories, respectively.
For every map $\beta: \mW'^\ot \to \mW^\ot$ of $\infty$-operads
there is a canonical right $\mW'$-enriched equivalence
$$ \mW'^\ot \times_{\mW^\ot} \Enr\Fun_{\mV, \emptyset}(\mM,\mN)^\circledast \simeq \Enr\Fun_{\mV, \emptyset}(\mM,\beta^* \mN)^\circledast.$$

\end{remark}
Lemma \ref{faith} implies the following remark:

\begin{remark}

Let $\mM^\circledast \to \mV^\ot, \mN^\circledast \to \mV^\ot \times \mW^\ot, \mN'^\circledast \to \mV^\ot \times \mW^\ot$ be weakly left, weakly bienriched $\infty$-categories, respectively.
For every $\mV,\mW$-enriched embedding $ \mN^\circledast \to \mN'^\circledast$
the induced right $\mW$-enriched functor
$\Enr\Fun_{\mV, \emptyset}(\mM,\mN)^\circledast \to \Enr\Fun_{\mV, \emptyset}(\mM,\mN')^\circledast$
is an embedding.

\end{remark}

Lemma \ref{innerho}, Remark \ref{reuil} and Lemma \ref{colas} imply the following corollary:

\begin{corollary}\label{innerhost} Let $\mK \subset \Cat_\infty$ be a full subcategory and $\mM^\circledast \to \mV^\ot, 
\mN^\circledast \to \mV^\ot \times \mW^\ot$ weakly left enriched and 
weakly bienriched $\infty$-categories. 
If $\mN^\circledast \to \mV^\ot \times \mW^\ot$ is a right tensored $\infty$-category compatible with $\mK$-indexed colimits, 
$\Enr\Fun_{\mV, \emptyset}(\mM,\mN)^\circledast \to \mW^\ot$ is a right tensored $\infty$-category compatible with $\mK$-indexed colimits and the right $\mW$-enriched functor
$\Enr\Fun_{\mV, \emptyset}(\mM,\mN)^\circledast \to (\mN^\mM)^\circledast $ is right $\mW$-linear and preserves $\mK$-indexed colimits.
		
		
		
	
\end{corollary}

The latter corollary admits the following refinement:

\begin{proposition}\label{innerhosty} Let $\mK \subset \Cat_\infty$ be a full subcategory and $\mM^\circledast \to \mV^\ot, 
\mN^\circledast \to \mV^\ot \times \mW^\ot$ weakly left enriched and weakly bienriched $\infty$-categories.
If $\mN^\circledast \to \mV^\ot \times \mW^\ot$ admits right tensors
and forming right tensors preserves $\mK$-indexed colimits, $\Enr\Fun_{\mV, \emptyset}(\mM,\mN)^\circledast \to \mW^\ot$ admits right tensors and $\mK$-indexed conical colimits and the right $\mW$-enriched functor
$\Enr\Fun_{\mV, \emptyset}(\mM,\mN)^\circledast \to (\mN^\mM)^\circledast $ is right $\mW$-linear and preserves $\mK$-indexed colimits.
		
		

\end{proposition}

\begin{proof}
By Lemma \ref{colas} it is enough to prove the case that $\mK$ is empty.
By Corollary \ref{emcolf} (2) (for $\kappa=\emptyset$) the weakly bienriched $\infty$-category $\mN^\circledast \to \mV^\ot \times \mW^\ot$ that admits right tensors is the pullback of a unique right tensored $\infty$-category
$\mN'^\circledast \to \mV^\ot \times \Env(\mW)^\ot$.
By Proposition \ref{innerho} the weakly right enriched $\infty$-category $\Enr\Fun_{\mV, \emptyset}(\mM,\mN')^\circledast \to \Env(\mW)^\ot$ is a right tensored $\infty$-category and the right $\Env(\mW)$-enriched functor
$\Enr\Fun_{\mV, \emptyset}(\mM,\mN')^\circledast \to (\mN'^\mM)^\circledast $ is right linear. 
So using Remark \ref{leman} the pullback $\Enr\Fun_{\mV, \emptyset}(\mM,\mN)^\circledast \simeq \Enr\Fun_{\mV, \emptyset}(\mM,\mN')^\circledast \times_{\Env(\mW)^\ot} \mW^\ot  \to \mW^\ot$ is a weakly right enriched $\infty$-category that admits right tensors and the following right $\mW$-enriched functor is right $\mW$-linear:
$$\Enr\Fun_{\mV, \emptyset}(\mM,\mN)^\circledast \simeq \Enr\Fun_{\mV, \emptyset}(\mM,\mN')^\circledast \times_{\Env(\mW)^\ot} \mW^\ot \to \Fun(\mM,\mN)^\circledast \simeq \Fun(\mM,\mN')^\circledast \times_{\Env(\mW)^\ot} \mW^\ot.$$

\end{proof}
 




\subsection{Bienriched $\infty$-categories of enriched functors}

In this subsection we extend Notation \ref{bbbb} by refining the enriched $\infty$-category of enriched functors to a bienriched $\infty$-category (Notation \ref{silop} and Theorem \ref{vvvl}).

\begin{lemma}\label{kkpp}
Let $\mM^\circledast \to \mU^\ot \times \mV^\ot, \mN^\circledast \to \mV^\ot \times \mW^\ot$ be weakly bienriched $\infty$-categories.
\begin{enumerate}
\item Then $\mM^\circledast \times_{\mV^\ot} \mN^\circledast \to \mU^\ot \times \mW^\ot$ is a weakly bienriched $\infty$-category.

\vspace{1mm}
\item If $\mM^\circledast \to \mU^\ot \times \mV^\ot$ is a left tensored $\infty$-category and $ \mN^\circledast \to \mV^\ot \times \mW^\ot$ is a right tensored $\infty$-category, then $\mM^\circledast \times_{\mV^\ot} \mN^\circledast \to \mU^\ot \times \mW^\ot$ is a bitensored $\infty$-category.

\vspace{1mm}
\item If $\mM^\circledast \to \mU^\ot \times \mV^\ot$ is a right tensored $\infty$-category and $ \mN^\circledast \to \mV^\ot \times \mW^\ot$ is a left tensored $\infty$-category, then $\mM^\circledast \times_{\mV^\ot} \mN^\circledast \to \mV^\ot $ is a map of  cocartesian fibrations over $\Ass$.

\vspace{1mm}
\item If $\mM^\circledast \to \mU^\ot \times \mV^\ot$ is a left tensored $\infty$-category compatible with small colimits
and $ \mN^\circledast \to \mV^\ot \times \mW^\ot$ is a right tensored $\infty$-category compatible with small colimits, then $\mM^\circledast \times_{\mV^\ot} \mN^\circledast \to \mU^\ot \times \mW^\ot$ is a bitensored $\infty$-category compatible with small colimits.
\end{enumerate}

\end{lemma}

\begin{proof}
Via the enveloping bitensored $\infty$-category (1) follows from (2).
(2): The functors $\mM^\circledast \times \mW^\ot \to \mU^\ot \times \mV^\ot \times \mW^\ot, \mU^\ot \times \mN^\circledast \to \mU^\ot \times \mV^\ot \times \mW^\ot$ are maps of cocartesian fibrations relative to the collection of triples,
whose first and third component is cocartesian and whose second component is an equivalence.
Thus the pullback $$(\mM^\circledast \times \mW^\ot) \times_{(\mU^\ot \times \mV^\ot \times \mW^\ot)} (\mU^\ot \times \mN^\circledast) \simeq \mM^\circledast \times_{\mV^\ot} \mN^\circledast \to \mU^\ot \times \mW^\ot$$ is a cocartesian fibration relative to the collection of pairs whose first and second component is cocartesian.
So it is enough to observe that for every $\U \in \mU, \W \in \mW$
the cocartesian lift $\U \to \U'$ in $ \mU$ of the morphism $\{\n\} \subset [\n]$
in $\Delta$ and the cocartesian lift $\W \to \W'$ in $\mW$ of the morphism $\{0\} \subset [\m]$ in $\Delta$ induce an equivalence
$$(\mM^\circledast \times_{\mV^\ot} \mN^\circledast)_{\U,\W} \simeq \mM_\U^\circledast \times_{\mV^\ot} \mN_\W^\circledast \to (\mM^\circledast \times_{\mV^\ot} \mN^\circledast)_{\U',\W'}\simeq \mM_{\U'}^\circledast \times_{\mV^\ot} \mN_{\W'}^\circledast$$
since the functors $\mM_\U^\circledast \to \mM_{\U'}^\circledast$ and $\mN_\W^\circledast \to 
\mN_{\W'}^\circledast$ are equivalences.

(3): The functors $\mM^\circledast \to \mV^\ot, \mN^\circledast \to \mV^\ot $ are maps of cocartesian fibrations over $\Ass$, where we see $\mM^\circledast$ over $\Ass$ via projection to the second factor and $\mN^\circledast$ over $\Ass$ via projection to the first factor.
Hence the pullback $ \mM^\circledast \times_{\mV^\ot} \mN^\circledast \to \mV^\ot $ is a map of  cocartesian fibrations over $\Ass$.

(4): Let $\U \to \U'$ be a cocartesian lift in $ \mU$ of the morphism $\{0\} \subset [\n]$ in $\Delta$ and $\W \to \W'$ a cocartesian lift in $\mW$ of the morphism $\{\m\} \subset [\m]$ in $\Delta$.
The induced functor
$$(\mM^\circledast \times_{\mV^\ot} \mN^\circledast)_{\U,\W} \simeq \mM_\U^\circledast \times_{\mV^\ot} \mN_\W^\circledast \to (\mM^\circledast \times_{\mV^\ot} \mN^\circledast)_{\U',\W'}\simeq \mM_{\U'}^\circledast \times_{\mV^\ot} \mN_{\W'}^\circledast$$ preserves small colimits because the functors
$\mM_\U^\circledast \to \mM_{\U'}^\circledast$ and $\mN_\W^\circledast \to \mN_{\W'}^\circledast$ preserve small colimits.

\end{proof}

\begin{lemma}\label{bopp}
Let $\mM^\circledast \to \mU^\ot \times \mV^\ot$ be a right tensored $\infty$-category and $ \mN^\circledast \to \mV^\ot \times \mW^\ot$ a left tensored $\infty$-category.	
The embedding $\mO^\circledast := \mM^\circledast \times_{\mV^\ot} \mN^\circledast \subset \L\Env(\mM)^\circledast \times_{\mV^\ot} \R\Env(\mN)^\circledast $
induces an equivalence
$$ \B\Env(\mO)^\circledast \simeq \L\Env(\mM)^\circledast \times_{\mV^\ot} \R\Env(\mN)^\circledast.$$

\end{lemma}

\begin{proof}
By Notation \ref{ene} the induced enriched $\Env(\mU), \Env(\mW)$-linear functor $$ \B\Env(\mO)^\circledast \to \L\Env(\mM)^\circledast \times_{\mV^\ot} \R\Env(\mN)^\circledast$$
identifies with the canonical equivalence 
$$\text{Min} \times_{\Fun(\{0\}, \Ass)} (\mM^\circledast \times_{\mV^\ot} \mN^\circledast) \times_{\Fun(\{0\}, \Ass) } \Max \simeq (\mathrm{Min} \times_{\Fun(\{0\},\Ass)} \mM^\circledast) \times_{\mV^\ot} (\mN^\circledast \times_{\Fun(\{0\},\Ass)} \Max).$$

\end{proof}


\begin{notation}\label{pppu} Let $\mM^\circledast \to \mU^\ot \times \mV^\ot$ be a right tensored $\infty$-category and $ \mN^\circledast \to \mV^\ot \times \mW^\ot$ a left tensored $\infty$-category.
\begin{enumerate}
\item Let $$\Enr\Fun^\mV_{\mU,\mW}(\mM \times_{\mV} \mN, \mO) \subset \Enr\Fun_{\mU,\mW}(\mM \times_{\mV} \mN, \mO)$$ be the full subcategory spanned by the $\mU,\mW$-enriched functors that invert $\sigma$-cocartesian morphisms, where $\sigma$ is the canonical functor $ \mM^\circledast \times_{\mV^\ot} \mN^\circledast\to \mV^\ot \to \Ass.$

\item Let $\mM^\circledast \to \mU^\ot \times \mV^\ot, \mN^\circledast \to \mV^\ot \times \mW^\ot$ be bitensored $\infty$-categories.
Let $$\LinFun^\mV_{\mU,\mW}(\mM \times_{\mV} \mN, \mO) \subset \LinFun_{\mU,\mW}(\mM \times_{\mV} \mN, \mO)$$ be the full subcategory of $\mU,\mW$-linear functors that invert $\sigma$-cocartesian morphisms. 
\end{enumerate}

\end{notation}

\begin{remark}\label{rembrol}

By Lemma \ref{kkpp} under the assumptions of Notation \ref{pppu} the functor $\sigma$ is a cocartesian fibration whose cocartesian morphisms lie over equivalences in $\mU^\ot \times \mW^\ot.$ This implies that a morphism of $ \mM^\circledast \times_{\mV^\ot} \mN^\circledast$ is $\sigma$-cocartesian if and only if it is the image of a morphism of $ \emptyset^\ot \times_{\mU^\ot}\mM^\circledast \times_{\mV^\ot} \mN^\circledast \times_{\mW^\ot} \emptyset^\ot $ that is $\sigma'$-cocartesian, where $\sigma'$ is the functor $ \emptyset^\ot \times_{\mU^\ot}\mM^\circledast \times_{\mV^\ot} \mN^\circledast \times_{\mW^\ot} \emptyset^\ot  \to \mV^\ot.$
Consequently, $$\Enr\Fun^\mV_{\mU,\mW}(\mM \times_{\mV} \mN, \mO), \LinFun^\mV_{\mU,\mW}(\mM \times_{\mV} \mN, \mO)$$ are likewise the full subcategories of enriched functors whose underlying functors preserve $\sigma'$-cocartesian morphisms.
\end{remark}

\begin{notation}Let $\mM^\circledast \to \mU^\ot \times \mV^\ot, \mN^\circledast \to \mV^\ot \times \mW^\ot$ be small bitensored $\infty$-categories.
We write $$(\mM \boxtimes_{\mV} \mN)^\circledast \to \mU^\ot \times \mW^\ot$$
for the relative tensor product of $\mM^\circledast \to \mU^\ot \times \mV^\ot $ and $\mN^\circledast \to \mV^\ot\times \mW^\ot $ induced by the product of $\Cat_\infty.$	

\end{notation}


\begin{lemma}\label{retaco}
Let $\mM^\circledast \to \mU^\ot \times \mV^\ot, \mN^\circledast \to \mV^\ot \times \mW^\ot$ be small bitensored $\infty$-categories.
There is a $\mU, \mW$-linear functor \begin{equation}\label{idd} \mM^\circledast \times_{\mV^\ot} \mN^\circledast \to (\mM \boxtimes_{\mV} \mN)^\circledast\end{equation} that induces for every bitensored $\infty$-category $\mO^\circledast \to \mU^\ot \times \mW^\ot$ an equivalence $$
\LinFun_{\mU,\mW}(\mM \boxtimes_{\mV} \mN, \mO) \to \LinFun^\mV_{\mU,\mW}(\mM \times_{\mV} \mN, \mO).$$

\end{lemma}

\begin{proof}
Let $\sigma$ be the canonical functor $ \mM^\circledast \times_{\mV^\ot} \mN^\circledast\to \mV^\ot \to \Ass.$
Let $ \mZ^\circledast$ be the localization of $\mM^\circledast \times_{\mV^\ot} \mN^\circledast$ with respect to the set of $\sigma$-cocartesian morphisms. 
The functor $\mM^\circledast \times_{\mV^\ot} \mN^\circledast \to \mU^\ot \times \mW^\ot$ induces a functor $\theta: \mZ^\circledast \to \mU^\ot \times \mW^\ot$.

The functor $\mM^\circledast \times_{\mV^\ot} \mN^\circledast \to \mU^\ot \times \mW^\ot$ is a map of cocartesian fibrations over $ \Ass \times \Ass$. This guarantees by \cite[Proposition 2.1.4.]{HinichDwyer} that $\theta$ is a map of cocartesian fibrations over 
$ \Ass \times \Ass$ and the universal functor $ \mM^\circledast \times_{\mV^\ot} \mN^\circledast \to \mZ^\circledast$ is a map of cocartesian fibrations over $ \Ass \times \Ass$
over $\mU^\ot \times \mW^\ot$.  
Moreover by \cite[Proposition 2.1.4.]{HinichDwyer} the fiber of $\theta$ over any $\U \in \mU, \W \in \mW$ is the localization
of $ \emptyset^\ot \times_{\mU^\ot} \mM^\circledast \times_{\mV^\ot} \mN^\circledast \times_{\mW^\ot} \emptyset^\ot$ with respect to the set of $\sigma$-cocartesian morphisms.
In particular, for every morphism $\U \to \U'$ in $\mU$ and $\W \to \W'$ in $\mW$
the induced functor $\mZ^\circledast_{\U,\V} \to \mZ^\circledast_{\U',\V'}$ identifies with the identity of the localization of $ \emptyset^\ot \times_{\mU^\ot} \mM^\circledast \times_{\mV^\ot} \mN^\circledast \times_{\mW^\ot} \emptyset^\ot$ with respect to the set of $\sigma$-cocartesian morphisms.
So $\theta$ is a bitensored $\infty$-category and the universal functor
$ \mM^\circledast \times_{\mV^\ot} \mN^\circledast \to \mZ^\circledast $ is a $\mU,\mW$-linear functor.
By definition the latter induces for every bitensored $\infty$-category $\mO^\circledast \to \mU^\ot \times \mW^\ot$ an equivalence
$$ \LinFun_{\mU,\mW}(\mZ, \mO) \to \LinFun^\mV_{\mU,\mW}(\mM \times_{\mV} \mN, \mO).$$

Projection gives a canonical equivalence $$\Fun'_{\mU^\ot \times \Ass \times \mW^\ot}(\mM^\circledast \times_{\mV^\ot} \mN^\circledast, \Ass \times \mO^\circledast) \simeq \LinFun^\mV_{\mU,\mW}(\mM \times_{\mV} \mN, \mO),$$
where the left hand side is the full subcategory of functors over $\mU^\ot \times\Ass \times \mW^\ot$ preserving cocartesian lifts of morphisms of $\mU^\ot \times\Ass \times \mW^\ot$
whose first and third component is cocartesian.

The map $\mM^\circledast \times_{\mV^\ot} \mN^\circledast \to \mU^\ot \times \Ass \times \mW^\ot$ of cocartesian fibrations over $ \Ass $ classifies the functor
$$\rho: \Ass \to {_\mU\BMod_\mW}, [\n] \mapsto \mM^\circledast \times_{\mV^\ot} \emptyset^\ot \times \mV^{\times \n} \times \emptyset^\ot \times_{\mV^\ot} \mN^\circledast$$
and there is a canonical equivalence 
$$\Fun'_{\mU^\ot \times \Ass \times \mW^\ot}(\mM^\circledast \times_{\mV^\ot} \mN^\circledast, \Ass \times \mO^\circledast)^\simeq \simeq \Fun(\Ass, {_\mU\BMod_\mW})(\rho, \delta(\mO)),$$
where $\delta: {_\mU\BMod_\mW} \to  \Fun(\Ass, {_\mU\BMod_\mW})$ is the diagonal functor.
So $\mZ^\circledast \to \mU^\ot \times \mW^\ot$ is the colimit of $\rho,$ 
which by \cite[Theorem 4.4.2.8.]{lurie.higheralgebra} is the the relative tensor product $(\mM \boxtimes_{\mV} \mN)^\circledast \to \mU^\ot \times \mW^\ot.$



\end{proof}

\begin{notation}\label{notlino} Let $\mV^\ot \to \Ass,\mW^\ot \to \Ass$ be small monoidal $\infty$-categories.
The functor $$ {_\mV\B\Mod_\mU} \times {_\mU\B\Mod_\mW} \to {_\mV\B\Mod_\mW}, (\mO^\circledast \to \mV^\ot \times \mU^\ot, \mM^\circledast \to \mU^\ot \times \mW^\ot) \mapsto ((\mO \boxtimes_\mU \mM)^\circledast \to \mV^\ot \times \mW^\ot) $$ 
preserves component-wise small colimits by the description of the relative tensor product \cite[Definition 4.4.2.10.]{lurie.higheralgebra} and so admits component-wise right adjoints that we denote by 
$$\L\LinFun_{\mV, \emptyset}(\mO,-)^\circledast: {_\mV\B\Mod_\mW} \to {_\mU\B\Mod_\mW},$$
$$\R\LinFun_{\emptyset,\mW}(\mM,-)^\circledast: {_\mV\B\Mod_\mW} \to {_\mV\B\Mod_\mU}.$$

\end{notation}

The following remark and Corollary \ref{nokiop} show that Notation \ref{notlino} is not in conflict with Notation \ref{bbbb}.

\begin{remark}\label{nocon} Let $\mV^\ot \to \Ass,\mW^\ot \to \Ass$ be small monoidal $\infty$-categories.
The functor $$  {_\mV\L\Mod_\emptyset} \times {_\emptyset\R\Mod_\mW} \simeq {_\mV\B\Mod_*} \times {_*\B\Mod_\mW} \to {_\mV\B\Mod_\mW},$$$$ (\mO^\circledast \to \mV^\ot, \mM^\circledast \to \mW^\ot) \mapsto ((\mO \boxtimes_* \mM)^\circledast \to \mV^\ot \times \mW^\ot) $$
identifies with the functor 
$$ {_\mV\L\Mod_\emptyset} \times {_\emptyset\R\Mod_\mW} \to {_\mV\B\Mod_\mW}, (\mO^\circledast \to \mV^\ot, \mM^\circledast \to \mW^\ot) \mapsto (\mO^\circledast \times \mM^\circledast \to \mV^\ot \times \mW^\ot).$$
Consequently, Notation \ref{notlino} for $\mU^\ot=\Ass$ agrees with Notation \ref{bbbb}.

\end{remark}

\begin{lemma}
Let $\mV^\ot \to \Ass,\mW^\ot \to \Ass,\mV'^\ot \to \Ass, \mW'^\ot \to \Ass, \mU^\ot \to \Ass,\mU'^\ot \to \Ass $ be small monoidal $\infty$-categories, $\alpha: \mV'^\ot \to \mV^\ot, \beta: \mW'^\ot \to \mW^\ot, \gamma: \mU'^\ot \to \mU^\ot $ monoidal functors
and $\mO^\circledast \to \mV^\ot \times \mU^\ot,  \mM^\circledast \to \mU^\ot \times \mW^\ot, \mN^\circledast \to \mV^\ot \times \mW^\ot$ bitensored $\infty$-categories.
There is a canonical $\mU',\mW'$-linear equivalence $$\mU'^\ot \times_{\mU^\ot} \L\LinFun_{\mV, \emptyset}(\mO,\mN)^\circledast \times_{\mW^\ot} {\mW'^\ot} \simeq \L\LinFun_{\mV, \emptyset}(\gamma^*\mO,\beta^*\mN)^\circledast$$
and a canonical $\mV',\mU'$-linear equivalence $$\mV'^\ot \times_{\mV^\ot} \R\LinFun_{\emptyset,\mW}(\mM,\mN)^\circledast \times_{\mU^\ot} \mU'^\ot \simeq \R\LinFun_{\emptyset,\mW}(\gamma^* \mM ,  \alpha^*\mN)^\circledast. $$

\end{lemma}

\begin{proof}
Let $\mM^\circledast \to \mU'^\ot \times \mW'^\ot$ be a bitensored $\infty$-category.
The first equivalence is represented by the canonical equivalence
$$ \LinFun_{\mU',\mW'}(\mM, (\gamma,\beta)^*\L\LinFun_{\mV, \emptyset}(\mO,\mN)) \simeq \LinFun_{\mU,\mW}(\mU \boxtimes_{\mU'} \mM \boxtimes_{\mW'} \mW, \L\LinFun_{\mV, \emptyset}(\mO,\mN))$$$$ \simeq 
\LinFun_{\mV,\mW}(\mO \boxtimes_\mU \mU \boxtimes_{\mU'} \mM \boxtimes_{\mW'} \mW, \mN) \simeq 
\LinFun_{\mV,\mW}(\gamma^*\mO \boxtimes_{\mU'}\mM \boxtimes_{\mW'} \mW, \mN) \simeq $$$$ \LinFun_{\mV,\mW'}(\gamma^*\mO \boxtimes_{\mU'} \mM,  \beta^*\mN) \simeq \LinFun_{\mU',\mW'}(\mM,\L\LinFun_{\mV, \emptyset}(\gamma^*\mO,\beta^*\mN)).$$

Let $\mO^\circledast \to \mV'^\ot \times \mU'^\ot$ be a bitensored $\infty$-category.	
The second equivalence is represented by the canonical equivalence
$$ \LinFun_{\mV',\mU'}(\mO, (\gamma,\alpha)^*\R\LinFun_{\emptyset,\mW}(\mM,\mN)) \simeq \LinFun_{\mV,\mU}(\mV \boxtimes_{\mV'} \mO \boxtimes_{\mU'} \mU, \R\LinFun_{\emptyset,\mW}(\mM,\mN)) \simeq $$$$
\LinFun_{\mV,\mW}(\mV \boxtimes_{\mV'} \mO \boxtimes_{\mU'} \mU \ot_{\mU} \mM,\mN) \simeq 
\LinFun_{\mV,\mW}(\mV \boxtimes_{\mV'} \mO \boxtimes_{\mU'} \gamma^*\mM,\mN) \simeq $$$$
\LinFun_{\mV',\mW}(\mO \boxtimes_{\mU'}\gamma^* \mM, \alpha^*\mN) \simeq \LinFun_{\mV',\mU'}(\mO,\R\LinFun_{\emptyset,\mW}(\gamma^* \mM , \alpha^* \mN)).$$

\end{proof}

\begin{corollary}\label{nokiop}
Let $\mV^\ot \to \Ass,\mW^\ot \to \Ass,\mV'^\ot \to \Ass, \mW'^\ot \to \Ass, \mU^\ot \to \Ass,\mU'^\ot \to \Ass $ be small monoidal $\infty$-categories and $\mO^\circledast \to \mV^\ot \times \mU^\ot, \mM^\circledast \to \mU^\ot \times \mW^\ot, \mN^\circledast \to \mV^\ot \times \mW^\ot$ bitensored $\infty$-categories.
There is a canonical right $\mW$-linear equivalence $$\emptyset^\ot \times_{\mU^\ot} \L\LinFun_{\mV, \emptyset}(\mO,\mN)^\circledast \simeq \L\LinFun_{\mV, \emptyset}(\mO,\mN)^\circledast$$
and a canonical left $\mV$-linear equivalence $$ \R\LinFun_{\emptyset,\mW}(\mM,\mN)^\circledast \times_{\mU^\ot} \emptyset^\ot \simeq \R\LinFun_{\emptyset,\mW}(\mM , \mN)^\circledast, $$
where the left hand sides are Notation \ref{notlino} and the right hand sides are Notation \ref{bbbb}.
\end{corollary}

\begin{remark}\label{hido}
Let $\mM^\circledast \to \mV^\ot$ be a left tensored $\infty$-category and $\V \in \mV$.
Under the canonical equivalence $\L\LinFun_{\mV, \emptyset}(\mV,\mM) \simeq \mM$ evaluating at the tensor unit the functor $\V \ot (-): \mM \to \mM$ corresponds to the functor
$\L\LinFun_{\mV, \emptyset}(\mV,\mM) \to \L\LinFun_{\mV, \emptyset}(\mV,\mM)$ precomposing along the left $\mV$-enriched functor $(-)\ot \V: \mV \to \mV.$

Consequently, for every bitensored $\infty$-categories $\mO^\circledast \to \mV^\ot \times \mU^\ot, \mN^\circledast \to \mV^\ot \times \mW^\ot$ and $\X \in \mU$ the functor $ \X \ot (-): \L\LinFun_{\mV, \emptyset}(\mO,\mN) \to \L\LinFun_{\mV, \emptyset}(\mO,\mN)$
precomposes along the left $\mV$-linear functor $$(-)\ot \X \simeq \mO \boxtimes_\mU ((-)\ot \X) : \mO \simeq \mO \boxtimes_\mU \mU \to \mO \boxtimes_\mU \mU \simeq \mO .$$

\end{remark}

\begin{notation}\label{silop}
Let $\mU^\ot \to \Ass$ be a monoidal $\infty$-category, $\mM^\circledast \to \mU^\ot \times \mW^\ot$ a left tensored $\infty$-category, $\mO^\circledast \to \mV^\ot \times \mU^\ot$ a right tensored $\infty$-category  and $\mN^\circledast \to \mV^\ot \times \mW^\ot$ a weakly bienriched $\infty$-category.

\begin{enumerate}
\item Let $$\hspace{4mm}\Enr\Fun_{\emptyset,\mW}(\mM,\mN)^\circledast \subset \mV^\ot \times_{\Env(\mV)^\ot} \R\LinFun_{\emptyset,\Env(\mW)}(\R\Env(\mM),\B\Env(\mN))^\circledast \to \mV^\ot \times \mU^\ot$$
be the weakly bienriched subcategory of right $\Env(\mW) $-enriched functors that carry $\mM$ to $\mN.$ 

\item Let $$\hspace{3mm}\Enr\Fun_{\mV,\emptyset}(\mO,\mN)^\circledast \subset \L\LinFun_{\Env(\mV),\emptyset}(\L\Env(\mO),\B\Env(\mN))^\circledast \times_{\Env(\mW)^\ot} \mW^\ot \to \mU^\ot \times \mW^\ot$$
be the weakly bienriched subcategory of left $\Env(\mV)$-enriched functors that carry $\mO$ to $\mN.$ 

\end{enumerate}

\end{notation}

The next lemma guarantees that Notation \ref{silop} is not in conflict with Notation \ref{bbbb}:

\begin{lemma}
Let $\mU^\ot \to \Ass$ be a monoidal $\infty$-category, $\mM^\circledast \to \mU^\ot \times \mW^\ot$ a left tensored $\infty$-category, $\mO^\circledast \to \mV^\ot \times \mU^\ot$ a right tensored $\infty$-category  and $\mN^\circledast \to \mV^\ot \times \mW^\ot$ a weakly bienriched $\infty$-category.

\begin{enumerate}
\item The underlying weakly left enriched $\infty$-category of $\Enr\Fun_{\emptyset,\mW}(\mM,\mN)^\circledast \to \mV^\ot \times \mU^\ot$ is canonically equivalent to $ \Enr\Fun_{\emptyset,\mW}(\mM,\mN)^\circledast \to \mV^\ot$ of Notation \ref{bbbb}:

\item The underlying weakly right enriched $\infty$-category of
$\Enr\Fun_{\mV,\emptyset}(\mO,\mN)^\circledast \to \mU^\ot \times \mW^\ot$
is canonically equivalent to $\Enr\Fun_{\mV,\emptyset}(\mO,\mN)^\circledast \to \mW^\ot$
of Notation \ref{bbbb}.
\end{enumerate}	

\end{lemma}

\begin{proof}
We prove (1). The proof of (2) is similar.
By Remark \ref{nocon} the underlying weakly left enriched $\infty$-category of $\Enr\Fun_{\emptyset,\mW}(\mM,\mN)^\circledast \to \mV^\ot \times \mU^\ot$ is by definition the full weakly left enriched subcategory of $\mV^\ot \times_{\Env(\mV)^\ot} \R\LinFun_{\emptyset,\Env(\mW)}(\R\Env(\mM),\B\Env(\mN))^\circledast \to \mV^\ot $ spanned by the right $\Env(\mW) $-enriched functors that carry $\mM$ to $\mN, $ where we use Notation \ref{bbbb} for the pullback.
By Remark \ref{silu} the embedding $\mM^\circledast \subset \R\Env(\mM)^\circledast$ induces a left $\Env(\mV)$-linear functor $$\R\LinFun_{\emptyset,\Env(\mW)}(\R\Env(\mM),\B\Env(\mN))^\circledast \to \R\LinFun_{\emptyset,\mW}(\mM,\B\Env(\mN))^\circledast $$ that is an equivalence since it induces an equivalence on underlying $\infty$-categories by Proposition \ref{unb}.
The pullback of the latter equivalence to $\mV^\ot$ restricts to an equivalence between the full weakly left enriched subcategory spanned by the right $\Env(\mW) $-enriched functors that carry $\mM$ to $\mN$ and
$ \Enr\Fun_{\emptyset,\mW}(\mM,\mN)^\circledast \to \mV^\ot$.

\end{proof}

\begin{lemma}\label{klmoto}

Let $\mU^\ot \to \Ass$ be a monoidal $\infty$-category, $\mM^\circledast \to \mU^\ot \times \mW^\ot$ a left tensored $\infty$-category, $\mO^\circledast \to \mV^\ot \times \mU^\ot$ a right tensored $\infty$-category and $\mN^\circledast \to \mV^\ot \times \mW^\ot$ a weakly bienriched $\infty$-category.

\begin{enumerate}

\item If $\mN^\circledast \to \mV^\ot \times \mW^\ot$ is a left tensored $\infty$-category,
then $\Enr\Fun_{\emptyset,\mW}(\mM,\mN)^\circledast\to \mV^\ot \times \mU^\ot$ is the full weakly bienriched subcategory of
$$  \R\LinFun_{\emptyset,\Env(\mW)}(\R\Env(\mM),\R\Env(\mN))^\circledast \to \mV^\ot \times \mU^\ot$$ spanned by the right $\Env(\mW) $-enriched functors that carry $\mM$ to $\mN.$ 

\vspace{1mm}

\item If $\mN^\circledast \to \mV^\ot \times \mW^\ot$ is a bitensored $\infty$-category,
then $\Enr\Fun_{\emptyset,\mW}(\mM,\mN)^\circledast\to \mV^\ot \times \mU^\ot$ is the full weakly bienriched subcategory of
$$ \R\LinFun_{\emptyset,\Env(\mW)}(\R\Env(\mM),\L^*\mN)^\circledast \to \mV^\ot \times \mU^\ot$$ spanned by the right $\Env(\mW) $-enriched functors, where $\L$ is the left adjoint relative to $\Ass$ of the embedding $\mW^\ot \subset \Env(\mW)^\ot.$

\vspace{1mm}

\item If $\mN^\circledast \to \mV^\ot \times \mW^\ot$ is a right tensored $\infty$-category,
then $\Enr\Fun_{\mV,\emptyset}(\mO,\mN)^\circledast \to \mU^\ot \times \mW^\ot$ is the full weakly bienriched subcategory of
$$\L\LinFun_{\Env(\mV),\emptyset}(\L\Env(\mO),\L\Env(\mN))^\circledast  \to \mU^\ot \times \mW^\ot$$
spanned by the left $\Env(\mV)$-enriched functors that carry $\mO$ to $\mN.$ 

\vspace{1mm}

\item If $\mN^\circledast \to \mV^\ot \times \mW^\ot$ is a bitensored $\infty$-category,
then $\Enr\Fun_{\mV,\emptyset}(\mO,\mN)^\circledast \to \mU^\ot \times \mW^\ot$ is the full weakly bienriched subcategory of
$$\L\LinFun_{\Env(\mV),\emptyset}(\L\Env(\mO),\L^*\mN)^\circledast  \to \mU^\ot \times \mW^\ot$$
spanned by the left $\Env(\mV)$-enriched functors, where $\L$ is the left adjoint relative to $\Ass$ of the embedding $\mV^\ot \subset \Env(\mV)^\ot.$

\end{enumerate}
\end{lemma}

\begin{proof}
We prove (1) and (2). The proofs of (3) and (4) are similar.	

(1): Since $\mN^\circledast \to \mV^\ot \times \mW^\ot$ is a left tensored $\infty$-category, by Lemma \ref{loccyy} and Remark \ref{rati} the right tensored $\infty$-category $\R\Env(\mN)^\circledast \to \mV^\ot \times \Env(\mW)^\ot$ is a bitensored $\infty$-category, the $\mV$-enriched embedding $\mN^\circledast \subset \R\Env(\mN)^\circledast$ is left $\mV$-linear
and the $\R\Env(\mW)$-enriched embedding $\R\Env(\mN)^\circledast \subset \B\Env(\mN)^\circledast$
admits an enriched left adjoint.
The left adjoint gives rise to a $\mV,\mU$-linear functor $$\mV^\ot \times_{\Env(\mV)^\ot} \R\LinFun_{\emptyset,\Env(\mW)}(\R\Env(\mM),\B\Env(\mN))^\circledast \to \R\LinFun_{\emptyset,\Env(\mW)}(\R\Env(\mM),\R\Env(\mN))^\circledast$$
whose underlying functor admits a fully faithful right adjoint, and which therefore admits a
$\mV,\mU$-enriched right adjoint that is an embedding.
This embedding identifies $\Enr\Fun_{\emptyset,\mW}(\mM,\mN)^\circledast \to \mV^\ot \times \mU^\ot$ with the full weakly bienriched subcategory of $\R\LinFun_{\emptyset,\Env(\mW)}(\R\Env(\mM),\R\Env(\mN))^\circledast \to \mV^\ot \times \mU^\ot$
spanned by the right $\Env(\mW) $-enriched functors that carry $\mM$ to $\mN.$ 	

(2): Since $\mN^\circledast \to \mV^\ot \times \mW^\ot$ is a bitensored $\infty$-category, by Lemma \ref{loccyy} 
the enriched embedding $\mN^\circledast \subset \R\Env(\mN)^\circledast$ admits a $\mV$-enriched left adjoint.
The left adjoint corresponds to a $\mV, \Env(\mW)$-linear functor $ \R\Env(\mN)^\circledast \to \L^*(\mN)^\circledast$ whose underlying functor admits a right adjoint and that therefore admits a $ \mV, \Env(\mW)$-enriched right adjoint . 
The left adjoint gives rise to a $\mV,\mU$-linear functor $$ \R\LinFun_{\emptyset,\Env(\mW)}(\R\Env(\mM),\R\Env(\mN))^\circledast \to \R\LinFun_{\emptyset,\Env(\mW)}(\R\Env(\mM),\L^*\mN)^\circledast$$
whose underlying functor admits a fully faithful right adjoint, and which therefore admits a $\mV,\mU$-enriched right adjoint that is an embedding.
In view of (1) this embedding identifies $$\Enr\Fun_{\emptyset,\mW}(\mM,\mN)^\circledast \to \mV^\ot \times \mU^\ot$$ with the full weakly bienriched subcategory of $\R\LinFun_{\emptyset,\Env(\mW)}(\R\Env(\mM),\L^*\mN)^\circledast \to \mV^\ot \times \mU^\ot$
spanned by the right $\Env(\mW) $-enriched functors.

\end{proof}

\begin{proposition}\label{klmo}
Let $\mU^\ot \to \Ass$ be a monoidal $\infty$-category, $\mM^\circledast \to \mU^\ot \times \mW^\ot$ a left tensored $\infty$-category, $\mO^\circledast \to \mV^\ot \times \mU^\ot$ a right tensored $\infty$-category and $\mN^\circledast \to \mV^\ot \times \mW^\ot$ a weakly bienriched $\infty$-category.	

\begin{enumerate}
\item The weakly bienriched $\infty$-category $\Enr\Fun_{\emptyset,\mW}(\mM,\mN)^\circledast \to \mV^\ot \times \mU^\ot$ is a right tensored $\infty$-category.
The right $\mU$-action sends a right $\mW$-enriched functor $\F: \mM^\circledast \to \mN^\circledast$ and an object $\Z \in \mU$ to the right $\mW$-enriched functor $ \mM^\circledast \xrightarrow{\Z \ot (-) }\mM^\circledast \xrightarrow{\F} \mN^\circledast.$	

\vspace{1mm}
\item The weakly bienriched $\infty$-category $\Enr\Fun_{\mV,\emptyset}(\mO,\mN)^\circledast \to \mU^\ot \times \mW^\ot$
is a left tensored $\infty$-category. The left $\mU$-action sends a left $\mV$-enriched functor $\F: \mO^\circledast \to \mN^\circledast$ and an object $\Z \in \mU$ to the left $\mV$-enriched functor $ \mO^\circledast \xrightarrow{ (-)\ot \Z}\mO^\circledast \xrightarrow{\F} \mN^\circledast.$

\vspace{1mm}
\item If $\mN^\circledast \to \mV^\ot \times \mW^\ot$ is a left tensored $\infty$-category,
$\Enr\Fun_{\emptyset,\mW}(\mM,\mN)^\circledast \to \mV^\ot \times \mU^\ot$ is a bitensored $\infty$-category. 
The left $\mV$-action sends a right $\mW$-enriched functor $\F: \mM^\circledast \to \mN^\circledast$ and an object $\Y \in \mV$ to the right $\mW$-enriched functor $ \mM^\circledast  \xrightarrow{\F} \mN^\circledast \xrightarrow{\Y \ot (-)}\mN^\circledast.$	
\vspace{1mm}
\item If $\mN^\circledast \to \mV^\ot \times \mW^\ot$ is a right tensored $\infty$-category,
$\Enr\Fun_{\mV,\emptyset}(\mO,\mN)^\circledast \to \mU^\ot \times \mW^\ot$ is a bitensored $\infty$-category. 
The right $\mW$-action sends a left $\mV$-enriched functor $\F: \mO^\circledast \to \mN^\circledast$ and an object $\Y \in \mW$ to the left $\mV$-enriched functor $ \mO^\circledast  \xrightarrow{\F} \mN^\circledast\xrightarrow{ (-)\ot \Y}\mN^\circledast.$

\end{enumerate}
\end{proposition}

\begin{proof}
We prove (1) and (3). The proofs of (2) and (4) are similar.
(1): Since $$\R\LinFun_{\emptyset,\Env(\mW)}(\R\Env(\mM),\B\Env(\mN))^\circledast \to \Env(\mV)^\ot \times \mU^\ot$$ is a bitensored $\infty$-category, $\mV^\ot \times_{\Env(\mV)^\ot} \R\LinFun_{\emptyset,\Env(\mW)}(\R\Env(\mM),\B\Env(\mN))^\circledast \to \mV^\ot \times \mU^\ot$
is a right tensored $\infty$-category. By Remark \ref{hido} the right $\mU$-action sends a right $\Env(\mW)$-linear functor $\F: \R\Env(\mM)^\circledast\to \B\Env(\mN)^\circledast$ and an object $\Z \in \mU$ to the right $\Env(\mW)$-linear functor $  \R\Env(\mM)^\circledast \xrightarrow{\Z \ot (-) }\R\Env(\mM)^\circledast\to \B\Env(\mN)^\circledast.$
The embedding $\mM^\circledast \subset \R\Env(\mM)^\circledast$ is left $\mU$-linear
because $\mM^\circledast \to \mU^\ot \times \mW^\ot$ a left tensored $\infty$-category.
Consequently, the right $\mU$-action restricts to the full subcategory of right $\Env(\mW) $-linear functors that carry $\mM$ to $\mN.$
The resulting right $\mU$-action sends a right $\mW$-enriched functor $\F: \mM^\circledast \to \mN^\circledast$ whose unique right $\Env(\mW)$-linear extension we denote by $\F': \R\Env(\mM)^\circledast \to \B\Env(\mN)^\circledast$, and an object $\Z \in \mU$ to the right $\mW$-enriched functor $$\mM^\circledast \subset  \R\Env(\mM)^\circledast \xrightarrow{\Z \ot (-) }\R\Env(\mM)^\circledast \xrightarrow{\F'} \B\Env(\mN)^\circledast,$$
which is $ \mM^\circledast \xrightarrow{\Z \ot (-) }\mM^\circledast \xrightarrow{\F} \mN^\circledast \subset \B\Env(\mN)^\circledast.$

(3): By Lemma \ref{klmoto} (1) the weakly bienriched $\infty$-category $\Enr\Fun_{\emptyset,\mW}(\mM,\mN)^\circledast \to \mV^\ot \times \mU^\ot$ is the full weakly bienriched subcategory of $\R\LinFun_{\emptyset,\Env(\mW)}(\R\Env(\mM),\R\Env(\mN))^\circledast \to \mV^\ot \times \mU^\ot$
spanned by the right $\Env(\mW) $-enriched functors that carry $\mM$ to $\mN.$ The weakly bienriched $\infty$-category $$\R\LinFun_{\emptyset,\Env(\mW)}(\R\Env(\mM),\R\Env(\mN))^\circledast \to \mV^\ot \times \mU^\ot$$
is by definition a bitensored $\infty$-category.
By Remark \ref{hido} the right $\mU$-action sends a right $\Env(\mW)$-enriched functor $\F: \R\Env(\mM)^\circledast \to \R\Env(\mN)^\circledast$ and an object $\Z \in \mU$ to the right $\Env(\mW)$-enriched functor $\R\Env(\mM)^\circledast \xrightarrow{\Z \ot (-)}\R\Env(\mM)^\circledast \xrightarrow{\F} \R\Env(\mN)^\circledast$, and the left $\mV$-action sends a right $\Env(\mW)$-enriched functor $\F: \mM^\circledast \to \mN^\circledast$ and an object $\Y \in \mV$ to the right $\Env(\mW)$-enriched functor $ \R\Env(\mM) ^\circledast\xrightarrow{\F} \R\Env(\mN)^\circledast \xrightarrow{\Y \ot (-)}\R\Env(\mN)^\circledast.$		

The $\mV,\mU$-biaction on $ \R\LinFun_{\emptyset,\Env(\mW)}(\R\Env(\mM),\R\Env(\mN))$
restricts to the full subcategory of right $\Env(\mW) $-enriched functors that carry $\mM$ to $\mN:$  the right $\mU$-action restricts for the same reason as in the proof of (1) and has the same description because the $\mV$-enriched embedding $\mM^\circledast \subset \R\Env(\mM)^\circledast$ is left $\mU$-linear.
The left $\mV$-action restricts because the $\mV$-enriched embedding $\mN^\circledast \subset \R\Env(\mN)^\circledast$ is left $\mV$-linear, and so has the claimed description.

\end{proof}

\begin{notation}\label{nozzk}
Let $\mV^\ot \to \Ass,\mW^\ot \to \Ass$ be monoidal $\infty$-categories, $\mM^\circledast \to \mU^\ot \times \mW^\ot$ a left tensored $\infty$-category, $\mO^\circledast \to \mV^\ot \times \mU^\ot$ a right tensored $\infty$-category  and $\mN^\circledast \to \mV^\ot \times \mW^\ot$ a weakly bienriched $\infty$-category.

\begin{enumerate}

\item Let $$\R\LinFun_{\emptyset,\mW}(\mM,\mN)^\circledast \subset\Enr\Fun_{\emptyset,\mW}(\mM,\mN)^\circledast \to \mV^\ot \times \mU^\ot$$
be the weakly bienriched subcategory spanned by the right $\mW$-linear functors.

\item Let $$\L\LinFun_{\mV,\emptyset}(\mO,\mN)^\circledast \subset \Enr\Fun_{\mV,\emptyset}(\mO,\mN)^\circledast \to \mU^\ot \times \mW^\ot$$
be the weakly bienriched subcategory spanned by the left $\mV$-linear functors.

\end{enumerate}

\end{notation}

\begin{corollary}
Let $\mV^\ot \to \Ass,\mW^\ot \to \Ass$ be monoidal $\infty$-categories, $\mM^\circledast \to \mU^\ot \times \mW^\ot$ a left tensored $\infty$-category, $\mO^\circledast \to \mV^\ot \times \mU^\ot$ a bitensored $\infty$-category  and $\mN^\circledast \to \mV^\ot \times \mW^\ot$ a weakly bienriched $\infty$-category.	

\begin{enumerate}
\item The weakly bienriched $\infty$-category $\R\LinFun_{\emptyset,\mW}(\mM,\mN)^\circledast \to \mV^\ot \times \mU^\ot$ is a right tensored $\infty$-category.
The right $\mU$-action sends a right $\mW$-linear functor $\F: \mM^\circledast \to \mN^\circledast$ and an object $\Z \in \mU$ to the right $\mW$-linear functor $ \mM^\circledast \xrightarrow{\Z \ot (-) }\mM^\circledast \xrightarrow{\F} \mN^\circledast.$	

\item The weakly bienriched $\infty$-category $\L\LinFun_{\mV,\emptyset}(\mO,\mN)^\circledast \to \mU^\ot \times \mW^\ot$
is a left tensored $\infty$-category. The left $\mU$-action sends a left $\mV$-linear functor $\F: \mO^\circledast \to \mN^\circledast$ and an object $\Z \in \mU$ to the left $\mV$-linear functor $ \mO^\circledast\xrightarrow{ (-)\ot \Z}\mO^\circledast \xrightarrow{\F} \mN^\circledast.$
\end{enumerate}
\end{corollary}

The next theorem guarantees that Notation \ref{nozzk} extends Notation \ref{notlino}.

\begin{theorem}\label{vvvl}
Let $\mV^\ot \to \Ass,\mW^\ot \to \Ass$ be monoidal $\infty$-categories, $\mM^\circledast \to \mU^\ot \times \mW^\ot$ a left tensored $\infty$-category, $\mO^\circledast \to \mV^\ot \times \mU^\ot$ a right tensored $\infty$-category  and $\mN^\circledast \to \mV^\ot \times \mW^\ot$ a weakly bienriched $\infty$-category.
There is a canonical equivalence $$ \R\LinFun_{\mV,\mU}(\mO,\Enr\Fun_{\emptyset,\mW}(\mM,\mN)) \simeq \Enr\Fun^\mU_{\mV,\mW}(\mO \times_\mU \mM,\mN) \simeq \L\LinFun_{\mU,\mW}(\mM,\Enr\Fun_{\mV,\emptyset}(\mO,\mN))$$
that restricts to equivalences 
$$ \LinFun_{\mV,\mU}(\mO,\Enr\Fun_{\emptyset,\mW}(\mM,\mN)) \simeq \L\LinFun_{\mU,\mW}(\mM,\L\LinFun_{\mV,\emptyset}(\mO,\mN)), $$
$$ \R\LinFun_{\mV,\mU}(\mO,\R\LinFun_{\emptyset,\mW}(\mM,\mN)) \simeq  \LinFun_{\mU,\mW}(\mM,\Enr\Fun_{\mV,\emptyset}(\mO,\mN)),$$
$$ \LinFun_{\mV,\mU}(\mO,\R\LinFun_{\emptyset,\mW}(\mM,\mN)) \simeq \LinFun^\mU_{\mV,\mW}(\mO \times_\mU \mM,\mN)  \simeq$$$$ \LinFun_{\mV,\mW}(\mO \boxtimes_\mU \mM,\mN) \simeq \LinFun_{\mU,\mW}(\mM,\L\LinFun_{\mV,\emptyset}(\mO,\mN)).$$
\end{theorem}

\begin{proof}
Since $\mM^\circledast \to \mU^\ot \times \mW^\ot$ a left tensored $\infty$-category
and $\mO^\circledast \to \mV^\ot \times \mU^\ot$ a right tensored $\infty$-category,
the embedding $\mM^\circledast \subset \R\Env(\mM)^\circledast$ is a left $\mU$-linear functor and
the embedding $\mO^\circledast \subset \L\Env(\mO)^\circledast$ is a right $\mU$-linear functor.
So the induced embedding $ \mZ^\circledast:=\mO^\circledast \times_{\mU^\ot} \mM^\circledast \subset  \L\Env(\mO)^\circledast\times_{\mU^\ot} \R\Env(\mM)^\circledast \simeq \B\Env(\mZ)^\circledast$, where the latter equivalence is by Lemma \ref{bopp}, preserves and detects morphisms cocartesian over $\Ass$ via the functor $\mU^\ot \to \Ass$.
Every object of $ \B\Env(\mZ)$ is of the form $\V_1 \ot ... \ot \V_\n \ot \X \ot \W_1 \ot ... \ot \W_\m$ for $\n,\m \geq 0$ and $\V_1,...,\V_\n \in \mV, \W_1,..., \W_\m \in \mW, \X \in \mZ$. 
This guarantees that every morphism of $\B\Env(\mZ)$ that is cocartesian over $\Ass$ via the functor $\mU^\ot \to \Ass$ is of the form
$\V_1 \ot ... \ot \V_\n \ot \f \ot \W_1 \ot ... \ot \W_\m$ for $\n,\m \geq 0$ and $\V_1,...,\V_\n \in \mV, \W_1,..., \W_\m \in \mW$ and $\f$ a morphism in $ \mZ$ cocartesian over $\Ass$ via the functor $\mU^\ot \to \Ass$.
Consequently, the induced equivalence $$\LinFun_{\Env(\mV),\Env(\mW)}(\L\Env(\mO) \times_\mU \R\Env(\mM),\B\Env(\mN)) \to \Enr\Fun_{\mV,\mW}(\mO \times_\mU \mM,\B\Env(\mN))$$
of Proposition \ref{unb} restricts to an equivalence
$$\LinFun^\mU_{\Env(\mV),\Env(\mW)}(\L\Env(\mO) \times_\mU \R\Env(\mM),\B\Env(\mN)) \simeq \Enr\Fun^\mU_{\mV,\mW}(\mO \times_\mU \mM,\B\Env(\mN)).$$

We obtain a canonical equivalence $$ \R\LinFun_{\mV,\mU}(\mO,\R\LinFun_{\emptyset,\Env(\mW)}(\R\Env(\mM),\B\Env(\mN))) \simeq$$$$
\LinFun_{\Env(\mV),\mU}(\L\Env(\mO),\R\LinFun_{\emptyset,\Env(\mW)}(\R\Env(\mM),\B\Env(\mN)))\simeq$$$$ \LinFun_{\Env(\mV),\Env(\mW)}(\L\Env(\mO) \boxtimes_\mU \R\Env(\mM),\B\Env(\mN)) \simeq$$
$$\LinFun^\mU_{\Env(\mV),\Env(\mW)}(\L\Env(\mO) \times_\mU \R\Env(\mM),\B\Env(\mN)) \simeq$$
$$\Enr\Fun^\mU_{\mV,\mW}(\mO \times_\mU \mM,\B\Env(\mN)),$$
where the first equivalence is by Proposition \ref{unb}, the second equivalence is tautological and the third equivalence is by Lemma \ref{retaco}.
This equivalence restricts to an equivalence
$$\R\LinFun_{\mV,\mU}(\mO,\Enr\Fun_{\emptyset,\mW}(\mM,\mN)) \simeq \Enr\Fun^\mU_{\mV,\mW}(\mO \times_\mU \mM,\mN).$$
Similarly, there is a canonical equivalence $$ \L\LinFun_{\mU,\mW}(\mM,\L\LinFun_{\Env(\mV),\emptyset}(\L\Env(\mO),\B\Env(\mN))) \simeq$$$$
\L\LinFun_{\mU,\Env(\mW)}(\R\Env(\mM),\L\LinFun_{\Env(\mV),\emptyset}(\L\Env(\mO),\B\Env(\mN)))
\simeq $$$$ \LinFun_{\Env(\mV),\Env(\mW)}(\L\Env(\mO) \boxtimes_\mU \R\Env(\mM),\B\Env(\mN)) \simeq$$
$$\Enr\Fun^\mU_{\mV,\mW}(\mO \times_\mU \mM,\B\Env(\mN))$$
that restricts to an equivalence
$$\L\LinFun_{\mU,\mW}(\mM,\Enr\Fun_{\mV,\emptyset}(\mO,\mN)) \simeq \Enr\Fun^\mU_{\mV,\mW}(\mO \times_\mU \mM,\mN).$$

It is clear that these equivalences restrict as claimed.

\end{proof}

\begin{lemma}\label{leujkp}
Let $\mU^\ot \to \Ass,\mU'^\ot \to \Ass $ be small monoidal $\infty$-categories, $\alpha: \mV'^\ot \to \mV^\ot, \beta: \mW'^\ot \to \mW^\ot$ maps of small $\infty$-operads, $\gamma: \mU'^\ot \to \mU^\ot $ a monoidal functor and $\mO^\circledast \to \mV^\ot \times \mU^\ot$ a left tensored $\infty$-category, $  \mM^\circledast \to \mU^\ot \times \mW^\ot$ a right tensored $\infty$-category and $ \mN^\circledast \to \mV^\ot \times \mW^\ot$ a weakly bienriched $\infty$-category.
There is a canonical $\mU',\mW'$-enriched equivalence $$\mU'^\ot \times_{\mU^\ot} \Enr\Fun_{\mV, \emptyset}(\mO,\mN)^\circledast \times_{\mW^\ot} {\mW'^\ot} \simeq \Enr\Fun_{\mV, \emptyset}(\gamma^*\mO,\beta^*\mN)^\circledast$$
and a canonical $\mV',\mU'$-enriched equivalence $$\mV'^\ot \times_{\mV^\ot} \Enr\Fun_{\emptyset,\mW}(\mM,\mN)^\circledast \times_{\mU^\ot} \mU'^\ot \simeq \Enr\Fun_{\emptyset,\mW}(\gamma^* \mM ,  \alpha^*\mN)^\circledast. $$

\end{lemma}

\begin{proof}	We prove the case of the first equivalence. The case of the second equivalenc is similar.
The universal $\mV, \mW$-enriched functor $$ \mO^\circledast \times_{\mU^\ot} \Enr\Fun_{\mV, \emptyset}(\mO,\mN)^\circledast \to \mN^\circledast $$ gives rise to a $\mV, \mW'$-enriched functor $$(\mO^\circledast \times_{\mU^\ot} \mU'^\ot) \times_{\mU'^\ot} (\mU'^\ot \times_{\mU^\ot}\Enr\Fun_{\mV, \emptyset}(\mO,\mN)^\circledast\times_{\mW^\ot} \mW'^\ot) \simeq $$$$  \mU'^\ot \times_{\mU^\ot}\mO^\circledast \times_{\mU^\ot}\Enr\Fun_{\mV, \emptyset}(\mO,\mN)^\circledast\times_{\mW^\ot} \mW'^\ot \to \mO^\circledast \times_{\mU^\ot} \Enr\Fun_{\mV, \emptyset}(\mO,\mN)^\circledast \times_{\mW^\ot} \mW'^\ot\to \mN^\circledast \times_{\mW^\ot} \mW'^\ot . $$
By Theorem \ref{vvvl} the latter corresponds to a $\mU',\mW' $-enriched functor
$$ \rho: (\gamma, \beta)^* \Enr\Fun_{\mV, \emptyset}(\mO,\mN)^\circledast \to  \Enr\Fun_{\mV, \emptyset}(\gamma^*\mO,\beta^*\mN)^\circledast.$$
The latter is $\mU'$-linear by the description of left actions of Proposition \ref{klmo}.
Hence $\rho$ is an equivalence if it induces an equivalence on underlying right enriched $\infty$-categories. The $\mU',\mW' $-enriched functor $\rho$ induces on underlying right enriched $\infty$-categories the canonical $\mW'$-enriched functor
$$ \beta^* \Enr\Fun_{\mV, \emptyset}(\mO,\mN)^\circledast \to  \Enr\Fun_{\mV, \emptyset}(\mO,\beta^*\mN)^\circledast,$$
which is an equivalende by Lemma \ref{leman}.

\end{proof} 

\begin{proposition}\label{klmou}
Let $\mU^\ot \to \Ass$ be a monoidal $\infty$-category, $\mM^\circledast \to \mU^\ot \times \mW^\ot$ a left tensored $\infty$-category, $\mO^\circledast \to \mV^\ot \times \mU^\ot$ a right tensored $\infty$-category and $\mN^\circledast \to \mV^\ot \times \mW^\ot$ a weakly bienriched $\infty$-category.	

\begin{enumerate}

\item If $\mN^\circledast \to \mV^\ot \times \mW^\ot$ is a left pseudo-enriched $\infty$-category,
$\Enr\Fun_{\emptyset,\mW}(\mM,\mN)^\circledast \to \mV^\ot \times \mU^\ot$ is a bipseudo-enriched $\infty$-category. 

\vspace{1mm}
\item If $\mN^\circledast \to \mV^\ot \times \mW^\ot$ is a right pseudo-enriched $\infty$-category,
$\Enr\Fun_{\mV,\emptyset}(\mO,\mN)^\circledast \to \mU^\ot \times \mW^\ot$ is a bipseudo-enriched $\infty$-category. 

\end{enumerate}
\end{proposition}

\begin{proof}
We prove (1). The proof of (2) is similar.
Let $$ \mN'^\circledast := \mP\L\Env(\mN)_{\L\Pseu}^\circledast\to \mP(\mV)^\ot \times \mW^\ot$$ be the left tensored $\infty$-category of Notation \ref{locpat2}
and $$ \mN''^\circledast := \mV^\ot \times_{\mP(\mV^\ot)} \mN'^\circledast \to \mV^\ot \times \mW^\ot.$$
By Proposition \ref{klmo} (3) the functor $\Enr\Fun_{\emptyset,\mW}(\mM,\mN')^\circledast \to \mP(\mV)^\ot \times \mU^\ot$ is a bitensored $\infty$-category. 
So also the pullback $$\mV^\ot \times_{\mP(\mV^\ot)}\Enr\Fun_{\emptyset,\mW}(\mM,\mN')^\circledast \to \mV^\ot \times \mU^\ot$$ is a bitensored $\infty$-category.
By Lemma \ref{leujkp} there is a canonical $\mV, \mU$-enriched equivalence $$\mV^\ot \times_{\mP(\mV^\ot)}\Enr\Fun_{\emptyset,\mW}(\mM,\mN')^\circledast \simeq \Enr\Fun_{\emptyset,\mW}(\mM,\mN'')^\circledast.$$

The $\mV,\mW$-enriched embedding $\mN^\circledast \subset \mN''^\circledast$
induces a $\mV,\mU$-enriched embedding $$\Enr\Fun_{\emptyset,\mW}(\mM,\mN)^\circledast \subset \Enr\Fun_{\emptyset,\mW}(\mM,\mN'')^\circledast.$$
Thus $\Enr\Fun_{\emptyset,\mW}(\mM,\mN)^\circledast \to \mV^\ot \times \mU^\ot$ is a bipseudo-enriched $\infty$-category.

\end{proof}

\begin{proposition}Let $\mN^\circledast \to \mV^\ot \times \mW^\ot$ be a right pseudo-enriched $\infty$-category.
There is a $\mV, \mW$-enriched functor $$ \Enr\Fun_{\emptyset,\mW}(\mW,\mN)^\circledast \to \mN^\circledast$$ 
evaluating at the tensor unit.

\end{proposition}

\begin{proof}
We assume first that $\mN^\circledast \to \mV^\ot \times \mW^\ot $ is a bitensored $\infty$-category, $\mN$ is  presentable and for every $\W \in \mW$ the right tensor $(-)\ot\W: \mN \to \mN$ preserves small colimits.
Under these conditions by Lemma \ref{colas} (3) the embedding $ \L\LinFun_{\emptyset,\mW}(\mW,\mN) \subset  \Enr\Fun_{\emptyset,\mW}(\R\Env(\mW),\mN)$ admits a right adjoint $\R.$
By Proposition \ref{klmo} (3) the weakly bienriched $\infty$-category $ \Enr\Fun_{\emptyset,\mW}(\mW,\mN)^\circledast \to \mV^\ot\times \mW^\ot$ is a bitensored $\infty$-category and by the description of biaction of Proposition \ref{klmo} the full subcategory $ \L\LinFun_{\emptyset,\mW}(\mW,\mN) \subset  \Enr\Fun_{\emptyset,\mW}(\mW,\mN)$
is closed under the biaction. Consequently, also $ \L\LinFun_{\emptyset,\mW}(\mW,\mN)^\circledast \to \mV^\ot\times \mW^\ot$ is a bitensored $\infty$-category and the $\mV,\mW$-enriched embedding
$ \L\LinFun_{\emptyset,\mW}(\mW,\mN)^\circledast  \subset  \Enr\Fun_{\emptyset,\mW}(\mW,\mN)^\circledast $ is linear and so by Lemma \ref{Adj2} admits a $\mV,\mW$-enriched right adjoint.

The composition $$ \theta: \Enr\Fun_{\emptyset,\mW}(\mW,\mN)^\circledast \to \L\LinFun_{\emptyset,\mW}(\mW,\mN)^\circledast \simeq \mN^\circledast $$ of $\mV, \mW$-enriched functors
sends $\F$ to $\R(\F)(\tu)$.
For every $\F \in \Enr\Fun_{\emptyset,\mW}(\mW,\mN)$ by Corollary \ref{explicas} the component of the counit $\R(\F)(\tu)\to \F(\tu)$ induces for every $\Y \in \mN$ an equivalence $$ \mN(\Y,\R(\F)(\tu)) \simeq \L\LinFun_{\emptyset,\mW}(\mW,\mN)(\Y\ot(-) , \R(\F)) \simeq $$$$\L\LinFun_{\emptyset,\mW}(\mW,\mN)(\Y\ot (-) , \F) \simeq \mN(\Y,\F(\tu))$$ and so is an  equivalence. 

Next we treat the general case. Let $$\mN'^\circledast:=\mP\B\Env(\mM)_{\R\Pseu}^\circledast  \times_{\mP(\mW)^\ot} \mW^\ot \to \mP\Env(\mV)^\ot \times \mW^\ot, \ \mN''^\circledast:=\mV^\ot \times_{\mP\Env(\mV)^\ot} \mN'^\circledast \to \mV^\ot \times \mW^\ot. $$
Then $\mN'^\circledast \to \mP\Env(\mV)^\ot \times \mW^\ot $ is a bitensored $\infty$-category, $\mN'$ is  presentable and for every $\W \in \mW$ the right tensor $(-)\ot\W: \mN'  \to \mN' $ preserves small colimits.
So by what we have proven, there is a $\mP\Env(\mV), \mW$-enriched functor $ \theta: \Enr\Fun_{\emptyset,\mW}(\mW,\mN')^\circledast \to \mN'^\circledast $ evaluating at the tensor unit.
The pullback of the latter to $\mV^\ot$ is a $\mV, \mW$-enriched functor $$\theta': \Enr\Fun_{\emptyset,\mW}(\mW,\mN'')^\circledast  \simeq  \mV^\ot \times_{\mP\Env(\mV)^\ot} \Enr\Fun_{\emptyset,\mW}(\mW,\mN')^\circledast \to  \mN''^\circledast $$ evaluating at the tensor unit, where the left hand equivalence is by Lemma \ref{leujkp}.

The canonical $\mV,\mW$-enriched embedding $\mN^\circledast \subset \mN''^\circledast$
induces a $\mV, \mW$-enriched embedding $$\Enr\Fun_{\emptyset,\mW}(\mW,\mN)^\circledast \subset \Enr\Fun_{\emptyset,\mW}(\mW,\mN'')^\circledast.$$
The $\mV, \mW$-enriched functor $ \theta'$ restricts to a  $\mV, \mW$-enriched functor $ \Enr\Fun_{\emptyset,\mW}(\mW,\mN)^\circledast \to \mN^\circledast $ evaluating at the tensor unit.


\end{proof}

\subsection{Enriched Kan-extensions}

In the following we use Proposition \ref{laan} to construct a left Kan-extension 
for $\infty$-categories enriched in any presentably monoidal $\infty$-category
(Proposition \ref{laaaaan}, Proposition \ref{siewa}).
Passing to opposite enriched $\infty$-categories we obtain a formally dual theory of right Kan-extensions for $\infty$-categories enriched in any presentably monoidal $\infty$-category.


\begin{proposition}\label{laaaaan}

\begin{enumerate}
\item Let $\mV^\ot \to \Ass, \mW^\ot \to\Ass$ be small $\infty$-operads and $\psi: \mM^\circledast \to \mN^\circledast $ a $\mV,\mW$-enriched functor from an absolute small to a locally small weakly bienriched $\infty$-category. Let $\mD^\circledast \to \mV^\ot \times \mW^\ot$ be a weakly bienriched $\infty$-category that admits left and right tensors and small conical colimits.
The induced functor $$\Enr\Fun_{\mV,\mW}(\mN,\mD) \to \Enr\Fun_{\mV, \mW}(\mM,\mD) $$ admits a left adjoint $\psi_!$, which is fully faithful if $\psi$ is an embedding.
	
\item Let $\mV^\ot \to \Ass$ be a presentably monoidal $\infty$-category, $\mW^\ot \to\Ass$ a small $\infty$-operad and $\psi: \mM^\circledast \to \mN^\circledast $ a $\mV,\mW$-enriched functor from a small to a locally small left enriched $\infty$-category. Let $\mD^\circledast \to \mV^\ot \times \mW^\ot$ be a left enriched $\infty$-category that admits left and right tensors and small conical colimits.
The induced functor $$\Enr\Fun_{\mV,\mW}(\mN,\mD) \to \Enr\Fun_{\mV, \mW}(\mM,\mD) $$ admits a left adjoint $\psi_!$, which is fully faithful if $\psi$ is an embedding.


\item Let $\mV^\ot \to \Ass, \mW^\ot \to\Ass$ be presentably monoidal $\infty$-categories and $\psi: \mM^\circledast \to \mN^\circledast $ a $\mV,\mW$-enriched functor from a small to a locally small bienriched $\infty$-category. Let $\mD^\circledast \to \mV^\ot \times \mW^\ot$ be a bienriched $\infty$-category that admits left and right tensors and small conical colimits.
The induced functor $$\Enr\Fun_{\mV,\mW}(\mN,\mD) \to \Enr\Fun_{\mV, \mW}(\mM,\mD) $$ admits a left adjoint $\psi_!$, which is fully faithful if $\psi$ is an embedding.

\item For every $\mV,\mW$-enriched functor $\F: \mM^\circledast \to \mD^\circledast$ 
and $\X \in \mN$ in all cases there is an equivalence
$$ \hspace{8mm}\psi_!(\F)(\X) \simeq \colim_{\V_1,..., \V_\n, \psi(\Y), \W_1,..., \W_{\m} \to \X}\ \bigotimes_{\bi=1}^\n \V_\bi \ot \F(\Y) \ot \bigotimes_{\bj=1}^\m \W_\bj.$$
\end{enumerate}		
\end{proposition} 

\begin{proof}
(1): By Corollary \ref{emcolf} the weakly bienriched $\infty$-category $\mD^\circledast \to \mP\Env(\mV)^\ot \times \mP\Env(\mW)^\ot$ is the pullback of a bitensored $\infty$-category $\bar{\mD}^\circledast \to \mV^\ot \times \mW^\ot $ compatible with  
small colimits. We apply Proposition \ref{laan}.
	
(2): The presentably monoidal $\infty$-category $\mV^\ot \to \Ass$ is $\kappa$-compactly generated for some small regular cardinal $\kappa$.
Let $\psi_\kappa: \mM_\kappa^\circledast \to \mN_\kappa^\circledast$ be the pullback of $\psi$ along $(\mV^\kappa)^\ot \subset \mV^\ot,$ which is a $\mV^\kappa,\mW$-enriched functor between absolute small weakly bienriched $\infty$-categories.
By Corollary \ref{cosqa} the induced functors 
$$ \Enr\Fun_{\mV,\mW}(\mN,\mD) \to \Enr\Fun_{\mV^\kappa,\mW}(\mN_\kappa,\mD_\kappa) ,$$
$$ \Enr\Fun_{\mV, \mW}(\mM,\mD) \to \Enr\Fun_{\mV^\kappa, \mW}(\mM_\kappa,\mD_\kappa)$$
are equivalences.
Hence (2) follows from (1), the statement that the induced functor
$$ 	\Enr\Fun_{\mV^\kappa,\mW}(\mN_\kappa,\mD_\kappa) \to \Enr\Fun_{\mV^\kappa,\mW}(\mM_\kappa,\mD_\kappa)$$
admits a left adjoint, which is fully faithful if $\psi_\kappa$ is an embedding.

(3): The proof of (3) is similar: the presentably monoidal $\infty$-categories $\mV^\ot \to \Ass, \mW^\ot \to\Ass$ are $\kappa$-compactly generated, $\tau$-compactly generated, respectively, for some small regular cardinals $\kappa, \tau$.
Let $\psi_{\kappa, \tau}: \mM_{\kappa, \tau}^\circledast \to \mN_{\kappa,\tau}^\circledast$ be the pullback of $\psi$ along $(\mV^\kappa)^\ot \times (\mW^\tau)^\ot \subset \mV^\ot \times \mW^\ot,$ which is a $\mV^\kappa,\mW^\tau$-enriched functor between absolute small weakly bienriched $\infty$-categories.
By Corollary \ref{cosqa} the induced functors 
$$ \Enr\Fun_{\mV,\mW}(\mN,\mD) \to \Enr\Fun_{\mV^\kappa,\mW^\tau}(\mN_{\kappa,\tau},\mD_{\kappa,\tau}) ,$$
$$ \Enr\Fun_{\mV, \mW}(\mM,\mD) \to \Enr\Fun_{\mV^\kappa, \mW^\tau}(\mM_{\kappa,\tau},\mD_{\kappa,\tau})$$
are equivalences.
Hence (3) follows from (1), the statement that the induced functor
$$ \Enr\Fun_{\mV^\kappa,\mW^\tau}(\mN_{\kappa,\tau},\mD_{\kappa,\tau})\to \Enr\Fun_{\mV^\kappa,\mW^\tau}(\mM_{\kappa,\tau},\mD_{\kappa,\tau})$$
admits a left adjoint, which is fully faithful if $\psi_{\kappa,\tau}$ is an embedding.
(4) holds by Proposition \ref{laan}.

\end{proof}

\begin{notation}\label{siew}

Let $\mM^\circledast \to \mV^\ot, \mM'^\circledast \to \mV^\ot $ be weakly left enriched $\infty$-categories, $\mN^\circledast \to \mV^\ot \times \mW^\ot, \mN'^\circledast \to \mV^\ot \times \mW^\ot$ be weakly bienriched $\infty$-categories, $\F:\mM'^\circledast \to \mM^\circledast$ a left $\mV$-enriched functor and $\G: \mN^\circledast \to \mN'^\circledast$ a $\mV,\mW$-enriched functor. 
Let $$\Enr\Fun_{\mV,\emptyset}(\F,\G): \Enr\Fun_{\mV,\emptyset}(\mM,\mN)^\circledast \to \Enr\Fun_{\mV,\emptyset}(\mM',\mN')^\circledast$$ be the right $\mW$-enriched functor corresponding to the $\mV,\mW$-enriched functor
$$\mM'^\circledast \times \Enr\Fun_{\mV, \emptyset}(\mM,\mN)^\circledast \xrightarrow{\F \times \Enr\Fun_{\mV, \emptyset}(\mM,\mN)^\circledast} \mM^\circledast \times \Enr\Fun_{\mV,\emptyset}(\mM,\mN)^\circledast \to \mN^\circledast \xrightarrow{\G} \mN'^\circledast.$$

\end{notation}

\begin{remark}
The right $\mW$-enriched functor $$\Enr\Fun_{\mV,\emptyset}(\F,\G): \Enr\Fun_{\mV,\emptyset}(\mM,\mN)^\circledast \to \Enr\Fun_{\mV,\emptyset}(\mM',\mN')^\circledast$$ induces on underlying $\infty$-categories the canonical functor $\Enr\Fun_{\mV,\emptyset}(\mM,\mN) \to \Enr\Fun_{\mV',\emptyset}(\mM',\mN').$

\end{remark}

\begin{remark}\label{silu}
If $\mN^\circledast \to \mV^\ot \times \mW^\ot, \mN'^\circledast \to \mV^\ot \times \mW^\ot$ admit right tensors and $\G: \mN^\circledast \to \mN'^\circledast$ is a right $\mW$-linear functor,
then $\Enr\Fun_{\mV,\emptyset}(\F,\G): \Enr\Fun_{\mV,\emptyset}(\mM,\mN)^\circledast \to \Enr\Fun_{\mV,\emptyset}(\mM',\mN')^\circledast$ is a right $\mW$-linear functor
between weakly bienriched $\infty$-categories that admit right tensors by Corollary \ref{innerhosty}.

\end{remark}

\begin{proposition}\label{siewa}

\begin{enumerate}
\item Let $\mM^\circledast \to \mV^\ot, \mM'^\circledast \to \mV^\ot $ be absolute small weakly left enriched $\infty$-categories, $\F:\mM'^\circledast \to \mM^\circledast$ a left $\mV$-enriched functor and $\mN^\circledast \to \mV^\ot \times \mW^\ot$ 
a weakly bienriched $\infty$-category that admits small conical colimits and left and right tensors.
The right $\mW$-enriched functor $$\Enr\Fun_{\mV,\emptyset}(\F,\mN): \Enr\Fun_{\mV,\emptyset}(\mM,\mN)^\circledast \to \Enr\Fun_{\mV,\emptyset}(\mM',\mN)^\circledast$$ admits a right $\mW$-enriched left adjoint $$\F_!: \Enr\Fun_{\mV,\emptyset}(\mM',\mN)^\circledast \to \Enr\Fun_{\mV,\emptyset}(\mM,\mN)^\circledast.$$

\item Let $\mM^\circledast \to \mV^\ot, \mM'^\circledast \to \mV^\ot $ be small left enriched $\infty$-categories, $\F:\mM'^\circledast \to \mM^\circledast$ a left $\mV$-enriched functor and $\mN^\circledast \to \mV^\ot \times \mW^\ot$ 
a weakly bienriched $\infty$-category that admits small conical colimits and left and right tensors whose underlying weakly left enriched $\infty$-category is a left enriched $\infty$-category. 
The right $\mW$-enriched functor $$\Enr\Fun_{\mV,\emptyset}(\F,\mN): \Enr\Fun_{\mV,\emptyset}(\mM,\mN)^\circledast \to \Enr\Fun_{\mV,\emptyset}(\mM',\mN)^\circledast$$ admits a right $\mW$-enriched left adjoint $$\F_!: \Enr\Fun_{\mV,\emptyset}(\mM',\mN)^\circledast \to \Enr\Fun_{\mV,\emptyset}(\mM,\mN)^\circledast.$$

\end{enumerate}

\end{proposition}
\begin{proof} 
The induced functor $\F^*: \Enr\Fun_{\mV,\emptyset}(\mM,\mN)\to \Enr\Fun_{\mV,\emptyset}(\mM',\mN)$ admits a left adjoint $\F_!$ under the assumptions of (1) by Proposition \ref{laaaaan} (1) and under the assumptions of (2) by Proposition \ref{laaaaan} (2).
By Remark \ref{silu} the right $\mW$-enriched functor $$\Enr\Fun_{\mV,\emptyset}(\mM,\mN)^\circledast\to \Enr\Fun_{\mV,\emptyset}(\mM',\mN)^\circledast$$ induced by $\F$ preserves right tensors and source and target of it admit right tensors since $\mN^\circledast \to \mV^\ot \times \mW^\ot$ admits right tensors.
Using Definition \ref{enradj} (2) it is enough to prove that for every $\m \geq 0$ and $\W_1,...,\W_\m \in \mW$ and left $\mV$-linear functor $\rH: \mM^\circledast \to \mN^\circledast$ the following canonical morphism is an equivalence:
$$ (\F_! \circ ((-) \ot \W_{\m}) \circ ... \circ ((-) \ot \W_1))(\rH) 
\to (((-) \ot \W_{\m}) \circ ... \circ ((-) \ot \W_1) \circ \F_!)(\rH).$$ 
By Proposition \ref{laan} (1) the latter morphism induces at $\X \in \mM$ the canonical morphism
$$ \colim_{\V_1,..., \V_\n, \F(\Y) \to \X}(\bigotimes_{\bi=1}^\n \V_\bi \ot (((-) \ot \W_{\m}) \circ ... \circ ((-) \ot \W_1) \circ\rH(\Y))) $$$$\to ((-) \ot \W_{\m}) \circ ... \circ ((-) \ot \W_1) \circ(\colim_{\V_1,..., \V_\n, \F(\Y) \to \X}\ \bigotimes_{\bi=1}^\n \V_\bi \ot \rH(\Y)).$$
The latter morphism is an equivalence since the functors $(-)\ot \W_\bi: \mN \to \mN$ preserve small conical colimits for $1 \leq \bi \leq \m.$

\end{proof}

\subsection{The monad of enriched functors}


In this subsection we prove that enriched functors are monadic over non-enriched functors and 
compute the monad, thereby proving an enriched Yoneda-lemma (Theorem \ref{expli}, Corollary \ref{explicas}). We use the description of monad to obtain monadic resolutions of enriched functors (Proposition \ref{mondec}) and deduce a totalization formula for morphism objects of enriched $\infty$-categories of enriched functors (Corollary \ref{enrhom}).

\begin{lemma}\label{hep}Let $\mB $ be a locally small $\infty$-category and $\mC$ an $\infty$-category that admits small colimits.
\begin{enumerate}
\item For every $\Z \in \mB$ the functor $\ev^\mB_\Z: \Fun(\mB,\mC)\to \mC$ evaluating at $\Z \in \mB$ admits a left adjoint that sends $\Y \in \mC$ to the functor $\mB(\Z,-)\ot\Y:\mB \to\mC$.

\item Let $\mB$ be a small space. For every functor $\rH: \mB \to \mC$ there is a canonical equivalence $\rH \simeq \colim_{\Z \in \mB} \mB(\Z,-) \ot \rH(\Z)$ in $\Fun(\mB,\mC).$	

\item  Let $\mB$ be a small $\infty$-category. The functor $\G: \Fun(\mB,\mC) \to \Fun(\mB^\simeq,\mC)$ restricting along the inclusion $\mB^\simeq \subset \mB$ admits a left adjoint and is monadic. The left adjoint sends the functor
$\mB^\simeq(\Z,-) \ot \Y:\mB^\simeq \to\mC$ for $\Y \in \mC$, $\Z \in \mB$
to the functor $\mB(\Z,-) \ot \Y:\mB \to\mC$.

\item Let $\mB$ be a small $\infty$-category. The $\infty$-category $\Fun(\mB,\mC)$ is generated under small colimits by $\{\mB(\Z,-)\ot\Y:\mB \to\mC \mid \Z \in \mB, \Y \in \mC\}$.
\end{enumerate}
\end{lemma}

\begin{proof}
(1): By the Yoneda-lemma for $\infty$-categories there is a canonical equivalence $$\Fun(\mB,\mC)(\mB(\Z,-) \ot \Y,-) \simeq \Fun(\mB,\mS)(\mB(\Z,-),-) \circ\mC(\Y,-)_* \simeq \ev^\mB_\Z \circ \mC(\Y,-)_* \simeq \mC(\Y,-) \circ \ev^\mB_\Z.$$

(2) Let $\mB$ be a locally small space. Then for all functors $\rH,\rH': \mB \to \mC$ the canonical equivalence $$ \Fun(\mB, \mC)(\colim_{\Z \in \mB} \mB(\Z,-) \ot \rH(\Z),\rH') \simeq \lim_{\Z \in \mB} \Fun(\mB, \mC)(\mB(\Z,-) \ot \rH(\Z),\rH') $$$$\simeq\lim_{\Z \in \mB} \mC(\rH(\Z), \rH'(\Z)) \simeq \Fun(\mB, \mC)(\rH,\rH')$$ represents an equivalence $\rH \simeq \colim_{\Z \in \mB} \mB(\Z,-) \ot \rH(\Z)$
in $\Fun(\mB,\mC).$	

(3): Let $\mB$ be a small $\infty$-category. 
To see that $\G$ admits a left adjoint, we need to show that for overy $\rH \in \Fun(\mB^\simeq,\mC)$
the functor $\Fun(\mB^\simeq,\mC)(\rH,-) \circ \G: \Fun(\mB,\mC) \to \mS$
is corepresentable. The full subcategory of $\Fun(\mB^\simeq,\mC)$ spanned by the
functors $\rH$ such that the functor $\Fun(\mB^\simeq,\mC)(\rH,-) \circ \G: \Fun(\mB,\mC) \to \mS$ is corepresentable, is closed under small colimits
because $\mB$ is small and $\mC$ and so $\Fun(\mB,\mC)$ admit small colimits.
By the first part of the proof the $\infty$-category 
$\Fun(\mB^\simeq,\mC)$ is generated under small colimits by the functors
$\mB^\simeq(\Z,-) \ot \Y:\mB^\simeq \to\mC$ for $\Y \in \mC$, $\Z \in \mB$.
Therefore $\G$ admits a left adjoint if for every $\Y \in \mC$, $\Z \in \mB$ the functor $\Fun(\mB^\simeq,\mC)(\mB^\simeq(\Z,-) \ot \Y,-) \circ \G: \Fun(\mB^\simeq,\mC) \to \mS$ is corepresentable.
There is a canonical equivalence, which implies (2): $$\Fun(\mB^\simeq,\mC)(\mB^\simeq(\Z,-) \ot \Y,-) \circ \G 
\simeq \mC(\Y,-)\circ \ev^{\mB^\simeq}_\Z \circ \G \simeq \mC(\Y,-)\circ \ev^\mB_\Z 
\simeq \Fun(\mB,\mC)(\mB(\Z,-) \ot \Y,-).$$

(4): The functor $\G$ is conservative because the inclusion $\mB^\simeq \subset \mB$ is essentially surjective.
The functor $\G$ preserves small colimits since such a formed object-wise.
Thus $\G$ is monadic \cite[Theorem 4.7.3.5]{lurie.higheralgebra}.
This guarantees that $\Fun(\mB,\mC)$ is generated under small colimits by the essential image of $\F.$ By the first part of the proof $\Fun(\mB^\simeq,\mC)$ is generated under small colimits by the functors
$\mB^\simeq(\Z,-) \ot \Y:\mB^\simeq \to\mC$ for $\Y \in \mC$, $\Z \in \mB$
so that $\Fun(\mB,\mC)$ is generated under small colimits by the images under $\F$ of all functors
$\mB^\simeq(\Z,-) \ot \Y:\mB^\simeq \to\mC$ for $\Y \in \mC$, $\Z \in \mB$. 

\end{proof}

\begin{theorem}\label{expli}
Let $\mM^\circledast \to \mV^\ot \times \mW^\ot$ be a weakly bienriched $\infty$-category and $\mN^\circledast \to \mV^\ot \times \mW^\ot$ a weakly bienriched $\infty$-category that admits left and right tensors and small conical colimits. 
\begin{enumerate}
\item For every $\rH \in \Enr\Fun_{\mV, \mW}(\mM,\mN), \X \in \mM, \N \in \mN$ and $ \V_1,...,\V_\n \in \mV, \W_1,...,\W_\m \in \mW$ for $\n,\m \geq 0$ the following map is an equivalence: $$\Enr\Fun_{\mV, \mW}(\mM,\mN)(\L\Mor_{\bar{\mM}}(\X,-)\ot\N,\rH) \to
\mN(\N, \rH(\X)).$$
\item If 
$\mM$ is small, then $\Enr\Fun_{\mV, \mW}(\mM,\mN)$ is generated under small colimits by $\{\L\Mor_{\bar{\mM}}(\X,-) \ot \N \mid \X \in \mM, \N \in \mN\}$ and the forgetful functor $\Enr\Fun_{\mV, \mW}(\mM,\mN) \to \Fun(\mM, \mN)$ admits a left adjoint and is monadic.
The left adjoint sends 
$\mM(\Z,-) \ot \N$ to $\L\Mor_{\bar{\mM}}(\Z,-) \ot \N$ for every $\Z \in \mM, \N \in \mN$.

\vspace{1mm}
\item If $\mM$ is small, the forgetful functor $\nu: \Enr\Fun_{\mV, \mW}(\mM,\mN) \to \Fun(\mM^\simeq, \mN)$ admits a left adjoint $\phi$ and is monadic and for every 
$\F : \mM^\simeq \to \mN$ the canonical morphism
$$ \colim_{\Z \in \mM^\simeq}(\L\Mor_{\bar{\mM}}(\Z,-) \ot \F(\Z)) \to \phi(\F)$$
in $\Enr\Fun_{\mV, \mW}(\mM,\mN) $ is an equivalence.

\end{enumerate}

\end{theorem}

\begin{proof}
(1): Let $\bar{\mN}^\circledast \subset \mP\B\Env(\mN)^\circledast$ be the full weakly bienriched subcategory spanned by $\mN$.
By Lemma \ref{lekko} (5) the bienriched $\infty$-category $\bar{\mN}^\circledast \to \mP\Env(\mV)^\ot \times \mP\Env(\mW)^\ot$ is a bitensored $\infty$-category compatible with small colimits.
Consider the $\mP\Env(\mV), \mP\Env(\mW)$-enriched adjunction $$(-)\ot \X\ot(-) : (\mP\Env(\mV) \ot \mP\Env(\mW))^\circledast \rightleftarrows \mP\B\Env(\mM)^\circledast: \L\Mor_{\mP\B\Env(\mM)}(\X,-).$$
By Proposition \ref{corok} the right adjoint preserves small colimits and is $\mP\Env(\mV), \mP\Env(\mW)$-linear.
Hence by Remark \ref{indadj} the latter $\mP\Env(\mV), \mP\Env(\mW)$-enriched adjunction induces an adjunction
$$ \L\Mor_{\mP\B\Env(\mM)}(\X,-)^*: \LinFun^\L_{\mP\Env(\mV), \mP\Env(\mW)}(\mP\B\Env(\mM),\bar{\mN}) \rightleftarrows $$$$ \LinFun^\L_{\mP\Env(\mV), \mP\Env(\mW)}(\mP\Env(\mV)\ot \mP\Env(\mW),\bar{\mN}): ((-)\ot \X)^*.$$
Consider the commutative square
$$\begin{xy}
\xymatrix{\LinFun^\L_{\mP\Env(\mV), \mP\Env(\mW)}(\mP\B\Env(\mM),\bar{\mN}) \ar[d]^{ ((-)\ot \X)^*} \ar[rrr]^\simeq
&&& \Enr\Fun_{\mV,\mW}(\mM,\mN) \ar[d]^{\ev_\X} 
\\
\LinFun^\L_{\mP\Env(\mV), \mP\Env(\mW)}(\mP\Env(\mV)\ot\mP\Env(\mW),\bar{\mN})\ar[rrr]^\simeq &&& \mN.}
\end{xy}$$
By Proposition \ref{envvcor} and Proposition \ref{Line} the horizontal functors are equivalences.
Consequently, the right vertical functor admits a left adjoint that sends
$\N\in \mN$ to the following composition of $\mV,\mW$-enriched functors:
$$\mM^\circledast\subset \mP\B\Env(\mM)^\circledast\xrightarrow{\L\Mor_{\mP\B\Env(\mM)}(\X,-)} (\mP\Env(\mV)\ot\mP\Env(\mW))^\circledast \xrightarrow{(-)\ot\N\ot (-)} \mN^\circledast.$$
The unit of this adjunction 
is induced by the canonical morphism $\tu \to \L\Mor_{\mP\B\Env(\mM)}(\X,\X)$
in $ \mP\Env(\mV)\ot\mP\Env(\mW).$
This proves (1).



(2): Let $\theta: \Enr\Fun_{\mV, \mW}(\mM,\mN) \to \Fun(\mM,\mN)$ be the forgetful functor, which is conservative by \cref{conserva}.
We first show that $\theta$ admits a left adjoint. Let $\Theta \subset \Fun(\mM,\mN)$ be the full subcategory spanned by those $\rH$ such that the functor $\Fun(\mM,\mN)(\rH,-)\circ \theta: \Enr\Fun_{\mV, \mW}(\mM,\mN) \to \mS$ is corepresentable. 
The functor $\theta$ admits a left adjoint if and only if $\Theta= \Fun(\mM,\mN)$.
By Lemma \ref{colas} (1) the $\infty$-category 
$\Enr\Fun_{\mV, \mW}(\mM,\mN)$ admits small colimits preserved by $\theta$. By 
This guarantees that $\Theta$ is closed under small colimits.
By Lemma \ref{hep} (4) the $\infty$-category $\Fun(\mM,\mN)$ is generated under small colimits by the full subcategory $\{\mM(\Z,-) \ot \N:\mM \to\mN \mid\N \in \mN,\Z \in \mM\}$.
Therefore it is enough to see that for every $\N \in \mN$, $\Z \in \mM$
the functor $\Fun(\mM,\mN)(\mM(\Z,-) \ot \N,-)\circ \theta: \Enr\Fun_{\mV, \emptyset}(\mM,\mN) \to \mS$ is corepresentable.
By Lemma \ref{hep} (1) for every $\Z \in \mM$ the functor $$ \mN \to \Fun(\mM, \mN), \N \mapsto \mM(\Z,-) \ot \N $$ is left adjoint to the functor $\ev_\Z$ evaluating at $\Z.$
Thus by (1) the functor $$\Fun(\mM,\mN)(\mM(\Z,-) \ot \N,-)\circ \theta \simeq \mN(\N,-)\circ \ev_\Z \circ \theta: \Enr\Fun_{\mV, \mW}(\mM,\mN) \to \mS$$ is equivalent to the corepresentable functor
$$\Enr\Fun_{\mV, \mW}(\mM,\mN)(\L\Mor_{\bar{\mM}}(\Z,-) \ot \N,-): \Enr\Fun_{\mV, \mW}(\mM,\mN) \to \mS.$$ So $\theta$ admits a left adjoint that sends 
$\mM(\Z,-) \ot \N$ to $\L\Mor_{\bar{\mM}}(\Z,-) \ot \N$ for every $\Z \in \mM, \N \in \mN$.

By \cite[4.7.3.5.]{lurie.higheralgebra} this guarantees that $\theta$ is monadic so that $\Enr\Fun_{\mV, \mW}(\mM,\mN)$ is generated under small colimits by the essential image of the left adjoint of $\theta.$
Since the $\infty$-category $\Fun(\mM,\mN)$ is generated under small colimits by the functors $\mM(\Z,-) \ot \N:\mM \to\mN$ for $\Z \in \mM$, $\N \in \mN$,
the $\infty$-category $\Enr\Fun_{\mV, \mW}(\mM,\mN)$ is generated under small colimits by the objects $\L\Mor_{\bar{\mM}}(\Z,-) \ot \N$ for $\Z \in \mM, \N \in \mN$.

(3): By (2), Lemma \ref{hep} (3) and \cite[4.7.3.5.]{lurie.higheralgebra} the forgetful functor $$\nu: \Enr\Fun_{\mV, \mW}(\mM,\mN) \to \Fun(\mM^\simeq, \mN)$$ has a left adjoint $\phi$ and is monadic.
The left adjoint sends $\mM^\simeq(\Z,-) \ot \N$ to $\L\Mor_{\bar{\mM}}(\Z,-) \ot \N$ for every $\Z \in \mM, \N \in \mN$.

For every functor $\F : \mM^\simeq \to \mN$ the unit $\F \to \nu(\phi(\F))$ in
$\Fun(\mM^\simeq,\mN)$ gives rise to a map 
$$\rho: \colim_{\Z \in \mM^\simeq}(\L\Mor_{\bar{\mM}}(\Z,-) \ot \F(\Z)) \to \phi(\F) $$
in $\Enr\Fun_{\mV, \mW}(\mM,\mN)$
whose summand $\L\Mor_{\bar{\mM}}(\Z,-) \ot \F(\Z)) \to \phi(\F)$ for $\Z \in \mM^\simeq$
corresponds to the morphism $\F(\Z) \to \phi(\F)(\Z)$ in $\mN.$
Since $\nu, \phi$ preserve small colimits and $\mN$ admits small conical colimits,
the full subcategory of $\Fun(\mM^\simeq,\mN)$ spanned by the functors $\F$, for which the $\rho$ is an equivalence is closed under small colimits. So by Lemma \ref{hep} (4)
we can assume that $\F= \mM^\simeq(\Z,-)\ot \N: \mM^\simeq \to \mN$ for $\N \in \mN$ and $\Z \in \mM.$
By Lemma \ref{hep} (2) there is a canonical equivalence $$ \F \simeq \colim_{\Z \in \mM^\simeq} \mM^\simeq(\Z,-) \ot \F(\Z)$$ in $\Fun(\mM^\simeq,\mN).$
Note that $\nu \circ \L\Mor_{\bar{\mN}}(\F(\Z),-)_* \simeq \L\Mor_{\bar{\mN}}(\F(\Z),-)_* \circ \nu$ so that by adjointness
$$\phi \circ ((-)\ot \F(\Z))_* \simeq ((-)\ot \F(\Z))_* \circ \phi.$$
Hence there is a canonical equivalence $$\colim_{\Z \in \mM^\simeq} \L\Mor_{\bar{\mM}}(\Z,-) \ot \F(\Z) \simeq \colim_{\Z \in \mM^\simeq} \phi(\mM^\simeq(\Z,-)) \ot \F(\Z) $$$$ \simeq \colim_{\Z \in \mM^\simeq} \phi( \mM^\simeq(\Z,-) \ot \F(\Z)) \simeq \phi(\colim_{\Z \in \mM^\simeq} \mM^\simeq(\Z,-) \ot \F(\Z))\simeq \nu(\phi(\F)) $$ in $\Enr\Fun_{\mV, \mW}(\mM,\mN),$
which canonically identifies with $\rho.$

\end{proof}

\begin{corollary}\label{explicas}

Let $\mM^\circledast \to \mV^\ot $ be a weakly left enriched $\infty$-category
and $\mN^\circledast \to \mV^\ot\times\mW^\ot $ a weakly bienriched $\infty$-category
that admits left tensors.
\begin{enumerate}
\item For any $\rH \in \Enr\Fun_{\mV, \emptyset}(\mM,\mN), \X \in \mM, \N \in \mN$ the following canonical morphism is an equivalence: $$\R\Mor_{\overline{\Enr\Fun_{\mV, \emptyset}(\mM,\mN)}}(\L\Mor_{\bar{\mM}}(\X,-) \ot \N, \rH) \to \R\Mor_{\bar{\mN}}(\N, \rH(\X)).$$


\item If $\mM$ is small and $\mN^\circledast \to \mV^\ot\times\mW^\ot $ admits small conical colimits, 
the right $\mW$-enriched forgetful functor $\Enr\Fun_{\mV, \emptyset}(\mM,\mN)^\circledast \to \Fun(\mM, \mN)^\circledast$ admits a right $\mW$-enriched left adjoint that sends 
$\mM(\X,-) \ot \N$ to $\L\Mor_{\bar{\mM}}(\X,-) \ot \N$ for every $\X \in \mM, \N \in \mN$.

\item If $\mM$ is small and $\mN^\circledast \to \mV^\ot\times\mW^\ot $ admits small conical colimits, the right $\mW$-enriched forgetful functor $\nu: \Enr\Fun_{\mV, \emptyset}(\mM,\mN)^\circledast \to \Fun(\mM^\simeq, \mN)^\circledast$ admits a right $\mW$-enriched left adjoint $\phi$ and for every $\F : \mM^\simeq \to \mN$ the following canonical morphism in $\Enr\Fun_{\mV, \emptyset}(\mM,\mN) $ is an equivalence:
$$ \colim_{\Z \in \mM^\simeq}(\L\Mor_{\bar{\mM}}(\Z,-) \ot \F(\Z)) \to \phi(\F).$$
\end{enumerate}

\end{corollary}

\begin{proof}
Statement (1) is equivalent to say that for every $\X\in \mM$ the right $\mW$-enriched functor $$\gamma: \Enr\Fun_{\mV, \emptyset}(\mM,\mN)^\circledast \to \mN^\circledast$$ evaluating at $\X$ admits a right $\mW$-enriched left adjoint and the unit of the adjunction is induced by the canonical morphism
$\tu \to \L\Mor_{\bar{\mM}}(\X,\X)$ in $\mP\Env(\mV).$ 
By Theorem \ref{expli} (1) the functor $\Enr\Fun_{\mV, \emptyset}(\mM,\mN) \to \mN$ evaluating at $\X$ admits a left adjoint that sends $\N\in \mN$ to $ \L\Mor_{\bar{\mM}}(\X,-)\ot\N$, and the unit of this adjunction is induced by the canonical morphism $\tu \to \L\Mor_{\bar{\mM}}(\X,\X)$ in $\mP\Env(\mV).$ 
By Proposition \ref{innerhosty} the right $\mW$-enriched functor $\gamma$ is a right $\mW$-linear functor between weakly right $\mW$-enriched $\infty$-categories that admit right tensors. Therefore by Lemma \ref{Adj} the result follows from the fact that 
for every $\N\in \mN, \W \in \mW$ the morphism
$$\kappa: \L\Mor_{\bar{\mM}}(\X,-)\ot(\N \ot \W) \to (\L\Mor_{\bar{\mM}}(\X,-)\ot\N) \ot \W $$ corresponding to the morphism $\tu \ot \N \ot \W \to \L\Mor_{\bar{\mM}}(\X,\X)\ot\N \ot \W$ is the canonical equivalence.

(2): By Proposition \ref{innerhosty} the right $\mW$-enriched forgetful functor $$\Enr\Fun_{\mV, \emptyset}(\mM,\mN)^\circledast \to \Fun(\mM, \mN)^\circledast$$ is a right $\mW$-linear functor between weakly right $\mW$-enriched $\infty$-categories that admit right tensors. By Theorem \ref{expli} (2) the forgetful functor $\Enr\Fun_{\mV, \emptyset}(\mM,\mN) \to \Fun(\mM, \mN)$ admits a left adjoint $\phi.$ 
Thus by Lemma \ref{Adj} it is enough to see that for every $\rH\in \Fun(\mM,\mN), \W \in \mW$ the canonical morphism
$\lambda: \phi(\rH\ot\W)\to \phi(\rH) \ot \W $ 
is an equivalence.
By Lemma \ref{hep} (2) the $\infty$-category $\Fun(\mM,\mN)$ is generated under small colimits by the full subcategory $\{\mM(\X,-) \ot \N:\mM \to\mN \mid\N \in \mN,\X \in \mM\}$.
By Proposition \ref{innerhosty} forming right tensors in $\Fun(\mM, \mN)$ and $\Enr\Fun_{\mV, \emptyset}(\mM,\mN)$ preserves small colimits. So it suffices to show that
$\lambda$ is an equivalence for $\rH=\mM(\X,-) \ot \N:\mM \to\mN$ for any $\N \in \mN,\X \in \mM$. In this case $\lambda$ identifies with $\kappa.$	
The proof of (3) is similar to the one of (2), where we use Theorem \ref{expli} (3). 

\end{proof}

\begin{corollary}\label{yotul}

Let $\mV^\ot \to \Ass$ be a monoidal $\infty$-category compatible with small colimits and $\mM^\circledast \to \mV^\ot $ a small left enriched $\infty$-category.
\begin{enumerate}
\item Then $\Enr\Fun_{\mV,\emptyset}(\mM,\mV)$ is generated under small colimits and left tensors by the full subcategory $$\{\L\Mor_{\bar{\mM}}(\X,-) \mid \X \in \mM \}.$$ 

\item For every $\X \in \mM$ the right $\mV$-enriched functor $$\R\Mor_{\overline{\Enr\Fun_{\mV, \emptyset}(\mM,\mV)}}(\L\Mor_\mM(\X,-),-):\Enr\Fun_{\mV, \emptyset}(\mM,\mV)^\circledast \to \mV^\circledast $$
is equivalent to the functor evaluating at $\X$ and so preserves small colimits and is right $\mV$-linear.

\end{enumerate}

\end{corollary}


\begin{remark}\label{Func}
Let $\mM^\circledast \to \mV^\ot, \mN^\circledast \to \mV^\ot $ be left enriched $\infty$-categories, $\F:\mM^\circledast \to \mN^\circledast$ a left $\mV$-enriched functor and $\X \in \mM.$
The morphism $\tu_\mV \to \L\Mor_\mN(\F(\X),\F(\X))$ in $\mV$
corresponds to a morphism $\L\Mor_\mM(\X,-)\to \L\Mor_\mN(\F(\X),-)\circ \F$
in $\Enr\Fun_{\mV,\emptyset}(\mM,\mV)$ whose component at any $\Y \in \mM$
is the canonical morphism $\L\Mor_\mM(\X,\Y)\to \L\Mor_\mN(\F(\X),\F(\Y))$ in $\mV.$
\end{remark}

\begin{proposition}\label{mondec}
Let $\mM^\circledast \to \mV^\ot \times \mW^\ot$ be a small weakly bienriched $\infty$-category, $\mN^\circledast \to \mV^\ot \times \mW^\ot$ a weakly bienriched $\infty$-category that admits left and right tensors and small conical colimits.
Every $\mV,\mW$-enriched functor $\F : \mM^\circledast \to \mN^\circledast$ is the colimit of a canonical simplicial object $\mX$ in $ \Enr\Fun_{\mV, \mW}(\mM,\mN)$ such that for every $\n \geq 0$ there is a canonical equivalence in $\Enr\Fun_{\mV, \mW}(\mM,\mN):$ $$\mX_\n \simeq \colim_{\Z_1, ...., \Z_\n \in \mM^\simeq}\L\Mor_{\bar{\mM}}(\Z_{\n},-)\ot \L\Mor_{\bar{\mM}}(\Z_{\n-1},\Z_\n)\ot ... \ot \L\Mor_{\bar{\mM}}(\Z_2,\Z_3) \ot \L\Mor_{\bar{\mM}}(\Z_1,\Z_2) \ot \F(\Z_1).$$
	
\end{proposition}

\begin{proof}

By Theorem \ref{expli} (3) the forgetful functor $\nu: \Enr\Fun_{\mV, \mW}(\mM,\mN) \to \Fun(\mM^\simeq, \mN)$ admits a left adjoint $\phi$ and is monadic.
This guarantees that $\F$ admits a monadic resolution, i.e. $\F$ is the colimit of a canonical simplicial object $\mX$ in $\Enr\Fun_{\mV, \mW}(\mM,\mN)$ such that for every $\n \geq 0$ there is a canonical equivalence in $\Enr\Fun_{\mV, \mW}(\mM,\mN) :$$$\mX_\n \simeq (\phi \circ \nu)^{\circ\n}(\F).$$ 
By Theorem \ref{expli} (3) and induction over $\n \geq 0$ there is a canonical equivalence in $\Enr\Fun_{\mV, \mW}(\mM,\mN): $ $$\mX_\n \simeq (\phi \circ \nu)^{\circ\n}(\F) \simeq $$$$ \colim_{\Z_1, ...., \Z_\n \in \mM^\simeq}\L\Mor_{\bar{\mM}}(\Z_{\n},-)\ot \L\Mor_{\bar{\mM}}(\Z_{\n-1},\Z_\n)\ot ... \ot \L\Mor_{\bar{\mM}}(\Z_2,\Z_3) \ot \L\Mor_{\bar{\mM}}(\Z_1,\Z_2) \ot \F(\Z_1).$$
	
\end{proof}

\begin{corollary}
Let $\mM^\circledast \to \mV^\ot \times \mW^\ot$ be a small weakly bienriched $\infty$-category and $\mN^\circledast \to \mV^\ot \times \mW^\ot$ a weakly bienriched $\infty$-category that admits left and right tensors.
For every $\mV,\mW$-enriched functors $\F,\G : \mM^\circledast \to \mN^\circledast$ there is a canonical equivalence of spaces $$\Enr\Fun_{\mV, \mW}(\mM,\mN)(\F,\G) \simeq $$$$ \underset{[\n]\in \Delta^\op}{\lim} \lim_{\Z_1, ...., \Z_\n \in \mM^\simeq} \mN(\L\Mor_{\bar{\mM}}(\Z_{\n-1},\Z_\n)\ot ... \ot \L\Mor_{\bar{\mM}}(\Z_1,\Z_2) \ot \F(\Z_1), \G(\Z_\n)).$$
	
\end{corollary}


Corollary \ref{explicas} (1) implies the following:
\begin{corollary}\label{enrhom}
Let $\mM^\circledast \to \mV^\ot \times \mW^\ot$ be a small weakly bienriched $\infty$-category and $\mN^\circledast \to \mV^\ot \times \mW^\ot$ a weakly bienriched $\infty$-category that admits left tensors.
For every left $\mV$-enriched functors $\F,\G : \mM^\circledast \to \mN^\circledast$ there is a canonical equivalence in $\mW:$ $$\R\Mor_{\overline{\Enr\Fun_{\mV, \emptyset}(\mM,\mN)}}(\F,\G) \simeq $$$$ \underset{[\n]\in \Delta^\op}{\lim} \lim_{\Z_1, ...., \Z_\n \in \mM^\simeq}  \R\Mor_{\bar{\mN}}(\L\Mor_{\bar{\mM}}(\Z_{\n-1},\Z_\n) \ot  ...\ot \L\Mor_{\bar{\mM}}(\Z_1,\Z_2) \ot \F(\Z_1),\G(\Z_\n)).$$
	
\end{corollary}

\subsection{Enriched Yoneda-embedding and opposite enrichment}

Next we use the enriched Yoneda-lemma (Corollary \ref{explicas}) to construct an opposite enriched $\infty$-category and an enriched Yoneda-embedding.


\begin{definition}\label{notori}
Let $\mM^\circledast \to \mV^\ot $ be a left enriched $\infty$-category.
The opposite right $\mV$-enriched $\infty$-category of $\mM^\circledast \to \mV^\ot$
is the full weakly right $\mV$-enriched subcategory 
$$(\mM^\op)^\circledast \subset \Enr\Fun_{\mV, \emptyset}(\mM,\mV)^\circledast $$ 
spanned by all left $\mV$-enriched functors $\Mor_\mM(\X,-):\mM^\circledast \to \mV^\circledast$ for $\X \in \mM.$

\end{definition}

The following remark, proposition and lemma justify the terminology of Definition \ref{notori}:

\begin{remark}
By construction \ref{Enros} and Corollary \ref{MorFu} for every weakly left enriched $\infty$-category $\mM^\circledast \to \mV^\ot$ there is an embedding
$$\xi: \mM^\op \to \Enr\Fun_{\mV, \emptyset}(\mM,\mP\Env(\mV))^\circledast$$
sending $\X\in \mM$ to $\L\Mor_{\bar{\mM}}(\X,-).$
If $\mM^\circledast \to \mV^\ot$ is a left enriched $\infty$-category, by Remark \ref{EEnr} the embedding $\xi$ induces an embedding
$\mM^\op \to \Enr\Fun_{\mV, \emptyset}(\mM,\mV)^\circledast \subset \Enr\Fun_{\mV, \emptyset}(\mM,\mP\Env(\mV))^\circledast$
denoted by the same name, which sends $\X\in \mM$ to $\L\Mor_{\bar{\mM}}(\X,-).$
The embedding $\xi$ 
induces an equivalence $\mM^\op \simeq (\mM^\op)^\circledast \times_{\mV^\ot} \emptyset^\ot $.
	
\end{remark}
\begin{proposition}\label{oppoen}

Let $\mM^\circledast \to \mV^\ot $ be a left enriched $\infty$-category.
The weakly right enriched $\infty$-category $(\mM^\op)^\circledast \to \mV^\ot$ exhibits $\mM^\op$ as right enriched in $\mV$, where for any $\X,\Y \in \mM$ there is a canonical equivalence $$\R\Mor_{\mM^\op}(\X,\Y) \simeq \L\Mor_\mM(\Y,\X).$$	
\end{proposition}

\begin{proof}
By Corollary \ref{explicas} for any $\X,\Y \in \mM$ and $\V_1,...,\V_\n \in \mV$ for $\n \geq 0$ the canonical map
$$\Mul_{\mM^\op}(\V_1,...,\V_\n, \L\Mor_\mM(\X,-);\L\Mor_\mM(\Y,-)) \to  \Mul_{\mV}(\V_1,...,\V_\n, \L\Mor_\mM(\Y,\X)) $$ is an equivalence.
So the result follows.
\end{proof}

\begin{notation}Let $\mV^\ot \to \Ass, \mW^\ot \to \Ass$ be small $\infty$-operads.
Via the canonical equivalences
$$ {_\mV\omega\B\Enr}_{\mW} \simeq {_{\mP(\Env(\mW)^\rev\times \Env(\mV))}\L\Enr_\emptyset}, {_\mW\omega\B\Enr}_{\mV} \simeq {_\emptyset\R\Enr}_{\mP(\Env(\mW)^\rev\times \Env(\mV))}$$
of Corollary \ref{coronn} we define for every weakly bienriched $\infty$-category  $\mM^\circledast \to \mV^\ot \times \mW^\ot $ an opposite weakly bienriched $\infty$-category $(\mM^\op)^\circledast \to \mW^\ot \times \mV^\ot $
extending Definition \ref{notori}.

\end{notation}

\begin{lemma}\label{aos}
Let $\mM^\circledast \to \mV^\ot, \mN^\ot \to \mW^\ot $ be left enriched $\infty$-categories.
The right $\mV\times \mW$-enriched functor $$\rho: \Enr\Fun_{\mV, \emptyset}(\mM,\mV)^\circledast \times_\Ass \Enr\Fun_{\mV, \emptyset}(\mN,\mV)^\circledast \to \Enr\Fun_{\mV \times \mW, \emptyset}(\mM \times \mN,\mV\times \mW)^\circledast$$
corresponding to the $\mV\times \mW,\mV\times \mW$-enriched functor $$(\mM^\circledast \times_\Ass \mN^\circledast) \times (\Enr\Fun_{\mV, \emptyset}(\mM,\mV)^\circledast \times_\Ass \Enr\Fun_{\mV, \emptyset}(\mN,\mV)^\circledast) \simeq $$$$ (\mM^\circledast \times \Enr\Fun_{\mV, \emptyset}(\mM,\mV)^\circledast) \times_{\Ass \times \Ass} (\mN^\circledast \times \Enr\Fun_{\mV, \emptyset}(\mN,\mV)^\circledast) \to \mV^\circledast \times_{\Ass \times \Ass} \mW^\circledast,$$
which is the product of the $\mV,\mV$-enriched and $\mW,\mW$-enriched evaluation functors,
restricts to a right $\mV\times \mW$-enriched equivalence $$ \rho': (\mM^\op)^\circledast \times_\Ass (\mN^\op)^\circledast \simeq ((\mM \times \mN)^\op)^\circledast.$$
\end{lemma}
\begin{proof}
	
The enriched functor $\rho$ sends $\L\Mor_\mM(\X,-), \L\Mor_\mN(\Y,-)$ for $\X \in \mM, \Y \in \mN$ to $$\L\Mor_\mM(\X,-) \times \L\Mor_\mN(\Y,-) \simeq \L\Mor_{\mM \times \mN}((\X,\Y),-)$$ and so restricts to $\rho'$. By Corollary \ref{yotul} the enriched functor $\rho'$ induces on right morphism objects the following equivalence for $\X' \in \mM, \Y' \in \mN$:
$$ \L\Mor_\mM(\X',\X)\times \L\Mor_\mN(\Y',\Y) \simeq $$$$\R\Mor_{\mM^\op}(\L\Mor_\mM(\X,-),\L\Mor_\mM(\X',-)) \times \R\Mor_{\mN^\op}(\L\Mor_\mN(\Y,-),\L\Mor_\mN(\Y',-))\to $$$$ \R\Mor_{(\mM \times \mN)^\op}(\L\Mor_{\mM \times \mN}((\X,\Y),-),\L\Mor_{\mM \times \mN}((\X',\Y'),-)) \simeq \L\Mor_{\mM \times \mN}((\X',\Y'),(\X,\Y)).$$
\end{proof}

\begin{proposition}\label{dua} Let $\mM^\circledast \to \mV^\ot $ be a left enriched $\infty$-category, $\mO^\circledast \to \mW^\ot$ a right enriched $\infty$-category  and $\mN^\circledast \to \mV \times \mW^\ot $ a bienriched $\infty$-category.

\begin{enumerate}
\item There is a $\mW,\mV$-enriched equivalence
$$ ((\mM \times \mO)^\op)^\circledast \simeq (\mM^\op)^\circledast \times (\mO^\op)^\circledast.$$

\item There is a right $\mW$-enriched equivalence $$ \Enr\Fun_{\mV, \emptyset}(\mM,\mN) \simeq \Enr\Fun_{\emptyset, \mV}(\mM^\op,\mN^\op)^\op.$$

\end{enumerate}

\end{proposition}

\begin{proof}(1): By Lemma \ref{aos} there is a left $\mV,\mW^\rev$-enriched equivalence
$$ (\mM \times \mO^\rev)^\op)^\circledast \simeq (\mM^\op)^\circledast \times_\Ass ((\mO^\op)^\rev)^\circledast.$$
Under the equivalence ${_\mV\B\P\Enr_\mW} \simeq {_{\mV\times\mW^\rev}\L\P\Enr}_\emptyset$
the latter equivalence corresponds to the $\mV,\mW$-enriched equivalence
$((\mM \times \mO)^\op)^\circledast \simeq (\mM^\op)^\circledast \times (\mO^\op)^\circledast.$

(2): The claimed equivalence is represented by the following chain of natural equivalences:
$${_\emptyset\omega\B\Enr_\mW}(\mO, \Enr\Fun_{\emptyset, \mV}(\mM^\op,\mN^\op)^\op) \simeq {_\mW\omega\B\Enr_\emptyset}(\mO^\op, \Enr\Fun_{\emptyset, \mV}(\mM^\op,\mN^\op))$$$$
\simeq {_\mW\omega\B\Enr_\mV}(\mO^\op \times \mM^\op,\mN^\op) 
\simeq {_\mW\omega\B\Enr_\mV}((\mM \times \mO)^\op,\mN^\op) \simeq $$$${_\mV\omega\B\Enr_\mW}(\mM \times \mO,\mN) \simeq {_\emptyset\omega\B\Enr_\mW}(\mO, \Enr\Fun_{\mV, \emptyset}(\mM,\mN)).$$	
	
\end{proof}



\begin{notation}\label{nnoo}Let $\mM^\circledast \to \mV^\ot $ be a left enriched $\infty$-category.
\begin{enumerate}
\item The $\mV,\mV$-enriched left morphism object functor $$\L\Mor_\mM: \mM^\circledast \times (\mM^\op)^\circledast \to \mV^\circledast$$ is the restricted $\mV,\mV$-enriched evaluation functor $$\mM^\circledast \times (\mM^\op)^\circledast \subset \mM^\circledast \times  \Enr\Fun_{\mV, \emptyset}(\mM,\mV)^\circledast \to \mV^\circledast.$$
	
\item The left $\mV$-enriched Yoneda-embedding $$\rho_\mM: \mM^\circledast \to \Enr\Fun_{\emptyset, \mV}(\mM^\op,\mV)^\circledast$$
is the left $\mV$-enriched functor corresponding to the left morphism object functor $\L\Mor_\mM: \mM^\circledast \times (\mM^\op)^\circledast \to \mV^\circledast $
under the equivalence of Theorem \ref{vvvl}.
\end{enumerate}
\end{notation}

\begin{remark}
Notation \ref{nnoo} also applies to any right $\mV$-enriched $\infty$-category
$\mM^\circledast \to \mV^\ot$ (viewed as a left $\mV^\rev$-enriched $\infty$-category).
In this case we obtain a $\mV,\mV$-enriched right morphism object functor $$\R\Mor_\mM: (\mM^\op)^\circledast \times \mM^\circledast \to \mV^\circledast,$$ which is the restricted $\mV,\mV$-enriched evaluation functor $$(\mM^\op)^\circledast  \times \mM^\circledast \subset \Enr\Fun_{\emptyset,\mV}(\mM,\mV)^\circledast \times \mM^\circledast \to \mV^\circledast,$$
and a corresponding right $\mV$-enriched Yoneda-embedding $\rho_\mM: \mM^\circledast \to \Enr\Fun_{\mV,\emptyset}(\mM^\op,\mV)^\circledast.$
\end{remark}

\begin{remark}\label{Reyy}
The $\mV,\mV$-enriched left morphism object functor $\L\Mor_\mM: \mM^\circledast \times (\mM^\op)^\circledast \to \mV^\circledast$
induces on underlying weakly left $\mV$-enriched $\infty$-categories the left 
$\mV$-enriched functor $\Gamma_\mM: \mM^\circledast \times \mM^\op  \to \mV^\circledast$
of Construction \ref{Enros} corresponding to $\xi.$ 
\end{remark}


\begin{lemma}\label{exhibl}
Let $\mV^\ot \to \Ass$ be a monoidal $\infty$-category compatible with small colimits, $\mM^\circledast \to \mV^\ot$ a small left enriched $\infty$-category
and $\Z\in \mM.$ By Theorem \ref{expli} the canonical morphism $$\tu_\mV \to \rho(\Z)(\xi(\Z)) \simeq \L\Mor_\mM(\Z,\Z)$$ in $\mV$ corresponds to a morphism $$\kappa: \R\Mor_{\mM^\op}(\xi(\Z),-) \to \rho(\Z)$$ 
in $\Enr\Fun_{\emptyset, \mV}(\mM^\op,\mV)$.
The morphism $\kappa$ is an equivalence.
\end{lemma}

\begin{proof}
The right $\mV$-enriched functor $\rho(\Z): (\mM^\op)^\circledast \to \mV^\circledast$ induces a morphism
$$\R\Mor_{\mM^\op}(\xi(\Z),-) \to \R\Mor_{\mP\Env(\mV)}(\L\Mor_\mM(\Z,\Z),-) \circ \rho(\Z) \to \rho(\Z)$$ in $\Enr\Fun_{\emptyset, \mV}(\mM^\op,\mV)$ whose component at $\rH \in \mM^\op$ is the canonical morphism $$\R\Mor_{\mM^\op}(\xi(\Z),\rH) \to \R\Mor_{\mP\Env(\mV)}(\L\Mor_\mM(\Z,\Z),\rH(\Z)) \to \rH(\Z)$$ in $\mV$, which is an equivalence by Corollary \ref{explicas}.

	
\end{proof}

\begin{proposition}\label{yofaith}
Let $\mM^\circledast \to \mV^\ot $ be a small left enriched $\infty$-category. The left $\mV$-enriched Yoneda-embedding $\rho: \mM^\circledast \to \Enr\Fun_{\emptyset, \mV}(\mM^\op,\mV)^\circledast $ induces an equivalence $\mM^\circledast \to ((\mM^\op)^\op)^\circledast.$	
\end{proposition}

\begin{proof}
By Lemma \ref{exhibl} the left $\mV$-enriched Yoneda-embedding $\rho_\mM: \mM^\circledast \to\Enr\Fun_{\mV,\emptyset}(\mM^\op,\mV)^\circledast$ induces a left $\mV$-enriched functor $\mM^\circledast \to ((\mM^\op)^\op)^\circledast$ whose underlying functor $\mM \to (\mM^\op)^\op$ is essentially surjective.	
By Remark \ref{Func} the left $\mV$-enriched functor $\rho: \mM^\circledast \to (\mM^\op)^\op)^\circledast$ induces a morphism
$$\kappa: \L\Mor_\mM(\Z ,-) \to \L\Mor_{(\mM^\op)^\op}(\rho(\Z ),-)\circ \rho$$
in $\Enr\Fun_{\mV,\emptyset}(\mM,\mV)$.	
By Remark \ref{Func} the left $\mV$-enriched functor $\ev_{\xi(\Z)}: \Enr\Fun_{\emptyset, \mV}(\mM^\op,\mV)^\circledast \to \mV^\circledast$ evaluating at $\xi(\Z) \in \mM^\op$
induces the following morphism $\lambda$ in $\Enr\Fun_{\mV,\emptyset}(\mM,\mV) $:
$$\L\Mor_{(\mM^\op)^\op}(\rho(\Z ),-)\circ \rho \to \L\Mor_{\mP\Env(\mV)}(\L\Mor_\mM(\Z,\Z),-)\circ\ev_{\xi(\Z)} \circ \rho \to $$$$ \ev_{\xi(\Z)} \circ \rho \simeq \xi(\Z)=\L\Mor_\mM(\Z ,-),$$
which is an equivalence by Corollary \ref{explicas}.
By Corollary \ref{explicas} the composition $\lambda \circ \kappa$
is the identity because $\lambda_\Z  \circ \kappa_\Z $ sends the identity of $\Z $ to the identity of $\Z .$ So the result follows.
\end{proof}

\begin{remark}
Let $\mM^\circledast \to \mV^\ot $ be a small left enriched $\infty$-category. 
By Proposition \ref{yofaith} the left $\mV$-enriched Yoneda-embedding $\rho_\mM: \mM^\circledast \to \Enr\Fun_{\emptyset, \mV}(\mM^\op,\mV)^\circledast $ induces an equivalence $\rho'_\mM: \mM^\circledast \to ((\mM^\op)^\op)^\circledast.$	
By definition $\rho_\mM: \mM^\circledast \to \Enr\Fun_{\emptyset, \mV}(\mM^\op,\mV)^\circledast$ corresponds to the
$\mV,\mV$-enriched left morphism object functor
$\L\Mor_\mM: \mM^\circledast \times (\mM^\op)^\circledast \to \mV^\circledast.$
Thus $\L\Mor_\mM: \mM^\circledast \times (\mM^\op)^\circledast \to \mV^\circledast$
factors as $\mV,\mV$-enriched functors $$\mM^\circledast \times (\mM^\op)^\circledast \xrightarrow{\rho_\mM\times (\mM^\op)^\circledast}\Enr\Fun_{\emptyset, \mV}(\mM^\op,\mV)^\circledast \times (\mM^\op)^\circledast \to \mV^\circledast,$$
which factors as 
$$\mM^\circledast \times (\mM^\op)^\circledast \xrightarrow{\rho'_\mM\times (\mM^\op)^\circledast} (\mM^\op)^\op)^\circledast \times (\mM^\op)^\circledast \xrightarrow{\R\Mor_{\mM^\op}} \mV^\circledast.$$
In particular, $\rho_{\mM^\op}: (\mM^\op)^\circledast \to \Enr\Fun_{\mV,\emptyset}((\mM^\op)^\op,\mV)^\circledast \simeq \Enr\Fun_{\mV,\emptyset}(\mM,\mV)^\circledast,$ where the latter equivalence precomposes with $\rho'$, is the tautological embedding $(\mM^\op)^\circledast \subset \Enr\Fun_{\mV,\emptyset}(\mM,\mV)^\circledast.$


\end{remark}

\begin{lemma}\label{siewal}
Let $\mV^\ot \to \Ass$ be a monoidal $\infty$-category compatible with small colimits, $\mM^\circledast \to \mV^\ot, \mN^\circledast \to \mV^\ot $ small left enriched $\infty$-categories and $\F:\mM^\circledast \to \mN^\circledast$ a left $\mV$-enriched functor.
For every $\X \in \mM$ the morphism $$\kappa: \F_!(\L\Mor_\mM(\X,-)) \to \L\Mor_\mN(\F(\X),-)$$ in $\Enr\Fun_{\mV,\emptyset}(\mN,\mV)$
corresponding to the 
morphism $\tu_\mV \to \L\Mor_\mN(\F(\X),\F(\X))$ given by the identity is an equivalence.
 
\end{lemma} 

\begin{proof}
The morphism $\kappa$ induces for every $\rH \in \Enr\Fun_{\mV,\emptyset}(\mN,\mV)$
the identity
$$ \rH(\F(\X)) \simeq \Enr\Fun_{\mV,\emptyset}(\mN,\mV)(\Mor_\mN(\F(\X),-),\rH) \to \Enr\Fun_{\mV,\emptyset}(\mN,\mV)(\F_!(\Mor_\mM(\X,-)),\rH)$$$$ \simeq \Enr\Fun_{\mV,\emptyset}(\mM,\mV)(\Mor_\mM(\X,-),\rH \circ \F) \simeq \rH(\F(\X)),$$
where we use Corollary \ref{explicas}.
\end{proof}

\begin{lemma}\label{siewalt}
Let $\mV^\ot \to \Ass$ be a presentably monoidal $\infty$-category, $\mM^\circledast \to \mV^\ot, \mN^\circledast \to \mV^\ot $ small left enriched $\infty$-categories and $\F:\mM^\circledast \to \mN^\circledast$ a left $\mV$-enriched functor.
There is a commutative square:
$$\begin{xy}
\xymatrix{\mM^\op  \ar[d]^{\xi_\mM}\ar[rr]^{\F^\op}
&& \mN^\op \ar[d]^{\xi_\mN} 
\\
\Enr\Fun_{\mV,\emptyset}(\mM,\mV) \ar[rr]^{\F_!} && \Enr\Fun_{\mV,\emptyset}(\mN,\mV).}
\end{xy}$$
	
\end{lemma} 

\begin{proof}

The enriched adjunction $\F_!: \mP\L\Env(\mM)^\circledast\rightleftarrows\mP\L\Env(\mN)^\circledast:\F^*$
gives rise to an adjunction $(\F^*)^*: \Enr\Fun^\R_{\mP\Env(\mV),\emptyset}(\mP\L\Env(\mM),\mP\Env(\mV)) \rightleftarrows \Enr\Fun^\R_{\mP\Env(\mV),\emptyset}(\mP\L\Env(\mN),\mP\Env(\mV)): (\F_!)^*.$

There is a commutative diagram
$$\begin{xy}
\xymatrix{\mM^\op\ar[d] \ar[r]^{\F^\op} & \mN^\op \ar[d]
\\
\mP\L\Env(\mM)^\op \ar[d]^\simeq\ar[r]^{}
& \mP\L\Env(\mN)^\op \ar[d]^{\simeq} 
\\
\Enr\Fun^\L_{\mP\Env(\mV),\emptyset}(\mP\Env(\mV), \mP\L\Env(\mM))^\op \ar[r]^{}\ar[d]^\simeq & \Enr\Fun^\L_{\mP\Env(\mV),\emptyset}(\mP\Env(\mV),\mP\L\Env(\mN))^\op\ar[d]^\simeq
\\
\Enr\Fun^\R_{\mP\Env(\mV),\emptyset}(\mP\L\Env(\mM), \mP\Env(\mV)) \ar[r]^{(\F^*)^*} 
& \Enr\Fun^\R_{\mP\Env(\mV),\emptyset}(\mP\L\Env(\mN),\mP\Env(\mV))}
\end{xy}$$	
Let $\alpha_\mM$ be the left vertical functor in the diagram so that $\alpha_\mN$ is the right vertical functor in the diagram.
The diagram provides an equivalence
$(\F^*)^* \circ \alpha_\mM \simeq \alpha_\mN \circ \F^\op.$
Composing the composition $$ \alpha_\mM \to (\F_!)^*\circ (\F^*)^* \circ \alpha_\mM \simeq (\F_!)^*\circ \alpha_\mN \circ \F^\op $$
of maps of functors $\mM^\op \to \Enr\Fun^\R_{\mP\Env(\mV),\emptyset}(\mP\L\Env(\mM), \mP\Env(\mV))$
with the functor $$\Enr\Fun^\R_{\mP\Env(\mV),\emptyset}(\mP\L\Env(\mM), \mP\Env(\mV))
\to \Enr\Fun_{\mV,\emptyset}(\mM, \mP\Env(\mV))$$
gives a map $ \xi_\mM \to \F^* \circ \xi_\mN \circ \F^\op $
of functors $\mM^\op \to \Enr\Fun_{\mV,\emptyset}(\mM, \mV) \subset \Enr\Fun_{\mV,\emptyset}(\mM, \mP\Env(\mV)).$
The composition 
$$ \F_! \circ \xi_\mM \to \F_!\circ \F^* \circ \xi_\mN \circ \F^\op \to \xi_\mN \circ \F^\op $$
induces at $\Z \in \mM$ 
the morphism $\kappa: \F_!(\L\Mor_\mM(\Z,-)) \to \L\Mor_\mN(\F(\Z),-)$
in $\Enr\Fun_{\mV,\emptyset}(\mN,\mV)$
corresponding to the morphism $\tu_\mV \to \L\Mor_\mN(\F(\Z),\F(\Z))$ given by the identity. By Lemma \ref{siewal} the morphism $\kappa$ is an equivalence.


\end{proof}

\begin{corollary}\label{siewalty}
Let $\mV^\ot \to \Ass$ be a monoidal $\infty$-category compatible with small colimits, $\mM^\circledast \to \mV^\ot, \mN^\circledast \to \mV^\ot $ small left enriched $\infty$-categories and $\F:\mM^\circledast \to \mN^\circledast$ a left $\mV$-enriched functor.
The right $\mV$-enriched functor $$\F_!: \Enr\Fun_{\mV,\emptyset}(\mM,\mV)^\circledast \to \Enr\Fun_{\mV,\emptyset}(\mN,\mV)^\circledast$$
restricts to a right $\mV$-enriched functor $(\mM^\op)^\circledast \to (\mN^\op)^\circledast$, which we denote by $\F^\op.$	
	
\end{corollary} 

\begin{construction}\label{Varz}
Let $\mV^\ot\to \Ass$ be a monoidal $\infty$-category compatible with small colimits.
By Proposition \ref{siewa} the functor $\Enr\Fun_{\mV,\emptyset}(-,\mV)^\circledast: (_\mV\omega\B\Enr_{\emptyset})^\op \to {_\emptyset\omega\widehat{\B\Enr}_\mV} $
lands in ${_\emptyset\omega\widehat{\B\Enr}^\R_\mV}.$
By Corollary \ref{kurt} there is a canonical equivalence ${_\emptyset\omega\widehat{\B\Enr}^\R_\mV}\simeq ({_\emptyset\omega\widehat{\B\Enr}^\L_\mV})^\op.$
Corollary \ref{siewalty} implies that there is a subfunctor $\zeta_\mV: {_\mV\omega\B\Enr_{\emptyset}} \to {_\emptyset\omega\B\Enr_\mV} \subset {_\emptyset\omega\widehat{\B\Enr}_\mV}$ of the composition $$ _\mV\omega\B\Enr_{\emptyset} \xrightarrow{\Enr\Fun_{\mV,\emptyset}(-,\mV)^\circledast} ({_\emptyset\omega\B\Enr^\R_\mV})^\op \simeq {_\emptyset\omega\widehat{\B\Enr}^\L_\mV} \subset {_\emptyset\omega\widehat{\B\Enr}_\mV},$$
where $\zeta$ sends $\mM^\circledast \to \mV^\ot$ to $(\mM^\op)^\circledast \to \mV^\ot.$
	
\end{construction}

\begin{proposition}\label{Dungo}
Let $\mV^\ot\to\Ass$ be a monoidal $\infty$-category compatible with small colimits.
There is a canonical 
equivalence $$\id \to \zeta_{\mV^\rev} \circ \zeta_\mV$$ of endofunctors of ${_\mV\B\Enr_\emptyset}$ whose component at $\mM^\circledast \to \mV^\ot$
is $\rho_\mM: \mM^\circledast \to ((\mM^\op)^\op)^\circledast.$

Replacing $\mV^\ot \to \Ass$ by $(\mV^\rev)^\ot \to \Ass$
there is also an equivalence of endofunctors of ${_\emptyset\B\Enr_\mV}$: $$\id \to \zeta_{\mV} \circ \zeta_{\mV^\rev}.$$ 

\end{proposition}

\begin{proof}


Let $\mM^\circledast \to \rS \times \mV^\ot$ be a map of cocartesian fibrations
over $\rS$ 
classifying a functor $\rS \to {_\mV\B\Enr_\emptyset} \subset \Cat_{\infty/\mV^\ot}$.
Let $$ (\mM^\op)^\circledast\subset \Enr\Fun^\rS_{\mV,\emptyset}(\mM,\mV)^\circledast\subset \Fun^{\rS\times\mV^\ot}_{\rS\times\mV^\ot\times\mV^\ot}(\mM^\circledast\times\mV^\ot,\rS\times\mV^\circledast)$$
be the full subcategories spanned by the objects of $$\Enr\Fun_{\mV,\emptyset}(\mM_\s,\mV)^\circledast\subset \Fun^{\rS\times\mV^\ot}_{\rS\times\mV^\ot\times\mV^\ot}(\mM^\circledast\times\mV^\ot,\rS\times\mV^\circledast)_\s \simeq \Fun^{\mV^\ot}_{\mV^\ot\times\mV^\ot}(\mM_\s^\circledast\times\mV^\ot,\mV^\circledast),$$$$(\mM_\s^\op)^\circledast\subset \Enr\Fun_{\mV,\emptyset}(\mM_\s,\mV)^\circledast\subset \Fun^{\rS\times\mV^\ot}_{\rS\times\mV^\ot\times\mV^\ot}(\mM^\circledast\times\mV^\ot,\rS\times\mV^\circledast)_\s \simeq \Fun^{\mV^\ot}_{\mV^\ot\times\mV^\ot}(\mM_\s^\circledast\times\mV^\ot,\mV^\circledast),$$ respectively, for some $\s \in \rS.$


By Remark \ref{reuil} the functor $\Fun^{\rS\times\mV^\ot}_{\rS\times\mV^\ot\times\mV^\ot}(\mM^\circledast\times\mV^\ot,\rS\times\mV^\circledast) \to \rS \times \mV^\ot$ is a map of cartesian fibrations
over $\rS$. 
Since the fiber transports restrict, also $ \Enr\Fun^\rS_{\mV,\emptyset}(\mM,\mV)^\circledast \to \rS \times \mV^\ot$ is a map of cartesian fibrations over $\rS$ that induces on the fiber over any $\s \in \rS$
the weakly right enriched $\infty$-category $\Fun^{\mV^\ot}_{\mV^\ot\times\mV^\ot}(\mM_\s^\circledast\times\mV^\ot,\mV^\circledast)\to \mV^\ot$. 
This implies that $ \Enr\Fun^\rS_{\mV,\emptyset}(\mM,\mV)^\circledast \to \mV^\ot$ is also a map of cocartesian fibrations relative to the collection of
cocartesian lifts of inert morphism of $\Ass$ preserving the minimum,
whose fiber transports preserve cartesian morphisms over $\rS$ being projections.
This guarantees that the map $ \Enr\Fun^\rS_{\mV,\emptyset}(\mM,\mV)^\circledast \to \rS \times \mV^\ot$ of cartesian fibrations over $\rS$ classifies a functor $\alpha: \rS^\op \to {_\emptyset\widehat{\B\Enr}_\mV}$.
Proposition \ref{siewa} guarantees that $\alpha$ lands in ${_\emptyset\widehat{\B\Enr}^\R_\mV}$.
This implies by \cite[Corollary 5.2.2.5.]{lurie.HTT} that $ \Enr\Fun^\rS_{\mV,\emptyset}(\mM,\mV)^\circledast \to \rS \times \mV^\ot$ is also a map of cocartesian fibrations over $\rS$
whose fiber transports preserve cocartesian lifts of inert morphisms preserving the minimum, and so classifies a functor $\beta: \rS \to {_\emptyset\widehat{\B\Enr}^\L_\mV}$,
which by construction factors as $\rS \xrightarrow{\alpha^\op}({_\emptyset\widehat{\B\Enr}^\R_\mV})^\op \simeq {_\emptyset\widehat{\B\Enr}^\L_\mV}\subset {_\emptyset\widehat{\B\Enr}_\mV}$.

By Corollary \ref{siewalty} the fiber transports of the map
$ \Enr\Fun^\rS_{\mV,\emptyset}(\mM,\mV)^\circledast \to \rS \times \mV^\ot$ of cocartesian fibrations over $\rS$ restrict to $(\mM^\op)^\circledast $ so that also $(\mM^\op)^\circledast \to \rS \times \mV^\ot$ is a map of cocartesian fibrations over $\rS$
classifying a functor $\rS \to{_\emptyset\widehat{\B\Enr}_\mV} .$

The restricted evaluation functor $ \mM^\circledast \times_\rS (\mM^\op)^\circledast \subset  \mM^\circledast \times_\rS \Enr\Fun^\rS_{\mV,\emptyset}(\mM,\mV)^\circledast\to \mV^\circledast$
over $\rS \times \mV^\ot \times \mV^\ot$ corresponds to a functor
$\sigma: \mM^\circledast \to \Enr\Fun^\rS_{\emptyset,\mV}(\mM^\op,\mV)^\circledast$
over $\rS \times \mV^\ot$ that induces on the fiber over any $\s \in \rS$
the embedding $\rho: \mM_\s^\circledast \simeq ((\mM_\s^\op)^\op)^\circledast \subset \Enr\Fun_{\emptyset,\mV}(\mM_\s^\op,\mV)^\circledast$ of Proposition \ref{yofaith}.
Lemma \ref{siewal} guarantees that $\sigma$ is a map of cocartesian fibrations over $\rS$.
Hence $\sigma$ is an embedding whose essential image is $((\mM^\op)^\op)^\circledast$.
Thus $\sigma$ induces an equivalence $\sigma': \mM^\circledast \to ((\mM^\op)^\op)^\circledast$
that induces on the fiber over any $\s \in \rS$
the equivalence $\rho: \mM_\s^\circledast \simeq ((\mM_\s^\op)^\op)^\circledast.$

If $\mM^\circledast \to \rS \times \mV^\ot$ classifies the identity of $ {_\mV\B\Enr_\emptyset}$, 
the map $ \Enr\Fun^\rS_{\mV,\emptyset}(\mM,\mV)^\circledast \to \rS \times \mV^\ot$
of cartesian fibrations over $\rS$ classifies the functor $$\Enr\Fun_{\mV,\emptyset}(-,\mV)^\circledast: (_\mV\omega\B\Enr_{\emptyset})^\op \to {_\emptyset\omega\widehat{\B\Enr}^\R_\mV} \subset {_\emptyset\widehat{\B\Enr}_\mV} $$ of Construction \ref{Varz} by \cite[Proposition 7.3.]{articles}.
Thus the map $ \Enr\Fun^\rS_{\mV,\emptyset}(\mM,\mV)^\circledast \to \rS \times \mV^\ot$ of cocartesian fibrations over $\rS$ classifies the composition $$ _\mV\omega\B\Enr_{\emptyset} \xrightarrow{\Enr\Fun_{\mV,\emptyset}(-,\mV)^\circledast} ({_\emptyset\omega\B\Enr^\R_\mV})^\op \simeq {_\emptyset\omega\widehat{\B\Enr}^\L_\mV} \subset {_\emptyset\omega\widehat{\B\Enr}_\mV}$$ 
and the map $(\mM^\op)^\circledast \to \rS \times \mV^\ot$ of cocartesian fibrations over $\rS$ classifies the functor $\zeta_\mV: {_\mV\B\Enr_\emptyset} \to {_\emptyset\B\Enr_\mV} $ of Construction \ref{Varz}.

Sending a map $\mM^\circledast \to \rS \times \mV^\ot$ of cocartesian fibrations over $\rS$ classifying a functor $\rS \to {_\mV\B\Enr_\emptyset} $ to the functor $\rS \to {_\emptyset\B\Enr_\mV}$ classified by the map $(\mM^\op)^\circledast \to \rS \times \mV^\ot$ of cocartesian fibrations over $\rS$ gives a natural map $$\Ho(\widehat{\Cat}_\infty)(\rS, {_\mV\B\Enr_\emptyset}) \to \Ho(\widehat{\Cat}_\infty)(\rS,{_\emptyset\B\Enr_\mV}).$$
By the classical Yoneda-lemma this map postcomposes with
the functor $\zeta_\mV: {_\mV\B\Enr_\emptyset} \to {_\emptyset\B\Enr_\mV}$.
In particular for any map $\mM^\circledast \to \rS \times \mV^\ot$ of cocartesian fibrations over $\rS$ classifying a functor $\tau: \rS \to {_\mV\B\Enr_\emptyset} $
the map $(\mM^\op)^\op)^\circledast \to \rS \times \mV^\ot$ of cocartesian fibrations over $\rS$ classifies the functor $\zeta_{\mV^\rev} \circ \zeta_\mV \circ \tau: \rS \to {_\mV\B\Enr_\emptyset}$.
Therefore the equivalence $\sigma'$ classifies an equivalence
$\id \to \zeta_{\mV^\rev}\circ\zeta_\mV$ of functors 
${_\mV\B\Enr_\emptyset} \to {_\mV\B\Enr_\emptyset}.$


\end{proof}

Next we deduce naturality of the left $\mV$-enriched Yoneda-embedding $$\rho_\mM: \mM^\circledast \to \Enr\Fun_{\emptyset, \mV}(\mM^\op,\mV)^\circledast$$ (Corollary \ref{ohay}).
\begin{definition}
Let $\mV^\ot\to \Ass$ be a monoidal $\infty$-category compatible with small colimits.
Let $\mP_\mV$ be the composition
$$ _\mV\omega\B\Enr_{\emptyset} \xrightarrow{\zeta_\mV}
{_\emptyset\omega\B\Enr_\mV} \xrightarrow{\Enr\Fun_{\emptyset,\mV}(-,\mV)^\circledast} ({_\mV\omega\B\Enr^\R_\emptyset})^\op \simeq {_\mV\omega\widehat{\B\Enr}^\L_\emptyset} \subset {_\mV\omega\widehat{\B\Enr}_\emptyset}.$$
\end{definition}

\begin{corollary}\label{ohay}
Let $\mV^\ot\to \Ass$ be a monoidal $\infty$-category compatible with small colimits.	
There is a natural transformation $\id \to \mP_\mV$ of functors ${_\mV\omega\B\Enr_\emptyset}\to {_\mV\omega\widehat{\B\Enr}_\emptyset}$
whose component at any left enriched $\infty$-category $\mM^\circledast \to \mV^\ot$
is the left $\mV$-enriched Yoneda-embedding $\rho_\mM: \mM^\circledast \to \Enr\Fun_{ \emptyset, \mV}(\mM^\op,\mV)^\circledast$.
	
\end{corollary}

\begin{proof}
By Construction \ref{Varz} there is an embedding $\zeta_{\mV^\rev}\circ \zeta_\mV \hookrightarrow \mP_\mV$ of functors ${_\mV\omega\B\Enr_\emptyset}\to {_\mV\omega\widehat{\B\Enr}_\emptyset}.$
By Proposition \ref{Dungo} there is a canonical equivalence $\id \to \zeta_{\mV^\rev} \circ \zeta_\mV$ of endofunctors of ${_\mV\B\Enr_\emptyset}$.
We obtain an embedding $\id \hookrightarrow \mP_\mV$ of functors ${_\mV\omega\B\Enr_\emptyset}\to {_\mV\omega\widehat{\B\Enr}_\emptyset}$
whose component at any left enriched $\infty$-category $\mM^\circledast \to \mV^\ot$
is the left $\mV$-enriched Yoneda-embedding $\rho_\mM: \mM^\circledast \to ((\mM^\op)^\op)^\circledast \subset \Enr\Fun_{ \emptyset, \mV}(\mM^\op,\mV)^\circledast$ by Proposition \ref{Dungo}.
\end{proof}

\begin{remark}
Corollary \ref{ohay} gives naturality of the enriched Yoneda-embedding.
Naturality of the enriched Yoneda-embedding in Hinich's model of enrichment is the topic of \cite{BENMOSHE2024107625}.	
	
\end{remark}

Next we apply Corollary \ref{Univ} to deduce that our both models for the left enriched $\infty$-category of enriched presheaves are equivalent:

\begin{theorem}\label{unitol}
Let $\mV^\ot \to \Ass$ be a monoidal $\infty$-category compatible with small colimits and $\mM^\circledast \to \mV^\ot $ a left enriched $\infty$-category. 

\begin{enumerate}
\item 
The left $\mV$-enriched Yoneda-embedding $$\rho: \mM^\circledast \to \Enr\Fun_{\emptyset, \mV}(\mM^\op,\mV)^\circledast $$
uniquely extends to a left $\mV$-enriched equivalence $$\theta: \mP\L\Env(\mM)_{\L\Enr}^\circledast\simeq\Enr\Fun_{\emptyset, \mV}(\mM^\op,\mV)^\circledast.$$	

\item If $\mV^\ot \to \Ass$ is a presentably monoidal $\infty$-category, $\theta$ corresponds to the $\mV,\mV$-enriched functor $$\hspace{10mm}\mP\L\Env(\mM)_{\L\Enr}^\circledast\times (\mM^\op)^\circledast \subset\mP\L\Env(\mM)_{\L\Enr}^\circledast \times (\mP\L\Env(\mM)_{\L\Enr}^\op)^\circledast\xrightarrow{\L\Mor_{\mP\L\Env(\mM)_{\L\Enr}}(-,-)}\mV^\circledast.$$
\end{enumerate}	
	
\end{theorem}

\begin{proof}

(1): By Proposition \ref{innerhosty} and Theorem \ref{unipor3} there exists a unique left $\mV$-linear small colimits preserving extension $\theta$ of $\rho.$ By Corollary \ref{Univ} it suffices to see that the essential image of $\rho$ generates the target of $\rho$ 
under small colimits and left tensors and that for any
$\Z \in \mM$ the left $\mV$-enriched functor
$$\L\Mor_{\Enr\Fun_{\emptyset, \mV}(\mM^\op,\mV)}(\rho(\Z),-):
\Enr\Fun_{\emptyset, \mV}(\mM^\op,\mV)^\circledast \to \mV^\circledast$$ preserves small colimits and is $\mV$-linear.
This holds by Corollary \ref{yotul}.
	
(2): 
The $\mV,\mV$-enriched functor $$\mP\L\Env(\mM)_{\L\Enr}^\circledast\times (\mM^\op)^\circledast \subset\mP\L\Env(\mM)_{\L\Enr}^\circledast \times (\mP\L\Env(\mM)_{\L\Enr}^\op)^\circledast\xrightarrow{\L\Mor_{\mP\L\Env(\mM)_{\L\Enr}}(-,-)}\mV^\circledast$$
corresponds to a left $\mV$-enriched functor $\theta': \mP\L\Env(\mM)_{\L\Enr}^\circledast\to\Enr\Fun_{\emptyset, \mV}(\mM^\op,\mV)^\circledast.$
By Proposition \ref{corok}, Proposition \ref{innerhosty} and Lemma \ref{colas} the left $\mV$-enriched functor $\theta'$ preserves small colimits and is left $\mV$-linear.
Hence it is enough to see that $\theta'_{\mid \mM^\circledast}\simeq \theta_{\mid \mM^\circledast}$ by Proposition \ref{envvcor}.
Both latter enriched functors correspond to the $\mV,\mV$-enriched functor $$\L\Mor_{\mM}(-,-): \mM^\circledast\times (\mM^\op)^\circledast \subset\mP\L\Env(\mM)_{\L\Enr}^\circledast \times (\mP\L\Env(\mM)_{\L\Enr}^\op)^\circledast\xrightarrow{\L\Mor_{\mP\L\Env(\mM)_{\L\Enr}}(-,-)}\mV^\circledast.$$
	
\end{proof}

\begin{corollary}\label{saewo}
	
Let $\mV^\ot \to \Ass$ be a monoidal $\infty$-category compatible with small colimits, $\mM^\circledast \to \mV^\ot, \mN^\circledast \to \mV^\ot $ small left enriched $\infty$-categories and $\F:\mM^\circledast \to \mN^\circledast$ a left $\mV$-enriched functor.
There is a commutative square of $\infty$-categories left tensored over $\mV$:
$$\begin{xy}
\xymatrix{\mP\L\Env(\mM)_{\L\Enr}^\circledast  \ar[d]^{\simeq}\ar[rr]^{\bar{\F}}
&& \mP\L\Env(\mN)_{\L\Enr}^\circledast \ar[d]^{\simeq} 
\\
\Enr\Fun_{\emptyset, \mV}(\mM^\op,\mV)^\circledast \ar[rr]^{(\F^\op)_!} && \Enr\Fun_{\emptyset, \mV}(\mN^\op,\mV)^\circledast.}
\end{xy}$$
\end{corollary}



\subsection{A bienriched Yoneda-embedding}

Next we enhance the enriched Yoneda-embedding of the latter section to a bienriched  Yoneda-embedding.

\begin{lemma}\label{notori2}
Let $\mM^\circledast \to \mV^\ot \times \mW^\ot $ be a right tensored left enriched $\infty$-category.
The full weakly bienriched subcategory 
$$(\mM^\op)^\circledast \subset \Enr\Fun_{\mV, \emptyset}(\mM,\mV)^\circledast \to \mV^\ot \times \mV^\ot $$ 
spanned by all left $\mV$-enriched functors $\Mor_\mM(\X,-):\mM^\circledast \to \mV^\circledast$ for $\X \in \mM$ is closed under right cotensors.
	
\end{lemma}

\begin{lemma}
	
$\mM$ weakly bitensored over $\mV,\mV$ and has right tensors.

$\Enr\Fun_{\mV, \emptyset}(\mM,\mV)$ is bitensored over $\mV,\mV.$

$\mM^\op$ 

comparison between left tensors:

$ \L\Mor_\mM(\X,-) \circ ((-)\ot \V) \to  \L\Mor_\mM({\X^\V},-).$
	
\end{lemma}

The following remark, proposition and lemma justify the terminology of Definition \ref{notori}:

\begin{remark}
By construction \ref{Enros} and Corollary \ref{MorFu} for every weakly left enriched $\infty$-category $\mM^\circledast \to \mV^\ot$ there is an embedding
$$\xi: \mM^\op \to \Enr\Fun_{\mV, \emptyset}(\mM,\mP\Env(\mV))^\circledast$$
sending $\X\in \mM$ to $\L\Mor_{\bar{\mM}}(\X,-).$
If $\mM^\circledast \to \mV^\ot$ is a left enriched $\infty$-category, by Remark \ref{EEnr} the embedding $\xi$ induces an embedding
$\mM^\op \to \Enr\Fun_{\mV, \emptyset}(\mM,\mV)^\circledast \subset \Enr\Fun_{\mV, \emptyset}(\mM,\mP\Env(\mV))^\circledast$
denoted by the same name, which sends $\X\in \mM$ to $\L\Mor_{\bar{\mM}}(\X,-).$
The embedding $\xi$ 
induces an equivalence $\mM^\op \simeq (\mM^\op)^\circledast \times_{\mV^\ot} \emptyset^\ot $.
	
\end{remark}
\begin{proposition}\label{oppoen}
	
Let $\mM^\circledast \to \mV^\ot $ be a left enriched $\infty$-category.
The weakly right enriched $\infty$-category $(\mM^\op)^\circledast \to \mV^\ot$ exhibits $\mM^\op$ as right enriched in $\mV$, where for any $\X,\Y \in \mM$ there is a canonical equivalence $$\R\Mor_{\mM^\op}(\X,\Y) \simeq \L\Mor_\mM(\Y,\X).$$	
\end{proposition}

\begin{proof}
By Corollary \ref{explicas} for any $\X,\Y \in \mM$ and $\V_1,...,\V_\n \in \mV$ for $\n \geq 0$ the canonical map
$$\Mul_{\mM^\op}(\V_1,...,\V_\n, \L\Mor_\mM(\X,-);\L\Mor_\mM(\Y,-)) \to  \Mul_{\mV}(\V_1,...,\V_\n, \L\Mor_\mM(\Y,\X)) $$ is an equivalence.
So the result follows.
\end{proof}

\begin{lemma}\label{aos}
Let $\mM^\circledast \to \mV^\ot, \mN^\ot \to \mW^\ot $ be left enriched $\infty$-categories.
The right $\mV\times \mW$-enriched functor $$\rho: \Enr\Fun_{\mV, \emptyset}(\mM,\mV)^\circledast \times_\Ass \Enr\Fun_{\mV, \emptyset}(\mN,\mV)^\circledast \to \Enr\Fun_{\mV \times \mW, \emptyset}(\mM \times \mN,\mV\times \mW)^\circledast$$
corresponding to the $\mV\times \mW,\mV\times \mW$-enriched functor $$(\mM^\circledast \times_\Ass \mN^\circledast) \times (\Enr\Fun_{\mV, \emptyset}(\mM,\mV)^\circledast \times_\Ass \Enr\Fun_{\mV, \emptyset}(\mN,\mV)^\circledast) \simeq $$$$ (\mM^\circledast \times \Enr\Fun_{\mV, \emptyset}(\mM,\mV)^\circledast) \times_{\Ass \times \Ass} (\mN^\circledast \times \Enr\Fun_{\mV, \emptyset}(\mN,\mV)^\circledast) \to \mV^\circledast \times_{\Ass \times \Ass} \mW^\circledast,$$
which is the product of the $\mV,\mV$-enriched and $\mW,\mW$-enriched evaluation functors,
restricts to a right $\mV\times \mW$-enriched equivalence $$ \rho': (\mM^\op)^\circledast \times_\Ass (\mN^\op)^\circledast \simeq ((\mM \times \mN)^\op)^\circledast.$$
\end{lemma}
\begin{proof}
	
The enriched functor $\rho$ sends $\L\Mor_\mM(\X,-), \L\Mor_\mN(\Y,-)$ for $\X \in \mM, \Y \in \mN$ to $$\L\Mor_\mM(\X,-) \times \L\Mor_\mN(\Y,-) \simeq \L\Mor_{\mM \times \mN}((\X,\Y),-)$$ and so restricts to $\rho'$. By Corollary \ref{yotul} the enriched functor $\rho'$ induces on right morphism objects the following equivalence for $\X' \in \mM, \Y' \in \mN$:
$$ \L\Mor_\mM(\X',\X)\times \L\Mor_\mN(\Y',\Y) \simeq $$$$\R\Mor_{\mM^\op}(\L\Mor_\mM(\X,-),\L\Mor_\mM(\X',-)) \times \R\Mor_{\mN^\op}(\L\Mor_\mN(\Y,-),\L\Mor_\mN(\Y',-))\to $$$$ \R\Mor_{(\mM \times \mN)^\op}(\L\Mor_{\mM \times \mN}((\X,\Y),-),\L\Mor_{\mM \times \mN}((\X',\Y'),-)) \simeq \L\Mor_{\mM \times \mN}((\X',\Y'),(\X,\Y)).$$
\end{proof}

\begin{proposition}\label{dua} Let $\mM^\circledast \to \mV^\ot $ be a left enriched $\infty$-category, $\mO^\circledast \to \mW^\ot$ a right enriched $\infty$-category  and $\mN^\circledast \to \mV \times \mW^\ot $ a bienriched $\infty$-category.
	
\begin{enumerate}
\item There is a $\mW,\mV$-enriched equivalence
$$ ((\mM \times \mO)^\op)^\circledast \simeq (\mM^\op)^\circledast \times (\mO^\op)^\circledast.$$
		
\item There is a right $\mW$-enriched equivalence $$ \Enr\Fun_{\mV, \emptyset}(\mM,\mN) \simeq \Enr\Fun_{\emptyset, \mV}(\mM^\op,\mN^\op)^\op.$$
		
\end{enumerate}
	
\end{proposition}

\begin{proof}(1): By Lemma \ref{aos} there is a left $\mV,\mW^\rev$-enriched equivalence
$$ (\mM \times \mO^\rev)^\op)^\circledast \simeq (\mM^\op)^\circledast \times_\Ass ((\mO^\op)^\rev)^\circledast.$$
Under the equivalence ${_\mV\B\P\Enr_\mW} \simeq {_{\mV\times\mW^\rev}\L\P\Enr}_\emptyset$
the latter equivalence corresponds to the $\mV,\mW$-enriched equivalence
$((\mM \times \mO)^\op)^\circledast \simeq (\mM^\op)^\circledast \times (\mO^\op)^\circledast.$
	
(2): The claimed equivalence is represented by the following chain of natural equivalences:
$${_\emptyset\omega\B\Enr_\mW}(\mO, \Enr\Fun_{\emptyset, \mV}(\mM^\op,\mN^\op)^\op) \simeq {_\mW\omega\B\Enr_\emptyset}(\mO^\op, \Enr\Fun_{\emptyset, \mV}(\mM^\op,\mN^\op))$$$$
\simeq {_\mW\omega\B\Enr_\mV}(\mO^\op \times \mM^\op,\mN^\op) 
\simeq {_\mW\omega\B\Enr_\mV}((\mM \times \mO)^\op,\mN^\op) \simeq $$$${_\mV\omega\B\Enr_\mW}(\mM \times \mO,\mN) \simeq {_\emptyset\omega\B\Enr_\mW}(\mO, \Enr\Fun_{\mV, \emptyset}(\mM,\mN)).$$	
	
\end{proof}



\begin{notation}\label{nnoo}Let $\mM^\circledast \to \mV^\ot $ be a left enriched $\infty$-category.
\begin{enumerate}
\item The $\mV,\mV$-enriched left morphism object functor $$\L\Mor_\mM: \mM^\circledast \times (\mM^\op)^\circledast \to \mV^\circledast$$ is the restricted $\mV,\mV$-enriched evaluation functor $$\mM^\circledast \times (\mM^\op)^\circledast \subset \mM^\circledast \times  \Enr\Fun_{\mV, \emptyset}(\mM,\mV)^\circledast \to \mV^\circledast.$$
		
\item The left $\mV$-enriched Yoneda-embedding $$\rho_\mM: \mM^\circledast \to \Enr\Fun_{\emptyset, \mV}(\mM^\op,\mV)^\circledast$$
is the left $\mV$-enriched functor corresponding to the left morphism object functor $\L\Mor_\mM: \mM^\circledast \times (\mM^\op)^\circledast \to \mV^\circledast $
under the equivalence of Theorem \ref{vvvl}.
\end{enumerate}
\end{notation}

\begin{remark}
Notation \ref{nnoo} also applies to any right $\mV$-enriched $\infty$-category
$\mM^\circledast \to \mV^\ot$ (viewed as a left $\mV^\rev$-enriched $\infty$-category).
In this case we obtain a $\mV,\mV$-enriched right morphism object functor $$\R\Mor_\mM: (\mM^\op)^\circledast \times \mM^\circledast \to \mV^\circledast,$$ which is the restricted $\mV,\mV$-enriched evaluation functor $$(\mM^\op)^\circledast  \times \mM^\circledast \subset \Enr\Fun_{\emptyset,\mV}(\mM,\mV)^\circledast \times \mM^\circledast \to \mV^\circledast,$$
and a corresponding right $\mV$-enriched Yoneda-embedding $\rho_\mM: \mM^\circledast \to \Enr\Fun_{\mV,\emptyset}(\mM^\op,\mV)^\circledast.$
\end{remark}

\begin{remark}\label{Reyy}
The $\mV,\mV$-enriched left morphism object functor $\L\Mor_\mM: \mM^\circledast \times (\mM^\op)^\circledast \to \mV^\circledast$
induces on underlying weakly left $\mV$-enriched $\infty$-categories the left 
$\mV$-enriched functor $\Gamma_\mM: \mM^\circledast \times \mM^\op \to \mV^\circledast$
of Construction \ref{Enros} corresponding to $\xi.$ 
\end{remark}

\begin{proposition}\label{yofaith}
Let $\mM^\circledast \to \mV^\ot $ be a small left enriched $\infty$-category. The left $\mV$-enriched Yoneda-embedding $\rho: \mM^\circledast \to \Enr\Fun_{\emptyset, \mV}(\mM^\op,\mV)^\circledast $ induces an equivalence $\mM^\circledast \to ((\mM^\op)^\op)^\circledast.$	
\end{proposition}

\begin{proof}
By Lemma \ref{exhibl} the left $\mV$-enriched Yoneda-embedding $\rho_\mM: \mM^\circledast \to\Enr\Fun_{\mV,\emptyset}(\mM^\op,\mV)^\circledast$ induces a left $\mV$-enriched functor $\mM^\circledast \to ((\mM^\op)^\op)^\circledast$ whose underlying functor $\mM \to (\mM^\op)^\op$ is essentially surjective.	
By Remark \ref{Func} the left $\mV$-enriched functor $\rho: \mM^\circledast \to (\mM^\op)^\op)^\circledast$ induces a morphism
$$\kappa: \L\Mor_\mM(\Z ,-) \to \L\Mor_{(\mM^\op)^\op}(\rho(\Z ),-)\circ \rho$$
in $\Enr\Fun_{\mV,\emptyset}(\mM,\mV)$.	
By Remark \ref{Func} the left $\mV$-enriched functor $\ev_{\xi(\Z)}: \Enr\Fun_{\emptyset, \mV}(\mM^\op,\mV)^\circledast \to \mV^\circledast$ evaluating at $\xi(\Z) \in \mM^\op$
induces the following morphism $\lambda$ in $\Enr\Fun_{\mV,\emptyset}(\mM,\mV) $:
$$\L\Mor_{(\mM^\op)^\op}(\rho(\Z ),-)\circ \rho \to \L\Mor_{\mP\Env(\mV)}(\L\Mor_\mM(\Z,\Z),-)\circ\ev_{\xi(\Z)} \circ \rho $$$$ \to \ev_{\xi(\Z)} \circ \rho \simeq \xi(\Z)=\L\Mor_\mM(\Z ,-),$$
which is an equivalence by Corollary \ref{explicas}.
By Corollary \ref{explicas} the composition $\lambda \circ \kappa$
is the identity because $\lambda_\Z  \circ \kappa_\Z $ sends the identity of $\Z $ to the identity of $\Z .$ So the result follows.
\end{proof}

\begin{remark}
Let $\mM^\circledast \to \mV^\ot $ be a small left enriched $\infty$-category. 
By Proposition \ref{yofaith} the left $\mV$-enriched Yoneda-embedding $\rho_\mM: \mM^\circledast \to \Enr\Fun_{\emptyset, \mV}(\mM^\op,\mV)^\circledast $ induces an equivalence $\rho'_\mM: \mM^\circledast \to ((\mM^\op)^\op)^\circledast.$	
By definition $\rho_\mM: \mM^\circledast \to \Enr\Fun_{\emptyset, \mV}(\mM^\op,\mV)^\circledast$ corresponds to the
$\mV,\mV$-enriched left morphism object functor
$\L\Mor_\mM: \mM^\circledast \times (\mM^\op)^\circledast \to \mV^\circledast.$
Thus $\L\Mor_\mM: \mM^\circledast \times (\mM^\op)^\circledast \to \mV^\circledast$
factors as $\mV,\mV$-enriched functors $$\mM^\circledast \times (\mM^\op)^\circledast \xrightarrow{\rho_\mM\times (\mM^\op)^\circledast}\Enr\Fun_{\emptyset, \mV}(\mM^\op,\mV)^\circledast \times (\mM^\op)^\circledast \to \mV^\circledast,$$
which factors as 
$$\mM^\circledast \times (\mM^\op)^\circledast \xrightarrow{\rho'_\mM\times (\mM^\op)^\circledast} (\mM^\op)^\op)^\circledast \times (\mM^\op)^\circledast \xrightarrow{\R\Mor_{\mM^\op}} \mV^\circledast.$$
In particular, $\rho_{\mM^\op}: (\mM^\op)^\circledast \to \Enr\Fun_{\mV,\emptyset}((\mM^\op)^\op,\mV)^\circledast \simeq \Enr\Fun_{\mV,\emptyset}(\mM,\mV)^\circledast,$ where the latter equivalence precomposes with $\rho'$, is the tautological embedding $(\mM^\op)^\circledast \subset \Enr\Fun_{\mV,\emptyset}(\mM,\mV)^\circledast.$


\end{remark}

\begin{corollary}\label{siewalty}
Let $\mV^\ot \to \Ass$ be a monoidal $\infty$-category compatible with small colimits, $\mM^\circledast \to \mV^\ot, \mN^\circledast \to \mV^\ot $ small left enriched $\infty$-categories and $\F:\mM^\circledast \to \mN^\circledast$ a left $\mV$-enriched functor.
The right $\mV$-enriched functor $$\F_!: \Enr\Fun_{\mV,\emptyset}(\mM,\mV)^\circledast \to \Enr\Fun_{\mV,\emptyset}(\mN,\mV)^\circledast$$
restricts to a right $\mV$-enriched functor $(\mM^\op)^\circledast \to (\mN^\op)^\circledast$, which we denote by $\F^\op.$	

\end{corollary} 

Next we apply Corollary \ref{Univ} to deduce that our both models for the left enriched $\infty$-category of enriched presheaves are equivalent:

\begin{theorem}\label{unitol}
Let $\mV^\ot \to \Ass$ be a monoidal $\infty$-category compatible with small colimits and $\mM^\circledast \to \mV^\ot $ a left enriched $\infty$-category. 

\begin{enumerate}
\item The left $\mV$-enriched Yoneda-embedding $$\rho: \mM^\circledast \to \Enr\Fun_{\emptyset, \mV}(\mM^\op,\mV)^\circledast $$
uniquely extends to a left $\mV$-enriched equivalence $$\theta: \mP\L\Env(\mM)_{\L\Enr}^\circledast\simeq\Enr\Fun_{\emptyset, \mV}(\mM^\op,\mV)^\circledast.$$	
	
\item If $\mV^\ot \to \Ass$ is a presentably monoidal $\infty$-category, $\theta$ corresponds to the $\mV,\mV$-enriched functor $$\hspace{10mm}\mP\L\Env(\mM)_{\L\Enr}^\circledast\times (\mM^\op)^\circledast \subset\mP\L\Env(\mM)_{\L\Enr}^\circledast \times (\mP\L\Env(\mM)_{\L\Enr}^\op)^\circledast\xrightarrow{\L\Mor_{\mP\L\Env(\mM)_{\L\Enr}}(-,-)}\mV^\circledast.$$
\end{enumerate}	

\end{theorem}

\begin{proof}

(1): By Proposition \ref{innerhosty} and Theorem \ref{unipor3} there exists a unique left $\mV$-linear small colimits preserving extension $\theta$ of $\rho.$ By Corollary \ref{Univ} it suffices to see that the essential image of $\rho$ generates the target of $\rho$ 
under small colimits and left tensors and that for any
$\Z \in \mM$ the left $\mV$-enriched functor
$$\L\Mor_{\Enr\Fun_{\emptyset, \mV}(\mM^\op,\mV)}(\rho(\Z),-):
\Enr\Fun_{\emptyset, \mV}(\mM^\op,\mV)^\circledast \to \mV^\circledast$$ preserves small colimits and is $\mV$-linear.
This holds by Corollary \ref{yotul}.

(2): 
The $\mV,\mV$-enriched functor $$\mP\L\Env(\mM)_{\L\Enr}^\circledast\times (\mM^\op)^\circledast \subset\mP\L\Env(\mM)_{\L\Enr}^\circledast \times (\mP\L\Env(\mM)_{\L\Enr}^\op)^\circledast\xrightarrow{\L\Mor_{\mP\L\Env(\mM)_{\L\Enr}}(-,-)}\mV^\circledast$$
corresponds to a left $\mV$-enriched functor $\theta': \mP\L\Env(\mM)_{\L\Enr}^\circledast\to\Enr\Fun_{\emptyset, \mV}(\mM^\op,\mV)^\circledast.$
By Proposition \ref{corok}, Proposition \ref{innerhosty} and Lemma \ref{colas} the left $\mV$-enriched functor $\theta'$ preserves small colimits and is left $\mV$-linear.
Hence it is enough to see that $\theta'_{\mid \mM^\circledast}\simeq \theta_{\mid \mM^\circledast}$ by Proposition \ref{envvcor}.
Both latter enriched functors correspond to the $\mV,\mV$-enriched functor $$\L\Mor_{\mM}(-,-): \mM^\circledast\times (\mM^\op)^\circledast \subset\mP\L\Env(\mM)_{\L\Enr}^\circledast \times (\mP\L\Env(\mM)_{\L\Enr}^\op)^\circledast\xrightarrow{\L\Mor_{\mP\L\Env(\mM)_{\L\Enr}}(-,-)}\mV^\circledast.$$

\end{proof}

\begin{corollary}\label{saewo}

Let $\mV^\ot \to \Ass$ be a monoidal $\infty$-category compatible with small colimits, $\mM^\circledast \to \mV^\ot, \mN^\circledast \to \mV^\ot $ small left enriched $\infty$-categories and $\F:\mM^\circledast \to \mN^\circledast$ a left $\mV$-enriched functor.
There is a commutative square of $\infty$-categories left tensored over $\mV$:
$$\begin{xy}
\xymatrix{\mP\L\Env(\mM)_{\L\Enr}^\circledast  \ar[d]^{\simeq}\ar[rr]^{\bar{\F}}
&& \mP\L\Env(\mN)_{\L\Enr}^\circledast \ar[d]^{\simeq} 
\\
\Enr\Fun_{\emptyset, \mV}(\mM^\op,\mV)^\circledast \ar[rr]^{(\F^\op)_!} && \Enr\Fun_{\emptyset, \mV}(\mN^\op,\mV)^\circledast.}
\end{xy}$$
\end{corollary}

\subsection{Enriched $\infty$-categories of enriched functors as an internal hom}

In the following we prove that the weakly enriched $\infty$-categories of enriched functors of Notation \ref{bbbb} are in fact pseudo-enriched, enriched $\infty$-categories, respectively, under reasonable conditions (Theorem \ref{psinho}), and identify such as the internal hom for the tensor product of enriched $\infty$-categories
(Theorem \ref{cloff}).

\begin{theorem}\label{psinho}
\begin{enumerate}
\item Let $\mV^\ot \to \Ass$ be a small $\infty$-operad, $ (\mW^\otimes \to \Ass, \T)$ a small localization pair, $\mM^\circledast \to \mV^\ot$ a small weakly left enriched $\infty$-category and $\mN^\circledast \to \mV^\ot \times \mW^\ot$ a small 
right $\T$-enriched $\infty$-category. Then $\Enr\Fun_{\mV, \emptyset}(\mM, {\mN})^\circledast \to \mW^\ot $ is right $\T$-enriched. 

\vspace{1mm}
\item Let $\mV^\ot \to \Ass$ be an $\infty$-operad, $\mM^\circledast \to \mV^\ot$ a small weakly left enriched $\infty$-category, $\mN^\circledast \to \mV^\ot \times \mW^\ot $ a right enriched $\infty$-category and $ \mW^\ot \to \Ass$ an $\infty$-operad that exhibits $\mW$ as right enriched in $\mW$ such that $\mW$ admits small limits and for every $\W \in \mW$ the functor $\R\Mor_\mW(\W,-):\mW \to \mW$ preserves small limits. Then $\Enr\Fun_{\mV, \emptyset}(\mM, {\mN})^\circledast \to \mW^\ot $ is a right enriched $\infty$-category.

\end{enumerate}

\end{theorem}

\begin{proof}

(1): The $\mV, \mW$-enriched embedding $\mN^\circledast \hookrightarrow \widetilde{\B\Env}(\mN)_{\T}$ induces a right $\mW$-enriched embedding 
$$\Enr\Fun_{\mV, \emptyset}(\mM,\mN)^\circledast \to \Enr\Fun_{\mV, \emptyset}(\mM,\mP\widetilde{\B\Env}(\mN)_{\T})^\circledast.$$ 
Since $\mN'^\circledast:= \mV^\ot \times_{\mP\Env(\mV)^\ot}\mP\B\Env(\mN)_{\T}^\circledast \to \mV^\ot \times \T^{-1}\mP\Env(\mW)^\ot$ is a presentably right tensored $\infty$-category, by Corollary \ref{innerhost} and Lemma \ref{colas} also $ \Enr\Fun_{\mV, \emptyset}(\mM,\mN')^\circledast \to \T^{-1}\mP\Env(\mW)^\ot$ is a presentably right tensored $\infty$-category and so a right enriched $\infty$-category.
Thus by Proposition \ref{eqq} the pullback $$\mW^\ot \times_{ \T^{-1}\mP\Env(\mW)^\ot} \Enr\Fun_{\mV, \emptyset}(\mM,\mN')^\circledast \simeq \Enr\Fun_{\mV, \emptyset}(\mM,\mP\widetilde{\B\Env}(\mN)_{\T})^\circledast $$
is a right $\T$-enriched $\infty$-category, where the equivalence is by Remark \ref{leman}. So also $\Enr\Fun_{\mV, \emptyset}(\mM,\mN)^\circledast \to \mW^\ot$ is a right $\T$-enriched $\infty$-category.

(2): By Proposition \ref{eqq} the weakly left enriched $\infty$-category
$\mM^\circledast \to \mV^\ot$ is the pullback of a unique left enriched $\infty$-category $\bar{\mM}^\circledast \to \widehat{\mP}\Env(\mV)^\ot$.
The following right $\widehat{\mP}\Env(\mW)$-enriched functor is an equivalence since it
is linear by Corollary \ref{innerhost} and induces an equivalence on underlying $\infty$-categories by Proposition \ref{eqq}:
$$\Theta^\circledast := \Enr\Fun_{\widehat{\mP}\Env(\mV), \emptyset}(\bar{\mM},\widehat{\mP}\B\Env(\mN))^\circledast\to \Enr\Fun_{\mV, \emptyset}(\mM,\widehat{\mP}\B\Env(\mN))^\circledast.$$
We obtain a right $\mW$-enriched embedding:
$$\Enr\Fun_{\mV, \emptyset}(\mM,\mN)^\circledast \subset \Enr\Fun_{\mV, \emptyset}(\mM,\widehat{\mP}\widetilde{\B\Env}(\mN))^\circledast\simeq \mW^\ot \times_{\widehat{\mP}\Env(\mW)^\ot}  \Theta^\circledast.$$
By Corollary \ref{innerhost} and Lemma \ref{colas} the weakly right enriched $\infty$-category $ \Theta^\circledast \to \widehat{\mP}\Env(\mW)^\ot$ is a presentably right tensored $\infty$-category and so exhibits $\Theta$ as right $\widehat{\mP}\Env(\mW)$-enriched.
To prove (2) it is enough to see that for any $\G, \rH \in \Enr\Fun_{\mV, \emptyset}(\mM,\mN)$ the morphism object $\R\Mor_{\Theta}(\iota \circ \G,\iota \circ \rH) \in \widehat{\mP}\Env(\mW) $ belongs to $\mW,$ where $\iota: \mN^\circledast \to \widehat{\mP}\B\Env(\mN)^\circledast$ is the canonical embedding.

The following right $\widehat{\mP}\Env(\mW)$-enriched functor is an equivalence since it is linear by Corollary \ref{innerhost} and induces an equivalence on underlying $\infty$-categories by Corollary \ref{cosqai}:
$$\Theta^\circledast = \Enr\Fun_{\widehat{\mP}\Env(\mV), \emptyset}(\bar{\mM},\widehat{\mP}\B\Env(\mN))^\circledast \to \Enr\Fun_{\mP\Env(\mV), \emptyset}(\bar{\mM},\widehat{\mP}\B\Env(\mN))^\circledast. $$
By Remark \ref{silu} the linear small colimits preserving embedding $ \mP\B\Env(\mN)^\circledast \subset  \widehat{\mP}\B\Env(\mN)^\circledast$ induces a right $\mP\Env(\mW)$-linear small colimits preserving embedding $$ \Theta'^\circledast:= \Enr\Fun_{\mP\Env(\mV), \emptyset}(\bar{\mM},\mP\B\Env(\mN))^\circledast \subset \Enr\Fun_{\mP\Env(\mV), \emptyset}(\bar{\mM},\widehat{\mP}\B\Env(\mN))^\circledast $$$$\simeq \mP\Env(\mW)^\ot \times_{\widehat{\mP}\Env(\mW)^\ot} \Theta^\circledast.$$

Let $\Lambda \subset \Theta$ be the full subcategory of all $\G \in \Theta$
such that for every $\rH \in \Enr\Fun_{\mV, \emptyset}(\mM,\mN)$ the morphism object $\R\Mor_{\Theta}(\G,\iota \circ \rH) \in \widehat{\mP}\Env(\mW) $ belongs to $\mW.$
We will prove that $\Theta' \subset \Lambda.$
We prove first that $\Lambda $ is closed in $\Theta$ under small colimits.
This follows from the fact that the functor $\R\Mor_{\Theta}(-,-): \Theta^\op \times \Theta \to \widehat{\mP}\Env(\mW)$
preserves limits component-wise and that $\mW$ is closed in $\widehat{\mP}\Env(\mW)$ under small limits since $\mW$ admits small limits and both functors $\mW \subset \Env(\mW), \Env(\mW) \subset \widehat{\mP}\Env(\mW)$ preserve arbitrary limits:
the second functor does since it is the Yoneda-embedding, the first functor does because 
for every $\W_1,..., \W_\m $ for $\m \geq 0$ the functor
$\Mul_\mW(\W_1,..., \W_\m,-): \mW \to \widehat{\mS}$ preserves limits as it factors as
$\mW(\W_\m,-) \circ \R\Mor_\mW(\W_{\m-1},-) \circ ... \circ \R\Mor_\mW(\W_1,-): \mW \to \widehat{\mS}.$
Thus $\Lambda $ is closed in $\Theta$ under small colimits.

By Theorem \ref{expli} for every $\X \in \mM$ there is a $\widehat{\mP}\Env(\mW)$-enriched adjunction $ \mF_\X: \widehat{\mP}\B\Env(\mN)^\circledast \rightleftarrows \Theta^\circledast: \ev_\X$, where the right adjoint evaluates at $\X$ and the left adjoints sends $\Z \in \widehat{\mP}\B\Env(\mN)$ to $\L\Mor_\mM(\X,-)\ot \Z.$ 
Since $\mM$ is locally small and the enriched embedding $ \mP\B\Env(\mN)^\circledast \subset  \widehat{\mP}\B\Env(\mN)^\circledast$ is linear, the latter adjunction restricts to a $\mP\Env(\mW)$-enriched adjunction $ \mP\B\Env(\mN)^\circledast \rightleftarrows \Theta'^\circledast,$ where the right adjoint evaluates at $\X$.
Since $\mM$ is small, by Corollary \ref{explicas} the $\infty$-category $ \Theta'$ is generated under small colimits by the objects $\mF_\X(\Z)$ for $\X \in \mM, \Z \in \mP\B\Env(\mN).$
Because $\mP\B\Env(\mN)$ is by definition generated under small colimits by $\B\Env(\mN)$, the $\infty$-category $ \Theta'$ is generated under small colimits by the objects $\mF_\X(\V_1 \ot ... \ot \V_\n \ot \Z \ot \W_1 \ot ... \ot \W_\m)$ for $\X \in \mM, \Z \in \mN, \V_1,...,\V_\n \in \mV, \W_1,...,\W_\m \in \mW$ for $\n, \m \geq 0.$
Since the embedding $\Theta' \subset \Theta$ preserves small colimits, $\Theta'$ is the smallest full subcategory of $\Theta$ closed under small colimits containing the objects $\mF_\X(\V_1 \ot ... \ot \V_\n \ot \Z \ot \W_1 \ot ... \ot \W_\m)$ for $\X \in \mM, \Z \in \mN, \V_1,...,\V_\n \in \mV, \W_1,...,\W_\m \in \mW$ for $\n, \m \geq 0.$
Consequently, $\Theta' \subset \Lambda$ if these objects belong to $\Lambda,$ i.e. for every $\X \in \mM, \Z \in \mN, \V_1,...,\V_\n \in \mV, \W_1,...,\W_\m \in \mW$ for $\n, \m \geq 0$ the morphism object $$\R\Mor_{\Theta}(\mF_\X(\V_1 \ot ... \ot \V_\n \ot \Z \ot \W_1 \ot ... \ot \W_\m),\iota \circ \rH) \simeq $$$$\R\Mor_{\widehat{\mP}\B\Env(\mN)}(\V_1 \ot ... \ot \V_\n \ot \Z \ot \W_1 \ot ... \ot \W_\m,\rH(\X))$$ belongs to $\mW.$ 
We first reduce to the case that $\m=0.$
For that we note that the canonical embedding $\bj: \mW^\circledast \to \widehat{\mP}\Env(\mW)^\circledast$ preserves right morphism objects since 
$\mW$ is right enriched in itself: for every $\W_1,...\W_\m,\X,\Y \in \mW$ for $\m \geq 0$ the induced map $$ \widehat{\mP}\Env(\mW)(\bj(\W_1) \ot ... \ot \bj(\W_\m,), \bj(\R\Mor_\mW(\X,\Y)))\to $$$$\widehat{\mP}\Env(\mW)(\bj(\W_1) \ot ... \ot \bj(\W_\m), \R\Mor_{\mP\Env(\mW)}(\bj(\X),\bj(\Y)))$$ identifies with the canonical equivalence $$\Mul_{\widehat{\mP}\Env(\mW)}(\bj(\W_1), ... , \bj(\W_\m); \bj(\R\Mor_\mW(\X,\Y)))\simeq \Mul_{\mW}(\W_1, ... , \W_\m ; \R\Mor_\mW(\X,\Y))$$$$ \simeq \Mul_{\mW}(\W_1, ... , \W_\m, \X; \Y)\simeq \Mul_{\widehat{\mP}\Env(\mW)}(\bj(\W_1), ... , \bj(\W_\m), \bj(\X); \bj(\Y)).$$
The morphism object $ \R\Mor_{\widehat{\mP}\B\Env(\mN)}(\V_1 \ot ... \ot \V_\n \ot \Z \ot \W_1 \ot ... \ot \W_\m,\rH(\X))$
is the image of 
$$ \R\Mor_{\widehat{\mP}\B\Env(\mN)}(\V_1 \ot ... \ot \V_\n \ot \iota(\Z),\rH(\X))$$ under the functor
$$\R\Mor_{\widehat{\mP}\Env(\mW)}(\bj(\W_\m),-) \circ ... \circ \R\Mor_{\widehat{\mP}\Env(\mW)}(\bj(\W_{1}),-): \widehat{\mP}\Env(\mW) \to \widehat{\mP}\Env(\mW).$$ 
Since $\bj: \mW^\circledast \to \widehat{\mP}\Env(\mW)^\circledast$ preserves right morphism objects, we can assume that $\m=0.$ In this case we like to see that $ \R\Mor_{\widehat{\mP}\B\Env(\mN)}(\V_1 \ot ... \ot \V_\n \ot \Z,\rH(\X))$ belongs to $\mW$.
This follows from the following equivalence: 
$$\R\Mor_{\widehat{\mP}\B\Env(\mN)}(\V_1 \ot ... \ot \V_\n \ot \Z,\rH(\X)) \simeq \R\Mul\Mor_{\widehat{\mP}\B\Env(\mN)}(\V_1, ..., \V_\n, \Z,\rH(\X)) $$$$ \simeq \R\Mul\Mor_{\mN}(\V_1, ... , \V_\n, \Z,\rH(\X)).$$

\end{proof} 

\begin{corollary}
Let $\kappa$ be a regular cardinal, $\mV^\otimes \to \Ass$ an $\infty$-operad, $ \mW^\ot \to \Ass$ a monoidal $\infty$-category compatible with $\kappa$-small colimits, $\mM^\circledast \to \mV^\ot$ a weakly left enriched $\infty$-category and $\mN^\circledast \to \mV^\ot \times \mW^\ot$ a right $\kappa$-enriched $\infty$-category. Then $\Enr\Fun_{\mV, \emptyset}(\mM, {\mN})^\circledast \to \mW^\ot $ is right $\kappa$-enriched.
	
\end{corollary}

\begin{corollary}\label{ps} Let $\mV^\ot \to \Ass,\mW^\ot \to \Ass$ be small monoidal $\infty$-categories. The functor $$
{_\mV\P\L\Enr}_\emptyset \times {_\emptyset \P\R\Enr}_{\mW} \to {_\mV\P\B\Enr_{\mW}} \ (\mM^\circledast \to \mV^\ot, \mN^\circledast \to \mW^\ot) \mapsto \mM^\circledast \times \mN^\circledast \to \mV^\ot \times \mW^\ot $$
admits component-wise right adjoints.
\end{corollary}

\begin{construction}
Let $\mV^\ot \to \Ass, \mW^\ot \to \Ass$ be monoidal $\infty$-categories
and $\mu: \mV^\ot \times_\Ass \mW^\ot \to \mW^\ot$ a monoidal functor.
By Proposition \ref{biii} the pullback $\mu^\ast(\mN)^\circledast \to \mV^\ot \times_\Ass \mW^\ot $ along $\mu$ is the pullback of a unique bipseudo-enriched $\infty$-category $\mN_\mu^\circledast \to \mV^\ot\times (\mW^\rev)^\ot$.
Let $$\widetilde{\mN}_\mu^\circledast \to \mW^\ot \times (\mV^\rev)^\ot$$ be the corresponding bipseudo-enriched $\infty$-category via Notation \ref{invo}.
\end{construction}
\begin{remark}
If for every $\V \in \mV$ the functor $\mu(\V,-): \mW \to \mW $ admits a right adjoint $\Gamma_\V$ and $\mN^\circledast \to \mW^\ot$ exhibits $\mN$ as right enriched in $\mW$, then $\mN_\mu^\circledast \to \mV^\ot\times (\mW^\rev)^\ot$ exhibits $\mN$ as right enriched in $\mW^\rev$, where the right multi-morphism object for $\V_1,...,\V_\n \in \mV, \X,\Y \in \mN$ is $$\R\Mul\Mor_{\mN_\mu}(\V_1,...,\V_\n, \X,\Y) =\Gamma_{\V_1 \ot ... \ot \V_\n}( \R\Mor_\mN(\X,\Y)).$$

If for every $\W \in \mW$ the functor $\mu(-,\W): \mV \to \mW $ admits a right adjoint $\Xi_\W$ and $\mN^\circledast \to \mW^\ot$ exhibits $\mN$ as right enriched in $\mW$, then $\mN_\mu^\circledast \to \mV^\ot\times (\mW^\rev)^\ot$ exhibits $\mN$ as left enriched in $\mV$, where the left multi-morphism object for $\W_1,...,\W_\m \in \mW, \X,\Y \in \mN$ is $$\L\Mul\Mor_{\mN_\mu}(\X,\W_1,...,\W_\m,\Y) =\Xi_{\W_1 \ot ... \ot \W_\m}( \R\Mor_\mN(\X,\Y)).$$

\end{remark}
\begin{notation}\label{Arat}
Let $\mV^\ot \to \Ass, \mW^\ot \to \Ass$ be monoidal $\infty$-categories
and $\mu: \mV^\ot \times_\Ass \mW^\ot \to \mW^\ot$ a monoidal functor.
Let $\mN^\circledast \to \mW^\ot$ be a left pseudo-enriched $\infty$-category and $\mM^\circledast \to \mV^\ot, \mO^\circledast \to \mW^\ot$ weakly left enriched $\infty$-categories.

\begin{enumerate}
\item Let $$\Enr\Fun_\mV(\mM,\mN)^\circledast \to \mW^\ot $$ be the weakly left enriched $\infty$-category corresponding to the weakly right enriched $\infty$-category $$\Enr\Fun_{\mV, \emptyset}(\mM,\mN_\mu)^\circledast \to (\mW^\rev)^\ot.$$

\item Let $$\Enr\Fun_\mW(\mO,\mN)^\circledast \to \mV^\ot $$ be the weakly left enriched $\infty$-category corresponding to the weakly right enriched $\infty$-category $$\Enr\Fun_{\mW, \emptyset}(\mO,\widetilde{\mN}_\mu)^\circledast \to (\mV^\rev)^\ot. $$
\end{enumerate}

\end{notation}

Theorem \ref{psinho} has the following consequences: 
\begin{corollary}\label{psinhosp}
Let $\kappa$ be a small regular cardinal, 
$\mV^\ot \to \Ass, \mW^\ot \to \Ass$ monoidal $\infty$-categories compatible with $\kappa$-small colimits and $\mu: \mV^\ot \times_\Ass \mW^\ot \to \mW^\ot$ a monoidal functor preserving $\kappa$-small colimits component-wise.
Let $\mN^\circledast \to \mW^\ot$ be a left $\kappa$-enriched $\infty$-category and $\mM^\circledast \to \mV^\ot, \mO^\circledast \to \mW^\ot$ weakly left enriched $\infty$-categories.
The following weakly left enriched $\infty$-categories are left $\kappa$-enriched:
$$\Enr\Fun_{\mV}(\mM, {\mN})^\circledast \to \mW^\ot, \ \Enr\Fun_\mW(\mO,\mN)^\circledast \to \mV^\ot.$$

\end{corollary}

\begin{corollary}\label{psinhosp2}
Let $\mV^\ot \to \Ass, \mW^\ot \to \Ass$ be monoidal $\infty$-categories, $\mu: \mV^\ot \times_\Ass \mW^\ot \to \mW^\ot$ a monoidal functor, $\mM^\circledast \to \mV^\ot, \mO^\circledast \to \mW^\ot$ small weakly left enriched $\infty$-categories and $\mN^\circledast \to \mW^\ot$ a left enriched $\infty$-category. 

\begin{enumerate}
\item If $\mW$ admits small limits, the monoidal structure on $\mW$ is closed and for every $\mV \in \mV$ the functor $\mu(\V,-): \mW \to \mW$ admits a right adjoint, the weakly left enriched $\infty$-category		
$$\Enr\Fun_{\mV}(\mM, {\mN})^\circledast \to \mW^\ot$$
is a left enriched $\infty$-category.

\item If $\mV$ admits small limits, the monoidal structure on $\mV$ is closed and for every $\mW \in \mW$ the functor $\mu(-,\W): \mV \to \mW$ admits a right adjoint, the weakly left enriched $\infty$-category		
$$\Enr\Fun_\mW(\mO,\mN)^\circledast \to \mV^\ot $$ is a left enriched $\infty$-category.	
\end{enumerate}
\end{corollary}

\begin{proposition}\label{clof}
Let $\mV^\ot \to \Ass, \mW^\ot \to \Ass$ be monoidal $\infty$-categories
and $\mu: \mV^\ot \times_\Ass \mW^\ot \to \mW^\ot$ a monoidal functor.
Let $\mN^\circledast \to \mW^\ot$ be a left pseudo-enriched $\infty$-category and $\mM^\circledast \to \mV^\ot, \mO^\circledast \to \mW^\ot$ weakly left enriched $\infty$-categories.	
There are canonical equivalences
$$\Enr\Fun_{\mW, \emptyset}(\mO, \Enr\Fun_{\mV}(\mM, \mN)) \simeq \Enr\Fun_{\mW, \emptyset}(\mu_!(\mM \times \mO), \mN) \simeq \Enr\Fun_{\mV, \emptyset}(\mM, \Enr\Fun_{\mW}(\mO, \mN)).$$

\end{proposition}

\begin{proof}

There is a canonical equivalence
$$\Enr\Fun_{\mW, \emptyset}(\mu_!(\mM \times \mO), \mN) \simeq \Enr\Fun_{\mV \times \mW, \emptyset}(\mM \times \mO, \mu^\ast(\mN))\simeq \Enr\Fun_{\mV, \mW^\rev}(\mM \times \mO^\rev, \mN_\mu).$$

There are canonical equivalences:
$$\Enr\Fun_{\mW, \emptyset}(\mO, \Enr\Fun_{\mV}(\mM, \mN)) \simeq \Enr\Fun_{\emptyset, \mW^\rev}(\mO^\rev, \Enr\Fun_{\mV, \emptyset}(\mM, \mN_\mu)) \simeq $$$$ \Enr\Fun_{\mV, \mW^\rev}(\mM \times \mO^\rev, \mN_\mu),$$		
$$\Enr\Fun_{\mV, \emptyset}(\mM, \Enr\Fun_{\mW}(\mO, \mN)) \simeq \Enr\Fun_{\emptyset, \mV^\rev}(\mM^\rev, \Enr\Fun_{\mW, \emptyset}(\mO, \widetilde{\mN}_\mu)) \simeq $$$$ \Enr\Fun_{\mW, \mV^\rev}(\mO \times \mM^\rev, \widetilde{\mN}_\mu) \simeq \Enr\Fun_{\mV, \mW^\rev}(\mM \times \mO^\rev, \mN_\mu).$$		

\end{proof}

Next we recall the construction of the tensor product of enriched $\infty$-categories
\cite{GEPNER2015575}, \cite[Corollary 5.7.12.]{haugseng2023tensor}.
For that we will use braided monoidal $\infty$-categories and left actions of a braided monoidal $\infty$-category on a monoidal $\infty$-category. In the following we use Definition 
\ref{monob}.

\begin{definition}
A braided monoidal $\infty$-category is a cocartesian fibration $\mV^\boxtimes \to \Ass \times \Ass$ whose pullbacks along the functors $\Ass \times \{[\n] \} \to \Ass \times \Ass, \{[\n]\}\times \Ass \to \Ass \times \Ass$ for every $\n \geq 0$ are monoidal $\infty$-categories.
	
\end{definition}

\begin{notation}
We write $\mV^\ot \to \Ass$ for the pullback of a braided monoidal $\infty$-category $\mV^\boxtimes \to \Ass \times \Ass$ along the functor $\Ass \times \{[1]\} \subset \Ass \times \Ass $. Then $\mV^\ot \to \Ass$ is a monoidal $\infty$-category, which agrees with the pullback along the functor $ \{[1]\} \times \Ass \subset \Ass \times \Ass $.
\end{notation}

A braided monoidal $\infty$-category $\mV^\boxtimes \to \Ass \times \Ass$
classifies a monoid object $\Ass \to \Mon$ (Definition \ref{monob}) corresponding to an associative algebra $\Ass \to \Mon^\times$ for the cartesian structure
by Remark \ref{carst}.

\begin{definition}
A monoidal left action of a braided monoidal $\infty$-category $\mV^\boxtimes \to \Ass \times \Ass$ on a monoidal $\infty$-category $\mW^\ot \to \Ass$ is a map $\mW^\star \to \mV^\boxtimes$ 
of cocartesian fibrations over $\Ass \times \Ass$ that classifies a left action object $\Ass \times [1] \to \Mon$ (Definition \ref{monob}). 
\end{definition}

\begin{remark}Let $\mV^\boxtimes \to \Ass \times \Ass$ be a braided monoidal $\infty$-category
corresponding to an associative algebra $\beta: \Ass \to \Mon^\times$.
A monoidal left action of $\mV^\boxtimes \to \Ass \times \Ass$ on a monoidal $\infty$-category $\mW^\ot \to \Ass$ corresponds to a map $ \alpha \to \beta$ of functors $ \Ass \to \Mon^\times$ that lies over the map $(-)^{\triangleright}\to \id $ of functors $ \Ass \to \Ass$ such that $\alpha$ is a map of cocartesian fibrations relative to the collection of inert morphisms preserving the minimum (Remark \ref{carst}).
\end{remark}


The canonical functor $\L\Enr_\emptyset \to \Op_\infty$ preserves finite products and so gives rise to a monoidal functor on cartesian structures $\L\Enr_\emptyset^\times \to \Op_\infty^\times$.
\begin{proposition}\label{tenpro}
	
The monoidal functor $\L\Enr_\emptyset^\times \to \Op_\infty^\times$ is a cocartesian fibration.
	
\end{proposition}

\begin{proof}
By Theorem \ref{bica} the functor $\phi: \L\Enr_\emptyset \to \Op_\infty$ is a cocartesian fibration.
So it is enough to see that the product of two $\phi$-cocartesian morphisms in $\L\Enr_\emptyset $ is again $\phi$-cocartesian. This follows immediately from the description of $\phi$-cocartesian morphsisms of Theorem \ref{bica}.
	
\end{proof}

\begin{notation}Let $\mW^\ot \to \Ass$ be a monoidal $\infty$-category equipped with a monoidal left action of a braided monoidal $\infty$-category $\mV^\boxtimes \to \Ass\times \Ass$.
	
\begin{enumerate}
\item Then $\mV^\boxtimes \to \Ass\times \Ass$ corresponds to an associative algebra $\beta: \Ass \to \Mon^\times$. 
Let $$_\mV\L\Enr_\emptyset^\ot \to \Ass$$ be the pullback of the cocartesian fibration
$\L\Enr_\emptyset^\times \to \Op_\infty^\times$ along $\beta: \Ass \to \Mon^\times \subset \Op_\infty^\times,$	which is a monoidal $\infty$-category.

\item The left action of $ \mV^\boxtimes \to \Ass\times\Ass$ on $\mW^\ot \to \Ass$ in $\Mon $ corresponds to a map $\sigma: \alpha\to \beta$ of functors $ \Ass \to \Mon^\times$ that lies over the map $(-)^{\triangleright}\to \id $ of functors $ \Ass \to \Ass$ such that 
$\alpha$ is a map of cocartesian fibrations relative to the collection of inert morphisms preserving the minimum.
Let $$_\mW\L\Enr_\emptyset^\circledast \to {_\mV\L\Enr_\emptyset^\ot} $$ be the functor over $\Ass$, which is a left tensored $\infty$-category, classified by the pullback $([1]\times \Ass) \times_{\Op_\infty^\times} \L\Enr_\emptyset^\times\to [1]\times \Ass $ along $\sigma$, which is a map of cocartesian fibrations $ $ over $[1].$

\end{enumerate}	
\end{notation}

The tensor product and left action are defined as in the following notation:

\begin{notation}\label{opot}Let $\mW^\ot \to \Ass$ be a monoidal $\infty$-category equipped with a monoidal left action of a braided monoidal $\infty$-category $\mV^\boxtimes \to \Ass \times \Ass$.
Let $\mM^\circledast \to \mV^\ot, \mN^\circledast \to \mV^\ot, \mO^\circledast \to \mW^\ot$ be small left enriched $\infty$-categories.

\begin{enumerate}
\item Let $(\mM\ot\mN)^\circledast \to \mV^\ot$ be the tranfer of enrichment 
of the left $\mV \times \mV$-enriched $\infty$-category
$\mM^\circledast \times_\Ass \mN^\circledast \to \mV^\ot \times_\Ass \mV^\ot$
along the monoidal tensor product functor $ \mV^\ot \times_\Ass \mV^\ot \to \mV^\ot$.

\item Let $(\mM\ot\mO)^\circledast \to \mW^\ot$ be the tranfer of enrichment 
of the left $\mV \times \mW$-enriched $\infty$-category
$\mM^\circledast \times_\Ass \mN^\circledast \to \mV^\ot \times_\Ass \mW^\ot$
along the monoidal left action functor $\mV^\ot \times_\Ass \mW^\ot \to \mW^\ot$.
\end{enumerate}

\end{notation}

\begin{theorem}\label{cloff}Let $\mW^\ot \to \Ass$ be a closed monoidal $\infty$-category equipped with a monoidal left action of a closed braided monoidal $\infty$-category $\mV^\boxtimes \to \Ass \times \Ass$ such that $\mV,\mW$ admit small limits
and the tensor product $\mV \times \mV \to \mV$ and left action $\mV \times \mW \to \mW$ admit component-wise right adjoints.

\begin{enumerate}
\item Let $\mM^\circledast \to \mV^\ot, \mN^\circledast \to \mV^\ot, \mO^\circledast \to \mV^\ot$ be small left enriched $\infty$-categories.	
The small weakly left enriched $\infty$-categories $\Enr\Fun_{\mV}(\mM, \mN)^\circledast \to \mV^\ot, \Enr\Fun_{\mV}(\mO, \mN)^\circledast \to \mV^\ot$
are left enriched $\infty$-categories and there are canonical equivalences
$$\hspace{12mm}\Enr\Fun_{\mV, \emptyset}(\mO, \Enr\Fun_{\mV}(\mM, \mN)) \simeq \Enr\Fun_{\mW, \emptyset}(\mM \otimes \mO, \mN) \simeq \Enr\Fun_{\mV, \emptyset}(\mM, \Enr\Fun_{\mV}(\mO, \mN)).$$

\item Let $\mM^\circledast \to \mV^\ot, \mN^\circledast \to \mV^\ot, \mO^\circledast \to \mW^\ot$ be small left enriched $\infty$-categories.	
The small weakly left enriched $\infty$-categories $\Enr\Fun_{\mV}(\mM, \mN)^\circledast \to \mW^\ot, \Enr\Fun_{\mW}(\mO, \mN)^\circledast \to \mV^\ot$
are left enriched $\infty$-categories and there are canonical equivalences
$$\hspace{12mm}\Enr\Fun_{\mW, \emptyset}(\mO, \Enr\Fun_{\mV}(\mM, \mN)) \simeq \Enr\Fun_{\mW, \emptyset}(\mM \otimes \mO, \mN) \simeq \Enr\Fun_{\mV, \emptyset}(\mM, \Enr\Fun_{\mW}(\mO, \mN)).$$
\end{enumerate}
	
\end{theorem}
Every monoidal $\infty$-category $\mV^\ot \to \Ass$ compatible with small colimits carries a canonical monoidal left action of the cartesian braided monoidal structure $\mS^\times \to \Ass \times \Ass$ \cite[Remark 4.8.1.8]{lurie.higheralgebra}.
The monoidal left action functor is the universal monoidal functor $\mu: \mS^\times \times_\Ass \mV^\ot \to (\mS \ot \mV)^\ot \simeq \mV^\ot$ preserving small colimits component-wise.
The latter monoidal left action gives rise to a left action of $\Cat_\infty \simeq {_\mS\L\Enr_\emptyset} $ on $ {_\mV\L\Enr_\emptyset},$ which we use in the following corollary:


\begin{corollary}\label{space}
Let $\mV^\otimes \to \Ass$ be a presentably monoidal $\infty$-category.
Let $\mM^\circledast \to \mV^\ot, \mN^\circledast \to \mV^\ot$ be left $\mV$-enriched $\infty$-categories and $\K \in \Cat_\infty.$
There is a canonical equivalence $$\Fun(\K, \Enr\Fun_{\mV, \emptyset}(\mM,\mN)) \simeq \Enr\Fun_\mV(\mu_!(\K \times \mM),\mN).$$		
\end{corollary}

\begin{proof}By Corollary \ref{uhol} (3) the forgetful functor $_\mS\L\Enr_\emptyset \to \Cat_\infty$ is an equivalence. Let $\K^\circledast \to \mS^\times$ be the unique left enriched $\infty$-category whose underlying $\infty$-category is $\K.$
		
By Proposition \ref{clof} and Corollary \ref{uhol} there is a canonical equivalence $$\Enr\Fun_{\mV, \emptyset}(\mu_!(\K \times \mM),\mN) \simeq \Enr\Fun_{\mS, \emptyset}(\K, \Enr\Fun_\mV(\mM,\mN)) \simeq \Fun(\K, \Enr\Fun_{\mV, \emptyset}(\mM,\mN)).$$		

\end{proof}

\subsection{Enriched profunctors}

\begin{definition}Let $\mV^\ot \to \Ass$ be a monoidal $\infty$-category compatible with small colimits and $\mC^\circledast \to \mV^\ot, \mD^\circledast \to \mV^\ot$ right $\mV$-enriched $\infty$-categories. A $\mV$-enriched profunctor $\mC \to \mD$
is a $\mV,\mV$-enriched functor $ (\mD^\op)^\circledast \times \mC^\circledast \to \mV^\circledast.$	

\end{definition}

\begin{remark}
	
Haugseng \cite{haugseng_2016} defines enriched bimodules, another model for enriched profunctors. Hinich \cite[8.1.]{HINICH2020107129} defines enriched
profunctors under the name enriched correspondences and identifies such with enriched functors to the interval \cite[Proposition 8.3.2.]{HINICH2020107129}. 

\end{remark}

\begin{remark}
Let $\mC^\circledast \to \mV^\ot, \mD^\circledast \to \mV^\ot$ be right enriched $\infty$-categories. A $\mV$-enriched profunctor $\mC \to \mD$
corresponds to a right $\mV$-enriched functor 
$ \mC^\circledast \to \Enr\Fun_{\mV,\emptyset}(\mD^\op, \mV)^\circledast,$		
which corresponds to a right $\mV$-linear small colimits preserving functor
$ \Enr\Fun_{\mV,\emptyset}(\mC^\op, \mV)^\circledast \to \Enr\Fun_{\mV,\emptyset}(\mD^\op, \mV)^\circledast$ by Theorem \ref{unitol}.

\end{remark}

\begin{example}Let $\mC^\circledast \to \mV^\ot$ be a right enriched $\infty$-category.
The $\mV,\mV$-enriched morphism object functor $\L\Mor: (\mC^\op)^\circledast \times \mC^\circledast \to \mV^\circledast $ of Notation \ref{nnoo} is a $\mV$-enriched profunctor $\mC \to \mC$,
which corresponds to the identity of $\Enr\Fun_{\mV,\emptyset}(\mC^\op, \mV)^\circledast$
and to the right $\mV$-enriched Yoneda-embedding $\mC^\circledast \to \Enr\Fun_{\mV,\emptyset}(\mC^\op, \mV)^\circledast.$

\end{example}

\begin{example}\label{relto} Let $\mM^\circledast \to \mV^\ot $ be a right tensored $\infty$-category compatible with small colimits and $\A,\B$ associative algebras in $\mV.$
By \cite[Theorem 4.8.4.1.]{lurie.higheralgebra} and Corollary \ref{corz} there is a canonical left $\mV$-enriched equivalence $$\Enr\Fun_{\emptyset,\mV}(B\A, \mM)^\circledast \simeq \LinFun^\L_{\emptyset,\mV}(\LMod_\A(\mV), \mM)^\circledast \simeq \RMod_\A(\mM)^\circledast.$$
Hence there is a canonical left $\mV$-enriched equivalence $\Enr\Fun_{\emptyset,\mV}(B\A, \mV)^\circledast \simeq \RMod_\A(\mV)^\circledast$ and so dually a right $\mV$-enriched equivalence $\Enr\Fun_{\mV,\emptyset}(B\A^\op, \mV)^\circledast \simeq \LMod_\A(\mV)^\circledast.$

Thus a $\mV$-enriched profunctor $B\A \to B\B$ corresponds to a right $\mV$-linear small colimits preserving functor
$\LMod_\A(\mV)^\circledast \to \LMod_\B(\mV)^\circledast,$ which corresponds to an object of
$ \RMod_\A(\LMod_\B(\mV)) $ identified with an $\A,\B$-bimodule in $\mV$ by Remark \ref{bimo}.	

\end{example}

\begin{definition} Let $\mC^\circledast \to \mV^\ot, \mD^\circledast \to \mV^\ot, \mE^\circledast \to \mV^\ot$ be right enriched $\infty$-categories and 
$\F: \mC \to \mD, \G: \mD \to \mE$ be $\mV$-enriched profunctors corresponding to right $\mV$-linear small colimits preserving functors $\Enr\Fun_{\mV,\emptyset}(\mC^\op, \mV)^\circledast \to \Enr\Fun_{\mV,\emptyset}(\mD^\op, \mV)^\circledast, \Enr\Fun_{\mV,\emptyset}(\mD^\op, \mV)^\circledast \to \Enr\Fun_{\mV,\emptyset}(\mE^\op, \mV)^\circledast.$
The relative tensor product of $\F$ and $\G$, denoted by $\F \ot_\mD \G$, is the $\mV$-enriched profunctor $\mC \to \mE$ corresponding to the right $\mV$-linear small colimits preserving functor $$\Enr\Fun_{\mV,\emptyset}(\mC^\op, \mV)^\circledast \to \Enr\Fun_{\mV,\emptyset}(\mD^\op, \mV)^\circledast \to \Enr\Fun_{\mV,\emptyset}(\mE^\op, \mV)^\circledast.$$
\end{definition}

\begin{example}
There are canonical equivalences $\F \ot_\mD \L\Mor_{\mD} \simeq \F,\ \L\Mor_{\mD} \ot_\mD \G \simeq\G.$	

\end{example}

\begin{corollary}\label{reltu} Let $\mC^\circledast \to \mV^\ot, \mD^\circledast \to \mV^\ot, \mE^\circledast \to \mV^\ot$ be right enriched $\infty$-categories and 
$\F: \mC \to \mD, \G: \mD \to \mE$ be $\mV$-enriched profunctors.
For every $\X \in \mC, \Y \in \mE$ there is a canonical equivalence $$ (\F \ot_\mD \G)(\Y,\X) \simeq $$$$ \underset{[\n]\in \Delta^\op}{\colim}(\colim_{\Z_1, ...., \Z_\n \in \mD^\simeq} \F(\Z_1,\X) \ot\L\Mor_{\mD}(\Z_2,\Z_1) \ot .... \ot \L\Mor_{\mD}(\Z_{\n},\Z_{\n-1})\ot \G(\Y,\Z_{\n}))).$$
\end{corollary}

\begin{proof}By Proposition \ref{mondec} there is a canonical equivalence
$$\F(-,\X) \simeq $$$$\underset{[\n]\in \Delta^\op}{\colim}(\colim_{\Z_1, ...., \Z_\n \in \mD^\simeq} \F(\X,\Z_1) \ot\L\Mor_{\mD}(\Z_2,\Z_1) \ot .... \ot \L\Mor_{\mD}(\Z_{\n},\Z_{\n-1})\ot \L\Mor_{\mD}(-,\Z_{\n}))$$ in $\Enr\Fun_{\mV,\emptyset}(\mD^\op,\mV)$
and so a canonical equivalence in $\Enr\Fun_{\mV,\emptyset}(\mE^\op,\mV):$
$$ (\F \ot_\mD \G)(-,\X) \simeq \bar{\G}(\F(,-\X)) \simeq $$$$ \underset{[\n]\in \Delta^\op}{\colim} (\colim_{\Z_1, ...., \Z_\n \in \mD^\simeq} \F(\Z_1,\X) \ot\L\Mor_{\mD}(\Z_2,\Z_1) \ot .... \ot \L\Mor_{\mD}(\Z_{\n},\Z_{\n-1})\ot \bar{\G}(\L\Mor_{\mD}(-,\Z_{\n}))) $$$$ \simeq  \underset{[\n]\in \Delta^\op}{\colim}(\colim_{\Z_1, ...., \Z_\n \in \mD^\simeq} \F(\Z_1,\X) \ot\L\Mor_{\mD}(\Z_2,\Z_1) \ot .... \ot \L\Mor_{\mD}(\Z_{\n},\Z_{\n-1})\ot \G(-,\Z_{\n}))).$$
\end{proof}

\subsection{An end formula for morphism objects of enriched functor $\infty$-categories}

In the following we describe morphism objects in the enriched $\infty$-category of enriched functors as an enriched end (Theorem \ref{end}).

In the following we will use symmetric monoidal $\infty$-categories \cite[Definition 2.0.0.7.]{lurie.higheralgebra}, which are cocartesian fibrations
$\mV^\boxtimes \to \Comm$ to the category $\Comm$ of finite pointed sets
such that for every $\n \geq 0$ the canonical induced functor $\mV^\boxtimes_{\{1,...,\n,*\}} \to \mV^{\times\n}$ is an equivalence, where $ \mV:=\mV^\boxtimes_{\{1,*\}}.$
There is a canonical functor $\theta: \Ass \to \Comm, [\n]\mapsto \{1,...,\n,*\} $
\cite[Construction 4.1.2.9.]{lurie.higheralgebra} and we write $\mV^\ot \to \Ass$ for the pullback of a symmetric monoidal $\infty$-category $\mV^\boxtimes \to \Comm$ along $\theta,$ which is a monoidal $\infty$-category.

If $\mV^\ot \to \Ass$ is the pullback of a symmetric monoidal $\infty$-category $\mV^\boxtimes \to \Comm$, there is a canonical monoidal equivalence
$(\mV^\rev)^\ot \simeq \mV^\ot$ since $\theta$ factors as $ \Ass \xrightarrow{(-)^\op} \Ass \xrightarrow{\theta}\Comm$.
Consequently, there is no need to distinguish between (weakly) left and right $\mV$-(pseudo)-enriched $\infty$-categories, which we therefore call (weakly) $\mV$-(pseudo)-enriched $\infty$-categories. 
Moreover there is no need to distinguish between left and right morphism objects, which we call morphism objects, and between left and right (co)tensors, which we call (co)tensors.

\begin{notation}Let $\mV^\boxtimes \to \Comm$ be a symmetric monoidal $\infty$-category and $\mM^\circledast \to \mV^\ot $ a right $\mV$-enriched $\infty$-category. Let $$\Mor_\mM: (\mM^\op  \otimes \mM)^\circledast \to \mV^\circledast$$
be the left $\mV$-enriched functor corresponding to the left $\mV \times \mV$-enriched functor $ (\mM^\op)^\circledast \times_\Ass \mM^\circledast \to \ot_\mV^*(\mV)^\circledast$ corresponding to the $\mV,\mV$-enriched functor $\R\Mor_\mM: (\mM^\op)^\circledast \times \mM^\circledast \to \mV^\circledast$ of Notation \ref{nnoo}. 
\end{notation}



\begin{definition}\label{Aho} Let $\mV^\ot \to \Comm$ be a symmetric monoidal $\infty$-category, $\mM^\circledast \to \mV^\ot, \mJ^\circledast \to \mV^\ot$ pseudo-enriched $\infty$-categories and $\F: (\mJ^\op\ot\mJ)^\circledast \to \mM^\circledast$ a $\mV$-enriched functor.
	
\begin{itemize}
\item The $\mV$-enriched end of $\F$, denoted by $\int_\mJ \F$, is the object representing the $\mV$-enriched presheaf
$$\mM^\op \to \mV, \X \mapsto \Mor_{\Enr\Fun_{\mV}(\mJ^\op \otimes \mJ,\mV)}(\Mor_\mJ, (\Mor_{\mM}(\X,-)) \circ \F).$$ 

\item The $\mV$-enriched coend of $\F$, denoted by $\int^\mJ \F$, is the object corepresenting the $\mV$-enriched functor
$$\mM \to \mV, \X \mapsto \Mor_{\Enr\Fun_{\mV}(\mJ \otimes \mJ^\op,\mV)}(\Mor_{\mJ^\op}, \Mor_{\mM}(-,\X) \circ \F^\op).$$ 	
\end{itemize}	
	
\end{definition}

\begin{remark}
By universal property the $\mV$-enriched end of $\F$ is the $\mV$-enriched coend of $\F^\op$.
	
\end{remark}

\begin{remark}\label{ioo}
Let $\mV^\ot \to \Comm$ be a symmetric monoidal $\infty$-category, $\mJ^\circledast \to \mV^\ot $ a pseudo-enriched $\infty$-category and $\F: (\mJ^\op\ot\mJ)^\circledast \to \mV^\circledast$ a $\mV$-enriched functor.
By Definition \ref{Aho} there is an equivalence
$$ \int_\mJ \F \simeq \Mor_{\Enr\Fun_{\mV}(\mJ^\op \otimes \mJ,\mV)}(\Mor_\mJ, \F).$$
\end{remark}

\begin{remark}
	
Let $\mV^\ot \to \Comm$ be a symmetric monoidal $\infty$-category, $\mM^\circledast \to \mV^\ot, \mJ^\circledast \to \mV^\ot$ pseudo-enriched $\infty$-categories and $\F: (\mJ^\op\ot\mJ)^\circledast \to \mM^\circledast$ a $\mV$-enriched functor.
A $\mV$-enriched functor $\phi: \mM^\circledast \to \mN^\circledast$ that admits a $\mV$-enriched right adjoint $\gamma$ preserves $\mV$-enriched coends: the canonical morphism $\int^\mJ \phi \circ \F \to \phi(\int_\mJ\F)$
is corepresented by the canonical equivalence
$$ \Mor_{\Enr\Fun_{\mV}(\mJ \otimes \mJ^\op,\mV)}(\Mor_{\mJ^\op}, \Mor_{\mM}(-,\gamma(\X)) \circ \F^\op) \to $$$$ \Mor_{\Enr\Fun_{\mV}(\mJ \otimes \mJ^\op,\mV)}(\Mor_{\mJ^\op}, \Mor_{\mN}(-,\X) \circ \phi^\op \circ \F^\op).$$
	
\end{remark}

\begin{notation}
Let $\mJ$ be an $\infty$-category. Let $\q: \Tw(\mJ)\to \mJ^\op\times \mJ$ be the left fibration classifying the mapping space functor $\mJ(-,-):\mJ^\op\times\mJ \to \mS.$	
	
\end{notation}

\begin{lemma}
Let $\mM^\circledast \to \mS^\times$ be an enriched $\infty$-category.
The $\mS$-enriched end of a $\mS$-enriched functor
$\F: (\mJ^\op)^\circledast\times_{\Ass}\mJ^\circledast\simeq (\mJ^\op \otimes \mJ)^\circledast \to \mM^\circledast$ is the limit of the functor
$$ \Tw(\mJ) \xrightarrow{\q} \mJ^\op \times \mJ \xrightarrow{\F} \mM.$$
	
	
\end{lemma}

\begin{proof}
For every $\X \in \mM$ let $\mX \to \mJ^\op \times\mJ$ be the left fibration classifying the functor $\mM(\X,-)\circ \F: \mJ^\op\times\mJ \to \mS$.
By Remark \ref{ioo} there is a canonical equivalence
$$ \mM(\X,\int_\mJ \F) \simeq  \Fun(\mJ^\op\times\mJ,\mS)(\mJ(-,-),\mM(\X,\F(-))) \simeq \Fun_{\mJ^\op \times \mJ}(\Tw(\mJ),\mX) \simeq $$$$\Fun_{\Tw(\mJ)}(\Tw(\mJ),\Tw(\mJ) \times_{\mJ^\op \times \mJ} \mX) \simeq \lim(\mM(\X,-)\circ \F \circ \q) \simeq \mM(\X,\lim(\F \circ \q))$$
representing an equivalence
$ \int_\mJ \F \simeq \lim(\F \circ \q).$
	
	
\end{proof}

\begin{theorem}\label{end}Let $\mV^\ot \to \Comm$ be a presentably symmetric monoidal $\infty$-category, $\mM^\circledast \to \mV^\ot$ a small enriched $\infty$-category and $ \mJ^\circledast \to \mV^\ot $ a small weakly enriched $\infty$-category. Let $\F,\G: \mJ^\circledast \to \mM^\circledast$ be $\mV$-enriched functors.	
Then $\Mor_{\Enr\Fun_{\mV}(\mJ,\mM)}(\F,\G)$
is the $\mV$-enriched end $ \int_\mJ \Mor_\mM \circ (\F^\op \otimes \G).$
	
\end{theorem}

\begin{proof}
We use Remark \ref{ioo}. There is a chain of equivalences:
$$\Mor_{\Enr\Fun_{\mV}(\mJ^\op\otimes\mJ,\mV)}(\Mor_\mJ, \Mor_\mM \circ (\F^\op \otimes \G)) \simeq $$$$
\Mor_{\Enr\Fun_{\mV}(\mJ,\Enr\Fun_{\mV}(\mJ^\op,\mV))}(\rho_\mJ, \Enr\Fun_{\mV}(\F^\op,\mV) \circ\rho_\mM \circ \G) \simeq $$$$ \Mor_{\Enr\Fun_{\mV}(\mJ,\Enr\Fun_{\mV}(\mM^\op,\mV))}(\F_! \circ \rho_\mJ,  \rho_\mM \circ \G)\simeq $$
$$\Mor_{\Enr\Fun_{\mV}(\mJ,\Enr\Fun_{\mV}(\mM^\op,\mV))}(\rho_\mM \circ \F, \rho_\mM \circ \G) \simeq \Mor_{\Enr\Fun_{\mV}(\mJ,\mM)}(\F, \G).$$
The first equivalence is by definition of the $\mV$-enriched functor 
$$\Enr\Fun_{\mV}(\F^\op,\mV): \Enr\Fun_{\mV}(\mM^\op,\mV)^\circledast \to \Enr\Fun_{\mV}(\mJ^\op,\mV)^\circledast $$ of Notation \ref{siew} and Notation \ref{Arat}. The second equivalence is by adjointness. The third equivalence is by Corollary \ref{saewo}.
The fourth equivalence holds by Lemma \ref{faith} and the fact that the enriched Yoneda-embedding is an embedding by Proposition \ref{yofaith}. 

\end{proof}

\begin{corollary}\label{uhoz}
Let $\mV^\ot \to \Comm$ be a presentably symmetric monoidal $\infty$-category, $\mJ^\circledast \to \mV^\ot $ a small enriched $\infty$-category and $\mM^\circledast \to \mV^\ot $ an enriched $\infty$-category that admits small conical colimits and tensors.
Every $\mV$-enriched functor $\F:\mJ^\circledast \to \mM^\circledast$ is the $\mV$-enriched coend of the $\mV$-enriched functor $$\ot \circ (\F \ot \iota_{\mJ^\op}): (\mJ \ot \mJ^\op)^\circledast \xrightarrow{\F \ot \iota_{\mJ^\op}} (\mM \ot \Enr\Fun_\mV(\mJ,\mV))^\circledast \xrightarrow{\ot} \Enr\Fun_\mV(\mJ,\mM)^\circledast,$$
where $\ot$ is the $\mV$-enriched functor
$$(\mM \ot \Enr\Fun_\mV(\mJ,\mV))^\circledast \simeq (\LinFun^\L_\mV(\mV,\mM) \ot \Enr\Fun_\mV(\mJ,\mV))^\circledast\xrightarrow{\circ} \Enr\Fun_\mV(\mJ,\mM)^\circledast.$$

\end{corollary}

\begin{proof}Let $\G:\mJ^\circledast \to \mM^\circledast$ be a $\mV$-enriched functor.
By the enriched Yoneda-lemma (Corollary \ref{explicas}) the composition $$(\mJ^\op \ot \mJ)^\circledast \xrightarrow{(\ot \circ (\F \ot \iota_{\mJ^\op}))^\op} (\Enr\Fun_\mV(\mJ,\mM)^\op)^\circledast \xrightarrow{\Mor_{\Enr\Fun_\mV(\mJ,\mM)}(-,\G)} \mV^\circledast$$
is the $\mV$-enriched functor $ \Mor_\mM \circ (\F^\op \otimes \G).$
Hence by Theorem \ref{end} there is a canonical equivalence $$\Mor_{\Enr\Fun_\mV(\mJ,\mM)}(\F,\G) \simeq \Mor_{\Enr\Fun_{\mV}(\mJ^\op\otimes\mJ,\mV)}(\Mor_\mJ, \Mor_\mM \circ (\F^\op \otimes \G)) \simeq $$$$\Mor_{\Enr\Fun_{\mV}(\mJ^\op\otimes\mJ,\mV)}(\Mor_\mJ, \Mor_{\Enr\Fun_\mV(\mJ,\mM)}(-,\G) \circ (\ot \circ (\F \ot \iota_{\mJ^\op}))^\op).$$
	
	
\end{proof}

\begin{corollary}
Let $\mV^\ot \to \Comm$ be a presentably symmetric monoidal $\infty$-category, $\mJ^\circledast \to \mV^\ot $ a small $\mV$-enriched $\infty$-category, $\mN^\circledast \to \mV^\ot$ a $\mV$-enriched $\infty$-category that admits small conical colimits and tensors and let $\F: (\mJ^\op)^\circledast \to \mV^\circledast $, $\G:\mJ^\circledast \to \mN^\circledast$ be $\mV$-enriched functors. Let $\bar{\G}: \Enr\Fun_{\mV}(\mJ^\op,\mV)^\circledast \to \mN^\circledast$ be the unique $\mV$-enriched left adjoint extension $\bar{\G}: \Enr\Fun_{\mV}(\mJ^\op,\mV)^\circledast \to \mN^\circledast$ 
of $\G$. There is a canonical equivalence in $\mN: $$$ \bar{\G}(\F)
\simeq \int^\mJ \ot \circ (\F \ot \G).$$ 
	
\end{corollary}

\begin{proof}
	
By Corollary \ref{uhoz} there is a canonical equivalence
$\F \simeq \int^\mJ \ot \circ (\F \ot \iota_{\mJ})$.
Since $\bar{\G}$ is a $\mV$-enriched left adjoint, it preserves $\mV$-enriched coends.
Moreover because $\bar{\G}$ preserves tensors, the canonical map
$\ot \circ (\F \ot (\bar{\G} \circ \iota_{\mJ})) \to \bar{\G} \circ (\ot \circ (\F \ot \iota_{\mJ})) $ is an equivalence.
We obtain a canonical equivalence $$\bar{\G}(\F) \simeq \bar{\G}(\int^\mJ \ot \circ (\F \ot \iota_{\mJ})) \simeq \int^\mJ \bar{\G}\circ (\ot \circ (\F \ot \iota_{\mJ}))$$$$ \simeq \int^\mJ \ot \circ (\F \ot (\bar{\G} \circ \iota_{\mJ})) \simeq \int^\mJ \ot \circ (\F \ot \G).$$
	
\end{proof}

\begin{corollary}\label{endrel}
Let $\mV^\ot \to \Comm$ be a presentably symmetric monoidal $\infty$-category, $\mC^\circledast \to \mV^\ot,\mD^\circledast \to \mV^\ot,\mE^\circledast \to \mV^\ot $ small $\mV$-enriched $\infty$-categories and $\F$ a $\mV$-enriched profunctor $\mC \to \mD$ and
$\G$ a $\mV$-enriched profunctor $\mD \to \mE$.
For every $\X \in \mC, \Y \in \mE$ there is a canonical equivalence in $\mV: $$$ (\F \ot_{\mD} \G)(\X,\Y)
\simeq \int^\mJ \ot \circ (\F(-,\X) \ot \G(\Y,-)).$$ 

\end{corollary}

\bibliographystyle{plain}
\bibliography{ma}
	
\end{document}